\documentclass[12pt,twoside]{report}

\usepackage{amsmath}
\usepackage{a4wide, amsthm, amsfonts, amssymb, latexsym, epsfig, cite, fancyhdr}
\usepackage[mathscr]{eucal}
\usepackage{graphicx}
\usepackage[titles]{tocloft}
\usepackage{amsmath,amscd}

\newfont{\hge}{hge scaled 1800}
\newtheorem{theorem} {{\textsf{Theorem}}}[chapter]
\newtheorem{proposition}[theorem]{{\textsf{Proposition}}}
\newtheorem{lemma}[theorem]{{\textsf{Lemma}}}
\newtheorem{corollary}[theorem]{{\textsf{Corollary}}}

\theoremstyle{definition}
\newtheorem{definition}[theorem]{{\textsf{Definition}}}
\newtheorem{remark}[theorem]{{\textsf{Remark}}}
\newtheorem{example}[theorem]{{\textsf{Example}}}

\numberwithin{equation}{chapter}
\newcommand{\bSB}{\boldsymbol{B}}

\begin{document}

\thispagestyle{empty}
\begin{center}
{\LARGE {\bf Contractivity, Complete Contractivity\\
\vspace{.12in}and Curvature inequalities}}
\end{center}
\vspace{1.2in}

\begin{center}
{\large  A Dissertation \\
submitted in partial fulfilment \\
\vspace{0.04in}
of the requirements for the award of the  }\\
\vspace{0.04in}
{\large{ degree of}} \\
\vspace{0.06in}
{\hge{Doctor of Philosophy}} \\
\vspace{0.75in}
{\em by}\\
\vspace{0.05in}
{\large Avijit Pal}\\
\vspace{1.5in}
\includegraphics[width=3cm]{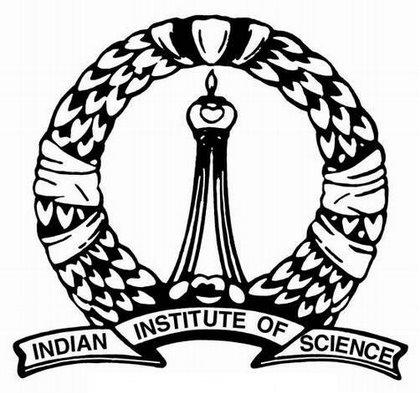}\\
\vspace{0.2in}
{\large{Department of Mathematics}}\\
{\large{Indian Institute of Science}} \\
{\large{Bangalore \,-\, 560012}}\\
{April 2014}
\end{center}

%\newpage

\thispagestyle{empty}
\cleardoublepage
\pagenumbering{roman}

\addcontentsline{toc}{chapter}{Declaration}
\chapter*{Declaration}

\vspace*{.80in}

I hereby declare that the work reported in this thesis is entirely
original and has been carried out by me under the supervision of
Professor Gadadhar Misra at the Department of Mathematics, Indian
Institute of Science, Bangalore. I further declare that this work
has not been the basis for the award of any degree, diploma,
fellowship, associateship or similar title of any University or
Institution. \vspace*{.5in}

\begin{flushright}
Avijit Pal\\
S. R. No.6910-110-081-06284
\end{flushright}
Indian Institute of Science,\\
Bangalore,\\
April, 2014.
\smallskip

\begin{flushright}
 Professor Gadadhar Misra \\
 (Research advisor)
\end{flushright}

\newpage

\thispagestyle{empty}
\cleardoublepage
\thispagestyle{empty}
%\pagenumbering{roman}
\vspace*{8cm}
\begin{center}
\em{DEDICATED TO MY PARENTS, ELDER BROTHER AND MY WIFE ANJALI
 }
\end{center}

\thispagestyle{empty}
\cleardoublepage
\addcontentsline{toc}{chapter}{Acknowledgement}
\chapter*{Acknowledgements}
Foremost, I would like to express my sincere gratitude to my
advisor Professor Gadadhar Misra for the continuous support of my
Ph.D study and research, for his patience, motivation and immense
knowledge. I have been extremely lucky to have his continuous
guidance, active participation and stimulating discussion on
various aspects of operator theory and its connection to different
parts of mathematics. The joy and enthusiasm, he has for his
research was motivational for me, even during tough times in the
Ph.D. I especially thank him for the infinite patience that he has
shown in correcting the mistakes in my writings.\smallskip

\noindent I would like to thank Dr. Cherian Varughese for
cheerfully explaining many obscure points and valuable comments on
my thesis.\smallskip

\noindent I am grateful to all the instructors of the courses I
did at the Department of Mathematics of the Indian Institute of
Science. I learnt a great deal of mathematics from these
courses.\smallskip

\noindent I am deeply indebted to Prof. Michael Dritschel and
Prof. Harald  Upmeier for several illuminating discussions
relating to the topics presented here.\smallskip

\noindent I would like to thank Prof. Dimitry Yakubovich for
giving me valuable comments on initial draft of my thesis.

\noindent I thank my teacher Prof. K. C. Chattopadhaya at Burdwan
university, Burdwan for motivating me to do mathematics and giving
me constant mental support. \smallskip

\noindent I would like to thank Sayan Bagchi for giving valuable
suggestion on my thesis.

\noindent I would also like to thank Sayan Bagchi and Prahllad Deb
for helping me to survive in this isolated place. My sincerest
gratitude also goes to Amit, Dinesh, Santanu, Soma, Sourav,
Subhamay, Tapan, Gururaja, Arpan, Kartick, Ramiz, Rajib, Ratna,
Shibu da, Subrata da, Sayani, Mousumi di, Atryee di and Koushik da
who were an integral part of my life during the last five years at
IISc. I express my sincerest gratitude to my parents, uncle,
brothers, sisters-in-law, sisters, father-in-law, mother-in-law,
grand father and grand mother whose love and affection has given
me the courage to pursue my dreams. Last but not the least, I
thank all my well wishers and all of my friends. I should mention
my wife Anjali who always shared all my disappointments and
frustrations when things did not work.

\noindent A note of thanks goes to IISc, UGC-NET and IFCAM for
providing me financial support.

\thispagestyle{empty}
\cleardoublepage
\addcontentsline{toc}{chapter}{Abstract}
\chapter*{Abstract}
Let $\|\cdot\|_{\mathbf A}$ be a norm on $\mathbb C^m$ given by
the formula $\|(z_1,\ldots,z_m)\|_{\mathbf
A}=\|z_1A_1+\cdots+z_mA_m\|_{\rm op}$ for some choice of an
$m$-tuple of $n\times n$ linearly independent matrices $\mathbf
A=(A_1, \ldots, A_m).$ Let $\Omega_\mathbf A\subset \mathbb C^m$
be the unit ball with respect to the norm $\|\cdot\|_{\mathbf A}.$
%For a holomorphic function $f$ on $\Omega_\mathbf A,$ let
%$\rho_{V}(f):=\left (
%\begin{smallmatrix}
%f(w)I_p& \sum_{i=1}^{m} \partial_if(w)V_{i} \\
%0  & f(w)I_q
%\end{smallmatrix}\right ),$ where $V_1, \ldots, V_m$ are $p\times q$
%matrices.
Given $p\times q$ matrices $V_1, \ldots, V_m$ and a function $f
\in \mathcal O(\Omega_\mathbf A),$ the algebra of function
holomorphic on an open set $U$ containing the closed unit ball
$\bar{\Omega}_\mathbf A$ define $$\rho_{V}(f):=\left (
\begin{smallmatrix}
f(w)I_p& \sum_{i=1}^{m} \partial_if(w)V_{i} \\
0  & f(w)I_q
\end{smallmatrix}\right ),$$ $w\in \Omega_\mathbf A.$ Clearly, $\rho_{V}$ defines an algebra homomorphism. We
study contractivity (resp. complete contractivity) of such
homomorphisms.

%We study homomorphisms ($\rho_{V}(f)=\left (
%\begin{smallmatrix}
%f(w)I_n& \sum_{i=1}^{m} \partial_if(w)V_{i} \\
%0  & f(w)I_n
%\end{smallmatrix}\right ), f \in \mathcal O(\Omega_\mathbf A)$) defined on $\mathcal
%O(\Omega_\mathbf A)$, where  $\Omega_\mathbf A$ is a bounded
%domain of the form:
%\begin{eqnarray*}
%\Omega_\mathbf A & := &\{(z_1 ,z_2, \ldots, z_m) :\|z_1 A_1
%+\cdots + z_mA_m \|_{\rm op} < 1\}
%\end{eqnarray*}
%for some choice of a linearly independent set of $n\times n$
%matrices $\{A_1, \ldots, A_m\}.$  Clearly,
% inherited from $\mathcal M(\mathbb C).$
The homomorphism $\rho_{V}$ induces a linear map $L_{V}:(\mathbb
C^m,\|\cdot\|^*_{\mathbf A})\to \mathcal M_{p \times q}(\mathbb
C),$
$$L_{V}(w)=w_1V_1+\cdots+w_mV_m.$$ The contractivity (resp. complete contractivity) of the
homomorphism $\rho_{V}$ determines the contractivity (resp.
complete contractivity) of the linear map $L_{V}$ and vice-versa.
It is known that contractive homomorphisms of the disc and the
bi-disc algebra are completely contractive, thanks to the dilation
theorems of B. Sz.-Nagy and Ando respectively.  However,  examples
of contractive homomorphisms $\rho_{V}$ of the (Euclidean) ball
algebra which are not completely contractive was given by G. Misra.

From the work of V. Paulsen and E. Ricard, it follows that if
$m\geq 3$ and $\mathbb B$ is any ball in $\mathbb C^m$ with
respect to some norm, say $\|\cdot\|_{\mathbb B},$ then there
exists a contractive linear map $L:(\mathbb
C^m,\|\cdot\|^*_{\mathbb B})\to \mathcal B(\mathcal H)$ which is
not complete contractive. The characterization of those balls in
$\mathbb C^2$ for which contractive linear maps are always
completely contractive remained open. We answer this question for
balls of the form $\Omega_\mathbf A$ in $\mathbb C^2.$

The class of homomorphisms of the form $\rho_V$ arise from
localization of  operators in the Cowen-Douglas class of $\Omega.$
The (complete) contractivity of a homomorphism in this class
naturally produces inequalities for the curvature of the
corresponding Cowen-Douglas bundle. This connection and some of
its very interesting consequences are discussed.

\thispagestyle{empty}
\cleardoublepage
%\addcontentsline{toc}{chapter}{Contents}
\tableofcontents
\cleardoublepage

\pagestyle{fancy}
\renewcommand{\chaptermark}[1]{\markboth{\textsl{\thechapter.\ #1}}{}}
\renewcommand{\sectionmark}[1]{\markright{\textsl{\thesection.\ #1}}}
\fancyhf{}
\renewcommand{\headrulewidth}{.05pt}
%\fancyfoot[CE,CO]{\thepage}
%\fancyhead[CO]{\nouppercase{\leftmark}} \fancyhead[CE]{\rightmark}
\fancyhead[LE]{\thepage}
\fancyhead[RO]{\thepage}
\fancyhead[RE]{\leftmark}
\fancyhead[LO]{\rightmark}

\fancypagestyle{plain}{% this is for beginning of chapter: empty'''''''''
\fancyhead{} % get rid of headers
\renewcommand{\headrulewidth}{0pt} % and the line
 }

%\renewcommand{\headrulewidth}{.1pt}
%
%\newpage
%\mbox{} \thispagestyle{empty}
\pagenumbering{arabic}

\chapter{{Introduction}} \label{C:Intro}
In 1936 von Neumann (see \cite[Chapter 1, Corollary 1.2]{paulsen})
proved that if $T$ is a bounded linear operator on a separable
complex Hilbert space $\mathcal H,$ then
$$\|p(T)\|\leq \|p\|_{\infty, \mathbb D}:=\sup\{|p(z)|: |z|<1\}$$ if and only if $\|T\|\leq 1.$
The original proof of this inequality is intricate. A couple of
decades later, Sz.-Nazy (see \cite[Chapter 4 , Theorem
4.3]{paulsen}) proved that a bounded linear operator $T$ admits a
unitary (power) dilation if and only  there exists a unitary
operator $U$ on a Hilbert space $\mathcal K \supseteq \mathcal H$
such that
$$P_\mathcal H\,p(U)_{|\mathcal H}=p(T),$$ for all polynomials $p.$
The existence of such a dilation may be established by actually
constructing a unitary operator $U,$ dilating $T.$ This
construction is due to Schaffer \cite{shaffer}. Clearly, the von
Neumann inequality follows from the existence of a power dilation
via the spectral theorem for unitary operators.

The von Neumann inequality says that the homomorphism $\rho_{T}$
induced by $T$ on the polynomial ring $ P[z]$ by the rule
$\rho_{T}(p)=p(T)$ is contractive. The homomorphism $\rho_{T}$
therefore extends to the closure of the polynomial ring $\mathcal
P[z]$ with respect to the sup norm $\|p\|_{\infty, \mathbb D}.$
This is the disc algebra which consists  of all continuous
functions on the closed unit disc $\bar{\mathbb D},$ which are
holomorphic on the open unit disc $\mathbb D.$

Over the years many questions related to the von Neumann
inequality have been studied. Typically, these questions involve
replacing, the polynomial ring $(P[z], \|\cdot \|_{\infty, \mathbb
D})$ with some other ring of functions. For instance,  the rings
of rational functions $\mbox{\rm Rat}(\Omega)$ with poles off
$\bar{\Omega}$ on some open, bounded, connected subset of $\mathbb
C,$ equipped with the supremum norm on $\Omega.$

Suppose $T$ is an operator with $\sigma(T)\subseteq \bar{\Omega}.$
Set $r(T):=p(T)q(T)^{-1},$ for $r\in \mbox{\rm Rat}(\Omega).$
Since $q$ does not vanish on $\bar{\Omega}$ and
$\sigma(T)\subseteq \bar{\Omega},$ it follows that $r(T)$ is
well-defined. It is natural to ask, prompted by the inequality of
von Neumann, when the homomorphism $\rho_{T},$ defined by the rule
$\rho_T(r)=r(T),$ is contractive on $\mbox{\rm Rat}(\Omega).$
There is no good answer to this question, in general. Also, in this more general setting, let us say that a
homomorphism $\rho_T:\mbox{\rm Rat}(\Omega) \to \mathcal
B(\mathcal H)$ admits a normal dilation if there exists a normal
operator $N:\mathcal K \to \mathcal K,$ $\mathcal H\subseteq
\mathcal K$ and $\sigma(N)\subseteq \partial \Omega$ such that
$$P_{\mathcal H}r(N)_{| \mathcal H}=r(T).$$
Clearly, if there exists a normal dilation, then it would follow that
$\|r(N)\|_{\rm op} \leq \|r\|_{\infty, \Omega}$ making the
homomorphism $r \mapsto r(N)$ contractive. However, in most cases, the converse fails.

In 1984, J. Agler (see \cite{agler}) proved that if $\Omega$ is the annulus $\mathbb A$,
then every contractive homomorphism of the ring $\mbox{\rm Rat}(\mathbb A)$ admits a normal dilation.
Recently, M. Dristchell and S. McCullough \cite{michel} have shown that for a domain of connectivity $\geq 2$
 the converse statement is false in general.

In the very fundamental work of Arveson  \cite{A,AW}, he studied the normal
dilation in detail and showed that the existence of a
normal dilation is equivalent to complete contractivity of the
homomorphism $\rho_{T}:$

Let $R=\left(\!(r_{ij})\!\right), r_{ij} \in \mbox{\rm Rat}(\Omega)$  be a matrix valued rational function. Let
$$\|R\|=\sup\{\|\left(\!(r_{ij}(z))\!\right)\|_{\rm op}: z \in
\Omega\}.$$ Define $R(T)$ naturally to be the operator
$\left(\!(r_{ij}(T))\!\right).$ The homomorphism $\rho_T$  is said
to be completely contractive if $\|R(T)\| \leq \|R\|_{\infty,
\Omega}$ for all $R\in \mbox{\rm Rat}(\Omega)\otimes \mathcal
M_k(\mathbb C), k=1,2, \ldots,n,\ldots.$

A deep theorem proved by Arveson says that $T$ has a normal
dilation if and only if $\rho_T$ is completely contractive.
Clearly, if $\rho_T$ is completely contractive, then it is
contractive. The dilation theorems due to Sz.-Nazy and Agler give
the non-trivial converse. Thus for the case of the disc and the
annulus algebras contractive homomorphisms are always completely
contractive.

Most of these notions apply to the rings of polynomials in more
than one variable, or even to the ring of holomorphic functions,
in a neighborhood of $\bar{\Omega},$ where $\Omega$ is some open
bounded connected subset of $\mathbb C^m.$ Indeed, the theorem of
Arveson remains valid in this more general setting.

The first dilation theorem for a commuting pair of contractions
was proved by Ando.
%multi-variate context is due to
He showed that
%. He
%showed that
if $T_1, T_2$ are a pair of commuting contractions, then there
exists a pair of commuting unitaries, which dilate $T_1, T_2$
simultanously, that is, $$P_{\mathcal H}(p(U_1,U_2))|_{\mathcal
H}=p(T_1,T_2)$$(see \cite[Chapter 5, Theorem 5.5]{paulsen}). In
otherwords, every contractive homomorphisms of the bi-disc algebra
is completely contractive. The only other dilation theorem, in the
multi-variable context is due to Agler and Young which is for the
Symmetrized bi-disc \cite{young}. Soon after, Parrott showed that
there are three commuting contractions for which it is impossible
to find commuting unitaries dilating them. This naturally leads to
the question, in view of Arveson's theorem, for which function
algebras $\mathcal A(\Omega),$ all contractive homomorphisms must
be necessarily  completely contractive. At the moment, this is
known to be true of the disc, bi-disc,  Symmetrized bi-disc and
the annulus algebras. Counter examples are known for domains of
connectivity $\geq 2$ and the ball algebra and any balls in
$\mathbb C^m,$ $m\geq 3,$ as we will explain below.

Neither Ando's proof of the existence of a unitary dilation for a
pair of commuting contractions, nor the counter example to such an
existence theorem due to Parrott involved the notion of complete
contractivity directly. However, G. Misra in the papers
\cite{G},\cite{GM} and \cite{sastry} began the study of Parrott like
examples, comparing the norm and the cb-norm,  on domains $\Omega
\subset \mathbb C^m$ other than the tri-disc. This was further
studied in depth by V. Paulsen \cite{vern}, where he showed that
the question of contractive vs completely  contractive for Parrott
like homomorphisms $\rho_{V}$ includes the question of contractive
vs completely contractive for linear maps $L_{V}$ from some finite
dimensional Banach space $X$ to $\mathcal M_n(\mathbb C).$ The
counter examples we mentioned in the previous paragraph were found
by him for such linear maps for $m\geq 5.$ Such examples were
found for $m=3,4$ later by E. Ricard leaving the question of what
happens when $m=2$ open. This is the question we answer in this
thesis. We point out that the results of Paulsen used deep ideas
from geometry of finite dimensional Banach spaces. In contrast,
our results are elementary in nature, although the computations,
at times, are somewhat involved.

\section{Preliminaries}
Let $\Omega$ be a bounded domain (open connected set) in $\mathbb
C^m$ and $\mathcal O(\Omega)$ be the algebra of bounded functions
holomorphic in some neighborhood of $\overline{\Omega}$.
 We equip the algebra $\mathcal O(\Omega)$ with the sup norm, that is,
$$\|f\|_{\infty}=\sup_{z\in \Omega}|f(z)|,
f\in\mathcal O(\Omega).$$ For $i=1,\ldots ,m$ and any choice of
$V_i$ in $\mathcal M_{p,q}(\mathbb C),$ let $T_i =\Big (
\begin{matrix} w_i I_p & V_i \\ 0 & w_iI_q \end{matrix}\Big ),$ $w=(w_1, \ldots, w_m)\in \Omega.$ The
$m$-tuple $T = (T_1, \ldots ,T_m)$  of linear transformations on
$\mathbb C^{p+q}$ is commuting and defines a  homomorphism
$\rho_{V}:\mathcal O(\Omega)\rightarrow \mathcal M_{p+q}(\mathbb
C)$ given by the formula
$$\rho_{V}(f):=f(T_1, \ldots, T_m)=\left (
\begin{smallmatrix}
f(w)I_p& \sum_{i=1}^{m} \partial_if(w)V_{i} \\
0  & f(w)I_q
\end{smallmatrix}\right ), f \in \mathcal O(\Omega),$$
where V denotes the $m$-tuple $(V_1, \ldots , V_m).$ The
homomorphism $\rho_V$ induces the linear map $L_{V}:\mathbb
C^m\rightarrow \mathcal M_{p,q}(\mathbb C)$  given by the formula
$$L_{V}(z)=z_1V_1+\cdots +z_mV_m.
$$
For $v$ in $\mathbb C^m,$
$$\mathcal C_{\Omega, w}(v):=\sup\{|\sum
v_i\partial_if(w)|:f\in\mathcal O(\Omega),
f(w)=0,\|f\|_{\infty}\leq 1\}$$ defines a norm on $\mathbb C^m.$
It is the Carath$\acute{e}$odory norm of $\Omega$ at $w.$ We see
that  $\|\rho_{V}\|\leq 1$ if and only if $\|L_{V}\|_{(\mathbb
C^m,\mathcal C_{\Omega, w})\to (\mathcal M_{p,q}, \|\cdot\|_{\rm
op})}\leq 1$ (here $\|\cdot\|_{1, 2}$ denotes the operator norm
from $(X, \|\, \cdot\,\|_1)$ to $(Y, \|\, \cdot\,\|_2)$). We can
say a little more after tensoring with $\mathcal M_k.$  Let
$\rho_{V}^{(k)}$ be the operator $\rho_{V}\otimes I_k:\mathcal
O(\Omega)\otimes \mathcal M_k\rightarrow (\mathcal M_{p+q}(\mathbb
C)\otimes \mathcal M_{k} , \|\cdot\|_{\rm op}),$ where for $F\in
\mathcal O(\Omega)\otimes \mathcal M_k.$ We define
$\|F\|=\sup_{z\in \Omega}\|(\!(f_{ij}(z))\!)\|, f_{ij}\in \mathcal
O(\Omega).$ Similarly, set $L_{V}^{(k)} := L_{V}\otimes I_k.$ Now,
we have $\|\rho_{V}^{(k)}\|\leq 1$ if and only if
$\|L^{(k)}_{V}\|_{(\mathbb C^m \otimes \mathcal M_{k},
\|\cdot\|_{k}^* )\rightarrow (\mathcal M_{k}\otimes \mathcal
M_{p+q} , \|\cdot\|_{\rm op})}\leq 1$ (cf. \cite[Proposition
2.1]{bagchi} and \cite[Proposition 3.5]{vern}).

Here we study homomorphisms $\rho_{V}$ defined on $\mathcal
O(\Omega_\mathbf A)$, where  $\Omega_\mathbf A$ is a bounded
domain of the form
\begin{eqnarray*}
\Omega_\mathbf A & := &\{(z_1 ,z_2, \ldots, z_m) :\|z_1 A_1
+\cdots + z_mA_m \|_{\rm op} < 1\}
\end{eqnarray*}
for some choice of a linearly independent set of $n\times n$  matrices $\{A_1, \ldots, A_m\}.$

By definition, $\Omega_\mathbf A$ is the unit ball obtained via an
isometric embedding into \mbox{$(\mathcal M_{n},\|\cdot\|_{\rm
op}).$} It is therefore a unit ball in $\mathbb C^m$ with respect
to some norm, say, $\|\, \cdot\,\|_\mathbf A.$ It also has a
natural operator space structure obtained via this embedding.
However, it is possible to pick different isometric embeddings of
a $(\mathbb C^m, \|\cdot\|_\mathbf A)$ into the operators on some
Hilbert space. Whether these different (isometric) embeddings give
the same operator space structure is an interesting question on
its own right. Picking $A_1 = \big (\begin{smallmatrix} 1 & 0\\ 0
& 0
\end{smallmatrix} \big )$ and $A_2 = \big ( \begin{smallmatrix} 0
& 1\\ 0 & 0 \end{smallmatrix} \big )$ gives the embedding of the
Euclidean ball as the space of ``row vectors", while if we pick
the transpose of $A_1$ and $A_2$, we would be embedding it as the
space of column vectors. As is well known, these two embeddings
give rise to different operator space structures leading to an
example of a contractive homomorphism on the ball algebra which is
not completely contractive. While for any $n$ in $\mathbb N$, if
we pick $A_1=I_{2n}, A_2= \big ( \begin{smallmatrix} 0 & B\\ 0 & 0
\end{smallmatrix} \big ),$ then they determine the same norm
(independent of $n$)  on $\mathbb C^2$  as long as $\|B\| =1.$
However, the operator space structure is also independent of $n,$
which  we show in Chapter $3.$ Following the example of the ball,
even if we pick the new pair to be $A_1 = I_{2n}$ and $A_2 = \big
(
\begin{smallmatrix} 0 & B\\ 0 & 0
\end{smallmatrix} \big )^{\rm t},$ the operator space we obtain
remains the same. The ball $\Omega_\mathbf A,$ in this case,  is
the set $\{(z_1, z_2)\in \mathbb C^2: |z_1|^2 + |z_2| < 1 \}.$ We
see that it has several distinct isometric embeddings  into
$\mathcal M_{2n}(\mathbb C).$ Surprisingly, all of these give the
same operator space structure on $(\mathbb C^2, \|\cdot
\|_{\mathbf A}).$ Therefore, unlike the case of the Euclidean
ball, we have to find some other way of  showing the existence of
distinct operator space structures on this normed space, which  we
do in Chapter $3.$

From the work of V. Paulsen and E. Ricard  (cf. \cite{vern},
\cite{pisier}), it follows that if $m\geq 3$ and $\mathbb B$ is
any ball in $\mathbb C^m$, then there exists a contractive linear
map which is not completely contractive. It is known that
contractive homomorphisms of the disc and the bi-disc algebra  are
completely contractive, thanks to the dilation theorem of B.
Sz.-Nagy and Ando.  However, an example of a contractive
homomorphism of the (Euclidean) ball algebra which is not
completely contractive was given in \cite{GM, G}. The
characterization of those balls in $\mathbb C^2$ for which
``contractive linear maps  are always completely contractive"
remained open. We answer this question in Chapter $5$ for domains
of the form $ \Omega_\mathbf A ,$ $\mathbf A=(A_1,A_2)$ in
$\mathbb C^2 \otimes \mathcal M_2(\mathbb C)$. Along the way we
obtain some interesting applications for domains of the form $
\Omega_\mathbf A $ in $\mathbb C^m, m\in \mathbb N.$

The (linear) polynomial $P_{\mathbf A}$ defined by the rule
$$P_{\mathbf A}(z_{1}, z_{2}, \ldots,
z_{m})=z_{1}A_{1}+z_{2}A_{2}+\cdots+z_{m}A_{m},$$
%where $\mathbf
%A=(A_1,\ldots ,A_m)$ is in $\mathbb C^m \otimes \mathcal
%M_n(\mathbb C),$
maps the ball $\Omega_\mathbf A$ into $(\mathcal M_{n}(\mathbb C),
\|\cdot\|_{\rm op})_1$ by definition. %We develop several methods
%to determine if $\|\rho_V \| \leq 1.$ We show that $\|\rho_V\|
%\leq \|\rho_V^{(n)} (P_\mathbf A)\|.$
%$$\|L^{(n)}_{V}(P_\mathbf A)\|_{\rm op}:=\|V_1\otimes A_1+\cdots
%+V_m\otimes A_m\|_{\rm op}$$ for any choice of $V$ such that
%$\|L_{V}(\lambda_1, \cdots, \lambda_m)\|_{\rm op}\leq 1$ for
%$(\lambda_1, \cdots, \lambda_m)\in \Omega_{\mathbf A}^{*}.$
We develop several methods to determine when $\|L_V \| \leq 1.$ We
recall that $\|L_V \| \leq 1$ if and only if $\|\rho_V \| \leq 1.$
We show that $\|L_V\| \leq \|L_V^{(n)} (P_\mathbf A)\|.$ Finding a
$V$ such that $$\|L_{V}\|_{(\mathbb C^m,\mathcal C_{\Omega, w})\to
(\mathcal M_{p,q}(\mathbb C), \|\cdot\|_{\rm op})}\leq 1$$ for
which $\|L^{(n)}_{V}(P_{\mathbf A})\|_{\rm op}>1$ gives an example
of a contractive homomorphism on $\mathcal O(\Omega_\mathbf A)$
which is not completely contractive. However, finding such a $V$
is far from obvious, as we will see.

Furthermore, in Chapter 2, we show that for homomorphisms of our
class, the property ``contractivity implies complete
contractivity", remains unaffected under bi-holomorphic
equivalence. Thus we describe some natural bi-holomorphic,
actually linear, equivalence for domains of the form
$\Omega_\mathbf A $ and work with a convenient representative from
each equivalence class. We give a list of such representatives for
the class of domains $\Omega_\mathbf A$ in Chapter 2.

The class of homomorphisms of the form $\rho_V$ arise from
localization of  operators in the Cowen-Douglas class of $\Omega.$
The (complete) contractivity of a homomorphism in this class
naturally produces inequalities for the curvature of the
corresponding Cowen-Douglas bundle ({cf. \cite[Theorem5.2]{GM}}).
This connection and some of its very interesting consequences are
discussed in Chapter 4.

In the paper \cite{parrott}, Parrott showed that if $U_1$ and $U_2$ are
a pair of non-commuting unitaries then the homomorphism  $\rho_V,$  $V=(I,U_1,U_2),$
is contractive on the tri-disc algebra $\mathcal A(\mathbb D^3)$ which is not completely contractive.
Equivalently, he shows that there does not exist commuting unitaries dilating
the commuting contractions
$\big (\begin{smallmatrix} I_2 &0 \\0& I_2 \end{smallmatrix}\big),\,(\begin{smallmatrix} 0 &U_1
\\0&0 \end{smallmatrix}\big)$ and $(\begin{smallmatrix} 0 & U_2 \\0&0 \end{smallmatrix}\big).$
This shows that Ando's theorem does not generalize to $m > 2.$ The
Parrott examples were further studied in a series of papers
\cite{GM, G, sastry, pati, vern}.

%Although, we are unable to use the following what appears to be an
%interesting criterion for the contractivity of $L_{V}$ at present,
%we record it here in hope that we will be able to make use of it
%in further work. $L_{V}$ is contractive if and only if
%$$\|\sum_{i=1}^{m}V_i\langle A_i\alpha, \beta\rangle\|_2\leq 1$$
%for all $\alpha, \beta$ with $\|\alpha\|_2=\|\beta\|_2=1.$

Let $\Omega_{\mathbf A}^{*}$ be the unit ball for the dual norm
$\|\cdot\|^*_\mathbf A $. We point out that contractive
homomomorphisms of our class are completely contractive for
$\mathcal O(\Omega_\mathbf A)$ if and only if it is true for
$\mathcal O(\Omega_{\mathbf A}^{*})$ (see \cite{pati} and
\cite{vern}).

Let $P_{\mathbf A}:\Omega_{\mathbf A}\rightarrow (\mathcal
M_{n}(\mathbb C))_{1}$ be the matrix valued polynomial on
$\Omega_{\mathbf A}$ defined by $P_{\mathbf A}(z_{1}, z_{2},
\ldots, z_{m})=z_{1}A_{1}+z_{2}A_{2}+\cdots+z_{m}A_{m},$ where
$(\mathcal M_n(\mathbb C))_1$ is the matrix unit ball with respect
to the operator norm. For $(z_{1}, z_{2}, \ldots, z_{m})$ in
$\Omega_{\mathbf A},$ the norm $$\|P_{\mathbf A}\|_\infty :=
\sup_{(z_1,\ldots ,z_m)\in \Omega_\mathbf A} \|P_\mathbf A(z_1,
\ldots, z_m)\|_{\rm op}$$ is at most $1$ by definition of the
polynomial $P_\mathbf A.$ We say that $P_{\mathbf A}$ is a
defining function for $\Omega_{\mathbf A}.$ As we have indicated
earlier, we  detect the failure of complete contractivity by
checking if $\|\rho_V(P_\mathbf A)\| \leq 1$ or not. Often, one
works with a defining function which is assumed to be smooth. Our
defining function takes values in $\mathcal M_n(\mathbb C),$  it
is holomorphic, indeed, it is a linear map.

%\subsubsection{Preliminaries}

For $(\alpha, \beta) \in \mathbb B \times \mathbb B,$ define
$p^{(\alpha,\beta)}_{\mathbf A}:\Omega_{\mathbf A}\rightarrow
\mathbb D$ to be the linear map $$p^{(\alpha,\beta)}_{\mathbf
A}(z_1,\ldots, z_{m})=\left\langle P_{\mathbf A}(z_1, \ldots,
z_{m})\alpha, \beta\right\rangle.$$
%$Then p^{(\alpha\beta)}_{\mathbf A}$ is holomorphic on %\Omega_{\mathbf
%A}$ and $p^{(\alpha\beta)}_{\mathbf A}(0,\ldots, 0)=0.$
The sup norm $\|p^{(\alpha,\beta)}_{\mathbf A}\|_{\infty}$ on
$\Omega_\mathbf A,$ for any pair of vectors $(\alpha, \beta)$ in
$\mathbb B\times \mathbb B,$  is at most $1$  by definition. Let
$\wp^{(\alpha,\beta)}_{\mathbf A}$ denote the set of linear
functions
%\begin{align*}
$\{p^{(\alpha,\beta)}_{\mathbf A}:(\alpha, \beta)\in \mathbb
B\times \mathbb B\}.$
%\end{align*}
Let $V=(V_1, \ldots, V_m),$ $V_{i}\in \mathcal M_{p,q},$   and
 $\rho_{V}:\wp
^{(\alpha,\beta)}_{\mathbf A}\longmapsto \mathcal B(\mathbb
C^p\oplus \mathbb C^q) \cong \mathcal B(\mathbb C^{p+q}) $ be the
homomorphism defined by
$$\rho_{V}(p^{(\alpha,\beta)}_{\mathbf A})=\begin{pmatrix}
    p^{(\alpha,\beta)}_{\mathbf A}(0)I_p & \partial_{1}p^{(\alpha,\beta)}_{\mathbf A}(0)V_{1}+\ldots+
    \partial_{m}p^{(\alpha,\beta)}_{\mathbf A}(0)V_{m}\\
    0    & p^{(\alpha,\beta)}_{\mathbf A}(0)I_q
\end{pmatrix},\, p^{(\alpha, \beta)}_\mathbf A \in \wp^{(\alpha,\beta)}_{\mathbf A}. $$

\begin{lemma}\label{lem:adfd}
$\sup_{\|\alpha\|=\|\beta\|=1}\|\rho_V(p^{(\alpha,\beta)}_{\mathbf
A})\| \leq \|\rho_V^{(n)}(P_{\mathbf A})\|.$\end{lemma}
\begin{proof} The proof is a straightforward computation:
{\small\begin{align}\label{eqn:1} \sup_{|\alpha\|=\|\beta\|=1} \|
\rho_{V}(p^{(\alpha,\beta)}_{\mathbf A})
\|\nonumber&=\sup_{\|\alpha\|=\|\beta\|=1}
\|\,\partial_{1}p^{(\alpha,\beta)}_{\mathbf A}(0)V_1+ \cdots+
\partial_{m}p^{(\alpha,\beta)}_{\mathbf A}(0)V_{m}\,\|\\ \nonumber
&=\sup_{\|\alpha\|=\|\beta\|=1} \|\,\langle
A_{1}\alpha, \beta\rangle V_1+ \cdots+ \langle A_{m}\alpha,
\beta\rangle V_{m}\,\|
\\ \nonumber
&=\sup_{\|\alpha\|=\|\beta\|=\|u\|=\|v\|=1}
 |\langle
A_{1}\alpha, \beta\rangle \langle V_1u, v\rangle+ \cdots+ \langle
A_{m}\alpha, \beta\rangle \langle V_mu,
v\rangle|\\\nonumber&=\sup_{\|\alpha\|=\|\beta\|=\|u\|=\|v\|=1}|\langle(A_1
\otimes V_1+\cdots+A_m \otimes V_m)\alpha \otimes u, \beta \otimes
v \rangle|\\\nonumber
&=\sup_{\|\alpha\|=\|\beta\|=\|u\|=\|v\|=1}|\langle
\rho_V^{(n)}(P_{\mathbf A}) \alpha \otimes u, \beta \otimes v
\rangle|\\
&\leq \|\rho_V^{(n)}(P_{\mathbf A})\|.
\end{align}}This completes the proof.
\end{proof}
Since $p^{(\alpha, \beta)}_\mathbf A$ is linear,
the derivative $Dp^{(\alpha, \beta)}_\mathbf A(0) = p^{(\alpha, \beta)}_\mathbf A.$
The set of vectors
$$\{(\langle A_1 \alpha, \beta\rangle, \ldots , \langle A_m \alpha, \beta\rangle):\alpha, \beta \in \mathbb B^2\}
\subseteq \mathbb C^m$$ is  a subset of the dual unit ball $\Omega_\mathbf A^*$ by definition.
We will not distinguish between this set of vectors and the set of linear maps $\wp^{(\alpha, \beta)}_\mathbf A$
induced by them.

The linear map $L_V,$ $V=(V_1, \ldots , V_m),$ is contractive if and only if
\begin{align*}
\sup_{(\lambda_1, \ldots, \lambda_m)\in \Omega_{\mathbf
A}^{*}}\|\lambda_1 V_1 + \cdots + \lambda_m V_m\|_{\rm op}
=\sup_{(\lambda_1, \ldots, \lambda_m)\in \Omega_{\mathbf A}^{*}}
\sup_{ \|u\|_2 =1} \|\lambda_1 V_1 u + \cdots + \lambda_m V_m
u\|_2 \leq \|(\lambda_1, \ldots, \lambda_m)\|^*_\mathbf A.
\end{align*}
Or, equivalently,
$$
\sup_{(\lambda_1, \ldots, \lambda_m)\in \Omega_{\mathbf
A}^{*}}\sup_{ \|u\|_2 =1=\|v\|_2} |\lambda_1 \langle V_1 u,
v\rangle  + \cdots + \lambda_m \langle V_m u , v \rangle | \leq
\|(\lambda_1, \ldots, \lambda_m)\|_\mathbf A^*,$$ that is,
$\|L_V\|\leq 1$ if and only if $(\langle V_1 u, v\rangle, \ldots ,
\langle V_m u, v\rangle)$ is in $\Omega_\mathbf A$ for every pair
of unit vectors $u$ and $v.$ We find that
\begin{align}
\sup_{\|u\|_2=1=\|v\|_2} \|(\langle V_1 u, v\rangle, \ldots ,
\langle V_m u, v\rangle)\|^2_\mathbf A \nonumber &=\|\langle V_1
u, v\rangle A_1 + \cdots + \langle V_m u, v\rangle A_m)\|^2_{\rm
op} \\ \nonumber
%&=\sup_{\|u\|_2=1=\|v\|_2}\sup_{\|\alpha\|_2=1=\|\beta\|_2}|\langle
%V_1 u, v\rangle\langle A_1 \alpha, \beta\rangle+\cdots +\langle
%V_m u, v\rangle\langle A_m \alpha, \beta\rangle|^2
%\nonumber
&=\sup_{\|u\|_2=1=\|v\|_2} \sup_{\|\alpha\|_2=1=\|\beta\|_2}
|\langle \sum_{j=1}^m  \langle A_j \alpha  , \beta \rangle V_ju,
v\rangle\ |^2 \\\nonumber &=\sup_{\|u\|_2=1}
\sup_{\|\alpha\|_2=1=\|\beta\|_2} \|\sum_{j=1}^m  \langle A_j
\alpha  , \beta \rangle V_ju\ \|^{2}_{2}\\\nonumber
&=\sup_{\|u\|_2=1} \sup_{\|\alpha\|=\|\beta\|=1} \|L_{V }\big ( (
\langle A_1 \alpha, \beta\rangle, \ldots , \langle A_m \alpha ,
\beta \rangle) \big ).u\|_2^2 \\&= \sup_{\|\alpha\|=\|\beta\|=1}
\|L_{V }\big ( ( \langle A_1 \alpha, \beta\rangle, \ldots ,
\langle A_m \alpha , \beta \rangle) \big )\|_{\rm op}
\end{align}
We have seen that $\{(\langle V_1 u, v\rangle, \ldots ,  \langle
V_m u, v\rangle): \|u\|_2\leq 1,\, \|v\|_2 \leq 1\} \subseteq
\Omega_\mathbf A$ for any fixed but arbitrary $m$ tuple $V$ for
which $L_V$ is contractive. However, it is not clear if there is a
collection of contractive homomorphisms  which produce all of
$\Omega_\mathbf A.$ Similarly, the set $\{( \langle A_1 \alpha,
\beta\rangle, \ldots , \langle A_m \alpha , \beta \rangle) :
\|\alpha\|_2 \leq 1,\, \|\beta\|_2\leq 1 \} \subseteq
\Omega_\mathbf A^*.$ Again, we don't know if for some choice of
$\mathbf A$ equality occurs.

Thus we have shown that $L_{V}$ is contractive if and only if
it is contractive on the set $\wp^{(\alpha, \beta)}_\mathbf A.$  However, as we have pointed out earlier,
$L_{V}$ is contractive  if and only if the homomorphism $\rho_V$ is contractive. Similarly,
$L_{V}$ is contractive on the set $\wp^{(\alpha, \beta)}_\mathbf A$  if and only if the restriction
${\rho_V}_{|\wp^{(\alpha, \beta)}_\mathbf A}$ of the homomorphism $\rho_V$ to $\wp^{(\alpha, \beta)}_\mathbf A$
is contractive. Therefore we have proved the following.
\begin{proposition}\label{LVRHO}The following conditions are equivalent.
\begin{enumerate}
\item[(i)] $\|\rho_V\| = \sup_{\|p\|_{\infty}\leq 1}
\{\|\rho_{V}(p)\|: p \in \mathcal O(\Omega_\mathbf A),
\,p(0)=0\}\leq 1$

\item[(ii)] $\sup_{\|\alpha\|=\|\beta\|=1}\{ \|
\rho_{V}(p^{(\alpha, \beta)}_\mathbf A) \|: p_\mathbf A^{(\alpha,
\beta)} \in \wp_\mathbf A^{(\alpha, \beta)}\} \leq 1$

\item [(iii)]$\|L_{V}\|_{(\mathbb C^m, \|\cdot\|^* _\mathbf A)\to
(\mathcal M_n, \|\cdot\|_{\rm op})} \leq 1$

\item[(iv)] $\sup_{\|\alpha\|=\|\beta\|=1} \|L_{V }\big ( (
\langle A_1 \alpha, \beta\rangle, \ldots , \langle A_m \alpha ,
\beta \rangle) \big )\|_{\rm op}\leq 1$
\end{enumerate}
\end{proposition}
\begin{corollary}\label{corolll}
If $\|\rho^{(n)}_{V}(P_{\mathbf A})\|\leq 1$ then $\rho_V$ is
contractive.
\end{corollary}
\begin{proof}
It is enough to check the contractivity of the restriction of $\rho_V$
to the set $\wp^{(\alpha, \beta)}_\mathbf A.$ On this set,
as we have shown in Lemma \ref{lem:adfd}, the norm of $\rho_V$ is bounded above by $\|\rho_V^{(n)}(P_\mathbf A)\|.$
\end{proof}
\begin{remark}This proposition says that checking the contractivity of $\rho_{V}$
on the algebra $\mathcal O(\Omega_\mathbf A)$ may be reduced to
checking it on $\wp ^{(\alpha,\beta)}_{\mathbf A}.$ Thus this
class of homomorphisms $\wp ^{(\alpha,\beta)}_{\mathbf A}$ serves
as a class of ``Test functions". Apart from this, for this class
of homomorphisms  $\rho_V$, we have the property $\|\rho_V\| \leq
\|\rho_V^{(n)} (P_\mathbf A)\|.$ This inequality often happens to
be strict making it possible to construct examples of contractive
homomorphisms which are not completely contractive.
\end{remark}
%We give a different proof of Proposition \ref{rho_contra}, which has its own merit.
%\begin{proposition}\label{Pcont}If $\|\rho_{V}(P_{\mathbf A})\|\leq 1$ then
%$\rho$ is contractive.
%\end{proposition}
%\begin{proof}By Proposition \ref{LVRHO} we have{\small\begin{align*}
%\sup_{\|p\|_{\infty}\leq 1,
%p(0)=0}\|\rho_{V}(p)\|^2&=\sup_{\|\alpha\|=\|\beta\|=\|u\|=1}
%\sum_{j=1}^{m}|\langle B_{j}\alpha, \beta\rangle|^2\\&\leq
%\sup_{\|\alpha\|=\|u\|=1} \sum_{j=1}^{m}|
%B_{j}\alpha|^2\\&=\sup_{\|\alpha\|=\|u\|=1}\|(A_1 \otimes
%V_1+\cdots+A_m \otimes V_m)\alpha \otimes u\|^2\\&\leq\|A_1
%\otimes V_1+\cdots+A_m \otimes V_m\|^2\\&= \|\rho_{V}(P_{\mathbf
%A})\|^2\end{align*}} completing the proof.
%\end{proof}
%\begin{proposition}\label{LV}
%$\|L_{V}\|=\sup_{\|u\|=1}\|L_{w}\|.$
%\end{proposition}
%\begin{proof} We have
%\begin{align*}\|L_{V}\|&=\sup_{\|(z_1,
%\ldots,z_m)\|^*_{\Omega_{\mathbf A}}\leq 1}\|z_1 V_1+ \cdots + z_m
%V_m\|\\&=\sup_{\|u\|=1}\sup_{\|(z_1,
%\ldots,z_m)\|_{\Omega_{\mathbf A}}^*\leq 1}\|z_1 V_1u+ \cdots +
%z_m V_mu\|\\&=\sup_{\|u\|=1}\|L_{w}\|.
%\end{align*} This completes the proof.
%\end{proof}

Choosing $\mathbf A=\left(\left (
\begin{smallmatrix}
1 & 0   \\
0  &  0
\end{smallmatrix}\right ) ,
\left ( \begin{smallmatrix}
0&  1\\
0  &  0
\end{smallmatrix}\right )\right),$ we see that  $\Omega_\mathbf A$
defines the Euclidean ball in $\mathbb C^2$. Choose
$V_1=(v_{11}\,\, v_{12}), V_2= (v_{21}\,\, v_{22}).$ We will prove
that
$$\sup_{\|\alpha\|=\|\beta\|=1}\| \rho_{V}(p^{(\alpha,
\beta)}_\mathbf A) \|< \|\rho_V(P_{\mathbf A})\|_{\rm op}.$$ This
example, of a contractive homomorphism of the ball
algebra which is not completely contractive, was found
in \cite{GM, G}.
\begin{theorem}\label{thm1} For $\Omega_\mathbf A=\mathbb B^2,$ we have
$$\sup_{\|\alpha\|=\|\beta\|=1}\| \rho_{V}(p^{(\alpha
,\beta)}_\mathbf A) \|< \|\rho_V(P_{\mathbf A})\|_{\rm op}.$$
\end{theorem}
\begin{proof} By definition of $\rho_V,$ we have
\begin{align*}
\sup_{\|\alpha\|=\|\beta\|=1}\|\rho_{V}(p^{(\alpha
,\beta)}_\mathbf A)\|^2&=\sup_{\|\alpha\|=\|\beta\|=\|u\|=\|v\|=1}
 |\langle
A_{1}\alpha, \beta\rangle \langle V_1u, v\rangle+ \langle
A_{2}\alpha, \beta\rangle \langle V_2u,
v\rangle|^2\\&=\sup_{\|\alpha\|=\|\beta\|=\|u\|=1}|\alpha_1(v_{11}u_1+v_{12}u_2)+\alpha_2(v_{21}u_1+v_{22}u_2)|^2
|\beta_1|^2\\
&=\sup_{\|\alpha\|=\|u\|=1}|\alpha_1(v_{11}u_1+v_{12}u_2)+\alpha_2(v_{21}u_1+v_{22}u_2)|^2
\\&=\sup_{\|u\|=1}|v_{11}u_1+v_{12}u_2|^2+|v_{21}u_1+v_{22}u_2|^2\\&=\big\|\left(\begin{smallmatrix} v_{11}
& v_{12}\\
v_{21} & v_{22}\end{smallmatrix}\right)\big\|^{2}_{\rm
op}.\end{align*}
On the other hand, we have  $$\|\rho_V(P_{\mathbf A})\|_{\rm op}^2=\|V_1\|^2+\|V_2\|^2$$ where
 $\|V_1\|^2=|v_{11}|^2+|v_{12}|^2,
 \|V_2\|^2=|v_{21}|^2+|v_{22}|^2.$ It follows that $$\big\|\left(\begin{smallmatrix} v_{11}
& v_{12}\\
v_{21} & v_{22}\end{smallmatrix}\right)\big\|^{2}_{\rm op}
<\|V_1\|^2+\|V_2\|^2.$$ Hence we have
$$\sup_{\|\alpha\|=\|\beta\|=1}\| \rho_{V}(p^{(\alpha,
\beta)}_\mathbf A) \|< \|\rho_V(P_{\mathbf A})\|_{\rm op}.$$
\end{proof}
\begin{remark}\label {rem2}
It is therefore natural to ask which of the domains $\Omega_\mathbf A \subset
\mathbb C^2$ has the property
$$\sup_{\|\alpha\|=\|\beta\|=1}\|\rho_V(p^{(\alpha,\beta)}_{\mathbf
A})\|< \|\rho_V^{(n)}(P_{\mathbf A})\|.$$
If the answer is affirmative, then there is a possibility of producing an example of a contractive homomorphism of
$\mathcal O(\Omega_\mathbf A)$ which is not completely contractive.
However, as we will see, unlike the case of the Euclidean ball, this requires lot more work in general.
\end{remark}
%In chapter $4$ we describe the notion of complete
%contraction and contraction in view of operator space.
If a normed linear space $(\mathbb C^m, \|\cdot\|_\mathbf A)$
admits only one operator space structure, then every contractive
linear map from $(\mathbb C^m, \|\cdot\|_\mathbf A)$ into
$\mathcal M_k(\mathbb C),$ $k\in \mathbb N$ must be completely
contractive. As before, for some linearly independent set of
$n\times n $ matrices $\{A_1,\ldots ,A_m\},$ setting  $$\|(z_1,
\cdots, z_m)\|_{\mathbf A}:=\|z_1A_1+\cdots+z_mA_m\|_{\rm op},$$
we obtain an $m$-dimensional normed linear space $
\mathbf V_{\mathbf A}.$ This makes the map
$$(z_1, \ldots, z_m)\rightarrow z_1A_1+\cdots+z_mA_m$$ an isometry
from $\mathbf V_{\mathbf A}$ into $(\mathcal M_n, \|\cdot\|_{\rm
op}).$ Therefore, $\mathbf V_{\mathbf A}$ inherits  an operator
space structure from $\mathcal M_n.$ Similarly we can think of $
\mathbf V_{\mathbf A^{\rm t}}$  as an operator space via the
isometric embedding
$$(z_1, \ldots, z_m)\rightarrow z_1A_1^{\rm t}+\cdots+z_mA_m^{\rm
t}$$ into $\mathcal M_n(\mathbb C),$ where $\mathbf A^{\rm
t}=(A_1^{\rm t}, \ldots, A_m^{\rm t})$ is obtained by taking the
transpose.

Let $\mathbf A=\left(\left (
\begin{smallmatrix}
1 & 0   \\
0  &  0
\end{smallmatrix}\right ) ,
\left ( \begin{smallmatrix}
0&  1\\
0  &  0
\end{smallmatrix}\right )\right).$ The norm it determines on $ \mathbf V_\mathbf A (\cong \mathbb C^2)$ is the
$\ell_2$ norm. Note that $P_{\mathbf A}:\mathbf V_{\mathbf
A}\rightarrow \mathcal M_{2}(\mathbb C)$ defines a linear
isometric embedding. Suppose $ V=(\!(\mathbf v_{ij})\!)
\in\mathcal M_{k}( \mathbf V_{\mathbf A}),$ where $\mathbf
v_{ij}\in \mathbf V_{\mathbf A}.$ Define $P_{\mathbf
A}^{(k)}:=P_{\mathbf A}\otimes I_k:\mathcal M_{k}(\mathbf
V_{\mathbf A})\rightarrow \mathcal M_{k}(\mathcal M_{2}(\mathbb
C))$ by $P_{\mathbf A}^{(k)}( V)=(\!(P_{\mathbf A}(\mathbf
v_{ij}))\!).$ Let $\mathbf v_{ij}=(v^{1}_{ij}\,\, v^{2}_{ij}),
i,j=1,\ldots, k,$ then
$$P_{\mathbf A}^{(2)}( V)=\left(\begin{smallmatrix} V_1  &  V_2\\
0 & 0\end{smallmatrix}\right),$$ where $ V_1=(\!(v^{1}_{ij})\!)$
and $ V_2=(\!(v^{2}_{ij})\!).$ Similarly if we take $\mathbf
A^t=\left(\left (
\begin{smallmatrix}
1 & 0   \\
0  &  0
\end{smallmatrix}\right ) ,
\left ( \begin{smallmatrix}
0&  0\\
1  &  0
\end{smallmatrix}\right )\right)$ then $\mathbf
V_{\mathbf A^{\rm t}}$ becomes an operator space. Therefore we
have $$P_{\mathbf A^{\rm t}}^{(2)}( V) =\left(\begin{smallmatrix}
V_1 & 0  \\
 V_2 &  0
\end{smallmatrix}\right ).$$
For the record, the norm of $P_{\mathbf A}^{(2)}(V)$ and
$P_{\mathbf A^{\rm t}}^{(2)}( V)$ are given in the following
lemma.
\begin{lemma}\label{A Lem}
If $\mathbf v_1=(v_{11}\,\, v_{12}), \mathbf v_2= (v_{21}\,\,
v_{22}),$ then
$$\big\|\left(\begin{smallmatrix}
\mathbf v_1 & \mathbf v_2  \\
0  &  0
\end{smallmatrix}\right
)\big\|^2=\|\mathbf v_1\|^2+\|\mathbf v_2\|^2=|v_{11}|^2+|v_{12}|^2+|v_{21}|^2+|v_{22}|^2$$
and $$\big\|\left(\begin{smallmatrix}
\mathbf v_1 &0 \\
\mathbf v_2 &  0
\end{smallmatrix}\right
)\big\|^2=\big\|\begin{pmatrix} v_{11}
& v_{12}\\
v_{21} & v_{22}\end{pmatrix}\big\|^{2}_{\rm op}.$$
\end{lemma}

Consequently, for this choice of $ V,$ the norms $\|P_{\mathbf
A}^{(2)}( V)\|$ and  $\|P_{\mathbf A^{\rm t}}^{(2)}( V)\|$ are not
equal.  The existence of two distinct operator space structures on
$\mathbf V_{\mathbf A}$ follows from this.

However, most of the time, this trick doesn't work, that is, the
operator space structures induced by $\mathbf A$ and $\mathbf
A^{\rm t}$ are completely isometric. In that situation, the
following algorithm is adopted, which involves a careful ``case by
case'' analysis. Fix $\mathbf v_1=(v,0), \mathbf v_2=(0,w).$ Let
$L_{(\mathbf v_1, \mathbf v_2)}:(\mathbb C^2,\|\,\cdot\,\|^{*}
_{\Omega _{\mathbf A}})\rightarrow (\mathbb C^2,\|\,\cdot\,\|_2)$
be the linear map $(z_1,z_2)\mapsto(z_1v,z_2w).$
\begin{enumerate}
\item [(i)] For $\beta$ in $\mathbb C^2,$  and $\mathbf v_1 =
(v,0),\,\mathbf v_2 = (0,w)$ as above, let
$$g_{(v,w)}(\beta):=\{1-|v|^2\|A_{1}^*\beta\|^2
-|w|^2\|A_{2}^*\beta\|^2+
|vw|^2(\|A_{1}^*\beta\|^2\|A_{2}^*\beta\|^2 - |\left\langle
A_1A_{2}^*\beta, \beta\right\rangle |^2)\}.$$ We  show that
$L_{(\mathbf v_1, \mathbf v_2)}:(\mathbb C^2,\|\,\cdot\,\|^{*}
_{\Omega _{\mathbf A}})\rightarrow (\mathbb C^2,\|\,\cdot\,\|_2)$
is contractive if and only if $|v|^2\leq \frac{1}{\|A_1^{*}\|^2}$
and $(v,w)$ is in $\mathcal{E}:=\{(v, w):\inf_{\beta,
\|\beta\|_2=1}g_{(v, w)}(\beta)\geq 0\}.$

\item [(ii)] We then show that there exists $\mathbf v_1, \mathbf
v_2$ for which $L_{(\mathbf v_1, \mathbf v_2)}$ is contractive
while $L^{(2)}_{(\mathbf v_1, \mathbf v_2)}(P_\mathbf A)$ is not
contractive. Therefore, this contractive linear map, namely,
$L_{(\mathbf v_1, \mathbf v_2)}$ cannot be completely contractive.

\item [(iii)] The contractivity of $L_{(\mathbf v_1, \mathbf
v_2)}(P_{\mathbf A})$ is shown to be equivalent to the condition
$$\inf_{\beta}\{
1-|v|^2\|A_1^{*}\beta\|^2-|w|^2\|A_2^{*}\beta\|^2:\|\beta\|_2 = 1\}\geq 0.$$

%\item [(iv)] Let $\mathcal
%E_0:=\{(v,w):\inf_{\beta}g_{(v,w)}(\beta)=0\}.$

\item [(iv)] There exists $\beta \in \mathbb C^2$  such that
either $(A_2^*-\mu A_1^*)\beta=0$ or $(A_1^*-\nu A_2^*)\beta=0$
for some
 $\mu, \nu$ in $\mathbb C.$ The  set
$$\mathcal B:=\{ \beta: \|\beta\|_2=1,\,(A_2^*-\mu A_1^*)\beta = 0\, \mbox{\rm or}
\,(A_1^*-\nu A_2^*) \beta=0 \,\mbox{\rm for some } \mu, \nu \in
\mathbb C\}$$ of these vectors is non-empty.
\end{enumerate}
%Goal: Find a $b\in B_{(v,w)}$ for some $(v,w)$ such that
%$\|A_{1}^*\beta\|^2\|A_{2}^*\beta\|^2 - |\left\langle
%A_1A_{2}^*\beta, \beta\right\rangle |^2>0.$

%To achieve this goal that we
%have set ourselves, pick $(v, w)$ such that $|w|=\lambda |v|,
%\lambda >0.$ Having made this choice of $v, w,$  set
%$g_{(v,w)}(\beta)=g_{\lambda}(\beta, v).$
%The main bulk of our work
In the last chapter
%lies in showing the existence of
 we show that there exists a $\lambda >
0,$ say $\lambda_0,$ such that $(v, \lambda_0 v)$ is in $\mathcal
E$ with the property:

$g_{(v,\lambda_0v)}(\beta^{\prime\prime})>g_{(v,\lambda_0v)}(\beta^{\prime})>g_{(v,\lambda_0v)}(\beta)$
or
$g_{(v,\lambda_0v)}(\beta^{\prime})>g_{(v,\lambda_0v)}(\beta^{\prime\prime})>g_{(v,\lambda_0v)}(\beta)$
whenever $\beta^{\prime} ,\beta^{\prime\prime}\in \mathcal B.$

Also, we then  prove that there exists a $v$
($|v|<\frac{1}{\|A_1^*\|},$ this is necessary for contractivity),
say $v_0,$ such that  $(v_0,\lambda v_0)$ is in $\mathcal
E_0:=\{(v,w):\inf_{\beta}g_{(v,w)}(\beta)=0\}.$

Hence there exists a $v_0, \lambda_0$ and $\beta_0$ such that
$$
1-|v_0|^2\|A_{1}^*\beta_0\|^2 -|\lambda_0
v_0|^2\|A_{2}^*\beta_0\|^2+
\lambda_0^2|v_0|^4(\|A_{1}^*\beta_0\|^2\|A_{2}^*\beta_0\|^2 -
|\left\langle A_1A_{2}^*\beta_0, \beta_0\right\rangle |^2)=0$$
which is equivalent to $\|L_{(\mathbf v_1, \mathbf
v_2)}(P_{\mathbf A})\|>1.$

We now discuss the relationship of homomorphisms of the form
$\rho_V$ with $m$ tuple of operators $T$ in the Cowen-Douglas
class $B_1(\Omega),$ $\Omega \subseteq \mathbb C^m.$ In the papers
\cite{cowen} and \cite{curto}, it is shown that the operator $T$
can be realized as the adjoint of the commuting tuple $\boldsymbol
M=(M_1,\ldots M_m)$ of multiplication operators defined by the
coordinate functions on a reproducing kernel Hilbert space
$(\mathcal H, K)$ consisting of holomorphic functions on
$\Omega^*:=\{\bar{w}:w\in \Omega\}.$ It then follows that the
joint kernel $\cap_{i=1}^m\ker{(M_i-w_i)}^*$ is spanned the vector
$K_w.$ We think of $w\mapsto K_w$ as a frame for a
anti-holomorphic line bundle $\mathcal L_\mathbf M$ on $\Omega.$
The Hermitian metric of this line bundle is $K_w(w)$ on the fiber
at $w.$

Fix an operator $T$ in the Cowen-Douglas class $B_1(\Omega).$ This
is the same as fixing a Hilbert space $\mathcal H$ of holomorphic
functions on $\Omega$ and a positive definite kernel $K,$ which is
holomorphic in the first variable and anti-holomorphic in the
second, on $\Omega.$ Then the operator $T$ is unitarily equivalent
to $\boldsymbol M^*.$ Let $\mathcal N(w) = \cap_{i=1}^m\ker {(M_i
- w_i)^*}^2$ and let $N_i(w)$ be the commuting tuple of finite
dimensional operators obtained by restricting $M_{i}^{*}$ to
$\mathcal N(w),$ $i=1,\ldots ,m.$ The commuting tuple $N(w)$ is of
the form
$$ \Big (\, \Big (\begin{matrix} \bar{w}_1 & \mathbf v_1 \\ 0 & \bar{w}_1 I\end{matrix} \Big ) , \ldots ,
\Big ( \begin{matrix} \bar{w}_m & \mathbf v_m \\ 0 & \bar{w}_m
I\end{matrix} \Big )\, \Big ),$$ it is the localization of $T$ at
$w.$ These pairwise commuting operators induce a homomorphism
$\rho_V$ except that the $\mathbf v_1, \ldots , \mathbf v_m$ are
of size $1\times m.$ (It is also possible to add several rows of
zeros
 to each of these vectors making them $m\times m$ matrices.) It is easy to show that the
 $(m+1)$ dimensional space $\mathcal N(w)$ is spanned by the vectors
 $\{K_w, \bar{\partial_1} K_w, \ldots , \bar{\partial}_m K_w\}.$
 It therefore has a natural inner product, which it inherits from the Hilbert space $\mathcal H,$
 namely, $\langle \bar{\partial}_i K_w , \bar{\partial_j} K_w \rangle = (\partial_j\bar{\partial_i} K_w) (w),$
$i,j=0,1,\ldots, m,$ where $\bar{\partial}_0 K_w:=K_w.$ The
curvature of the line bundle is a $(1,1)$ form given by the
formula $\sum_{i,j=1}^m \frac{\partial^2}{\partial
w_i\partial\bar{w}_j}\log\|\gamma(w)\|^2 dw_i \wedge d\bar{w}_j.$
We let $\mathcal K(w)$ denote the matrix of the coefficients of
the curvature $(1,1)$ form.

There is a close relationship between the operators $N_1(w), \ldots , N_m(w)$ and the curvature $\mathcal K(w),$ namely,
$$
-\big (\!\! \big ({\rm tr} N_i(w) N_j(w)^* \big )\!\!\big )^{\rm
t} = \mathcal K(w)^{-1}.
$$
This relationship was derived in \cite{cowen} for $m=1$ and  in \cite{douglas} for $m=2$.

Suppose $K$ is a positive definite kernel. Then for any natural
number $n,$ the kernel $K^n(z,w)$,
 the point-wise product of the kernel $K,$ is  positive definite.  This is no longer true if we
 replace the natural numbers $n$ by positive real numbers $\lambda.$ However, we show that
%$\langle \bar{\partial}_i K^\lambda_w , \bar{\partial_j} K^\lambda_w \rangle := $
$\big (\!\! \big (\, (
\partial_j\bar{\partial}_i K^\lambda)(w,w)\big )\!\!\big )_{i,j=0}^m$
is positive definite for all positive real numbers $\lambda$ and
therefore it defines an inner product on the space $\mathcal
N^{(\lambda)}(w),$ the linear span of the vectors
$\{K^\lambda(\cdot,w), \bar{\partial_1} K^\lambda(\cdot,w), \ldots
, \bar{\partial}_m K^\lambda(\cdot,w)\}.$ We conclude that the
first order jet bundle determined by these vectors possesses a
non-degenerate Hermitian inner product. The Hermitian metric
induced  on the jet bundle of order $k$ by the kernel $K^\lambda$
( $\lambda >0$) need not be non-degenerate in general for $k>1.$

We define, for $i=1,\ldots , m,$ the operators $N^{(\lambda)}_i(w)$ on $\mathcal N^{(\lambda)}(w)$ by the rule
$$
\big (N^{(\lambda)}_i(w) - \bar{w}_iI_{m+1}\big ) (\partial_i K^\lambda)(\cdot, w)
= \begin{cases} K^\lambda(\cdot,w) & \mbox{if}\,\, i\not = 0 \\ 0 & \mbox{if}\,\, i=0\end{cases}.
$$
These are pairwise commuting nilpotent operators. However, they
need not be the localization of some operator in $B_1(\Omega)$
unless $\lambda$ is a natural number. We study the contractivity
(resp. complete contractivity) properties of the homomorphism
induced by the operators $N^{(\lambda)}$ starting with a fixed
operator $T$ in $B_1(\Omega).$ The contractivity properties of the
homomorphism induced by the localization operators is equivalent
to a curvature inequality. We study the Bergman kernel of the
matrix unit ball and some of its open subsets. This provides
examples to show that the curvature inequality does not
necessarily imply the stronger inequality $\|p(T) \| \leq
\|p\|_\infty $ for all polynomials $p$ in $m$ variables.

\chapter{Biholomorphic equivalence}
\section{ Linear equivalence}
We describe a natural class of domains in $\mathbb C^m$ which
admit an isometric embedding into the normed linear space
$(\mathcal M_n(\mathbb C), \|\cdot\|_{\rm op})$, where $
\|\cdot\|_{\rm op}$ denotes the operator norm on the space of
$n\times n$ complex matrices. For any $m$-tuple of matrices
$\mathbf A=(A_1,\ldots ,A_m)$ in $\mathbb C^m \otimes \mathcal
M_n(\mathbb C)$, let
\begin{eqnarray*}
\Omega_\mathbf A & := &\{(w_1 ,w_2, \ldots, w_m) :\|w_1 A_1
+\cdots + w_mA_m \|_{\rm op} < 1\}.
\end{eqnarray*}
Clearly, $\Omega_\mathbf A = (\mathbb C^m, \|\cdot \|_\mathbf
A)_1$ is the unit ball in $\mathbb C^m$ with respect to some norm.
Similarly, let $\Omega_{\tilde{\mathbf A}}$ be the ball in
$\mathbb C^m$ defined by the  $m$-tuple of matrices
$\tilde{\mathbf A}=(\tilde{A}_1, \ldots , \tilde{A}_m)$ in
$\mathbb C^m \otimes \mathcal M_n(\mathbb C),$ that is,
\begin{eqnarray*}
\Omega_{\tilde{\mathbf A}} & := &  \{(z_1 ,z_2, \ldots, z_m)
:\|z_1\tilde{A}_1 +\cdots+ z_m\tilde{A}_m \|_{\rm op} < 1\}.
\end{eqnarray*}
Again, $\Omega_{\tilde{\mathbf A}} = (\mathbb C^m, \|\cdot
\|_{\tilde{\mathbf A}})_1$ with respect to some norm
 $\|\cdot\|_{\tilde{\mathbf A}}$.
\begin{proposition}\label{biholo}
The two domains $\Omega_{\tilde{\mathbf A}}$ and $\Omega_{\mathbf
A}$ are bi-holomorphic  via an invertible linear map
 $R:\mathbb C^m \to \mathbb C^m$ if and only if  $(R\otimes I)(\mathbf A)=\tilde{\mathbf A}.$
\end{proposition}
\begin{proof}
Suppose ${\Omega}_{\tilde{\mathbf A}}$ is biholomorphic to
$\Omega_{\mathbf A}.$ Let $e_1, \ldots , e_m$ be the standard
basis for $\mathbb C^m$. Let $R:\mathbb C^m\rightarrow \mathbb C^m
$ be a linear map. Set $w:=Rz,$ $z\in \mathbb C^m.$  Since
$\Omega_{\tilde{\mathbf A}}$ is biholomorphic  to $\Omega_{\mathbf
A},$ via the invertible linear map $R$, it follows that
\begin{align*}
(R\otimes I)(\mathbf A) & =\big (R\otimes I \big )(e_1\otimes
A_1+\ldots+e_m\otimes A_m)\\& =R(e_1)\otimes
A_1+\ldots+R(e_m)\otimes A_m\\&=
(R_{11}e_1+\ldots+R_{m1}e_m)\otimes
A_1+\ldots+(R_{1m}e_1+\ldots+R_{mm}e_m)\otimes A_m \\&=e_1\otimes
(R_{11}A_1+\ldots+R_{1m}A_m)+\ldots+e_m \otimes
(R_{m1}A_1+\ldots+R_{mm}A_m)\\& = e_1\otimes
\tilde{A_1}+\ldots+e_m\otimes \tilde{A_m},
\end{align*}where $\tilde{A_i} =\sum_{j=1}^{m}R_{ij}A_{j}.$
This shows that $(R\otimes I)(\mathbf A) = \tilde{\mathbf A}.$

Conversely, assume that  $(R\otimes I)(\mathbf A)=\tilde{\mathbf
A}$ for some invertible linear map $R:\mathbb C^m \to \mathbb
C^m$, that is, $\tilde{A_i} =\sum_{j=1}^{m}R_{ij}A_{j}.$ Now,
$\sum_{i=1}^{m}z_i\tilde{A_i}=\sum_{i=1}^{m}z_i\sum_{j=1}^{m}R_{ij}A_{j}=\sum_{j=1}^{m}(\sum_{i=1}^{m}R_{ij}z_i)A_{j}
=\sum_{j=1}^{m}w_{j}A_{j}$, where $\sum_{i=1}^{m}R_{ij}z_i=w_{j}.$
Since $R$ is invertible, it follows that $\Omega_{\tilde{\mathbf
A}}$ is  bi-holomorphic to $\Omega_{\mathbf A}$ via the linear map
$R.$
\end{proof}

This Proposition prompts the following Definition.
\begin{definition}The $m$-tuple of matrices $\mathbf A=(A_1,\ldots ,A_m)$ is equivalent to another
$m$-tuple of matrices $\tilde{\mathbf A}=(\tilde{A}_1, \ldots ,
\tilde{A}_m)$ if there exist a invertible linear map $R:\mathbb
C^m \to \mathbb C^m$  such that $(R\otimes I)(\mathbf
A)=\tilde{\mathbf A}.$ Thus $\mathbf A$ and $\tilde{\mathbf A}$
belong to the same equivalence class if and only if
$\tilde{\mathbf A}_i$ is in the span of $\{A_1,\ldots ,A_m\}$ for
each $i.$
\end{definition}
\subsection{Examples}
\begin{example}
Let $\mathbb D^2=\{(z_1, z_2): \max\{|z_1|, |z_2|\}<1\}.$ Then
$\mathbb D^2$ is of the form $\Omega_{\mathbf A},$ where $\mathbf
A=\left(\left ( \begin{smallmatrix}
1 & 0   \\
0  &  0
\end{smallmatrix}\right ) ,
\left ( \begin{smallmatrix}
0&  0 \\
0  &  1
\end{smallmatrix}\right )\right).$

Pick $a,b,c,d$ in $\mathbb C$ with the property that $\det \left (
\begin{smallmatrix} a & c \\ b & d\end{smallmatrix}\right ) \not =
0.$ Then the pair of  $2\times 2 $ matrices   $\tilde{\mathbf
A}=\left(\left ( \begin{smallmatrix}
a & 0   \\
0  &  b
\end{smallmatrix}\right ) ,
\left ( \begin{smallmatrix}
c&  0 \\
0  &  d
\end{smallmatrix}\right )\right)$ defines a domain $\Omega_{\tilde{\mathbf A}}$ in $\mathbb C^2$ bi-holomorphic
to $\Omega_\mathbf A$ via the linear map $R=\left (
\begin{smallmatrix}
a&  c \\
b  &  d
\end{smallmatrix}\right ).$
\end{example}

\begin{example} The Euclidean ball $\mathbb B^2=\{(z_1, z_2): |z_1|^2+
|z_2|^2<1\}$ is determined by the pair $\mathbf A=\left(\left (
\begin{smallmatrix}
1 & 0   \\
0  &  0
\end{smallmatrix}\right ) ,
\left ( \begin{smallmatrix}
0&  1\\
0  &  0
\end{smallmatrix}\right )\right)$,
which is equivalent to $\tilde{\mathbf A}=\left(\left (
\begin{smallmatrix}
a & c   \\
0  &  0
\end{smallmatrix}\right ) ,
\left ( \begin{smallmatrix}
b&  d \\
0  &  0
\end{smallmatrix}\right )\right)$ for any choice of $a,b,c$ and $d$ in $\mathbb C$ with
$\det \left ( \begin{smallmatrix} a & b \\ c & d\end{smallmatrix}\right ) \not = 0.$

The biholomorphic equivalent copy of the ball $\mathbb B^2$ is the
ellipsoid:
$$\Omega_{\tilde{\mathbf A}}=\{(z_1,z_2): |(a+b) z_1|^2 + |(c+d) z_2|^2 <1\}.$$
\end{example}

\begin{example}
Let $\Omega_{\mathbf A}=\{(z_1, z_2): |z_1|^2+ |z_2|<1\},$ where
$\mathbf A=\left(\left ( \begin{smallmatrix}
1 & 0   \\
0  &  1
\end{smallmatrix}\right ) ,
\left ( \begin{smallmatrix}
0&  1\\
0  &  0
\end{smallmatrix}\right )\right).$ The domain $\Omega_{\mathbf A}$ is biholomorphic
$\Omega_{\tilde{\mathbf A}}$ for any pair

$\tilde{\mathbf A}=\left(\left ( \begin{smallmatrix}
a & b   \\
0  &  a
\end{smallmatrix}\right ) ,
\left ( \begin{smallmatrix}
c&  d \\
0  &  c
\end{smallmatrix}\right )\right),$  $a,b,c,d\in \mathbb C$ with
 $\det \left ( \begin{smallmatrix} a & b \\ c & d\end{smallmatrix}\right ) \not = 0.$
\end{example}

\begin{corollary}
A domain $\Omega_\mathbf A$ in $\mathbb C^2$ is bi-holomorphic to
$\Omega_{\tilde{\mathbf A}},$ where $\tilde{ A}_1=p \left (
\begin{smallmatrix}
d_{1} & 0  \\
0  &  d_{2}
\end{smallmatrix}\right ) +q\left ( \begin{smallmatrix}
a&  b\\
c &  d
\end{smallmatrix}\right ),$ $ \tilde{A}_2=
r \left ( \begin{smallmatrix}
d_{1} & 0  \\
0  &  d_{2}
\end{smallmatrix}\right )+ s \left ( \begin{smallmatrix}
a&  b\\
c  &  d
\end{smallmatrix}\right )$
and  the equivalence is implemented via the linear map $R=\left (
\begin{smallmatrix}
p &  q \\
r  & s
\end{smallmatrix}\right ),$ which is assumed to be invertible.
\end{corollary}
\begin{proof}
Let $\Omega_\mathbf A$ be  a domain in $\mathbb C^2$ determined by
some pair of $2 \times 2$ matrices, say $\mathbf A=(A_1,A_2)$.
Clearly, if $U,V$ are unitaries on $\mathbb C^2$, then the pair
$(UA_1V, UA_2V)$ determines the same set $\Omega_\mathbf A$. So,
we may assume without loss of generality that $\mathbf A$ is of
the form $\left(\left ( \begin{smallmatrix}
d_{1} & 0   \\
0  &  d_{2}
\end{smallmatrix}\right ) ,
\left ( \begin{smallmatrix}
a&  b\\
c  &  d
\end{smallmatrix}\right )\right).$ The proof is completed by appealing to Proposition \ref{biholo}.
\end{proof}
As a consequence of the above corollary, we will prove the
following corollary.

\begin{corollary}
 Let $\mathbb A_1$ be of the form
 $\left ( \begin{smallmatrix}
1 & 0   \\
0  & d_{2}
\end{smallmatrix}\right )$ or $\left ( \begin{smallmatrix}
d_{1} & 0   \\
0  & 1
\end{smallmatrix}\right )$ and  $\mathbb A_2 $ be of the form  $\left ( \begin{smallmatrix}
 0 &  b \\
c &  0
\end{smallmatrix}\right ) ,\left ( \begin{smallmatrix}
1 &  b \\
c &  0
\end{smallmatrix}\right ) $ or $\left ( \begin{smallmatrix}
 0 &  b \\
 c &  1
\end{smallmatrix}\right ) $ with one of $b$ or $c$ positive real.
Any domain of the form $\Omega_\mathbf A$ in $\mathbb C^2$ is
bi-holomorphically equivalent to $\Omega_\mathbb A.$
\end{corollary}
\begin{proof}
Since $d_1$ and $d_2$ are not simultaneously zero, we let
$p=\frac{1}{d_{1}}$ or $p=\frac{1}{d_{2}}.$ If we choose $q=0,$
 then we have $\mathbb A_1=\left (
\begin{smallmatrix}
1 & 0   \\
0  & d_{2}
\end{smallmatrix}\right ),$  $\left ( \begin{smallmatrix}
d_{1} & 0   \\
0  & 1
\end{smallmatrix}\right ).$

Now, suppose $d_2 \neq 0.$  Choose $r=-\frac{d}{d_{2}}s$ then
there are two possibilities.
\begin{enumerate}
\item[(i)] If  $\det\left (
\begin{smallmatrix}
d_{1} & a   \\
d_{2}  & d
\end{smallmatrix}\right )=0,$ then we will get $\mathbb A_2=\left ( \begin{smallmatrix}
 0 &  b \\
c &  0
\end{smallmatrix}\right ).$

\item[(ii)] If $\det\left ( \begin{smallmatrix}
d_{1} & a  \\
d_{2}  & d
\end{smallmatrix}\right )\neq0$, that is, $s=\frac{1}{d_{1}d-d_{2}a},$ then $A_2=\left ( \begin{smallmatrix}
1 &  b\\
c &  0
\end{smallmatrix}\right ).$
\end{enumerate}
Similarly, if we assume $d_1\neq 0,$ then we may assume $A_2=\left
(\begin{smallmatrix}
0 &  b \\
c &  1
\end{smallmatrix}\right ).$

If we conjugate  $\mathbb A_1, \mathbb A_2$ by a diagonal unitary
$U=\left (
\begin{smallmatrix}
\exp(i\theta)&  0\\
0 & \exp(i\phi)
\end{smallmatrix}\right ),$ then we may assume one of $b$ or $c$ is positive real.
\end{proof}
\begin{example}
The upper triangular matrices in the unit ball of $\mathcal M_2$
corresponds to

$\mathbf A=\left(\left ( \begin{smallmatrix}
1 & 0   \\
0  &  0
\end{smallmatrix}\right ),
\left ( \begin{smallmatrix}
0 &  1 \\
0  & 0
\end{smallmatrix}\right ),  \left ( \begin{smallmatrix}
0 & 0   \\
0  & 1
\end{smallmatrix}\right )\right),$ that is, $\Omega_{\mathbf A}=\{(z_1, z_2, z_3):\|\left ( \begin{smallmatrix}
z_1 & z_2   \\
0  & z_3
\end{smallmatrix}\right )\| <1\}.$

This domain is biholomorphic to $\Omega_{\tilde{\mathbf A}},$
where $\tilde{\mathbf A}=\left(\left ( \begin{smallmatrix}
a_{1} & b_{1}   \\
0  & c_{1}
\end{smallmatrix}\right ),
\left ( \begin{smallmatrix}
a_{2} & b_{2}   \\
0  & c_{2}
\end{smallmatrix}\right ),  \left ( \begin{smallmatrix}
a_{3} & b_{3}   \\
0  & c_{3}
\end{smallmatrix}\right )\right)$ for any choice of $a_ {i},  b_ {i},
c_ {i} \in \mathbb C$, $i=1, 2, 3,$ with  $\det \left (
\begin{smallmatrix}
a_ {1} & a_{2} & a_{3} \\
b_ {1} & b_{2} & b_{3} \\
c_ {1} & c_{2} & c_{3}
\end{smallmatrix} \right )\not = 0$.
\end{example}

\begin{example}
The unit ball in $(\mathcal M_2, \|\cdot\|_{\rm op})$ corresponds
to the choice:
$$\mathbf A=\left(\left ( \begin{smallmatrix}
1 & 0   \\
0  &  0
\end{smallmatrix}\right ),
\left ( \begin{smallmatrix}
0 &  1 \\
0  & 0
\end{smallmatrix}\right ), \left ( \begin{smallmatrix}
0 & 0   \\
1 & 0
\end{smallmatrix}\right ), \left ( \begin{smallmatrix}
0 & 0   \\
0 & 1
\end{smallmatrix}\right )\right).$$
Pick $a_ {i} , b_{i}  ,c_{i},  d_{i}$,  $i=1, 2, 3, 4$, in
$\mathbb C$ such that the determinant of $R=\left (
\begin{smallmatrix}
a_ {1} & a_{2} & a_{3}& a_{4} \\
b_ {1} & b_{2} & b_{3} & b_{4} \\
c_ {1} & c_{2} & c_{3} & c_{4} \\
d_ {1} & d_{2} & d_{3} & d_{4}
\end{smallmatrix} \right )$ is not zero and let
$$\tilde{\mathbf A}=\left(\left ( \begin{smallmatrix}
a_{1} &  b_{1} \\
c_{1}  & d_{1}
\end{smallmatrix}\right ),
\left ( \begin{smallmatrix}
a_{2} &  b_{2} \\
c_{2}  & d_{2}
\end{smallmatrix}\right ), \left ( \begin{smallmatrix}
a_{3} &  b_{3} \\
c_{3}  & d_{3}
\end{smallmatrix}\right ), \left ( \begin{smallmatrix}
a_{4} &  b_{4} \\
c_{4}  & d_{4}
\end{smallmatrix}\right )\right).$$
Then the unit ball in $(\mathcal M_2, \|\cdot\|_{\rm op})$ is
biholomorphic to $\Omega_{\tilde{\mathbf A}}$ via the linear map
induced by $R$ on $\mathbb C^4.$
\end{example}
\section{Carath$\acute{e}$odory norm and contractive Homomorphisms}

 We recall the definition of Carath$\acute{e}$odory norm from  Jarnicki and Pflug(cf.
\cite{pflug}).
\begin{definition}For any $\mathbf v \in \mathbb C^m$ and $\Omega$ a domain in $\mathbb C^m,$ the
Carath$\acute{e}$odory norm $\mathcal C_{\Omega, w}(\mathbf v)$ of
the vector $\mathbf v$  at $w\in \Omega$ is defined to be the
extremal quantity
$$ \sup_{f}\{|f^{\prime}(w)\mathbf v|:f\in {\rm Hol}(\Omega , \mathbb D),
f(w)=0\}.$$ Let $\mathbb B$ be the open unit ball in $\mathbb C^m$
with respect to some norm, say $\|\,\cdot\,\|_{\mathbb B}$ in
$\mathbb C^{m}$. Thus  $\mathbb B$ is the open set

$(\mathbb C^m,  \|\,\cdot\,\|_{\mathbb B})_1 =\{(z_1, \ldots,
z_m)\in \mathbb C^{m}: \|(z_1, \ldots, z_m)\|_{\mathbb B}<1\}.$
\end{definition}
\begin{proposition}For any holomorphic function $f:\mathbb B \rightarrow \mathbb D$ with $f(0)=0$ and a vector
$\mathbf v\in \mathbb C^m,$ we have
$$|\sum_{i=1}^{m}(\partial_{i}f(0))v_{i}|\leq \|\mathbf v\|_{\mathbb B}.$$
\end{proposition}
\begin{proof}Let $g_\mathbf v:\mathbb D\rightarrow \mathbb B$ be the holomorphic function defined by
$$g_\mathbf v(\lambda)=\lambda\frac{\mathbf v}{\|\mathbf v\|_{\mathbb B}}, \lambda \in \mathbb D.$$
The Schwartz Lemma for the unit disc now applies to the function
$f \circ g_\mathbf v$ and gives
$$1\geq |(f \circ g_\mathbf v)^{\prime}(0)|=|f^{\prime}(g_\mathbf v(0))g^{\prime}_\mathbf v(0)|=
|f^{\prime}(0)\frac{\mathbf v}{\|\mathbf v\|_{\mathbb
B}}|=\frac{|f^{\prime}(0)\mathbf v|}{\|\mathbf v\|_{\mathbb B}}$$
completing the proof.
\end{proof}

Let $\Omega \subset \mathbb C^m$ be an open bounded and  connected
set, $w \in \Omega .$  Let
$$\mathcal D_{\Omega, w}=\{f^{\prime}(w):f\in {\rm Hol}(\Omega , \mathbb D), f(w)=0\} \subseteq \mathbb C^m.$$
The Proposition merely says that $\mathcal D_{\mathbb B , 0}$ is a
subset of the dual unit ball $(\mathbb C^m,
\|\,\cdot\,\|^{*}_{\mathbb B})_1.$

\begin{corollary}\label{dualballsc}The set $\mathcal D_{\mathbb B , 0}$  is the dual unit ball
$(\mathbb C^m,  \|\,\cdot\,\|^{*}_{\mathbb B})_1.$
\end{corollary}
\begin{proof} Clearly, any $l \in (\mathbb C^m,  \|\,\cdot\,\|^{*}_{\mathbb B})_1$
defines a holomorphic function $l:\mathbb B \rightarrow \mathbb D$
with $l(0)=0$ and $l^{'}(0)=l.$
\end{proof}
The set $\mathcal D_{\Omega, w}$ is the unit ball with respect to
some norm(cf. \cite [Proposition 3.1]{vern} \cite[Theorem
1.1]{bagchi}). Except when $\Omega$ is the ball with respect to
some norm and $w=0,$ describing  the set $\mathcal D_{\Omega, w}$
appears to be a hard problem.

The set $\mathcal D_{\Omega, w}$ is determined by calculating the
Carath$\acute{e}$odory norm for the domain $\Omega.$ It is the
unit ball with respect to the norm dual to the
Carath$\acute{e}$odory norm. The explicit form of the
Carath$\acute{e}$odory norm is known, for instance, in the case of
the annulus in $\mathbb C$ (cf. \cite{pflug}).

Let $f:\Omega \rightarrow \tilde\Omega$ be a holomorphic map.
Define the push forward $f_{*}(\mathbf v)$ of a vector $\mathbf v$
under the function $f$ to be the vector $f^{\prime}(w)\mathbf v.$

\begin{lemma}$\mathcal C_{\tilde{\Omega}, f(w)}(f_{*}(\mathbf v)) \leq \mathcal C_{\Omega, w}(\mathbf v).$
\end{lemma}
\begin{proof} The proof is straightforward:
\begin{align*}
\mathcal C_{\tilde{\Omega}, f(w)}(f_{*}(\mathbf v)) & =
\sup\{|g^{\prime}(f(w))f_{*}(\mathbf v)|:g\in {\rm
Hol}(\tilde{\Omega} , \mathbb D), g(f(w))=0\}\\&=
\sup\{|g^{\prime}(f(w))f^{'}(w)(\mathbf v)|:g\in {\rm
Hol}(\tilde{\Omega} , \mathbb D), g(f(w))=0\}\\ &=\sup\{|(g\circ
f)^{\prime}(w)\mathbf v|:g\in {\rm Hol}(\tilde{\Omega} , \mathbb
D), g(f(w))=0\}\\ &\leq \sup\{|h^{\prime}(w)\mathbf v|:h\in {\rm
Hol}(\Omega , \mathbb D), h(w)=0\}.
\end{align*}
\end{proof}
\begin{corollary}Suppose $\mathbf v \in \mathbb C^m$ with $\mathcal C_{\Omega, w}(\mathbf v) \leq 1.$
Then for any holomorphic function
$F:\Omega\rightarrow (\mathcal M_k)_1$ with $F(w)=0,$ we have
$\mathcal C_{(\mathcal M_k)_{1}, 0}(F_{*}(\mathbf v))\leq \mathcal
C_{\Omega, w}(\mathbf v) \leq 1.$
\end{corollary}
If we pick $\mathbf v=(v_1, \ldots, v_m)\in \mathbb C^m$ with
$\mathcal C_{\Omega, w}(\mathbf v) \leq 1,$ then the commuting
tuple
$$N(\mathbf v, w):=\left(
\left ( \begin{smallmatrix}
w_1 & v_1   \\
0  & w_1
\end{smallmatrix}\right ),\cdots , \left ( \begin{smallmatrix}
w_m & v_m   \\
0  & w_m
\end{smallmatrix}\right )\right)$$
defines a contractive homomorphism of the algebra $\mathcal
O(\Omega).$ This homomorphism is then completely contractive. To
prove this, notice that the induced homomorphism $\rho$ is given
by the formula
$$\rho_{\mathbf v}(f)=\left ( \begin{smallmatrix}
f(w)& f^{\prime}(w) \mathbf v \\
0  & f(w)
\end{smallmatrix}\right ), f \in \mathcal O(\Omega).$$
 We may assume $f(w)=0$ without loss of
generality (cf. \cite[Lemma 3.3]{G} \cite[Lemma 4.1]{vern}). Hence
contractivity of $\rho_{\mathbf v}$ amounts to

$$\|\rho_{\mathbf v}\|=\sup_{f}\{|f^{\prime}(w)\mathbf v|:f\in {\rm Hol}(\Omega , \mathbb D), f(w)=0\}
=\mathcal C_{\Omega, w}(\mathbf v) \leq 1.$$
Let $F:\Omega \rightarrow \mathcal M_k$ be a holomorphic function.
Now, we have
$$\rho_{\mathbf v}^{(k)}(F):=(\rho_{\mathbf v}(F_{ij}))=\left ( \begin{smallmatrix}
F(w)& F^{\prime}(w) \mathbf v \\
0  & F(w)
\end{smallmatrix}\right ).$$
We may again assume, without loss of generality,  that $F(w)=0$
(cf. \cite[Lemma 3.3]{G}). Hence we have (by the norm decreasing
property of the Carath$\acute{e}$odory norm) that
\begin{align*}
\|\rho_{\mathbf v}^{(k)}\|&=\sup_{F}\{\|F^{\prime}(w)\mathbf
v\|_{\rm op}:F\in
{\rm Hol}(\Omega , (\mathcal M_k)_1), F(w)=0\}\\
&=\sup_{F}\{\mathcal C_{(\mathcal M_k)_{1}, 0}(F_{*}(\mathbf
v)):F\in {\rm Hol}(\Omega , (\mathcal M_k)_1), F(w)=0\}\\& \leq
\mathcal C_{\Omega, w}(\mathbf v).
\end{align*}
This shows that $\|\rho_{\mathbf v}^{(k)}\| \leq 1$ whenever
$\|\rho_{\mathbf v}\|=\mathcal C_{\Omega, w}(\mathbf v)\leq 1.$
Here we have made essential use of the two properties: (i) the
Carath$\acute{e}$odory norm decreases under holomorphic maps and
(ii) $\mathcal C_{(\mathcal M_k)_1,0}=\|\cdot\|_{\rm op}.$

Before we proceed any further, we note that the set
$$\mathcal D^{(k)}_{\Omega, w}:=\{DF(w):
 F\in {\rm Hol}(\Omega , (\mathcal M_k)_1), F(w)=0\}\subseteq \mathbb C^m \otimes \mathcal M_{k}$$
is the unit ball in $\mathbb C^m \otimes \mathcal M_{k}$ with
respect to some norm, say, $ \|\cdot\|_{k}^* $ (cf. \cite{vern}).
Thus contractivity of $\rho_{\mathbf v}^{(k)},$ in this case, is
the same as the contractivity of the linear map
$$L^{(k)}_{N(\mathbf v, w)}:(\mathbb C^m \otimes \mathcal M_{k},
\|\cdot\|_{k}^* )\rightarrow (\mathcal M_{k}, \|\cdot\|_{\rm
op})$$ given by the formula $L^{(k)}_{N(\mathbf v,
w)}(\Theta)=v_1\Theta_1+\cdots +v_m\Theta_m,$ where
$\Theta=(\Theta_1, \ldots , \Theta_m)\in\mathbb C^m \otimes
\mathcal M_{k}.$ We have shown that $L^{(k)}_{N(\mathbf v, w)}$ is
contractive for all $k>1,$ without knowing anything about the norm
$(\mathbb C^m \otimes \mathcal M_{k}, \|\cdot\|_{k}^*),$ as long
as we know it is contractive for $k=1.$

What happens if we pick $V=\left(V_1, \ldots, V_m\right)$ in
$\mathbb C^m \otimes \mathcal M_{p,q}$ and consider the
homomorphism induced by the commuting tuple of matrices:

$$N(V, w):=\left(
\left ( \begin{smallmatrix}
w_1I_p & V_1   \\
0  & w_1I_q
\end{smallmatrix}\right ),\cdots , \left ( \begin{smallmatrix}
w_mI_p & V_m   \\
0  & w_mI_q
\end{smallmatrix}\right )\right).$$

As before, the contractivity of the induced homomorphism is the
requirement that

$$\sup\{\|\left ( \begin{smallmatrix}
f(w)& f^{\prime}(w) V \\
0  & f(w)
\end{smallmatrix}\right )\|:f\in {\rm Hol}(\Omega , \mathbb D), f(w)=0\}\leq 1,$$ where $ f^{\prime}(w) V=
(\partial_1f)(w)V_1+ \cdots +(\partial_mf)(w)V_m.$ Again we
assume, without loss of generality, that $ f(w)=0.$ Therefore the
contractivity of $\rho_{V}$ is equivalent to the contractivity of
the linear operator
$$L_{N(V, w)}:(\mathbb C^m , \mathcal C_{\Omega, w})^*\rightarrow (\mathcal M_{p,q}, \|\cdot\|_{\rm op}),$$
 where $L_{N(V, w)}(\theta)=V_1\theta_1+\cdots +V_m\theta_m$ for $\theta \in \mathbb C^m.$ For
 a holomorphic function $F:\Omega \rightarrow \mathcal M_k,$ we have

$$\rho_{V}^{(k)}(F):=(\rho_{V}(F_{ij}))=\left ( \begin{smallmatrix}
F(w)\otimes I & F^{\prime}(w) V \\
0  & F(w)\otimes I
\end{smallmatrix}\right ),$$ where  $F^{\prime}(w) V=(\partial_1F)(w)\otimes V_1+ \cdots +(\partial_mF)(w)\otimes V_m.$

For one final time, assume $F(w)=0,$ without loss of generality
(cf. \cite[Lemma 3.3]{G}). Hence
$$\|\rho_{V}^{(k)}\|:=\sup_{F}\{\|F^{\prime}(w)V\|_{\rm op}:F\in {\rm Hol}(\Omega , (\mathcal M_k)_1), F(w)=0\}.$$
This is the norm of the linear operator
$$L^{(k)}_{N(V, w)}:(\mathbb C^m \otimes \mathcal M_{k}, \|\cdot\|_{k}^* )\rightarrow
(\mathcal M_{k}\otimes \mathcal M_{p,q} , \|\cdot\|_{\rm op})$$
given by the formula $L^{(k)}_{N(V, w)}(\Theta)=\Theta_1\otimes
V_1+\cdots +\Theta_m\otimes V_m$ for $\Theta=(\Theta_1, \ldots ,
\Theta_m)\in\mathbb C^m \otimes \mathcal M_{k}.$ Clearly
$L^{(1)}_{N(V, w)}=L_{N(V, w)},$ which we have already
encountered. We will attempt to determine whether $\|L_{N(V,
w)}\|\leq 1$ implies that $\|L^{(k)}_{N(V, w)}\|\leq 1$ for $k>1.$

When $V$ is in $\mathbb C^m \otimes \mathbb C,$ this is easily
done as we have seen, by using the two basic properties of the
Carath$\acute{e}$odory norm listed above. However, in general, it
is the following question: Given that
$$\|\theta_1\otimes V_1+\cdots +\theta_m\otimes V_m\| \leq 1$$ for
$\theta=(\theta_1, \ldots, \theta_m) \in \mathcal D_{\Omega, w},$
does it follow that
$$\|\Theta_1\otimes V_1+\cdots +\Theta_m\otimes V_m\| \leq 1$$ for
 $\Theta=(\Theta_1, \ldots , \Theta_m)\in \mathcal D^{(k)}_{\Omega, w}$ for all $k>1?$

%Let $\mathcal L^{(k)}_{\Omega, w}\subseteq \mathbb C^m \otimes
%\mathcal M_{k} $ be the set

%$$\mathcal L^{(k)}_{\Omega, w}=\{L:(\mathbb C^m, \mathcal C_{\Omega, w})\to (\mathcal M_k, {\rm op})  : \|L\| \leq
%1\},$$ where $L:(\mathbb C^m, \mathcal C_{\Omega, w})\to (\mathcal
%M_k, {\rm op})$ is the linear map. Recall that if $\Omega$ is the
%unit ball in $\mathbb C^m$ for some norm $\|\cdot\|$ and $w=0$,
%then $\mathcal C_{\Omega, 0} =\|\cdot\|.$ Consequently, in this
%case , the set $\mathcal L^{(k)}_{\Omega, 0}$ coincides with the
%set  $D^{(k)}_{\Omega, 0}.$ However, they don't, in general. One
%obtains such an example by taking $\Omega=\{(z_1, z_2):
%|z_1|^2+|z_2 | \leq 1\}$ and $w \in \Omega$ sufficiently away from
%$(0, 0).$

%{\bf Open Question:} We always have the inclusion
%$D^{(k)}_{\Omega, w}\subseteq \mathcal L^{(k)}_{\Omega, w}.$ When
%do we have equality?

\subsection{Invariance of $L^{(k)}_{N(V, w)}, k\geq 1$ under bi-holomorphic
maps}

Let $\varphi:\widetilde{\Omega} \to \Omega $ be the bi-holomorphic
map with $\varphi(w)=z.$ The linear map $D\varphi(w):(\mathbb C^m,
\mathcal C_{\widetilde{\Omega}, w})\to (\mathbb C^m, \mathcal
C_{\Omega, z})$ is a contraction by definition. Since $\varphi$ is
invertible, $D\varphi^{-1}(z):(\mathbb C^m, \mathcal C_{\Omega,
z})\to (\mathbb C^m, \mathcal C_{\widetilde{\Omega}, w})$ is also
a contraction. However, since $D\varphi^{-1}(z)=D\varphi(w)^{-1}
,$ it follows that $D\varphi(w)$ must be an isometry.  The map
$F\rightarrow F\circ\varphi $ is a bijection from ${\rm
Hol}_{z}(\Omega, (\mathcal M_k)_1)$ onto ${\rm
Hol}_{w}(\widetilde{\Omega}, (\mathcal M_k)_1).$ Therefore for
each $w$ in $\widetilde{\Omega}$ and a bi-holomorphic $\varphi$
from $\widetilde{\Omega}$ to $\Omega$ such that $\varphi(w)=z$ we
have
$$\{DF(z):F\in {\rm Hol}(\Omega, (\mathcal M_k)_1),
F(z)=0\}=\{DF\circ\varphi(w):F\in {\rm Hol}(\widetilde{\Omega},
(\mathcal M_k)_1), F\circ\varphi(w)=0\}.$$ Set $D\varphi(w):=\left
(\begin{smallmatrix}\varphi_{11}& \varphi_{12} &\ldots
\varphi_{1m}\\
\vdots &\vdots  &\vdots\\
\varphi_{m1}& \varphi_{m2} &\ldots
\varphi_{mm}\end{smallmatrix}\right)$ and $DF(z)=(A_1, \ldots,
A_m).$ By the chain rule we have
$$DF(\varphi(w))D\varphi(w)=(\varphi_{11}A_1+\cdots+\varphi_{m1}A_m,
\ldots,\varphi_{1m}A_1+\cdots+\varphi_{mm}A_m ).$$  Thus
$D\varphi(w)\otimes I_{k}$ maps $(\mathbb C^m \otimes \mathcal
M_{k},\|\cdot\|_{\Omega, k}^*)$ onto $(\mathbb C^m \otimes
\mathcal M_{k},\|\cdot\|_{\widetilde{\Omega}, k}^*).$ Since
$D\varphi(w)$ is an isometry, it follows that $D\varphi(w)\otimes
I_{k}$ is an isometry with respect to the two norms
$\|\cdot\|_{\Omega, k}^*$ and $\|\cdot\|_{\widetilde{\Omega},
k}^*.$ Let $L_{V}:(\mathbb C^m,\mathcal C_{\widetilde{\Omega},
w})^{*} \rightarrow (\mathcal M_{p,q}, \|\cdot\|_{\rm op})$ be the
linear map induced by $V=\left( V_1^{\rm t}, \ldots, V_m^{\rm
t}\right)^{\rm t},$ where $\mathcal C_{\widetilde{\Omega}, w}$ is
the Carath$\acute{e}$odory norm of $\widetilde{\Omega}$ at a fixed
but arbitrary $w\in \widetilde{\Omega}.$ The linear map $L_{V}$
(which depends on the point $w$ in $\Omega$) is contractive if and
only if $L_{D\varphi(w).V}$ is contractive, where
$D\varphi(w).V=(D\varphi(w)\otimes I)(V).$

\begin{lemma}\label{lemmcon3}
$$\|\,L_{V}\,\|_{(\mathbb C^m, \mathcal C_{\widetilde{\Omega},
w} )^*\rightarrow (\mathcal M_{p,q}, \|\cdot\|_{\rm op})} \leq 1$$
if and only if
$$\|\,L_{(D\varphi(w)\otimes I)(V)}\,\|_{(\mathbb C^m, \mathcal
C_{\Omega, z})^*\rightarrow (\mathcal M_{p,q}, \|\cdot\|_{\rm
op})} \leq 1.$$
\end{lemma}
\begin{proof}
Let $q :\widetilde{\Omega}\longmapsto \mathbb D$ be a holomorphic
map with  $q(w)=0$ and $\|q\|_{\infty,\mathbb D} \leq 1,$ then
$Dq(w)$ is in $\mathcal D^{1}_ {\widetilde{\Omega}, w}.$ Thus
$\|\,L_{V}\,\|_{(\mathbb C^m, \mathcal C_{\widetilde{\Omega}, w}
)^*\rightarrow (\mathcal M_{p,q}, \|\cdot\|_{\rm op})} \leq 1$ is
equivalent to $|\langle Dq(w), V\rangle|\leq 1.$ Similarly, if $p
: \Omega\longmapsto \mathbb D$ is a holomorphic map with $p(z)=0,$
then $Dp(z)$ is also in $\mathcal D^{1}_ {\Omega, z}.$ Consider
the commutative diagram
$$\begin{array}[c]{ccccc}
\widetilde{\Omega}&\stackrel{\varphi}{\rightarrow}&\Omega\\
&\searrow\scriptstyle{q}&\downarrow\scriptstyle{p}\\
&&\mathbb D.
\end{array}$$
Thus $q=p\circ\varphi.$ We therefore have   $Dq(w)=D
p(\varphi(w))D\varphi(w).$ Hence, we have {\small
\begin{align*}  \| L_{V} \| & = \sup_{Dq(w)\in \mathcal D^{1}_ {\widetilde{\Omega}, w}}\|\langle Dq(w),
V\rangle|\\&=\sup_{Dq(w)\in \mathcal D^{1}_ {\widetilde{\Omega},
w}}|\langle Dp(\varphi(w))D\varphi(w) , V\rangle|
\\&=\sup_{Dp(z)\in \mathcal D^{1}_ {\Omega, z}} |\langle \left( \partial_1p(\varphi(w), \ldots , \partial_mp(\varphi(w)
\right)\left (\begin{smallmatrix}\varphi_{11}& \varphi_{12}
&\ldots
\varphi_{1m}\\
\vdots &\vdots  &\vdots\\
\varphi_{m1}& \varphi_{m2} &\ldots
\varphi_{mm}\end{smallmatrix}\right), V\rangle|\\&=\sup_{Dp(z)\in
\mathcal D^{1}_ {\Omega, z}}|\langle D p(\varphi(w)),
(D\varphi(w)\otimes I)(V) \rangle|
\end{align*}}
completing the proof.
\end{proof}The following corollary is a consequence of Lemma
\ref{lemmcon3}.
\begin{corollary}
$$\|\,L_{V}\,\|_{(\mathbb
C^m,\|\,\cdot\,\|_{\widetilde{\Omega}})^*\rightarrow (\mathcal
M_{p,q}, \|\cdot\|_{\rm op})} \leq 1$$ if and only if
$$\|\,L_{(D\varphi(w)\otimes I)(V)}\,\|_{(\mathbb
C^m,\|\,\cdot\,\|_{\Omega})^*\rightarrow (\mathcal M_{p,q},
\|\cdot\|_{\rm op})} \leq 1.$$
\end{corollary}
\begin{proof}

We know that $\mathcal
C_{\widetilde{\Omega},0}(v)=\|v\|_{\widetilde{\Omega}}$ and
$\mathcal C_{\Omega,0}(v)=\|v\|_{\Omega}$. Using Lemma
(\ref{lemmcon3}) we get the desired result.
\end{proof}
We have indicated that contractivity of $\rho_{V}|^{(k)}$ is the
same as contractivity of the linear map
$$L^{(k)}_{N(V, w)}:(\mathbb C^m \otimes \mathcal M_{k}, \|\cdot\|_{k}^* )\rightarrow
(\mathcal M_{k}\otimes \mathcal M_{p,q} , \|\cdot\|_{\rm op}).$$
We recall from  \cite[Theorem 2]{harris}  that for $Z, W$ in the
matrix ball $(\mathcal M_k)_1$ and $\mathbf u\in \mathbb
C^{k\times k}$, we have
$$D\psi_W(Z) \cdot \mathbf u =  (I-W W^*)^{\frac{1}{2}}(I-ZW^*)^{-1} \mathbf u
(I- W^*Z)^{-1}(I-W^*W)^{\frac{1}{2}}.$$ The following Proposition
characterizes the contractivity of $\rho^{(k)}$  (cf. \cite[Lemma
3.3]{G}).
\begin{proposition}\label{prop 1}
$\|\rho_{V}^{(k)}(P_{\mathbf A})\|\leq 1$ if and only if
$$\sup_{z \in \Omega_{\mathbf A}}\|((I-P_{\mathbf A}(z)P_{\mathbf A}(z)^*)^{-\frac{1}{2}}\otimes I_n) (A_1\otimes
V_1+\cdots+A_m\otimes V_m)((I-P_{\mathbf A}(z)^*P_{\mathbf
A}(z))^{-\frac{1}{2}}\otimes I_n)\|\leq 1.$$
\end{proposition}
The homomorphism $\rho_{V}^{(k)}$ is contractive if and only if
$\rho_{(D\varphi(w)\otimes I)(V)}^{(k)}$ is contractive for all
$w$ in $\widetilde{\Omega}.$

\begin{proposition}\label{prop ccon4}
$$\|L^{(k)}_{N(V, w)}\|_{(\mathbb C^m \otimes \mathcal
M_{k},\|\cdot\|_{\widetilde{\Omega}, k}^*) \rightarrow (\mathcal
M_{k}\otimes \mathcal M_{p,q} , \|\cdot\|_{\rm op})}\leq1$$ if and
only if
$$\|L^{(k)}_{N((D\varphi(w)\otimes I)(V),z)}\|_{(\mathbb C^m \otimes \mathcal
M_{k}, \|\cdot\|_{\Omega, k}^*)\rightarrow (\mathcal M_{k}\otimes
\mathcal M_{p,q} , \|\cdot\|_{\rm op})}\leq1.$$
\end{proposition}
\begin{proof}
Let $P: \widetilde{\Omega} \rightarrow (\mathcal M_{k}(\mathbb
C))_{1}$ be a matrix valued polynomial on $\widetilde{\Omega}$
which is of the form $P(w)=w_1P_{1}+\cdots+w_mP_{m}.$ Then it is
easy to see that $DP(w)$ is in $\mathcal D^{k}_
{\widetilde{\Omega}, w}.$ By Proposition \ref{prop 1} we have
$\|L^{(k)}_{N(V, w)}\|_{(\mathbb C^m \otimes \mathcal
M_{k},\|\cdot\|_{\widetilde{\Omega}, k}^*) \rightarrow (\mathcal
M_{k}\otimes \mathcal M_{p,q} , \|\cdot\|_{\rm op})}\leq1$ if and
only if
$$\|(I-P(w)P(w)^*)^{-\frac{1}{2}}(P_{1}\otimes{V_1}+\cdots
+P_{m}\otimes{V_m})(I-P(w)^*P(w))^{-\frac{1}{2}}\|\leq 1.$$

Similarly, if $Q: \Omega\rightarrow (\mathcal M_{k}(\mathbb
C))_{1}$ is a matrix valued polynomial on $\Omega$ which is of the
form $Q(z)=z_1Q_{1}+\cdots+z_mQ_{m},$ then $DQ(z)$ is in $\mathcal
D^{k}_ {\Omega, z}.$ Consider the commutative diagram
$$\begin{array}[c]{ccccc}
\widetilde{\Omega}&\stackrel{\varphi}{\rightarrow}&\Omega\\
&\searrow\scriptstyle{P}&\downarrow\scriptstyle{Q}\\
&&\mathbb (\mathcal M_k)_1.
\end{array}$$
Thus $P=Q \circ\varphi.$ Hence we have $DP(w)= DQ
(\varphi(w))D\varphi(w).$ Hence we have
\begin{eqnarray*}\lefteqn{ \|L^{(k)}_{N(V, w)}\|_{(\mathbb
C^m \otimes \mathcal M_{k},\|\cdot\|_{\widetilde{\Omega}, k}^*)
\rightarrow (\mathcal M_{k}\otimes \mathcal M_{p,q} ,
\|\cdot\|_{\rm op})}}\\&\phantom{
}=&(I-P(w)P(w)^*)^{-\frac{1}{2}}(P_1\otimes{V_1}+\cdots
+P_m\otimes{V_m})(I-P(w)^*P(w))^{-\frac{1}{2}}\\&\phantom{
}=&r\big( (\varphi_{11}Q_1+\cdots+\varphi_{m1}Q_m)\otimes V_1+
\cdots+(\varphi_{1m}Q_1+\cdots+\varphi_{mm}Q_m )\otimes V_m
\big)s\\
&\phantom{}=& r( Q_1\otimes
(\varphi_{11}V_1+\cdots+\varphi_{1m}V_m)+\cdots+
Q_m\otimes(\varphi_{m1}V_1+\cdots+\varphi_{mm}V_m )
)s\\&=&\|L^{(k)}_{N((D\varphi(w)\otimes I)(V),z)}\|_{(\mathbb C^m
\otimes \mathcal M_{k}, \|\cdot\|_{\Omega, k}^*)\rightarrow
(\mathcal M_{k}\otimes \mathcal M_{p,q} , \|\cdot\|_{\rm op})},
\end{eqnarray*} where
$r=(I-Q(\varphi(w))Q(\varphi(w))^*)^{-\frac{1}{2}}$ and
$s=(I-Q(\varphi(w))^*Q(\varphi(w)))^{-\frac{1}{2}}.$ This
completes the proof.
\end{proof}
We have seen that if  $\widetilde{\Omega}$ and $\Omega$ are
bi-holomorphic then
$\|\rho_{V}^{(k)}\|=\|\rho_{(D\varphi(w)\otimes I)(V)}^{(k)}\|.$
Evidently we have the following corollary.
\begin{corollary}\label{bi-holo} Suppose $\varphi:\widetilde{\Omega} \to \Omega $ with $\varphi(0)=0,$ is a
bi-holomorphic map. Then every contractive linear map from
$(\mathbb C^m, \|\cdot\|_{\Omega})$ to $\mathcal M_n(\mathbb C)$
is completely contractive if and only if every contractive linear
map of $(\mathbb C^m, \|\cdot\|_{\tilde{\Omega}})$ to $\mathcal
M_n(\mathbb C)$ is completely contractive.
\end{corollary}
\begin{proof}From  Lemmas \ref{lemmcon3} and  Proposition \ref{prop ccon4} it follows
that every contractive linear map of  $(\mathbb C^m,
\|\cdot\|^*_{\Omega})$ is completely contractive if and only if
every contractive linear map of $(\mathbb C^m,
\|\cdot\|^*_{\tilde{\Omega}})$ is also completely contractive.
This property does not change if we replace a ball with the dual
ball completing the proof.
\end{proof}
This means in finding contractive homomorphisms which are not
completely contractive, we need not distinguish between balls
which are  bi-holomorphically equivalent ball (with $0$ as a fixed
point).
\begin{example}
In Example $2.3$ we have seen that $\mathbb D^2$ is bi-holomorphic
to $\Omega_{\tilde{\mathbf A}}$  via the linear map $R=\left (
\begin{smallmatrix}
a&  c \\
b  &  d
\end{smallmatrix}\right ),$ where  $\tilde{\mathbf
A}=\left(\left ( \begin{smallmatrix}
a & 0   \\
0  &  b
\end{smallmatrix}\right ) ,
\left ( \begin{smallmatrix}
c&  0 \\
0  &  d
\end{smallmatrix}\right )\right).$ By Corollary \ref{bi-holo} we
see that $\alpha_{\Omega_{\tilde{\mathbf A}}}=1.$
\end{example}

\chapter{Contractivity and complete contractivity -- some examples}
\section{Dual norm computation}
We have discussed  the class of  domains $\Omega_\mathbf A=\{(z_1,
z_2):\left\|z_1A_1+z_2A_2\right\|_{\rm op} < 1\}$ in $\mathbb
C^2,$ where $A_1=\left (
\begin{smallmatrix}
1 & 0   \\
0  & d_{2}
\end{smallmatrix}\right )$ or $\left ( \begin{smallmatrix}
d_{1} & 0   \\
0  & 1
\end{smallmatrix}\right )$ and  $A_2$ is one of $ \left ( \begin{smallmatrix}
 0 &  b \\
c &  0
\end{smallmatrix}\right ) ,\left ( \begin{smallmatrix}
1 &  b \\
c &  0
\end{smallmatrix}\right ) $ or $\left ( \begin{smallmatrix}
 0 &  b \\
 c &  1
\end{smallmatrix}\right ) $ with $b \in \mathbb R^{+}.$ We have
seen that the contractivity of the homomorphism  $\rho_{V}$ is
equivalent to the contractivity of the linear map $L_{V}:(\mathbb
C^2, \|\cdot\|^*_{\Omega_{\mathbf A}}) \to (\mathbb
C^2,\|\,\cdot\,\|_2).$ Thus if we know the norm
$\|\cdot\|^*_{\Omega_{\mathbf A}}$ dual to the norm
$\|\cdot\|_{\Omega_{\mathbf A}},$ then we may be able to compute
the norm of $L_{V}.$ However, computing the dual norm
$\|\cdot\|^*_{\Omega_{\mathbf A}}$ appears to be a hard problem.
The $\ell^{\infty}_{2}$ and $\ell^{2}_{2}$ unit balls are of the
form $\Omega_\mathbf A,$ and one knows the duals of these norms.
Let $X$ be the two dimensional normed linear space with respect to
the norm
$$\|(x, y)\|=\frac{|y|+\sqrt{|y|^2+4|x|^2}}{2}, (x, y)\in X.$$ The unit ball with respect to the
norm is equal to $\Omega_{\mathbf A}=\{(x, y):|y|+|x|^2< 1\},$
where $\mathbf A=(A_1, A_2)$ with $A_1=I_2,
A_2=\left(\begin{smallmatrix}
0 & 1\\
    0    & 0
\end{smallmatrix}\right).$ Therefore $\{I_2, \left(\begin{smallmatrix}
0 & 1\\
    0    & 0
\end{smallmatrix}\right)\}$ forms a basis of $X.$
It is also easy to see that $\{\frac{1}{2}I_2, \left(\begin{smallmatrix}
0 & 0\\
    1   & 0
\end{smallmatrix}\right)\}$ forms a dual basis of $X.$ We recall the definition
of annihilator from Rudin  \cite{Walter}. Suppose $\widetilde{Y}$
is a Banach space, $M$ is a subspace of $\widetilde{Y}$ and $N$ is
a subspace of $\widetilde{Y}^*.$ Neither $M$ nor $N$ are assuned
to be closed.
\begin{definition}
The annihilators of $M^\bot$ and ${}^\perp N$ are defined as
follows:
$$M^\bot=\{\tilde{y}^* \in \widetilde{Y}^*: \langle \tilde{y}, \tilde{y}^* \rangle=0\, \forall \,\tilde{y} \in M \},$$
$${}^\perp N=\{\tilde{y} \in \widetilde{Y}: \langle \tilde{y}, \tilde{y}^* \rangle=0\, \forall \,\tilde{y}^* \in N\}.$$
\end{definition}
If  $M$ is assumed to be a closed subspace of $\widetilde{Y},$
then the quotient $\widetilde{Y}/M$ is also a Banach space, with
respect to the quotient norm. The duals of $M$ and of
$\widetilde{Y}/M$ can be describe using the annihilator $M^\bot$
of $M.$ The following Theorem (cf. \cite[Theorem 4.9]{Walter})
describes this relation explicitly.
\begin{theorem}\label{rudin}
Let $M$ be a closed subspace of a Banach space $\widetilde{Y}.$
The Hahn-Banach theorem ensures that each  $m^*$ in $M^*$ extends
to a linear functional $\tilde{y}^* \in \widetilde{Y}^*.$ Define
$$\sigma m^*=\tilde{y}^*+M^\bot.$$ Then $\sigma$ is an isometric
isomorphism of $M^*$ onto $\widetilde{Y}^*/M^\bot.$
\end{theorem}

This Theorem is used in the first computation of the dual of the
normed linear space $X.$ We also obtain  the dual norm by a direct
computation.
\begin{theorem}The dual norm of $X$ is given by
\[ \|(\alpha,
\beta)\| = \left\{ \begin{array}{ll}
        \frac{|\alpha|^2+4|\beta|^2}{4|\beta|}  & \mbox{if $|\beta|\geq \frac{|\alpha|}{2} $};\\
        |\alpha| & \mbox{if $|\beta|\leq \frac{|\alpha|}{2}$}.\end{array} \right. \]

\end{theorem}
\begin{proof}\noindent \text{(First proof)} Let $X^*$ be the dual
of $X.$ It is easy to verify that the annihilator of $X$ is
spanned by the elements of the form $\{\left(\begin{smallmatrix}
1 & 0\\
    0    & -1
\end{smallmatrix}\right), \left(\begin{smallmatrix}
0 & 1\\
    0    & 0
\end{smallmatrix}\right)\},$ that is, $X^\bot= span\{\left(\begin{smallmatrix}
1 & 0\\
    0    & -1
\end{smallmatrix}\right), \left(\begin{smallmatrix}
0 & 1\\
    0    & 0
\end{smallmatrix}\right)\}.$ The set $\{\frac{1}{2}I_2, \left(\begin{smallmatrix}
0 & 0\\
    1   & 0
\end{smallmatrix}\right)\}$ forms a  basis of $X^*.$ We
extends elements of $X^*$ to linear functional on $(\mathcal
M_2(\mathbb C), \|\cdot\|_{\rm op}).$ From Theorem \ref{rudin}, it
follows that
$$\|(\alpha, \beta)\|^2=\inf_{a, b}\big\|\left(\begin{smallmatrix}
\frac{\alpha}{2}+a & b\\
    \beta    & \frac{\alpha}{2}-a
\end{smallmatrix}\right)\big\|_{\rm tr}^2,$$ where $\|\cdot\|_{\rm tr}$ denote the trace norm.
Now, \begin{eqnarray}\label{dual norm}\nonumber\lefteqn{\inf_{a,
b}\big\|\left(\begin{smallmatrix}
\frac{\alpha}{2}+a & b\\
    \beta    & \frac{\alpha}{2}-a
\end{smallmatrix}\right)\big\|_{\rm tr}^2}\\\nonumber&\phantom{}=&\inf_{a,
b}\{\frac{|\alpha|^2}{2}+2|a|^2+|b|^2+|\beta|^2+2\big|\frac{\alpha^2}{4}-a^2-b\beta\big|\}
\\\nonumber&\phantom{}=&\inf_{a,
b}\{\frac{|\alpha|^2}{2}+2|a|^2+|b|^2+|\beta|^2+2\big||\frac{|\alpha|^2}{4}-b\beta|-|a|^2\big|\}.
\end{eqnarray}
%\inf_{a,
%b}\{\frac{|\alpha|^2}{2}+2|a|^2+|b|^2+|\beta|^2+2\sqrt{\big|\frac{|\alpha|^2}{4}-a^2-b\beta\big|^2}\}
%\\\nonumber&\phantom{}=&\inf_{a,
%b}\{\frac{|\alpha|^2}{2}+2|a|^2+|b|^2+|\beta|^2\\&\phantom{}&+2\sqrt{\big|\frac{|\alpha|^2}{4}-b\beta\big|^2+|a|^4
%-2\Re\big|(\frac{|\alpha|^2}{4}-b\beta)\big||\bar{a}|^2\exp{i(\theta-\phi)}}\},\end{eqnarray}
%where
%$(\frac{|\alpha|^2}{4}-b\beta)=\big|(\frac{|\alpha|^2}{4}-b\beta)\big|\exp{i\theta}$
%and $\bar{a}^2=|a|^2\exp{i\phi}.$ If we choose $\theta=\phi,$ then
%Equation (\ref{dual norm}) is equivalent to
%\begin{align}\label{dual norm1}\inf_{a, b}\big\|\left(\begin{smallmatrix}
%\frac{\alpha}{2}+a & b\\
%    \beta    & \frac{\alpha}{2}-a
%\end{smallmatrix}\right)\big\|_{\rm tr}^2&=\inf_{a,
%b}\{\frac{|\alpha|^2}{2}+2|a|^2+|b|^2+|\beta|^2+2\big||\frac{|\alpha|^2}{4}-b\beta|-|a|^2\big|\}.\end{align}
%Let $$\mathcal A=\inf_{a, b,\big|\frac{|\alpha|^2}{4}-
%b\beta\big|\geq|a|^2}\{\frac{|\alpha|^2}{2}+2|a|^2+|b|^2+|\beta|^2+2\big||\frac{|\alpha|^2}{4}-b\beta|-|a|^2\big|\}$$
%and
%$$\mathcal B=\inf_{a, b, \big|\frac{|\alpha|^2}{4}-
%b\beta\big|\leq|a|^2}\{\frac{|\alpha|^2}{2}+2|a|^2+|b|^2+|\beta|^2+2\big||\frac{|\alpha|^2}{4}-b\beta|-|a|^2\big|\}.$$
To compute this infimum, we consider the two cases
$\big|\frac{|\alpha|^2}{4}-b\beta\big|\geq |a|^2$ and $|a|^2\geq
\big|\frac{|\alpha|^2}{4}-b\beta\big|.$ In either case, we have
\begin{align}\label{dual norm2}\inf_{a,
b}\big\|\left(\begin{smallmatrix}
\frac{\alpha}{2}+a & b\\
    \beta    & \frac{\alpha}{2}-a
\end{smallmatrix}\right)\big\|_{\rm tr}^2\nonumber&=\inf_{b}\{\frac{|\alpha|^2}{2}+|b|^2+|\beta|^2+
2\big|\frac{|\alpha|^2}{4}-b\beta\big|\}\\&=\inf_{b}\{\frac{|\alpha|^2}{2}+|b|^2+|\beta|^2+
2\big|\frac{|\alpha|^2}{4}-|b\beta|\big|\}.\end{align}
%it follows
%that $\frac{|\alpha|^2}{4}-|b\beta|\geq 0.$ Consider the set
%$\mathcal B_{+}=\{b:\frac{|\alpha|^2}{4}-|b\beta|\geq 0\}.$ From
%Equation (\ref{dual}) we conclude that if $|\beta| \leq
%\frac{|\alpha|}{2},$ then
%\begin{align*}\|(\alpha,
%\beta)\|&=\inf_{a,
%b}\{\frac{|\alpha|^2}{2}+2|a|^2+|b|^2+|\beta|^2+2|\frac{|\alpha|^2}{4}-|a|^2-|b\beta||\}\\&=|\alpha|.
%\end{align*}
%Again consider the set $\mathcal
%B_{-}=\{b:\frac{|\alpha|^2}{4}-|b\beta|\leq 0\}.$ The set
%$\mathcal B_{-}\neq \phi,$ because $\frac{\alpha^2}{4\beta}\in
%\mathcal B_{-}.$ Putting $a=0$ and $b=\frac{\alpha^2}{4\beta}$ in
%Equation (\ref{dual}) we have
%$\frac{(|\alpha|^2+4|\beta|^2)^2}{16|\beta|^2}.$ Hence
%$$\|(\alpha, \beta)\| ={ \begin{array}{ll}\frac{|\alpha|^2+4|\beta|^2}{4|\beta|} &  \mbox{if $|\beta|\geq
%$\frac{|\alpha|}{2} $}\end{array}}.$$ This completes the proof.

%Therefore, we have $$\inf_{a, b}\|\left(\begin{smallmatrix}
%\frac{\alpha}{2}+a & b\\
%    \beta    & \frac{\alpha}{2}-a
%\end{smallmatrix}\right)\|= {
%\begin{array}{ll}|\alpha| & \mbox{if
%$|\beta|\leq \frac{|\alpha|}{2} $}\end{array}}.$$
Let
$$M=\inf_{ b,|b|\leq\frac{|\alpha|^2}{4|\beta|}}
\{\frac{|\alpha|^2}{2}+|b|^2+|\beta|^2+
2\big|\frac{|\alpha|^2}{4}-|b\beta|\big|\}$$ and
$$N=\inf_{ b,|b|\geq\frac{|\alpha|^2}{4|\beta|}}
\{\frac{|\alpha|^2}{2}+|b|^2+|\beta|^2+
2\big|\frac{|\alpha|^2}{4}-|b\beta|\big|\}.$$ Now,
\begin{align}\label{dualmain1}M\nonumber&=\inf_{ b,|b|\leq\frac{|\alpha|^2}{4|\beta|}}
\{\frac{|\alpha|^2}{2}+|b|^2+|\beta|^2+
2\big|\frac{|\alpha|^2}{4}-|b\beta|\big|\}\\&=\inf_{
b,|b|\leq\frac{|\alpha|^2}{4|\beta|}}\{|\alpha|^2+(|b|-|\beta|)^2\}.
\end{align}
If $\frac{|\alpha|^2}{4|\beta|}\geq |\beta|,$ then we can take
$|b|=|\beta|$ in Equation (\ref{dualmain1}) and the infimum is
$|\alpha|^2.$ If $\frac{|\alpha|^2}{4|\beta|}\leq |\beta|,$ then
the largest value $|b|$ can take is $\frac{|\alpha|^2}{4|\beta|}.$
In this case, $ M=(\frac{|\alpha|^2+4|\beta|^2}{4|\beta|})^2.$
Therefore,
\[ M = \left\{ \begin{array}{ll}
        (\frac{|\alpha|^2+4|\beta|^2}{4|\beta|})^2  & \mbox{if $|\beta|\geq \frac{|\alpha|}{2} $};\\
        |\alpha|^2 & \mbox{if $|\beta|\leq \frac{|\alpha|}{2}$}.\end{array} \right. \]

\begin{align}\label{dualmain2}N\nonumber&=\inf_{ b,|b|\geq\frac{|\alpha|^2}{4|\beta|}}
\{\frac{|\alpha|^2}{2}+|b|^2+|\beta|^2+
2\big|\frac{|\alpha|^2}{4}-|b\beta|\big|\}\\\nonumber&=\inf_{
b,|b|\geq\frac{|\alpha|^2}{4|\beta|}}\{(|b|+|\beta|)^2\}\\&=(\frac{|\alpha|^2+4|\beta|^2}{4|\beta|})^2.
\end{align}
If $|\beta|\leq \frac{|\alpha|}{2},$ we have
\begin{eqnarray*}\inf_{a,
b}\big\|\left(\begin{smallmatrix}
\frac{\alpha}{2}+a & b\\
    \beta    & \frac{\alpha}{2}-a
\end{smallmatrix}\right)\big\|_{\rm tr}^2&=&\min\{ |\alpha|^2
,(\frac{|\alpha|^2+4|\beta|^2}{4|\beta|})^2\}\\&=&|\alpha|^2
\end{eqnarray*}
and  if  $|\beta|\geq \frac{|\alpha|}{2},$ we have
$$\inf_{a, b}\big\|\left(\begin{smallmatrix}
\frac{\alpha}{2}+a & b\\
    \beta    & \frac{\alpha}{2}-a
\end{smallmatrix}\right)\big\|_{\rm
tr}^2=(\frac{|\alpha|^2+4|\beta|^2}{4|\beta|})^2.$$ Hence
\[ \|(\alpha,
\beta)\| = \left\{ \begin{array}{ll}
        \frac{|\alpha|^2+4|\beta|^2}{4|\beta|}  & \mbox{if $|\beta|\geq \frac{|\alpha|}{2} $};\\
        |\alpha| & \mbox{if $|\beta|\leq \frac{|\alpha|}{2}$}.\end{array} \right. \]

This completes the proof.

\noindent \text{(Second proof)} Let $f_{\alpha,
\beta}:X\rightarrow \mathbb C$ be a linear functional defined by
$f_{\alpha, \beta}(x, y)=\alpha x+\beta y.$ Now,\begin{align*}
\|f_{\alpha, \beta}\|^2&=\sup_{|x|^2+|y|=1}|\alpha x+\beta
y|^2\\&=\sup_{|x|^2+|y|=1}(|\alpha||x|+|\beta||y|)^2\\&=\sup_{|x|^2+|y|=1}(|\alpha||x|+|\beta|(1-|x|^2))^2\\
&=\sup_{0\leq |x|^2\leq1}|\alpha
x|^2+|\beta|^2(1-|x|^2)^2+2|\alpha||x||\beta|(1-|x|^2).
\end{align*}
Note that the supremum has to be taken with $|x|^2\leq 1-|y|.$ Let
$g(|x|)=|\alpha||x|+|\beta|(1-|x|^2).$ The derivative of $g$ with
respect $|x|$ is equal to $g^{\prime}(|x|)=|\alpha|-2|\beta||x|.$
Now, $g^{\prime}(|x|)=0$ is equivalent to
$|x|=\frac{|\alpha|}{2|\beta|}.$ Also
$g^{\prime\prime}(|x|)=-2|\beta|$ which is less than zero.
Therefore, the supremum of $g$ is attained at
$|x|=\frac{|\alpha|}{2|\beta|}.$ If $|\beta|\leq
\frac{|\alpha|}{2},$ then the derivative does not vanish for $x$
with  $0\leq |x|^2\leq1.$ In fact,  the derivative is positive in
the range $0\leq |x|^2\leq1$ and the supremum is attained at
$|x|=1.$ Therefore, we have two cases:
 \begin{enumerate}
\item[(a).]
$$\|(\alpha, \beta)\|= { \begin{array}{ll}\frac{|\alpha|^2+4|\beta|^2}{4|\beta|} &
\mbox{if $|\beta|\geq \frac{|\alpha|}{2} $}\end{array}}.$$

\item[(b).]$$\|(\alpha, \beta)\|= {
\begin{array}{ll}|\alpha| & \mbox{if
$|\beta|\leq \frac{|\alpha|}{2} $}\end{array}}.$$
\end{enumerate} This completes the
proof.

\end{proof}
\section{Contractivity and complete contractivity}
Let $\Omega_{\mathbf A}=\{(z_1,
z_2):\left\|z_1A_1+z_2A_2\right\|_{\rm op} <1\}$ in $\mathbb C^2,$
where $\mathbf A=(A_1, A_2)$ and $A_1= I_2,
A_2=\left(\begin{smallmatrix}
0 & 1\\
    0    & 0
\end{smallmatrix}\right).$
We now describe contractivity of the special class of
homomorphisms $\rho_{V}: \mathcal O(\Omega_{\mathbf A}) \to
\mathcal M_3(\mathbb C)$ induced by a pair of the form
$\left(\left(\begin{smallmatrix}
w_1 & \mathbf v_1\\
    0    & w_1I_2
\end{smallmatrix} \right), \left(\begin{smallmatrix}
w_2 & \mathbf v_2\\
    0    & w_2I_2
\end{smallmatrix} \right)\right),$ where $\mathbf v_i \in \mathbb C^2$ for $i=1,2.$
We have seen that $\|\,\rho_{V}\,\|_{\mathcal O(\Omega_{\mathbf
A})\rightarrow \mathcal M_3(\mathbb C)}\leq 1$ if and only if
$\|\,L_{V}\,\|_{(\mathbb C^2,\|\,\cdot\,\|^*_{\Omega_{\mathbf
A}})\rightarrow(\mathbb C^2,\|\,\cdot\,\|_2)} \leq 1$ if and only
if $\|L_{V}^*\|_{(\mathbb C^2,\|\,\cdot\,\|_2)\rightarrow(\mathbb
C^2,\|\,\cdot\,\|_{\Omega_{\mathbf A}})} \leq 1.$ Let
$L_{V}:(\mathbb C^2,\|\,\cdot\,\|^*_{\Omega_{\mathbf
A}})\rightarrow (\mathbb C^2,\|\,\cdot\,\|_2)$ be the linear map
induced by the pair $\mathbf v_1$ and $\mathbf v_2.$ The matrix
representing $L_{V}^*:(\mathbb C^2,\|\,\cdot\,\|_2)\rightarrow
(\mathbb C^2,\|\,\cdot\,\|_{\Omega_{\mathbf A}})$  is of the form
$\left (
\begin{smallmatrix}
v_{11} & v_{12}  \\
  v_{21 } & v_{22}
\end{smallmatrix} \right ), V:=\left(\begin{smallmatrix}\mathbf v_1\\ \mathbf v_2\end{smallmatrix}\right),$ where
$\mathbf v_1=(v_{11}\,\, v_{12})$ and $\mathbf v_2=(v_{21}
\,\,v_{22}).$ The following theorem provides a characterization of
contractivity of the homomorphism $\rho_{V}$ in terms of $V.$
\begin{theorem}\label{con homo}$\|\,L_{V}\,\|_{(\mathbb
C^2,\|\,\cdot\,\|^*_{\Omega_{\mathbf A}})\rightarrow(\mathbb
C^2,\|\,\cdot\,\|_2)} \leq 1$ if and only if
$\|L_{V}^*\|_{(\mathbb C^2,\|\,\cdot\,\|_2)\rightarrow (\mathbb
C^2,\|\,\cdot\,\|_{\Omega_{\mathbf A}})}\leq 1$ if and only if
$$\left(\|\mathbf v_1\|^2+\frac{\|\mathbf
v_2\|^2}{4}+\sqrt{(\|\mathbf v_1\|^2-\frac{\|\mathbf
v_2\|^2}{4})^2 +|\langle \mathbf v_1, \mathbf v_2
\rangle|^2}\right)^2\leq 4\sqrt{(\|\mathbf v_1\|^2-\frac{\|\mathbf
v_2\|^2}{4})^2+|\langle \mathbf v_1,\mathbf v_2 \rangle|^2}.$$
\end{theorem}
\begin{proof}\noindent \text{(First proof)}
Suppose $\|L_{V}^*\|_{(\mathbb C^2,\|\,\cdot\,\|_2)\rightarrow
(\mathbb C^2,\|\,\cdot\,\|_{\Omega_{\mathbf A}})}\leq 1.$ Since
the matrix representation of $L_{V}^*$ is of the form $\left (
\begin{smallmatrix}
v_{11} & v_{12} \\
 v_{21}    & v_{22}
\end{smallmatrix} \right ).$ We have $\left ( \begin{smallmatrix}
v_{11} & v_{12} \\
 v_{21}    & v_{22}
\end{smallmatrix} \right )\left ( \begin{smallmatrix}
x \\
 y
\end{smallmatrix} \right )\in (\mathbb C^2,\|\,\cdot\,\|_{\Omega_{\mathbf A}}).$
We are interested in the expression
\begin{eqnarray}\label{contractionA}
(v_{11}x+ v_{12}y) A_1+(v_{21}x+ v_{22}y) A_2\nonumber &=&
\left(\begin{smallmatrix}
    v_{11}x+ v_{12}y & v_{21}x+ v_{22}y \\
    0    & v_{11}x+ v_{12}y
\end{smallmatrix}\right)\nonumber\\&=&xA^{'}_1+yA^{'}_2,
\end{eqnarray} where $A^{'}_1=\left(\begin{smallmatrix}
    v_{11} & v_{21} \\
    0    & v_{11}
\end{smallmatrix}\right)$ and $A^{'}_2=\left(\begin{smallmatrix}
    v_{12} &v_{22} \\
    0    & v_{12}
\end{smallmatrix}\right).$
Thus  $\|L_{V}^*\|_{(\mathbb C^2,\|\,\cdot\,\|_2)\rightarrow
(\mathbb C^2,\|\,\cdot\,\|_{\Omega_{\mathbf A}})}\leq 1$ if and
only if $\sup_{|x|^2+|y|^2=1}    \|xA^{'}_1+yA^{'}_2\|_{\rm op}^2
\leq 1.$
\begin{eqnarray}\label{normcompu}
   \sup_{|x|^2+|y|^2=1}    \|xA^{'}_1+yA^{'}_2\|^2\nonumber &= &    \sup_{|x|^2+|y|^2=1}
   \sup_{\|\alpha\|_2=\|\beta\|_2=1 }
|\langle (xA^{'}_1+yA^{'}_2)\alpha,\beta \rangle|^2\\\nonumber&= &
\sup_{|x|^2+|y|^2=1} \sup_{\|\alpha\|_2=\|\beta\|_2=1 }| x \langle
A^{'}_1\alpha,\beta \rangle +  y\langle A^{'}_2\alpha,\beta
\rangle|^2\\&=& \sup_{\|\alpha\|_2=\|\beta\|_2=1 } ( | \langle
A^{'}_1\alpha,\beta \rangle|^2 +  |\langle A^{'}_2\alpha,\beta
\rangle|^2).
\end{eqnarray}
 Now,
$$\langle A^{'}_1\alpha,\beta \rangle=v_{11}\alpha_1\overline{\beta_1}
+v_{11} \alpha_2  \overline{\beta_2 }   +  v_{21}  \alpha_2
\overline{\beta_1} =v_{11} \langle \alpha ,  \beta   \rangle +
v_{21} \alpha_2  \overline{\beta_1}.$$ Similarly we have  $\langle
A^{'}_2\alpha,\beta \rangle =v_{12} \langle \alpha ,\beta
\rangle+v_{22} \alpha_2  \overline{\beta_1}$. Putting the value of
$\langle A^{'}_1\alpha,\beta \rangle$ and  $\langle
A^{'}_2\alpha,\beta \rangle$ in Equation (\ref{normcompu}) we have
\begin{align*}
\sup_{\|\alpha\|_2=\|\beta\|_2=1 } ( | \langle A^{'}_1\alpha,\beta
\rangle|^2 +  |\langle A^{'}_2\alpha,\beta
\rangle|^2)&=\sup_{\|\alpha\|_2=\|\beta\|_2=1 }\|\mathbf v_1\|^2|
|\langle \alpha,\beta \rangle|^2+\|\mathbf v_2\|^2|\alpha_2
\overline{\beta_1}|^2
\\&+2Re\left\langle \mathbf v_1, \mathbf v_2\right\rangle\langle \alpha,\beta
\rangle \overline{\alpha_2}\beta_1,
\end{align*}
where  $\|\mathbf v_1\|^2=|v_{11}|^2+|v_{12}|^2$ ,$\left\langle
\mathbf v_1, \mathbf
v_2\right\rangle=(v_{11}\bar{v}_{21}+v_{12}\bar{v}_{22})$ and
$\|\mathbf v_2\|^2=|\bar{v}_{21}|^2+|\bar{v}_{22}|^2$.
 Choosing  $\alpha'=(\alpha_1\exp(i\theta),
\alpha_2\exp(-i\theta)), \beta'=(\beta_1\exp(i\theta),
\beta_2\exp(-i\theta))$ we have
{\small\begin{eqnarray}\label{contract}\lefteqn{\nonumber
\sup_{\|\alpha\|_2=\|\beta\|_2=1 } ( | \langle A^{'}_1\alpha,\beta
\rangle|^2 +  |\langle A^{'}_2\alpha,\beta
\rangle|^2)}\\\nonumber&=&\sup_{\|\alpha\|_2=\|\beta\|_2=1
}\|\mathbf v_1\|^2| |\langle \alpha,\beta \rangle|^2+\|\mathbf
v_2\|^2|\alpha_2 \overline{\beta_1}|^2 +2|\left\langle \mathbf
v_1, \mathbf v_2\right\rangle\langle \alpha,\beta \rangle
\alpha_2\overline{\beta_1}|\\\nonumber
&=&\sup_{\|\alpha\|_2=\|\beta\|_2=1 }\|\mathbf v_1\|^2| |\langle U
\alpha, U\beta \rangle|^2+\|\mathbf v_2\|^2|(U\alpha)_2
\overline{(U\beta)_1}|^2+2|\left\langle \mathbf v_1, \mathbf
v_2\right\rangle\langle U\alpha, U\beta \rangle (U\alpha)_2
\overline{ (U\beta)_1}|\\\nonumber&=&\sup_ { ||\beta||=1, U }
\|\mathbf v_1\|^2 |\langle U e_2,U\beta \rangle|^2 +\|\mathbf
v_2\|^2|(Ue_2)_2 \overline{(U\beta)_1}|^2
   \\&& +2|\langle \mathbf v_1, \mathbf v_2\rangle\langle Ue_2,U\beta \rangle(Ue_2)_2
\overline{(U\beta)_1}|,
\end{eqnarray}} where $ U: \mathbb C^2\rightarrow \mathbb C^2$ is a unitary which is of the form
$U=\left(\begin{smallmatrix}
a   & -\overline{b} \\
b    & \overline{a}
\end{smallmatrix}\right)$ with $|a|^2+|b|^2=1$ and $e_2=\left(\begin{smallmatrix}
0  \\
1
\end{smallmatrix}\right).$
Let $\alpha =\left(\begin{smallmatrix}
\alpha_1   \\
\alpha_2
\end{smallmatrix}\right)$ and $\beta=\left(\begin{smallmatrix}
\beta_1  \\
\beta_2
\end{smallmatrix}\right),$
then $U\alpha=\left(\begin{smallmatrix}
a   & -\bar{b} \\
b    & \bar{a}
\end{smallmatrix}\right)\left(\begin{smallmatrix}
\alpha_1   \\
\alpha_2
\end{smallmatrix}\right)=\left(\begin{smallmatrix}
a\alpha_1 -\bar{b}\alpha_2  \\
b\alpha_1+\bar{a}\alpha_2
\end{smallmatrix}\right).$
Hence $(U\alpha)_2 =b\alpha_1+\bar{a}\alpha_2$  and
$\overline{(U\beta)_1}=(\bar{a}\bar{\beta_1} -b\bar{\beta_2})$.
Thus we have
\begin{align*}(U\alpha)_2
\overline{(U\beta)_1}&=\bar{a}b\alpha_1\bar{\beta_1}+\bar{a}^2\alpha_2\bar{\beta_1}
-b^2\alpha_1\bar{\beta_2}-\bar{a}b\alpha_2\bar{\beta_2}\\
&=\langle \left(\begin{smallmatrix}
 \bar{a}b  & \bar{a}^2  \\
 -b^2   & -\bar{a}b
\end{smallmatrix}\right)\left(\begin{smallmatrix}
\alpha_1   \\
\alpha_2
\end{smallmatrix}\right) ,\left(\begin{smallmatrix}
\beta_1  \\
\beta_2
\end{smallmatrix}\right) \rangle.\end{align*}
In particular, we have $(Ue_2)_2 \overline{(U\beta)_1}= \bar{a}^2
\bar{\beta_1} -\bar{a}b\bar{\beta_2}.$

We claim that $\sup_ {U }| (Ue_2)_2
\overline{(U\beta)_1}|=\frac{(1+|\beta_1|)}{2} .$

In order to prove the claim, it is sufficient to observe that
{\small\begin{align}\label{unitary} \sup_ {U}|(Ue_2)_2
\overline{(U\beta)_1}|^2\nonumber&=
\sup_{|a|^2+|b|^2=1}|\overline{a}^2 \overline{\beta_1}
-\overline{a}b\overline{\beta_2}|^2\\\nonumber&=\sup|(\cos{t})^2\,\cos{\psi}\,\exp{i(-2\theta-x)}-
\cos{t}\,\sin{t}\,\sin{\psi}\exp{i(-\theta-y+\phi)}|^2\\\nonumber&
=\sup(\cos{t})^4\,(\cos{\psi})^2+\frac{\sin{2t}^2}{4}\,(\sin{\psi})^2\\&-(\cos{t})^2\,\cos{\psi}\,\sin{2t}\,\sin{\psi}\cos(\theta+x+\phi-y)
\end{align}}
where  $a=\cos{t}\exp{i\theta}, b=\sin{t}\exp{i\phi},
\beta_1=\cos{\psi}\exp{ix}$ and $\beta_2=\sin{\psi}\exp{iy}.$ If
we choose  $\theta+x+\phi=y,$ then the right hand side of
(\ref{unitary}) is
 $$\sup_{t}((\cos{t})^2\,\cos{\psi}+\frac{\sin{2t}}{2}\,\sin{\psi})^2.$$

Let
$f(t)=\big((\cos{t})^2\,\cos{\psi}+\frac{\sin{2t}}{2}\,\sin{\psi}\big).$
The derivative of $f$ with respect to $t$ is
$f'(t)=-\sin{2t}\cos{\psi}+\cos{2t}\sin{\psi}.$ If we assume
$0=f'(t),$ then we have $ {\psi-2t}=n\pi,$ where $n\in Z.$ Also,
$f''(t)=-2\cos{2t}\,\cos{\psi}-2\sin{2t}\,\sin{\psi}=-2\cos({\psi-2t}).$
If $\psi-2t=2n\pi,$ then $f''(t)\leq 0.$ Therefore, we  conclude
that the maximum value of  $f(t)$  is achieved at $\psi-2t=2n\pi$
and the maximum value  of $f(t)$ is equal to
$\frac{(1+|\beta_1|)}{2}$. This proves the claim. Putting $\sup_
{U} | (Ue_2)_2 \overline{(U\beta)_1}|=\frac{(1+|\beta_1|)}{2}$ in
Equation $(\ref{contract})$ we have
{\small\begin{align}\label{contraction con}\sup_{\|\beta\|=1
}\|\mathbf v_1\|^2 |\beta_2|^2+\frac{\|\mathbf
v_2\|^2}{4}(1+|\beta_1|)^2 +|\langle \mathbf v_1, \mathbf
v_2\rangle || (1+|\beta_1|)|\beta_2|.
\end{align}}

Let   $\beta_1=\cos{t}\exp(ix_2)$, $\beta_2=\sin{t}\exp(ix_1)$
then the Equation $(\ref{contraction con})$ simplifies to the
equation $ax^2+cxy+by^2,$ where $a=\|\mathbf v_1\|^2,
b=\frac{\|\mathbf v_2\|^2}{4},c=|\langle \mathbf v_1, \mathbf v_2
\rangle |,y=1+\cos{t} ,x=\sin{t}$ and  supremum is taken over
$x^2+y^2=2y,$ that is, $x^2+(y-1)^2=1$. Also, note that
{\small\begin{align}\label{con
condition}\sup_{x^2+(y-1)^2=1}\|ax^2+cxy+by^2\|\nonumber &=
\sup_{x^2+(y-1)^2=1}\|\left(\begin{smallmatrix}
x  \\
y
\end{smallmatrix}\right)^{\rm t}\left(\begin{smallmatrix}
a  & c/2  \\
 c/2   & b
\end{smallmatrix}\right)\left(\begin{smallmatrix}
x  \\
y
\end{smallmatrix}\right)\|\\\nonumber&=  \sup_{x^2+(y-1)^2=1}\|\left(\begin{smallmatrix}
x  \\
y-1+1
\end{smallmatrix}\right)^{\rm t}\left(\begin{smallmatrix}
a  & c/2  \\
 c/2   & b
\end{smallmatrix}\right)\left(\begin{smallmatrix}
x  \\
y -1+1
\end{smallmatrix}\right)\|\\&=\sup_{x^2+w^2=1}\|\left(\begin{smallmatrix}
x  \\
w+1
\end{smallmatrix}\right)^{\rm t}\left(\begin{smallmatrix}
a  & c/2  \\
 c/2   & b
\end{smallmatrix}\right)\left(\begin{smallmatrix}
x  \\
w+1
\end{smallmatrix}\right)\|,
\end{align}} where $(y-1)=w.$ Let $u=\left(\begin{smallmatrix}
x  \\
w
\end{smallmatrix}\right),$ then $\left(\begin{smallmatrix}
x  \\
w+1
\end{smallmatrix}\right)=\left(\begin{smallmatrix}
x  \\
w
\end{smallmatrix}\right)+\left(\begin{smallmatrix}
0  \\
1
\end{smallmatrix}\right)=u+e_2.$ Also, let  $T=\left(\begin{smallmatrix}
a  & c/2  \\
 c/2   & b
\end{smallmatrix}\right),$ then $T$ is self-adjoint matrix.
Since $T$ is self-adjoint matrix, there exist a unitary matrix $U$
such that   $U^{-1}TU=D,$ where $D=\left(\begin{smallmatrix}
 c_1   & 0  \\
0   & c_2
\end{smallmatrix}\right).$ Therefore, Equation $(\ref{con condition})$ is equivalent to
{\small\begin{align}\label{conncondition}\sup_{\|u\|_{2}=1}\|(u+e_2)^{\rm
tr}T(u+e_2)\|\nonumber&=\sup_{\|u\|_{2}=1}\|(u+e_2)^{\rm
tr}U^{-1}TU(u+e_2)\|\\\nonumber&= \sup_{\|u\|_{2}=1}=\|u^{\rm
tr}Du+e_2^{\rm tr}Du+u^{\rm tr}De_2+e_2^{\rm tr}De_2\|
\\\nonumber&=\sup_{x^2+w^2=1}c_1x^2+c_2w^2+2c_2w+c_2\\\nonumber&
=\sup_{w}c_1(1-w^2)+c_2w^2+2c_2w+c_2\\&=\sup_{w}(c_2-c_1)w^2+2c_2w+c_2+c_1.
\end{align}}
Suppose $g(w)=\sup_{w}(c_2-c_1)w^2+2c_2w+c_2+c_1.$ The derivative
of $g$ with respect to $w$ is $g'(w)=2(c_2-c_1)w+2c_2.$ Now,
$0=g'(w)$ implies that $ w=\frac{-c_2}{(c_2-c_1)}.$ Also,
$g''(w)=2(c_2-c_1)$ and   $g''(w)\leq {
\begin{array}{ll} 0 &
\mbox{if $c_1>c_2 $}\end{array}}.$ Therefore, the maximum value of
$g(w)$ is achieved at $ w=\frac{-c_2}{(c_2-c_1)}$ and is equal to
$g(w) =\frac{c_1^2}{(c_1-c_2)}.$ The eigen values of
$\left(\begin{smallmatrix}
a  & c/2  \\
 c/2   & b
\end{smallmatrix}\right)$ are  equal to $c_1=\frac{\Big(a+b+\sqrt{(a-b)^2+c^2}\Big)}{2},
c_2=\frac{\Big(a+b-\sqrt{(a-b)^2+c^2}\Big)}{2}.$ Therefore, from
Equation $(\ref{conncondition})$ we have
{\small\begin{align}\label{mainnorm}g(w)\nonumber&=\frac{\Big(a+b+\sqrt{(a-b)^2+c^2}\Big)^2}{4\sqrt{(a-b)^2+c^2}}
\\&=\frac{\left(\|\mathbf v_1\|^2+\frac{\|\mathbf v_2\|^2}{4}+\sqrt{(\|\mathbf v_1\|^2-\frac{\|\mathbf v_2\|^2}{4})^2
+|\langle \mathbf v_1, \mathbf v_2
\rangle|^2}\right)^2}{4\sqrt{(\|\mathbf v_1\|^2-\frac{\|\mathbf
v_2\|^2}{4})^2+|\langle \mathbf v_1, \mathbf v_2 \rangle|^2}}.
\end{align}}
Hence  $\|L_{V}\|^2\leq 1$ if and only if $\frac{\left(\|\mathbf
v_1\|^2+\frac{\|\mathbf v_2\|^2}{4}+\sqrt{(\|\mathbf
v_1\|^2-\frac{\|\mathbf v_2\|^2}{4})^2 +|\langle \mathbf v_1,
\mathbf v_2 \rangle|^2}\right)^2}{4\sqrt{(\|\mathbf
v_1\|^2-\frac{\|\mathbf v_2\|^2}{4})^2+|\langle \mathbf v_1,
\mathbf v_2 \rangle|^2}}\leq 1$ which is equivalent to
$$\left(\|\mathbf v_1\|^2+\frac{\|\mathbf v_2\|^2}{4}+\sqrt{(\|\mathbf v_1\|^2-\frac{\|\mathbf v_2\|^2}{4})^2
+|\langle \mathbf v_1, \mathbf v_2 \rangle|^2}\right)^2\leq
4\sqrt{(\|\mathbf v_1\|^2-\frac{\|\mathbf v_2\|^2}{4})^2+|\langle
\mathbf v_1, \mathbf v_2 \rangle|^2}.$$

\noindent \text{(Second proof)} Suppose $\|\,L_{V}\,\|_{(\mathbb
C^2,\|\,\cdot\,\|^*_{\Omega_{\mathbf A}})\rightarrow(\mathbb
C^2,\|\,\cdot\,\|_2)} \leq 1.$ Since the matrix representation of
$L_{V}^*$ is of the form $\left (
\begin{smallmatrix}
v_{11} & v_{21} \\
 v_{12}    & v_{22}
\end{smallmatrix} \right ).$ Therefore, we have $\|\,L_{V}\,\|_{(\mathbb C^2,\|\,\cdot\,\|^*_{\Omega_{\mathbf
A}})\rightarrow(\mathbb C^2,\|\,\cdot\,\|_2)} \leq 1$ if and only
if $\sup_{\|(\alpha,
\beta)\|=1}|v_{11}\alpha+v_{21}\beta|^2+|v_{12}\alpha+v_{22}\beta|^2\leq
1.$ Now, \begin{align}\label{contractive}\sup_{\|(\alpha,
\beta)\|=1}|v_{11}\alpha+v_{21}\beta|^2+|v_{12}\alpha+v_{22}\beta|^2&=\sup_{\|(\alpha,
\beta)\|=1}\|\mathbf v_1\|^2|\alpha|^2+\|\mathbf
v_2\|^2|\beta|^2+2\Re\langle \mathbf v_1, \mathbf
v_2\rangle\alpha\bar{\beta},\end{align}where $\|\mathbf
v_1\|^2=|v_{11}|^2+|v_{12}|^2$ ,$\left\langle \mathbf v_1, \mathbf
v_2\right\rangle=(v_{11}\bar{v}_{21}+v_{12}\bar{v}_{22})$ and
$\|\mathbf v_2\|^2=|\bar{v}_{21}|^2+|\bar{v}_{22}|^2.$ Choosing
$\alpha$ and $\beta$ in such a way that $\Re\langle \mathbf v_1,
\mathbf v_2\rangle\alpha\bar{\beta}=|\langle \mathbf v_1, \mathbf
v_2\rangle\alpha\bar{\beta}|.$ Hence from Equation
$(\ref{contractive})$ we have
\begin{align}\label{contractivity}\sup_{\|(\alpha,
\beta)\|=1}|v_{11}\alpha+v_{21}\beta|^2+|v_{12}\alpha+v_{22}\beta|^2&=\sup_{\|(\alpha,
\beta)\|=1}\|\mathbf v_1\|^2|\alpha|^2+\|\mathbf
v_2\|^2|\beta|^2+2|\langle \mathbf v_1, \mathbf
v_2\rangle\alpha\bar{\beta}|.\end{align}

We have seen in previous section that if $|\beta|\leq
\frac{|\alpha|}{2},$ then  $\|(\alpha, \beta)\|= |\alpha|.$
Therefore, $\|(\alpha, \beta)\|=1$ implies that $|\beta|\leq
\frac{1}{2}. $ For this case, we have
$$\sup_{\|(\alpha,
\beta)\|=1}|v_{11}\alpha+v_{21}\beta|^2+|v_{12}\alpha+v_{22}\beta|^2=\|\mathbf
v_1\|^2+\frac{\|\mathbf v_2\|^2}{4}+\Re\langle \mathbf v_1,
\mathbf v_2\rangle.$$ Also, if $|\beta|\geq \frac{1}{2},$ then
$\|(\alpha, \beta)\|= \frac{|\alpha|^2+4|\beta|^2}{4|\beta|}.$
Hence $\|(\alpha, \beta)\|=1$ implies that
$|\alpha|^2+(2|\beta|-1)^2=1.$ Setting $x=|\alpha|$ and
$y=2|\beta|-1$ we have $|\beta|=\frac{y+1}{2}.$ Therefore, we have
$$\sup_{x^2+y^2=1}x^2\|\mathbf v_1\|^2+\|\mathbf v_2\|^2\frac{(y+1)^2}{4}+|\langle
\mathbf v_1, \mathbf v_2\rangle x(y+1)|.$$ Hence from Equation
(\ref{mainnorm}) we have
\begin{align*}\sup_{\|(\alpha,
\beta)\|=1}|v_{11}\alpha+v_{21}\beta|^2+|v_{12}\alpha+v_{22}\beta|^2
&=\frac{\left(\|\mathbf v_1\|^2+\frac{\|\mathbf
v_2\|^2}{4}+\sqrt{(\|\mathbf v_1\|^2-\frac{\|\mathbf
v_2\|^2}{4})^2 +|\langle \mathbf v_1, \mathbf v_2
\rangle|^2}\right)^2}{4\sqrt{(\|\mathbf v_1\|^2-\frac{\|\mathbf
v_2\|^2}{4})^2+|\langle \mathbf v_1, \mathbf v_2 \rangle|^2}}.
\end{align*}
Therefore, we have
\begin{align*}\|L_{V}\|^2&=\max\{\frac{\left(\|\mathbf
v_1\|^2+\frac{\|\mathbf v_2\|^2}{4}+\sqrt{(\|\mathbf v_1\|^2
-\frac{\|\mathbf v_2\|^2}{4})^2 +|\langle \mathbf v_1, \mathbf v_2
\rangle|^2}\right)^2}{4\sqrt{(\|\mathbf v_1\|^2-\frac{\|\mathbf
v_2\|^2}{4})^2+|\langle \mathbf v_1, \mathbf v_2 \rangle|^2}},
\|\mathbf v_1\|^2+\frac{\|\mathbf v_2\|^2}{4}+\Re\langle \mathbf
v_1, \mathbf v_2\rangle \}\\&=\frac{\left(\|\mathbf
v_1\|^2+\frac{\|\mathbf v_2\|^2}{4}+\sqrt{(\|\mathbf
v_1\|^2-\frac{\|\mathbf v_2\|^2}{4})^2 +|\langle \mathbf v_1,
\mathbf v_2 \rangle|^2}\right)^2}{4\sqrt{(\|\mathbf
v_1\|^2-\frac{\|\mathbf v_2\|^2}{4})^2+|\langle \mathbf v_1,
\mathbf v_2 \rangle|^2}}
\end{align*} This completes the second
proof.
\end{proof}
Let $P_{\mathbf A}:\Omega_{\mathbf A}\rightarrow (\mathcal M_2)_1$
be the matrix valued polynomial on $\Omega_{\mathbf A}$ defined
earlier. The contractivity of $\rho_{V}^{(2)}$ is equivalent to
$\|A_1 \otimes  \mathbf v_1+A_2\otimes  \mathbf v_2\|\leq 1.$ Now,
we will compute $\|A_1 \otimes  \mathbf v_1+A_2\otimes \mathbf
v_2\|.$
\begin{theorem}\label{complete con}
$\|A_1 \otimes  \mathbf v_1+A_2\otimes \mathbf v_2\|\leq 1$ if and
only if
$$2\|\mathbf v_1\|^2+\|\mathbf v_2\|^2+\sqrt{\|\mathbf v_2\|^4-4|\langle \mathbf v_1, \mathbf v_2
\rangle|^2}\leq 2.$$
\end{theorem}
\begin{proof}
Suppose $\|A_1 \otimes  \mathbf v_1+A_2\otimes \mathbf v_2\|\leq
1.$ Now, {\small\begin{align} \|A_1 \otimes  \mathbf
v_1+A_2\otimes \mathbf v_2\|^2\nonumber&=\|(v_{11}A_1+v_{21}A_2,
v_{12}A_1+v_{22}A_2)\|^2\\\nonumber&=\Big\|(v_{11}A_1+v_{21}A_2,
v_{12}A_1+v_{22}A_2)\begin{pmatrix}(v_{11}A_1+v_{21}A_2)^*\\(v_{12}A_1+v_{22}A_2)^*\end{pmatrix}\Big\|
\\\nonumber&=\|\|\mathbf v_1\|^2A_1A_1^{*}+\langle \mathbf v_1, \mathbf v_2\rangle A_1A_2^{*}
+\langle \mathbf v_2, \mathbf v_1\rangle A_2A_1^{*}+\|\mathbf
v_2\|^2A_2A_2^{*}\|\\\nonumber&=\Big\|\begin{pmatrix}
\|\mathbf v_1\|^2 +\|\mathbf v_2\|^2  & \langle \mathbf v_2, \mathbf v_1\rangle  \\
 \langle \mathbf v_1, \mathbf v_2\rangle   & \|\mathbf v_1\|^2
\end{pmatrix}\Big\|.
\end{align}}

Let $\tilde{C}=\begin{pmatrix}
\|\mathbf v_1\|^2 +\|\mathbf v_2\|^2  & \langle \mathbf v_2, \mathbf v_1\rangle  \\
 \langle \mathbf v_1, \mathbf v_2\rangle   & \|\mathbf v_1\|^2
\end{pmatrix}.$ $\tilde{C}$ is a self-adjoint matrix.
 The norm of $\tilde{C}$ is
$$\frac{2\|\mathbf v_1\|^2+\|\mathbf v_2\|^2+\sqrt{\|\mathbf v_2\|^4-4|\langle \mathbf v_1, \mathbf v_2
\rangle|^2}}{2}.$$ Hence $\|A_1 \otimes  \mathbf v_1+A_2\otimes
\mathbf v_2\|\leq 1$ is equivalent to the inequality
$$2\|\mathbf v_1\|^2+\|\mathbf v_2\|^2+\sqrt{\|\mathbf v_2\|^4-4|\langle \mathbf v_1, \mathbf v_2
\rangle|^2}\leq 2.$$

This completes the proof.
\end{proof}
As a consequence of this Theorem, we have the following corollary.
\begin{corollary} There exists a contractive homomorphism of
$\mathcal O(\Omega_{\mathbf A})$ which is not complete
contractive.
\end{corollary}
\begin{proof} Let $\mathbf v_1=(\frac{1}{\sqrt{2}}, 0)$ and
$\mathbf v_2=(0, 1).$ Then it follows from Theorem \ref{con homo}
that $\|\rho_{V}\|\leq 1$ for this pair $(\mathbf v_1, \mathbf
v_2).$ Also, we have $\|\rho_{V}^{(2)}(P_{\mathbf A})\|>1.$ This
completes the proof.
\end{proof}

Let $\Omega_{\mathbf A}=\{(z_1, z_2, z_3): \|z_1A_1+z_2
A_2+z_3A_3\| < 1\} \subset \mathbb C^3 ,$ where  $A_1=
\left(\begin{smallmatrix}
1 & 0\\
    0    & 0
\end{smallmatrix}\right),$
$A_2=\left(\begin{smallmatrix}
0 & 1\\
    0    & 0
\end{smallmatrix}\right)$ and $ A_3= \left(\begin{smallmatrix}
0 & 0\\
    0    & 1
\end{smallmatrix}\right).$  The matrix representing
$L_{V}^*:(\mathbb C^3,\|\,\cdot\,\|_2)\rightarrow (\mathbb
C^3,\|\,\cdot\,\|_{\Omega_{\mathbf A}})$  is of the form $\left (
\begin{smallmatrix}
v_{11} & 0 &0  \\
0 & v_{22}& 0\\
0 & 0& v_{33}
\end{smallmatrix} \right ), V:=\left(\begin{smallmatrix}\mathbf v_1\\\mathbf v_2\\\mathbf v_3\end{smallmatrix}\right),$
 where
$\mathbf v_1=(v_{11}, 0, 0), \mathbf v_2=(0, v_{22}, 0)$ and
$\mathbf v_3=(0, 0, v_{33}).$

\begin{theorem}\label{cont homom}
$\|\rho_{V}^*\|_{(\mathbb C^3,\|\,\cdot\,\|_2)\rightarrow (\mathbb
C^3,\|\,\cdot\,\|_{\mathbf A})}\leq 1$ if and only if
$|v_{11}|^2(1-|v_{33}|^2) \geq (|v_{22}|^2-|v_{33}|^2).$
\end{theorem}
\begin{proof}Suppose $\|\rho_{V}^*\|_{(\mathbb C^3,\|\,\cdot\,\|_2)\rightarrow
(\mathbb C^3,\|\,\cdot\,\|_{\mathbf A})}\leq 1.$ Note that
$\|\rho_{V}^*\|_{(\mathbb C^3,\|\,\cdot\,\|_2)\rightarrow (\mathbb
C^3,\|\,\cdot\,\|_{\mathbf A})}\leq 1$ is equivalent to
$$\inf_{\beta}\det\left (
\begin{smallmatrix}
1-|\beta_1|^2|v_{11}|^2& 0 & 0 \\0 & 1-|\beta_1|^2|v_{22}|^2 &-
\beta_2\bar{\beta_1}v_{22}\bar{v}_{33}
 \\
0  & -\bar{\beta_2}\beta_1\bar{v}_{22}v_{33}  & 1 -|\beta_2
v_{33}|^2
\end{smallmatrix}\right )\geq 0$$ with $|v_{11}|^2 \leq 1, |v_{22}|^2 \leq 1,
|v_{33}|^2\leq 1,$ where $\sum_{i=1}^{3}|\beta_i|^2=1.$ Now,
\begin{align}\label{contra1}\lefteqn{\nonumber
\inf_{\beta}\det\left (
\begin{smallmatrix}
1-|\beta_1|^2|v_{11}|^2& 0 & 0 \\0 & 1-|\beta_1|^2|v_{22}|^2 &-
\beta_2\bar{\beta_1}v_{22}\bar{v}_{33}
 \\
0  & -\bar{\beta_2}\beta_1\bar{v}_{22}v_{33}  & 1 -|\beta_2
v_{33}|^2
\end{smallmatrix}\right
)}\\&=\inf_{\beta}(1-|\beta_1|^2|v_{11}|^2)\{(1-|\beta_1|^2|v_{22}|^2)(1-|\beta_2|^2|v_{33}|^2)-
|\beta_2|^2|v_{33}|^2|\beta_1|^2|v_{22}|^2\}.
\end{align}  Putting $|\beta|^2=r$ in Equation $(\ref{contra1})$ we have
%\begin{align}\label{contr2}\lefteqn{\nonumber
%\inf_{0\leq r\leq 1}\det\left (
%\begin{smallmatrix}
%1-|\beta_1|^2|v_{11}|^2& 0 & 0 \\0 & 1-|\beta_1|^2|v_{22}|^2 &-
%\beta_2\bar{\beta_1}v_{22}\bar{v}_{33}
% \\
%0  & -\bar{\beta_2}\beta_1\bar{v}_{22}v_{33}  & 1 -|\beta_2
%v_{33}|^2
%\end{smallmatrix}\right
%)}\\&=
$$\inf_{0\leq r\leq
1}\{1-r(|v_{11}|^2+|v_{22}|^2-|v_{11}|^2|v_{33}|^2)
-(1-r)|v_{33}|^2+r^2|v_{11}|^2(|v_{22}|^2-|v_{33}|^2\}.$$
%\end{align}
Let
$$f(r)=\inf_{0\leq r\leq 1}\{1-r(|v_{11}|^2+|v_{22}|^2-|v_{11}|^2|v_{33}|^2)
-(1-r)|v_{33}|^2+r^2|v_{11}|^2(|v_{22}|^2-|v_{33}|^2)\}.$$ The
derivative of $f$ with respect $r$ is equal to
$$f^{\prime}(r)=-(|v_{11}|^2+|v_{22}|^2-|v_{33}|^2-|v_{11}|^2|v_{33}|^2)+2r|v_{11}|^2(|v_{22}|^2-|v_{33}|^2).$$Now,
if $f^{\prime}(r)=0,$ then we have
$$r=\frac{(|v_{11}|^2+|v_{22}|^2-|v_{33}|^2-|v_{11}|^2|v_{33}|^2)}{|v_{11}|^2(|v_{22}|^2-|v_{33}|^2)}.$$
Also, $f^{\prime\prime}(r)=|v_{11}|^2(|v_{22}|^2-|v_{33}|^2).$ If
$|v_{22}|^2 >|v_{33}|^2$ then $f^{\prime\prime}(r)>0.$ Therefore,
the infimum is
%achieved  at
%$$r=\frac{(|v_{11}|^2+|v_{22}|^2-|v_{33}|^2-|v_{11}|^2|v_{33}|^2)}{|v_{11}|^2(|v_{22}|^2-|v_{33}|^2)}$$
%and is
equal to $$|v_{11}|^2(1-|v_{33}|^2)- (|v_{22}|^2-|v_{33}|^2).$$
Hence $\|\rho_{V}^*\|_{(\mathbb C^3,\|\,\cdot\,\|_2)\rightarrow
(\mathbb C^3,\|\,\cdot\,\|_{\mathbf A})}\leq 1$ if and only if
$$|v_{11}|^2(1-|v_{33}|^2)\geq (|v_{22}|^2-|v_{33}|^2).$$
This completes the proof.
\end{proof}
Let $P_{\mathbf A}:\Omega_{\mathbf A}\rightarrow (\mathcal M_2)_1$
be the matrix valued polynomial on $\Omega_{\mathbf A}.$ Now we
want to estimate the norm of $\|A_1 \otimes  \mathbf
v_1+A_2\otimes
 \mathbf v_2+A_3 \otimes \mathbf v_3\|.$
\begin{theorem}
$\|A_1 \otimes  \mathbf v_1+A_2\otimes  \mathbf v_2+A_3 \otimes
\mathbf v_3\|\leq 1$ if and only if $|v_{11}|^2+|v_{22}|^2\leq 1$
and $|v_{33}|^2 \leq 1.$
\end{theorem}
\begin{proof}
Assume that $\|A_1 \otimes  \mathbf v_1+A_2\otimes  \mathbf
v_2+A_3 \otimes \mathbf v_3\|\leq 1$. Now, {\small \begin{align*}
\|A_1 \otimes \mathbf v_1+A_2\otimes  \mathbf v_2+A_3 \otimes
\mathbf v_3\|^2&=\|(v_{11}A_1, v_{22}A_2,
v_{33}A_3\|^2\\&= \Big\|\begin{pmatrix}|v_{11}|^2+|v_{22}|^2 & 0\\
0 & |v_{33}|^2\end{pmatrix}\Big\|\\&= \max\{|v_{11}|^2+|v_{22}|^2,
|v_{33}|^2\}.
\end{align*}}
Therefore, $\|A_1 \otimes  \mathbf v_1+A_2\otimes  \mathbf v_2+A_3
\otimes \mathbf v_3\|\leq 1$ implies that
$|v_{11}|^2+|v_{22}|^2\leq 1$ and $|v_{33}|^2 \leq 1.$
 This completes the proof.
\end{proof}
 As before we will prove the following corollary.
\begin{corollary}
There exists a contractive homomorphism of $\mathcal
O(\Omega_{\mathbf A})$ which is not complete contractive.
\end{corollary}
\begin{proof}
Let $\mathbf v_1=(\frac{1}{2},0 ,0), \mathbf  v_2=(0, 1,0) $ and
$\mathbf v_3=(0, 0, 1).$ Then it follows from the Theorem
\ref{cont homom} that $\|\rho_{V}\|\leq 1.$ Also, we have
$\|\rho_{V}^{(2)}(P_{\mathbf A})\|>1.$ This completes the proof.
\end{proof}

\thispagestyle{empty} \cleardoublepage

\thispagestyle{empty}
\cleardoublepage
\chapter{Operator spaces}
We recall the notion of an operator space. We describe two
distinguished operator spaces, namely, the MIN and MAX operator
spaces. These two operator spaces have played an important role in
the development of operator theory in the recent past.
\begin{definition} (cf. \cite[Chapter 13, 14]{paulsen}) \label{def1}
An abstract operator space is a linear space $\mathbf V$ together
with a family of norms $\|\cdot\,\|_{k}$ defined on $\mathcal
M_{k}(V)$, $k=1,2, 3, \ldots ,$ where $\|\,\cdot\|_1$ is simply a
norm on the linear space $\mathbf V$. These norms are required to
satisfy the following compatibility conditions:
\begin{enumerate}
\item $\|T \oplus S\|_{p+q}=\max\{\|T\|_p, \|S\|_q\}$ and \item
$\|ASB\|_{q} \leq \|A\|_{\rm op}\|S\|_p\|B\|_{\rm op}$
\end{enumerate}
for all $S \in \mathcal M_q(\mathbf V), T \in \mathcal M_p(\mathbf
V)$ and $A \in \mathcal M_{q\, p}(\mathbb C), B \in \mathcal
M_{p\, q}(\mathbb C).$
\end{definition}
Two such operator spaces $(\mathbf V, \|\cdot\|_{ k})$ and
$(\mathbf W, \|\cdot\|_{ k})$  are said to be completely isometric
if there is a linear bijection $T : \mathbf V \to \mathbf W$ such
that $T \otimes I_{k}:(\mathcal M_{k}(\mathbf V), \|\cdot\|_{
k})\to (\mathcal M_{k}(\mathbf W), \|\cdot\|_{ k})$ is an isometry
for every $k \in \mathbb N.$ Here we have identified $\mathcal
M_k(\mathbf V)$ with $\mathbf V \otimes \mathcal M_k$ in the usual
manner. We note that a normed linear space $(\mathbf V,
\|\,\cdot\|)$ admits an operator space structure if and only if
there is an isometric embedding of it into the algebra of
operators $\mathcal B(\mathcal H)$ on some Hilbert space $\mathcal
H$. This is the well-known theorem of Ruan (cf. \cite{ruan}). Here
we study ``different'' operator space structures on a finite
dimensional normed linear space which admit an isometric embedding
onto a subspace of $\mathcal M_n(\mathbb C)$. %Thus we set
%$$\|(z_1, \cdots, z_m)\|_{\mathbf
%A}:=\|z_1A_1+\cdots+z_mA_m\|_{\rm op},$$ where $\mathbf
%A=\{A_1,\ldots ,A_m\}$ is a set of linearly independent elements
%in $\mathcal M_n(\mathbb C).$
Let $\mathbf A=(A_1, \ldots, A_m),$ where $A_1, \ldots, A_m$ are
 $n\times n$ linearly independent matrices. For $(z_1,\ldots,z_m)$ in $\mathbb C^m,$ set
$\|(z_1,\ldots,z_m)\|_{\mathbf A}=\|z_1A_1+\cdots+z_mA_m\|_{\rm
op}.$ Let $\mathbf V_{\mathbf A}$ be the $m$-dimensional normed
linear space with respect to the  norm $\|(z_1, \cdots,
z_m)\|_{\mathbf A}.$ This makes the map $$(z_1, \ldots,
z_m)\rightarrow z_1A_1+\cdots+z_mA_m$$ an isometry from $\mathbf
V_{\mathbf A}$ into $(\mathcal M_n, \|\cdot\|_{\rm op}).$
Therefore, $\mathbf V_{\mathbf A}$ inherits  an operator space
structure from $\mathcal M_n.$ Recall that  $\Omega_{\mathbf A}$
is the unit ball with respect to the norm $\|(z_1, \cdots,
z_m)\|_{\mathbf A}.$  Let $\mathbf V_{\mathbf A^{\rm t}}$ be the
normed linear space obtained by using the transpose, namely,
$\mathbf A^{\rm t}:= (A_1^{\rm t}, \ldots A_m^{\rm t}).$ By
definition, $\mathbf V_{\mathbf A^{\rm t}}$ has an isometric
embedding into $\mathcal M_n$ giving it an operator space
structure. Let $\Omega_{\mathbf A^{\rm t}}$ be the unit ball with
respect to the norm $\|\,\cdot\|_{\mathbf A^{\rm t}}$. Evidently,
the two normed linear spaces
 $(\mathbf V_{\mathbf A}, \|\cdot\|_{ \mathbf A})$ and
$(\mathbf V_{\mathbf A^{\rm t}}, \|\cdot\|_{\mathbf A^{\rm t}})$ are isometric.
We ask if the operator space structures they inherit from $\mathcal M_n(\mathbb C)$
 via the embedding involving the map induced by $\mathbf A$ and $\mathbf A^{\rm t}$ are isometric.
 In case these operator space structures are isometric, what are other possible operator space structures
 on $(\mathbf V_{\mathbf A}, \|\,\cdot\|_{\mathbf A})?$
We answer this question, after recalling the notions of MIN and
MAX operator spaces and a measure of their distance, namely,
$\alpha(\mathbf V)$ following \cite[Chapter 14]{paulsen}).
\begin{definition}
The MIN operator structure $MIN(\mathbf V)$ on a (finite
dimensional) normed linear space is obtained by isometrically
embedding $ \mathbf V $ in the $C^*$ algebra $C\big ( (\mathbf
V^*)_1 \big ),$ of continuous functions on the unit ball $(\mathbf
V^*)_1$ of the dual space. Thus for $(\!(v_{ij})\!)$ in $\mathcal
M_k(\mathbf V)$, we set
$$\|(\!(v_{ij})\!)\|_{MIN}=\|(\!(\widehat{v_{ij}})\!)\|=\sup\{\|(\!(f(v_{ij}))\!)\|:f
\in (\mathbf V^*)_1\},$$ where the norm of a scalar matrix
$(\!(f(v_{ij}))\!)$ in $\mathcal M_k$ is the operator norm.
\end{definition}
For an arbitrary $k \times k$ matrix over $\mathbf V,$ we simply
write $\|(\!(v_{ij})\!)\|_{MIN(\mathbf V)}$ to denote its norm in
$\mathcal M_{k}(\mathbf V).$ This is the minimal way in which we
represent the normed space as an operator space. However, it is
not difficult to create a ``maximal" representation. We shall
denote it by $MAX(\mathbf V).$
\begin{definition}
The matrix normed space $MAX(\mathbf V)$ is defined by setting
$$\|(\!(v_{ij})\!)\|_{MAX}=\sup\{\|(\!(T(v_{ij}))\!)\|:T:\mathbf V \rightarrow B(\mathcal H) \},$$
 and the supremum is taken over all isometries $T$
and all Hilbert spaces $\mathcal H.$
\end{definition}
It is easy to verify that every operator space structure on a
normed linear space $\mathbf V$ lies between $MIN(\mathbf V)$ and
$MAX(\mathbf V).$ To aid the understanding of the extent to which
the two operator space structures $MIN(\mathbf V)$ and
$MAX(\mathbf V)$ differ, Paulsen introduced the constant
$\alpha(\mathbf V)$ (cf.\cite[Chapter 14]{paulsen}), which we
recall below.
\begin{definition}
The constant $\alpha(\mathbf V)$ is defined by setting
$$\alpha(\mathbf V)=\sup\{\|(\!(v_{ij})\!)\|_{MAX}:\|(\!(v_{ij})\!)\|_{MIN}
\leq 1,\,\, (\!(v_{ij})\!) \in \mathcal M_k(\mathbf V), k\in
\mathbb N\}.$$
\end{definition}
Thus $\alpha(\mathbf V)=1$ if and only if the identity map is a
complete isometry from $MIN(\mathbf V)$ to $MAX(\mathbf V).$
Equivalently, we conclude that there exist a unique operator space
structure on $\mathbf V$ whenever $\alpha(\mathbf V)$ is $1.$
Therefore, those normed linear spaces for which $\alpha(\mathbf
V)=1$ are rather special. Unfortunately, there aren't too many of
them. The known examples are the $(\mathbb C,|\cdot|)$ and
$(\mathbb C,\|\cdot\|_{\infty})$ (resp. the unit ball in $\mathbb
C^2$ with respect to the $\ell_1$ norm). Indeed, Paulsen has shown
that $\alpha(\mathbf V)>1$ whenever $\dim(\mathbf V) \geq 5.$
Following this, Eric Ricard (cf. \cite{pisier} \cite{vern}) has
shown that $\alpha(\mathbf V)>1$ for $\dim(\mathbf V) \geq 3.$
This leaves the question open for normed linear spaces  whose
dimension is $2$. This is the question we address here in some
special cases.

\section{Operator norm calculation}
 The operator norm of the block matrix $S=\left(\begin{smallmatrix}\alpha I_m & B\\
0 & \alpha I_n\end{smallmatrix}\right),$ where $B$ is an $m \times
n$ matrix and $\alpha \in \mathbb C$ is not hard to compute (cf.
\cite [Lemma 2.1]{G}). This computation easily extends to an
operator $T$ of the form
$T=\left(\begin{smallmatrix}\alpha_1 I_m & B_{m \times n}\\
0 & \alpha_2 I_n\end{smallmatrix}\right),$ where $\alpha_1 \neq
\alpha_2$ are in $\mathbb C.$ Here we provide the straightforward
computation following \cite [Lemma 2.1]{G}.
\begin{lemma}\label{lem:con} $$\|T\|_{\rm op}
=\frac{(|\alpha_2|^2+\|B\|^2+|\alpha_1|^2)+
\sqrt{(|\alpha_2|^2+\|B\|^2-|\alpha_1|^2)^2+4\|B\|^2|\alpha_1|^2}}{2}.$$
\end{lemma}
\begin{proof}
Note that $\det\left(\begin{smallmatrix}A & X\\
C & D\end{smallmatrix}\right)=\det D\det(A-XD^{-1}C)$ and

$$TT^*=|\alpha_1|^2I_{m+n}+\left(\begin{smallmatrix}BB^* &\bar{\alpha}_2B\\
\\\alpha_2B^* & (|\alpha_2|^2-|\alpha_1|^2)I_m \end{smallmatrix}\right).$$ For $x\in \mathbb C,$ we have
{\small\begin{align*}\det\Big(\begin{smallmatrix}BB^*-xI_n &\bar{\alpha}_2B\\
\\\alpha_2B^* & (a-x)I_m
\end{smallmatrix}\Big)&=\noindent\det\{BB^*-xI_m-\frac{|\alpha_2|^2}{(a-x)}BB^*\}
\det\{(a-x)I_n\}\\
&=\noindent(-1)^m\left(\det\{BB^*+\frac{(a-x)x}{|\alpha_1|^2+x}I_n\}\right)
(a-x)^{n-m}(|\alpha_1|^2+x)^m\\
&=\noindent(-1)^n\left(\det(U\{BB^*+\frac{(a-x)x}{|\alpha_1|^2+x}I_n\}U^*)\right)
(a-x)^{n-m}(|\alpha_1|^2+x)^m,
\end{align*}} where $U$ is a unitary which makes  $BB^*$ into a
diagonal matrix $D_1$ and $(|\alpha_2|^2-|\alpha_1|^2)=a.$
Thus the maximum eigenvalue of $\Big(\begin{smallmatrix}BB^* &\bar{\alpha}_2B\\
\\\alpha_2B^* & (|\alpha_2|^2-|\alpha_1|^2)I_m
\end{smallmatrix}\Big)$ is $$x=\frac{(|\alpha_2|^2+\|B\|^2-|\alpha_1|^2)+
\sqrt{(|\alpha_2|^2+\|B\|^2-|\alpha_1|^2)^2+4\|B\|^2|\alpha_1|^2}}{2}.$$
Using the spectral mapping theorem, the norm of $TT^*$ is
$$\frac{(|\alpha_2|^2+\|B\|^2+|\alpha_1|^2)+
\sqrt{(|\alpha_2|^2+\|B\|^2-|\alpha_1|^2)^2+4\|B\|^2|\alpha_1|^2}}{2}.$$
\end{proof}
\begin{corollary}\label{cont}$\big\|\left(\begin{smallmatrix}\alpha_1 I_m & B_{m \times n}\\
0 & \alpha_2 I_n\end{smallmatrix}\right)\big\|_{\rm op}\leq 1$ if
and only if $\|B\|^2\leq (1-|\alpha_1|^2)(1-|\alpha_2|^2).$
\end{corollary}
\begin{proof}From Lemma \ref{lem:con}, it follows that $\big\|\left(\begin{smallmatrix}\alpha_1 I_m & B_{m \times n}\\
0 & \alpha_2 I_n\end{smallmatrix}\right)\big\|\leq 1$ if and only
if
$$\frac{(|\alpha_2|^2+\|B\|^2+|\alpha_1|^2)+
\sqrt{(|\alpha_2|^2+\|B\|^2-|\alpha_1|^2)^2+4\|B\|^2|\alpha_1|^2}}{2}\leq
1$$ which is clearly equivalent to $$ \|B\|^2\leq
(1-|\alpha_1|^2)(1-|\alpha_2|^2).$$ This completes the proof.
\end{proof}
Let  $A_1=\left(\begin{smallmatrix}\alpha_1 & 0\\
0 & \alpha_2\end{smallmatrix}\right)$ and $A_2=\left(\begin{smallmatrix}0 & \beta\\
0 & 0\end{smallmatrix}\right)$ with $z_1, z_2 \in \mathbb C.$ The
norm computation in the preceding Corollary shows that
$\|z_1A_1+z_2A_2\|_{\rm op}\leq 1$ if and only if
$$|z_2\beta|^2\leq (1-|z_1 \alpha_1|^2)(1-|z_1 \alpha_2|^2).$$
However, we can say a little more. For
$B_1,B_2$ in $\mathcal M_{m+n}(\mathbb C)$ of the form  $B_1=\left(\begin{smallmatrix}\alpha_1I_m & 0\\
0 & \alpha_2I_m\end{smallmatrix}\right)$ and  $B_2=\left(\begin{smallmatrix}0& B\\
0 & 0\end{smallmatrix}\right),$ set $\|(z_1, z_2)\|_{\mathbf
B}:=\|z_1B_1+z_2B_2\|_{\rm op}.$ Proof of the following lemma is
the same as that of the corollary.
\begin{lemma}\label{lem:iso} $\|(z_1, z_2)\|_{\mathbf B}\leq 1$ if and only
if $|z_2|^2\|B\|^2\leq (1-|z_1 \alpha_1|^2)(1-|z_1 \alpha_2|^2).$
\end{lemma}

\begin{remark}Let $ \widetilde{\mathbf B}$ be another pair
$(\widetilde{B}_1, \widetilde{B}_2)$ in  $\mathbb C^2
\otimes\mathcal M_{m+n}(\mathbb C)$ with $B_1=\widetilde{B}_1$ and
\mbox{$\|B_2\|=\|\widetilde{B}_2\|.$} It then follows that
$\|\cdot\|_{\mathbf B}=\|\cdot\|_{ \widetilde{\mathbf B}}.$ In
conclusion, we have shown that the normed linear space
$(V_{\mathbf A}, \|\cdot\|_{\mathbf A}),$
where $\mathbf A=\left(\left(\begin{smallmatrix}\alpha_1 & 0\\
0 & \alpha_2\end{smallmatrix}\right), \left(\begin{smallmatrix}0 & \beta\\
0 & 0\end{smallmatrix}\right)\right)$ has several different
isometric embedding into $\mathcal M_{m+n}(\mathbb C), m, n \in
\mathbb N.$ To see this simply pick any $B$ with $\|B\|=|\beta|.$
However, it is not clear if any of the embedding give rise to
distinct operator space.
\end{remark}
Let $\mathbf V_{\mathbf B}$ be the two dimensional linear space
with respect to the  norm $\|(z_1, z_2)\|_{\mathbf B}.$ Since
$B_1, B_2 \in \mathcal M_{m+n}(\mathbb C),$ we embed $\mathbf
V_{\mathbf B}$ into $\mathcal M_{m+n}(\mathbb C)$ isometrically.
Therefore we think of $\mathbf V_{\mathbf B}$ as an operator
space. Note that $P_{\mathbf B}:\mathbf V_{\mathbf B}\rightarrow
\mathcal M_{m+n}(\mathbb C)$ defines a linear isometric embedding
into $\mathcal M_{m+n}(\mathbb C).$ Suppose $V=(\!(\mathbf
v_{ij})\!) \in\mathcal M_{k}(\mathbf V_{\mathbf B}),$ where
$\mathbf v_{ij}\in \mathbf V_{\mathbf B}.$ We define $P_{\mathbf
B}^{(k)}:P_{\mathbf B}\otimes I_k:\mathcal M_{k}(\mathbf
V_{\mathbf B})\rightarrow \mathcal M_{k}(\mathcal M_{m+n}(\mathbb
C))$ by $P_{\mathbf B}^{(k)}(V)=(\!(P_{\mathbf B}(\mathbf
v_{ij})\!).$ Let $\mathbf v_{ij}=(v^{1}_{ij} \,\, v^{2}_{ij})$
then
$$P_{\mathbf B}^{(k)}(V)=\left(\begin{smallmatrix}\alpha_1V_1 \otimes I_m & V_2\otimes B\\
0 & \alpha_2V_1 \otimes I_n \end{smallmatrix}\right),$$ where
$V_1=(\!(v^{1}_{ij})\!)$ and $V_2=(\!(v^{2}_{ij})\!).$ The cb-norm
of $P_{\mathbf B}$ is defined to be $\sup_{k}\|P_{\mathbf
B}^{(k)}\|.$ To compute $\|P_{\mathbf B}^{(k)}\|$ norm, we will
need the following Lemma (cf. \cite[Theorem 1.3.3]{bhatia}).

\begin{lemma}\label{thm:Gal}Let $M=\left(\begin{smallmatrix}A &  X\\
X^* & C \end{smallmatrix}\right)$ be the block decomposition of
$M$ in $\mathcal M_{m+n}(\mathbb C).$  Assume that $A, C$ are
 positive definite. The following properties hold:
\begin{enumerate}
\item $M$ is positive-definite if and only if $A-XC^{-1}X^*$ is
positive-definite.

\item If $C$ is positive-definite, then $M$ is positive
semi-definite if and only if $A-XC^{-1}X^*$ is positive
semi-definite.
\end{enumerate}
\end{lemma}
Recall that a closed (resp. open) subset $\Omega$ of $\mathbb C^m$
is the closed (resp. open) unit ball in some norm if and only if
it is bounded, balanced, convex, and absorbing . If we set
$$\|x\|=\inf\{x: t^{-1}x \in \Omega, t>0\},$$ then $\|\cdot\|$ is
a norm on $\mathbb C^m$ and $\Omega$ is the closed (resp. open)
unit ball with respect to this norm (cf. \cite[Theorem
1.34]{Walter}). Now we will compute the norm $\|P_{\mathbf
B}^{(k)}\|.$ Let $U, V$ be unitaries that diagonalize  $V_1V_1^*$
and $BB^*,$ that is,
$D_1:=UV_1V_1^*U^*=\left(\begin{smallmatrix}|d_1|^2 &0 &\cdots
&0\\\vdots&\vdots&\vdots&\vdots\\ 0& 0&\cdots
&|d_n|^2\end{smallmatrix}\right)$ and
$D_2:VBB^*V^*=\left(\begin{smallmatrix}|c_1|^2 &0 &\cdots
&0\\\vdots&\vdots&\vdots&\vdots\\ 0& 0&\cdots
&|\beta|^2\end{smallmatrix}\right).$
\begin{theorem}\label{thm:main}$\left\|P_{\mathbf B}^{(k)}\right\|< 1$
if and only if $I_n-|\alpha_2|^2D_1 >0,$
$$I_n-|\alpha_1|^2D_1-|\beta|^2UV_2V_2^*U^*-|\alpha_2\beta|^2UV_2V_1^*U^*(I_n-|\alpha_2|^2D_1)^{-1}UV_1V_2^*U^*>0,$$
where
%$UV_1V_1^*U^*=D_1$ with
$\|B\|=|\beta|.$
\end{theorem}
\begin{proof}Let $S=\left(\begin{smallmatrix}\alpha_1V_1 \otimes I_m & V_2\otimes B\\
0 & \alpha_2V_1 \otimes I_m \end{smallmatrix}\right).$
Note that $\det\left(\begin{smallmatrix}A & B\\
C & D\end{smallmatrix}\right)=\det D\det(A-BD^{-1}C)$ and
$$SS^*=\left(\begin{smallmatrix}|\alpha_1|^2V_1 V_1^*\otimes
I_m+V_2V_2^*\otimes BB^*
&\bar{\alpha}_2 V_2V_1^*\otimes B\\
\alpha_2V_1V_2^*\otimes B^*& |\alpha_2|^2 V_1V_1^* \otimes I_m
\end{smallmatrix}\right).$$
If $\|S\|<1$ then we have $I_{2n+2m}-SS^*>0$ and conversely. Hence
$I_{2n+2m}-SS^*>0$ is equivalent to the  condition that {\small
\begin{align}\label{ct}
\left(\begin{smallmatrix}I_n\otimes I_m-|\alpha_1|^2V_1
V_1^*\otimes I_m+V_2V_2^*\otimes BB^*
&-\bar{\alpha}_2 V_2V_1^*\otimes B\\
-\alpha_2V_1V_2^*\otimes B^*& I_n \otimes I_m-|\alpha_2|^2
V_1V_1^* \otimes I_m
\end{smallmatrix}\right)>0.\end{align}
 Let
$\widetilde{U}=\left(\begin{smallmatrix}U \otimes V
&0\\
0& U \otimes V\end{smallmatrix}\right).$ Then it is easy to see
that $\widetilde{U}$ is an unitary. Multiply both side of the
Equation (\ref{ct}) by $\widetilde{U}$ and $\widetilde{U}^*$ we
have
\begin{align}\label{ct1}
\begin{pmatrix}I_n\otimes I_m-|\alpha_1|^2D_1\otimes
I_m+UV_2V_2^*U^*\otimes D_2
&-\bar{\alpha}_2U V_2V_1^*U^*\otimes VBV^*\\
-\alpha_2UV_1V_2^*U^*\otimes V B^*V^*& I_n \otimes
I_m-|\alpha_2|^2 D_1\otimes I_m
\end{pmatrix}>0.\end{align}}
Using Lemma \ref{thm:Gal} together with the Equation (\ref{ct1})
we have $I_n-|\alpha_2|^2D_1
>0$ and
$$I_n-|\alpha_1|^2D_1-|\beta|^2UV_2V_2^*U^*-|\alpha_2\beta|^2UV_2V_1^*U^*(I_n-|\alpha_2|^2D_1)^{-1}UV_1V_2^*U^*>0.$$
The converse statement is easily verified.
\end{proof}
Let $A_1=\left(\begin{smallmatrix}\alpha_1 & 0\\
0 & \alpha_2\end{smallmatrix}\right), A_2=\left(\begin{smallmatrix}0 & \beta\\
0 & 0\end{smallmatrix}\right)$ and $\mathbf A=(A_1, A_2).$ We can
think of $\mathbf V_{\mathbf A}$ as an operator space via the
linear isometric embedding $P_{\mathbf A} :\mathbf V_{\mathbf
A} \to \mathcal M_2.$ Therefore, we have $$P_{\mathbf A}^{(k)}(V)=\left(\begin{smallmatrix}\alpha_1V_1  & \beta V_2\\
0 & \alpha_2V_1  \end{smallmatrix}\right).$$ Taking $B$ to be the
scalar operator $\beta$ and $m=n=1,$ we have:
\begin{theorem}\label{main:thm3}$\left\|P_{\mathbf A}^{(k)}\right\|< 1$ if and only
if $I_n-|\alpha_2|^2D_1 >0,$
$$I_n-|\alpha_1|^2D_1-|\beta|^2UV_2V_2^*U^*-|\alpha_2\beta|^2UV_2V_1^*U^*(I_n-|\alpha_2|^2D_1)^{-1}UV_1V_2^*U^*>0$$
where $UV_1V_1^*U^*=D_1.$
\end{theorem}

\begin{remark} \label{remm1}
Theorems \ref{thm:main} and \ref{main:thm3} together imply that
$\|P_{\mathbf A}^{(k)}\|=\|P_{\mathbf B}^{(k)}\|, k=1,2, \ldots.$
Therefore, the operator space structure on $\mathbf V_{\mathbf A}$
obtained via $P_{\mathbf B}$ is independent of $B.$
\end{remark}
\section{Domains in $\mathbb C^2$}
In this section we study the class domains $\Omega_\mathbf
A=\{(z_1, z_2):\left\|z_1A_1+z_2A_2\right\|_{\rm op}
< 1\}$ in $\mathbb C^2,$ where $A_1=\left(\begin{smallmatrix}d_1 & 0\\
0 & d_2 \end{smallmatrix}\right), A_2=\left(\begin{smallmatrix}a & b\\
c& d \end{smallmatrix}\right)$ and $\mathbf A=(A_1, A_2).$ We may
assume without loss of generality, as shown in Chapter $1$ that
$A_1=\left (
\begin{smallmatrix}
1 & 0   \\
0  & d_{2}
\end{smallmatrix}\right )$ or $A_1=\left ( \begin{smallmatrix}
d_{1} & 0   \\
0  & 1
\end{smallmatrix}\right )$ and  $A_2= \left ( \begin{smallmatrix}
 0 &  b \\
c &  0
\end{smallmatrix}\right ) ,\left ( \begin{smallmatrix}
1 &  b \\
c &  0
\end{smallmatrix}\right ) $ or $A_2=\left ( \begin{smallmatrix}
 0 &  b \\
 c &  1
\end{smallmatrix}\right ) $ with $b \in \mathbb R^+.$

Consider the case: $A_1=\left (
\begin{smallmatrix}
1 & 0   \\
0  & d_{2}
\end{smallmatrix}\right ) ,A_2= \left ( \begin{smallmatrix}
1 &  b \\
c &  0
\end{smallmatrix}\right ) $ and $\mathbf A=(A_1, A_2).$ In particular, for $m=2, n=2,$ we can give
an operator space structure on the normed linear space $\mathbf
V_{\mathbf A}$ via the linear isometric embedding $P_{\mathbf
A}:\mathbf V_{\mathbf A} \rightarrow \mathcal M_{2}.$ For
$V=(\!(\mathbf v_{ij})\!) \in\mathcal M_{k}(\mathbf V_{\mathbf
A}),$ we
have $$P_{\mathbf A}^{(k)}(V)=\left(\begin{smallmatrix}V_3 & bV_2\\
cV_2 & d_2V_1 \end{smallmatrix}\right),$$ where  $\mathbf
v_{ij}\in \mathbf V_{\mathbf A}$ and $V_1=(\!(v^{1}_{ij})\!),
V_2=(\!(v^{2}_{ij})\!), V_3=V_1+V_2.$  Similarly we can think
$\mathbf V_{\mathbf A^{\rm t}}$ as an operator space via the
linear isometric embedding $P_{\mathbf A^{\rm t}}:\mathbf
V_{\mathbf A^{\rm t}}\rightarrow \mathcal M_{2},$ where $\mathbf
A^{\rm t}=(A_1^{\rm t}, A_2^{\rm t}).$ Therefore, we have
 $$P_{\mathbf A^{\rm t}}^{(k)}(V)=\left(\begin{smallmatrix}V_3 & cV_2\\
bV_2 & d_2V_1 \end{smallmatrix}\right).$$ Therefore, it is natural
to ask if these two operator space structure are same. The
following theorem says, in particular, that  $\|P_{\mathbf
A}^{(2)}(V)\|_{\rm op}\neq\|P_{\mathbf A^{\rm t}}^{(2)}(V)\|_{\rm
op}$ if and only if   $b \neq |c|$ and $1\neq |d_2|,$ for some $V$
in $\mathcal M_{2}(\mathbf V_{\mathbf A}).$
\begin{theorem}\label{thm:main2}
For  $V_1=\left(\begin{smallmatrix}v_{11}  & v_{12}\\
0 & 0\end{smallmatrix}\right)$ and $V_2=\left(\begin{smallmatrix}v_{21}  & v_{22}\\
0 & 0\end{smallmatrix}\right), V=(V_1, V_2)$ we have $\|P_{\mathbf
A}^{(2)}(V)\|_{\rm op}=\|P_{\mathbf A^{\rm t}}^{(2)}(V)\|_{\rm
op}$ if and only if either $1=|d_2|$ or $b=|c|.$
\end{theorem}
\begin{proof}
Note that
\begin{align}\|P_{\mathbf
A}^{(2)}(V)\|_{\rm op}^2&=\left\|\left(\begin{smallmatrix}V_3 & bV_2\\
cV_2 & d_2V_1 \end{smallmatrix}\right)\left(\begin{smallmatrix}V_3^{*} & \bar{c}V_2^{*}\\
bV_2^{*} & \bar{d}_2V_1^{*}
\end{smallmatrix}\right)\right\|_{\rm op}\nonumber\\&=
\left\|\left(\begin{smallmatrix}V_3V_3^{*}+b^2V_2V_2^{*} & \bar{c}V_3V_2^{*}
+b \bar{d}_2V_2V_1^{*}\\
cV_2V_3^{*}+bd_2V_1V_2^{*} & |c|^2V_2V_2^{*}
+|\bar{d}_2|^2V_1V_1^{*}
\end{smallmatrix}\right)\right\|_{\rm op}.\end{align}
 Similarly we have \begin{align}\|P_{\mathbf A^{\rm t}}^{(2)}(V)\|_{\rm op}^2&=\left\|
 \left(\begin{smallmatrix}V_3V_3^{*}+|c|^2V_2V_2^{*} & bV_3V_2^{*}
+c \bar{d}_2V_2V_1^{*}\\
bV_2V_3^{*}+\bar{c}d_2V_1V_2^{*} & b^2V_2V_2^{*}
+|\bar{d}_2|^2V_1V_1^{*}
\end{smallmatrix}\right)\right\|_{\rm op}.\end{align}
We first assume that  $\|P_{\mathbf A}^{(2)}(V)\|_{\rm
op}^2=\|P_{\mathbf A^{\rm t}}^{(2)}(V)\|_{\rm op}^2.$ Therefore,
the above condition is equivalent to
{\small\begin{align}\label{eqan1}\left\|\left(\begin{smallmatrix}V_3V_3^{*}+b^2V_2V_2^{*}
& \bar{c}V_3V_2^{*}
+b \bar{d}_2V_2V_1^{*}\\
cV_2V_3^{*}+bd_2V_1V_2^{*} & |c|^2V_2V_2^{*}
+|\bar{d}_2|^2V_1V_1^{*}
\end{smallmatrix}\right)\right\|_{\rm op}=\left\|\left(\begin{smallmatrix}V_3V_3^{*}+|c|^2V_2V_2^{*} & bV_3V_2^{*}
+c \bar{d}_2V_2V_1^{*}\\
bV_2V_3^{*}+\bar{c}d_2V_1V_2^{*} & b^2V_2V_2^{*}
+|\bar{d}_2|^2V_1V_1^{*}
\end{smallmatrix}\right)\right\|_{\rm op}.\end{align}}
Putting  $\mathbf v_1=(v_{11},\,\ v_{12}), \mathbf v_2=(v_{21},\,\
v_{22})$ and $\mathbf v_3=\mathbf v_1+\mathbf v_2$ in Equation
(\ref{eqan1}) we have
{\small\begin{align}\label{eqan2}\left\|\left(\begin{smallmatrix}\|\mathbf
v_3\|^2+b^2\|\mathbf v_2\|^2 & \bar{c}\langle \mathbf v_3, \mathbf
v_2\rangle
+b \bar{d}_2\langle \mathbf v_2, \mathbf v_1\rangle\\
c\langle \mathbf v_2, \mathbf v_3\rangle+bd_2 \langle \mathbf v_1,
\mathbf v_2\rangle & |c|^2\|\mathbf v_2\|^2
+|\bar{d}_2|^2\|\mathbf v_1\|^2
\end{smallmatrix}\right)\right\|_{\rm op}=\left\|\left(\begin{smallmatrix}\|\mathbf v_3\|^2+|c|^2\|\mathbf v_2\|^2
& b\langle \mathbf v_3, \mathbf v_2\rangle
+c \bar{d}_2\langle \mathbf  v_2, \mathbf v_1\rangle\\
b\langle \mathbf v_2, \mathbf v_3\rangle+\bar{c}d_2 \langle
\mathbf v_1, \mathbf v_2\rangle & b^2\|\mathbf v_2\|^2
+|\bar{d}_2|^2\|\mathbf v_1\|^2
\end{smallmatrix}\right)\right\|_{\rm op}.\end{align}}
The maximum eigenvalue of $\left(\begin{smallmatrix}\|\mathbf
v_3\|^2+b^2\|\mathbf v_2\|^2 & \bar{c}\langle \mathbf v_3, \mathbf
v_2\rangle
+b \bar{d}_2\langle \mathbf v_2, \mathbf v_1\rangle\\
c\langle \mathbf v_2, \mathbf v_3\rangle+bd_2 \langle \mathbf v_1,
\mathbf v_2\rangle & |c|^2\|\mathbf v_2\|^2
+|\bar{d}_2|^2\|\mathbf v_1\|^2
\end{smallmatrix}\right)$ is
$$x=\frac{p_1+\sqrt{p_{1}^{2}-4q_1}}{2},$$
 where $p_1=\|\mathbf v_3\|^2+b^2\|\mathbf v_2\|^2
+|c|^2\|\mathbf v_2\|^2 +|\bar{d}_2|^2\|\mathbf v_1\|^2,
q_1=(\|\mathbf v_3\|^2+b^2\|\mathbf v_2\|^2)(|c|^2\|\mathbf
v_2\|^2 +|\bar{d}_2|^2\|\mathbf v_1\|^2) -|\bar{c}\langle \mathbf
v_3, \mathbf v_2\rangle +b \bar{d}_2\langle \mathbf v_2, \mathbf
v_1\rangle|^2.$

 Similarly the maximum eigenvalue of $\left(\begin{smallmatrix}\|\mathbf v_3\|^2+|c|^2\|\mathbf v_2\|^2
& b\langle \mathbf v_3, \mathbf v_2\rangle
+c \bar{d}_2\langle \mathbf v_2, \mathbf v_1\rangle\\
b\langle \mathbf v_2, \mathbf v_3\rangle+\bar{c}d_2 \langle
\mathbf v_1, \mathbf v_2\rangle & b^2\|\mathbf v_2\|^2
+|\bar{d}_2|^2\|\mathbf v_1\|^2
\end{smallmatrix}\right)$ is
$$y=\frac{p_1+\sqrt{p_{1}^{2}-4q_2}}{2},$$ where $q_2=(\|\mathbf v_3\|^2+|c|^2\|\mathbf v_2\|^2)(b^2\|\mathbf v_2\|^2
+|\bar{d}_2|^2\|\mathbf v_1\|^2) -|b\langle \mathbf v_3, \mathbf
v_2\rangle +c \bar{d}_2\langle \mathbf v_2, \mathbf
v_1\rangle|^2.$ Since we are assuming $x=y,$ it follows that
$q_1=q_2.$ This simplifies to the equation
\begin{align}\label{eqan3}(|c|^2-b^2)\{(\|\mathbf v_3\|\|\mathbf v_2\|^2-|\langle \mathbf v_3,
\mathbf v_2\rangle |^2)-|\bar{d}_2|^2(\|\mathbf v_2\|\|\mathbf
v_1\|^2-|\langle \mathbf v_2, \mathbf
v_1\rangle|^2)\}=0.\end{align} Since $\mathbf v_3=\mathbf
v_1+\mathbf v_2,$ we have $(\|\mathbf v_3\|\|\mathbf
v_2\|^2-|\langle \mathbf v_3, \mathbf v_2\rangle |^2)=(\|\mathbf
v_2\|\|\mathbf v_1\|^2-|\langle \mathbf v_2, \mathbf
v_1\rangle|^2).$ Hence either $b=|c|$ or  $1=|d_2|.$

Conversely if either $b=|c|$ or  $1=|d_2|,$ then $q_1=q_2.$
Therefore,, we have $x=y,$ that is, the two maximum eigenvalue are
equal. Hence $\|P_{\mathbf A}^{(2)}(V)\|_{\rm op}=\|P_{\mathbf
A^{\rm t}}^{(2)}(V)\|_{\rm op}.$ This completes our theorem.
\end{proof}
\begin{remark}\label{remm2}
From  Theorem \ref{thm:main2} it follows that if  $b \neq |c|$ and
$|d_2|\neq 1,$ then $\|P_{\mathbf A}^{(2)}(V)\|_{\rm
op}\neq\|P_{\mathbf A^{\rm t}}^{(2)}(V)\|_{\rm op}$ for $V \in
\mathcal M_2(\mathbf V_{\mathbf A})$ of the form $(\!(\mathbf
v_{ij})\!), \mathbf v_{ij} \in \mathbf V_{\mathbf A}, \mathbf
v_{ij}=0$ if $i>1.$ Thus $P_{\mathbf A}$ and $P_{\mathbf A^{\rm
t}}$ induce different operator space structure on $\mathbf
V_{\mathbf A}.$ Equivalently, there is a contractive homomorphism
on $\mathcal O(\Omega_{\mathbf A}),$ which is not completely
contractive. %Similarly if $b \neq |c|$ and $|d_2| \neq 1,$ then we
%can prove also either $A_1=\left (
%\begin{smallmatrix}
%d_{1} & 0   \\
%0  & 1
%\end{smallmatrix}\right ),\left (
%\begin{smallmatrix}
%1 & 0   \\
%0  & d_{2}
%\end{smallmatrix}\right )$ or $ A_2=\left ( \begin{smallmatrix}
% 0 &  b \\
 %c &  1
%\end{smallmatrix}\right ) $ and $A_1=\left (
%\begin{smallmatrix}
%1 & 0   \\
%0  & d_{2}
%\end{smallmatrix}\right ),  A_2=\left ( \begin{smallmatrix}
% 0 &  b \\
 %c &  1
%\end{smallmatrix}\right ) $ have two different operator space
%structure.
Let $A_1\in \{A_{11},A_{12}\}$ and $A_2\in \{A_{21},A_{22}\},$
where $A_{11}=\left (
\begin{smallmatrix}
1 & 0   \\
0  & d_{2}
\end{smallmatrix}\right ), A_{12}=\left (
\begin{smallmatrix}
d_{1} & 0   \\
0  & 1
\end{smallmatrix}\right ); A_{21}=\left ( \begin{smallmatrix}
 1 &  b \\
 c &  0
\end{smallmatrix}\right ),A_{22}=\left ( \begin{smallmatrix}
 0 &  b \\
 c &  1
\end{smallmatrix}\right ) .$ We have proved the theorem for
$A_1=A_{11}$ and $A_2=A_{21}.$ The proof in the remaining cases,
namely, $A_1=A_{11}$ and $A_2=A_{22};$ $A_1=A_{12}$ and
$A_2=A_{21}$ and $A_1=A_{12}$ and $A_2=A_{22}$ follow similarly.

\end{remark}
If we consider the case  $\mathbf A=\left(\left (
\begin{smallmatrix}
1 & 0   \\
0  & d_{2}
\end{smallmatrix}\right ), A_2= \left ( \begin{smallmatrix}
0 &  b \\
c &  0
\end{smallmatrix}\right )\right),$ then we have
 $P_{\mathbf A}^{(k)}(V)=\left(\begin{smallmatrix}V_1  & bV_2\\
cV_2 & d_2V_1 \end{smallmatrix}\right)$ and $P_{{\mathbf A}^{\rm t}}^{(k)}(V)=\left(\begin{smallmatrix}V_1 & cV_2\\
bV_2 & d_2V_1 \end{smallmatrix}\right).$ The following theorem is
similar to Theorem \ref{thm:main2}.
\begin{theorem}\label{thm:main3}
Suppose $V_1=\left(\begin{smallmatrix}v_{11}  & v_{12}\\
0 & 0\end{smallmatrix}\right)$ and $V_2=\left(\begin{smallmatrix}v_{21}  & v_{22}\\
0 & 0\end{smallmatrix}\right).$ Then $\|P_{\mathbf
A}^{(2)}(V)\|_{\rm op}=\| P_{{\mathbf A}^{\rm t}}^{(2)}(V)\|_{\rm
op}$ if and only if either $1=|d_2|$ or $b=|c|.$
\end{theorem}
\begin{proof}
Note that
$$\|P_{\mathbf A}^{(2)}(V)\|_{\rm op}^2=\left\|\left(\begin{smallmatrix}\|\mathbf v_1\|^2+b^2\|\mathbf v_2\|^2
& \bar{c}\langle \mathbf v_1, \mathbf v_2\rangle
+b \bar{d}_2\langle \mathbf v_2, \mathbf v_1\rangle\\
c\langle \mathbf v_2, \mathbf v_1\rangle+bd_2 \langle \mathbf v_1,
\mathbf v_2\rangle & |c|^2\|\mathbf v_2\|^2
+|\bar{d}_2|^2\|\mathbf v_1\|^2
\end{smallmatrix}\right)\right\|_{\rm op}$$ and

$$\|P_{{\mathbf A}^{\rm t}}^{(2)}(V)\|_{\rm op}^2=\left\|\left(\begin{smallmatrix}\|\mathbf v_1\|^2+|c|^2\|\mathbf v_2\|^2
& b\langle \mathbf v_1, \mathbf v_2\rangle
+c \bar{d}_2\langle \mathbf v_2, \mathbf v_1\rangle\\
b\langle \mathbf v_2, \mathbf v_1\rangle+\bar{c}d_2 \langle
\mathbf v_1, \mathbf v_2\rangle & b^2\|\mathbf v_2\|^2
+|\bar{d}_2|^2\|\mathbf v_1\|^2
\end{smallmatrix}\right)\right\|_{\rm op}.$$ Let $x, y$ the maximum eigen
value of $\|P_{\mathbf A}^{(2)}(V)\|_{\rm op}^2, \|P_{{\mathbf
A}^{\rm t}}^{(2)}(V)\|_{\rm op}^2$ respectively. Since we are
assuming $x=y,$ it follows that
\begin{align}(|c|^2-b^2)\{\|\mathbf v_1\|^2\|\mathbf v_2\|^2-|\langle \mathbf v_1, \mathbf v_2\rangle
|^2-|\bar{d}_2|^2(\|\mathbf v_2\|^2\|\mathbf v_1\|^2-|\langle
\mathbf v_2, \mathbf v_1\rangle|^2\}=0.\end{align} Therefore, we
have either $1=|d_2|$ or $b=|c|.$ The converse is also easy to
verify.
\end{proof}
 \begin{remark}
As before we conclude that if $b \neq |c|$ and $ 1 \neq |d_2|,$
then the  operator structures induced by $P_{\mathbf A}$ and
$P_{\mathbf A^{\rm t}}$ are not isomorphic. Let $A_1\in
\{A_{11},A_{12}\}$ and $A_2=\left (
\begin{smallmatrix}
 0 &  b \\
 c &  0
\end{smallmatrix}\right ),$ where $A_{11}=\left (
\begin{smallmatrix}
1 & 0   \\
0  & d_{2}
\end{smallmatrix}\right ), A_{12}=\left (
\begin{smallmatrix}
d_{1} & 0   \\
0  & 1
\end{smallmatrix}\right ).$ We have proved the theorem for
$A_1=A_{11}$ and $A_2=\left ( \begin{smallmatrix}
 0 &  b \\
 c &  0
\end{smallmatrix}\right ).$ The proof in the remaining case, namely,
$A_1=A_{12}$ and $A_2=\left ( \begin{smallmatrix}
 0 &  b \\
 c &  0
\end{smallmatrix}\right )$ follows similarly. Equivalently, we also say that there exists
a contractive linear map from $(\mathbb C^2,
\|\cdot\|_{\Omega_\mathbf A})$ to $\mathcal M_n(\mathbb C)$ which
is not completely contractive.%For either $1=|d_2|$
%or $b=|c|,$ we will prove in Chapter $5$ that the normed space
%$\mathbf V_{\mathbf A}$ has two operator space structures. We will
%also prove that if $1=|d_2|$ and $b=|c|,$ then the normed space
%$\mathbf V_{\mathbf A}$ has one operator space structure.
\end{remark}
This phenomenon also occurs for the Euclidean Ball as the
following example.
\begin{example}
Let $\mathbf A=\left(\left (
\begin{smallmatrix}
1 & 0   \\
0  &  0
\end{smallmatrix}\right ) ,
\left ( \begin{smallmatrix}
0&  1\\
0  &  0
\end{smallmatrix}\right )\right).$ Then $\Omega_{\mathbf A}$
defines Euclidean ball. As before we have $P_{\mathbf A}^{(k)}(V)=\left(\begin{smallmatrix}V_1  & V_2\\
0 & 0\end{smallmatrix}\right)$ and $P_{\mathbf A^{\rm t}}^{(k)}(V)
=\left(\begin{smallmatrix}
V_1 & 0  \\
V_2  &  0
\end{smallmatrix}\right ).$ For $V_1=(v_{11}\,\,v_{12}), V_2= (v_{21}\,\, v_{22})$ and
$V=(V_1\,\, V_2),$ it is easy to verify that $\|P_{\mathbf
A}^{(k)}(V)\|_{\rm op}\neq \|P_{\mathbf A^{\rm t}}^{(k)}(V)
\|_{\rm op}.$ Hence two embedding of $(\mathbf V_{\mathbf A},
\|\cdot\|_{2})$ into $\mathcal M_2(\mathbb C)$ give two distinct
operator space structure.
\end{example}

\chapter{Bergman kernel}

We recall the  definition of the well known class of operators
$\mathcal P_n(\Omega)$ which was introduced in the foundational
paper of Cowen and Douglas  \cite{cowen}. An alternative point of
view was discussed in the paper of Curto and Salinas \cite{curto}.
\begin{definition}
The class $\mathcal P_n(\Omega)$ consists of m-tuples of commuting
bounded operators  $T=(T_1, T_2, \ldots T_m)$ on a Hilbert space
$\mathcal H$ satisfying the following conditions:
\begin{itemize}
\item the operators $T_1, T_2, \ldots, T_m$ commute,

\item for $w=(w_1,\dots , w_m)\in\Omega,$ the dimension of the
joint kernel $\bigcap_{k=1}^{m}\ker(T_{k}-w_{k})$ is $n$,

\item for $w\in\Omega,$ the operator $D_{T-w}:\mathcal H
\rightarrow \mathcal H\oplus \cdots \oplus \mathcal H$ has closed
range, where the operator $D_T$ is defined by
$D_T\,h=\oplus_{k=1}^m T_{k}h$, $h\in \mathcal H$,

\item closed
span$\{\bigcap_{k=1}^{m}\ker(T_{k}-w_{k}):w\in\Omega\}=\mathcal
H.$
\end{itemize}
\end{definition}

Here we relate the contractivity of the homomorphism $\rho_{T}(w)$
naturally induced by the localization $N_{T}(w), w \in \Omega,$ of
an $m$-tuple of operator $T$ in $\mathcal P_1(\Omega)$ to that of
its curvature $\mathcal K(w)$ corresponding to the holomorphic
Hermtian bundle corresponding to the commuting tuple $T$.

For an $m$-tuple of operators $T$ in $\mathcal P_n(\Omega)$, Cowen
and Douglas establish the existence of a non-zero holomorphic map
$\gamma:\Omega_0\rightarrow\mathcal H$ with $\gamma(w)$ in
$\bigcap_{k=1}^{m}\ker(T_{k}-w_{k})$, $w$ in some open  subset
$\Omega_0$ of $\Omega.$ We fix such an open set and call it
$\Omega$. The map $\gamma$ defines a holomorphic Hermitian vector
bundle, say $E_T$, on $\Omega$. They show that the equivalence
class of the vector bundle $E_T$ determines the equivalence class
(with respect to unitary equivalence) of the operator $T$ and
conversely. The determination of the
 equivalence class of the operator $T$ in $\mathcal P_1(\Omega)$ then is  particularly simple since
 the curvature of the line bundle $E_T$
$$-\mathsf K(w) = \sum_{i,j=1}^m \frac{\partial^2}{\partial w_i\partial\bar{w}_j}\log\|\gamma(w)\|^2 dw_i
\wedge d\bar{w}_j$$ is a complete invariant. We reproduce the
well-known proof of this fact for the sake of completeness.

Suppose that $E$ is a holomorphic Hermitian line bundle over a
bounded domain $\Omega\subseteq \mathbb C^m$. Pick a holomorphic
frame  $\gamma$ for the line bundle $E$ and let $\Gamma(w)=
\langle{\gamma_w , {\gamma_w}\rangle}$ be the Hermitian metric.
The curvature $(1,1)$ form $\mathsf K(w)\equiv 0$ on an open
subset $\Omega_0 \subseteq \Omega$ if and only if $\log \Gamma$ is
harmonic on $\Omega_0$.  Let $F$ be a second line bundle over the
same domain $\Omega$ with the metric $\Lambda$ with respect to a
holomorphic frame $\eta$.  Suppose that the two curvatures
$\mathsf K_E$ and $\mathsf K_F$ are equal on the open subset
$\Omega_0$. It then follows that $u=\log (\Gamma/\Lambda)$ is
harmonic on this open subset.  Thus there exists a harmonic
conjugate $v$ of $u$ on $\Omega_0$, which we assume without loss
of generality to be simply connected. For $w\in\Omega_0$, define
$\tilde{\eta}_w = e^{(u(w)+iv(w))/2} \eta_w$. Then clearly,
$\tilde{\eta}_w$ is a new holomorphic frame for $F$. Consequently,
we have the metric $\Lambda(w) = \langle{\tilde{\eta}_w,
\tilde{\eta}_w \rangle}$ for $F$ and  we see that
\begin{align*}
\Lambda(w) &= \langle{\tilde{\eta}_w, \tilde{\eta}_w\rangle}\\
&= \langle{e^{(u(w)+iv(w))/2} {\eta}_w, e^{(u(w)+iv(w))/2}{\eta}_w \rangle}\\
&= e^{u(w)}\langle{{\eta}_w, {\eta}_w \rangle}\\
&= \Gamma(w).
\end{align*}
This calculation shows that the map $U:{\eta}_w \mapsto \gamma_w$
defines an isometric holomorphic bundle map between $E$ and $F$.
The map, as shown in (cf. \cite[Theorem 1]{douglas}),
\begin{equation}\label{CD unitary}
U\Big ( \sum_{|I| \leq n} \alpha_I (\bar{\partial}^I
\eta)({w_0})\Big ) = \sum_{|I| \leq n} \alpha_I (\bar{\partial}^I
\gamma)(w_0), \,\, \alpha_I \in \mathbb C,
\end{equation}
where $w_0$ is a fixed point in $\Omega$ and $I$ is a multi-index
of length $n,$ is well-defined, extends to a unitary operator on
the Hilbert space spanned by the vectors $(\bar{\partial}^I
\eta)({w_0})$ and intertwines  the two $m$-tuples of operators in
$\mathcal P_1(\Omega)$ corresponding to the vector bundles $E$ and
$F$.

It is natural to ask what other properties of $T$ are directly
reflected in the curvature $\mathsf K.$ One such property that we
explore here is the contractivity and complete contractivity of
the homomorphism induced by the $m$-tuple $T$ via the map
$\rho_{T}: f\to f(T)$, $f\in \mathcal O(\Omega),$ where $\mathcal
O(\Omega)$ is the set of all holomorphic function in the
neighborhood of $\overline{\Omega}.$

It will be useful for us to work with the matrix of the
co-efficient  of the $(1,1)$ - form defining the curvature
$\mathsf K$, namely,
$$\mathcal K(w):= - \left(\!\!\left (\frac{\partial^2}{\partial
w_i\partial\bar{w}_j
}\log\|\gamma(w)\|^2\right)\!\!\right)_{i,j=1}^{m}$$

We recall the curvature inequality from Misra and Sastry {cf.
\cite[Theorem 5.2]{GM}} and produce a large family of examples to
show that the ``curvature inequality'' does not imply
contractivity of the homomorphism $\rho_{T}$.

\section{Localization of Cowen-Douglas operators}
For $T$ in $\mathcal P_1(\Omega)$, we define $\mathcal N(w)$ to be
the subspace $\bigcap_{j, k=1}^{m}\ker\big ( (T_j-w_j)(T_k-w_k)
\big )$ of $\mathcal H$. The localization $N(w)$ of $T$ at $w$ is
the m-tuple $N(w)=(N_1(w), N_2(w), \ldots, N_m(w)),$ where
$N_k(w)= T_{k}-w_{k}|_{\mathcal N(w)}.$ The subspace $\mathcal
N(w)$ is easily seen to be spanned by the vectors
$$\{\gamma(w), \bar{\partial}_1\gamma(w), \ldots,\bar{\partial}_m\gamma(w)\}.$$
The localization $N(w)$ of $T_k$ at $w$ then has the matrix
representation (recall $(T_i-w_i)\gamma(w) =0$ and $(T_i-w_i)
(\partial_j\gamma)(w) = \delta_{ij} \gamma(w)$ for $1 \leq i,j
\leq m$) $N_k(w) = \left ( \begin{smallmatrix}
0 & e_k \\
 0 & \boldsymbol 0\\
\end{smallmatrix}\right ),$ $k=1,\ldots , m.$ Here $\{e_k\}_{k=1}^m$ is the standard basis of $\mathbb C^m$.
Representing $N_k(w)$ with respect to an orthonormal basis in
$\mathcal N(w)$, it is possible to read off the curvature of $T$
at $w$ using the relationship:
\begin{equation}\label{curvform}
-\big (\mathcal K(w)^{\rm t}\big )^{-1}=\big ( \!\! \big ( {\rm
tr}\big ( N_{k}(w)\overline{N_{j}(w)}^{\rm t}\big )\, \big
)\!\!\big )_{kj=1}^m = A(w)^{\rm t}\overline{A(w)},\end{equation}
where the $k^{\rm th}$-column of $A(w)$ is the vector $\boldsymbol
\alpha_k$ (depending on $w$) which appears in the matrix
representation of $N_k(w)$ with respect to any choice of an
orthonormal basis in $\mathcal N(w)$.

This formula is established for a pair of operators in $\mathcal
P_1(\Omega)$ (cf. \cite[Theorem 7]{douglas}). However, it is easy
to verify it for an $m$-tuple $T$ in $\mathcal P_1(\Omega)$.

Fix $w_0$ in $\Omega$. We may assume without loss of generality
that $\|\gamma(w_0)\|=1$. The function $\langle \gamma(w),
\gamma(w_0)\rangle$ is invertible in some neighborhood of $w_0$.
Then setting $\hat{\gamma}(w):=  \langle \gamma(w),
\gamma(w_0)\rangle^{-1} \gamma(w)$, we see that
$$\langle \partial_k \hat{\gamma}(w_0), \gamma(w_0) \rangle = 0, \,\, k=1,2,\ldots, m.$$
Thus $\hat{\gamma}$ is another holomorphic section of $E$.  The
norms of the two sections $\gamma$ and $\hat{\gamma}$ differ by
the absolute square of a holomorphic function, that is,
$\tfrac{\|\hat{\gamma}(w)\|}{\|\gamma(w)\|} = |\langle \gamma(w),
\gamma(w_0)\rangle|$. Hence the curvature is independent of the
choice of the holomorphic section.
 Therefore, without loss of generality, we will prove the claim assuming:
for a fixed but arbitrary $w_0$ in $\Omega$,
\begin{enumerate}
\item[\textit{(i)}] $\|\gamma(w_0)\|=1$, \item[\textit{(ii)}]
$\gamma(w_0)$ is orthogonal to $(\partial_k\gamma)(w_0)$,
$k=1,2,\ldots , m$.
\end{enumerate}

Let $G$ be the Grammian corresponding to the $m+1$ dimensional
space spanned by the vectors $$\{\gamma(w_0),
(\partial_1\gamma)(w_0),  \ldots , (\partial_m \gamma)(w_0)\}.$$
This is just the space $\mathcal N(w_0)$. Let $v, w$ be any two
vectors in $\mathcal N(w_0)$. Find $\boldsymbol c=(c_0,\ldots,
c_m), \boldsymbol d=(d_0, \ldots, d_m)$ in $\mathbb C^{m+1}$ such
that $v=\sum_{i=0}^{m}c_i {\partial}_i\gamma(w_0)$ and
$w=\sum_{j=0}^{m}d_j{\partial}_j \gamma(w_0).$  Here
$(\partial_0\gamma)(w_0) = \gamma(w_0)$. We have
\begin{align*}\langle v, w\rangle_{G}&=\langle\sum_{i=0}^{m}c_i{\partial}_i\gamma(w_0),
\sum_{j=0}^{m}d_j {\partial}_j \gamma(w_0)\rangle\\
&= \langle G^{\rm t}(w_0)\boldsymbol c, \boldsymbol d\rangle_{\mathbb C^{m+1}}\\
&= \langle (G^{\rm t})^{\frac{1}{2}}(w_0)\boldsymbol c, (G^{\rm
t})^{\frac{1}{2}}(w_0)\boldsymbol d\rangle_{\mathbb C^{m+1}}.
\end{align*}
Let $\{e_i\}_{i=0}^{m}$ be the standard orthonormal basis for
$\mathbb C^{m+1}$. Also, let $(G^{\rm
t})^{-\frac{1}{2}}(w_0)e_i:=\boldsymbol \alpha_i(w_0)$, where
$\boldsymbol \alpha_i(j)(w_0) = \alpha_{j i}(w_0)$, $i=0,1,\ldots,
m$. We see that the vectors  $\varepsilon_i:=\sum_{j=0}^m
\alpha_{ji} (\partial_j \gamma)(w_0)$, $i=0,1, \ldots ,m,$ form an
orthonormal basis in $\mathcal N(w_0)$:  \begin{align*}\langle
\varepsilon_i, \varepsilon_l\rangle &=\big\langle
\sum_{j=0}^{m}\alpha_{i j}{\partial}_j \gamma(w_0),
\sum_{p=0}^{m}\alpha_{lp}{\partial}_p \gamma(w_0)\big\rangle\\&=
\langle(G^{\rm t})^{-\frac{1}{2}}\boldsymbol \alpha_i, (G^{\rm
t})^{-\frac{1}{2}}(w_0)\boldsymbol
\alpha_l\rangle_{G}(w_0)\\&=\delta_{il},\end{align*} where
$\delta_{il}$ is the Kornecker delta. Since $N_k\big
(\,(\partial_j\gamma)(w_0) \,\big )= \gamma(w_0)$ for $j=k$ and
$0$ otherwise, we have $N_k(\varepsilon_i) = \Big (
\begin{smallmatrix} 0 & \boldsymbol \alpha_k^{\rm t}\\ 0 & 0
\end{smallmatrix} \Big )$. Hence
\begin{align*}
{\rm tr}\big ( N_i(w_0) N_j^*(w_0) \big ) &=  \boldsymbol{\alpha_i}(w_0)^{\rm t} \overline{\boldsymbol \alpha_j}(w_0)\\
&= \big ( (G^{\rm t})^{-\frac{1}{2}}(w_0)e_i\big )^{\rm t}\overline{\big ( (G^{\rm t})^{-\frac{1}{2}}(w_0)e_j\big )}\\
&= \langle {G}^{-\frac{1}{2}}(w_0)e_i , {G}^{-\frac{1}{2}}e_j(w_0)
\rangle = {(G^{\rm t})}^{-1}(w_0)_{ij}.
\end{align*}
Since the curvature, computed with respect to the holomorphic
section $\gamma$ satisfying the conditions \textit{(i)} and
\textit{(ii)}, is of the form
\begin{align*}
\mathcal K(w_0)_{ij} &= \frac{\partial^2}{\partial w_i \bar{\partial} w_j}\log \|\gamma(w)\|^2_{|w=w_0}\\
&= \Big ( \frac{\|\gamma(w)\|^2 \big ( \frac{\partial^2
\gamma}{\partial w_i \partial\bar{w}_j} \big )(w) - \big
(\frac{\partial \gamma}{\partial w_i}\big )(w)  \big (
\frac{\partial \gamma }{\partial \bar{w}_j}\big ) (w)}
{\|\gamma(w)\|^4}\Big )_{|w=w_0}\\
&= \big ( \frac{\partial^2 \gamma}{\partial w_i \partial
\bar{w}_j} \big )(w_0)=G(w_0)_{ij}, \end{align*} we have verified
the claim \eqref{curvform}.

%\begin{align*}
%\mathcal K(w_0)_{ij} &= \frac{\partial^2}{\partial w_i \bar{\partial} w_j}\log \|\gamma(w)\|^2_{|w=w_0}\\
%&= \Big ( \frac{\|\gamma(w)\|^2 \big ( \frac{\partial^2 \gamma}{\partial w_i \partial\bar{w}_j} \big )(w) - \big (\frac{\partial \gamma}{\partial w_i}\big )(w)  \big ( \frac{\partial \gamma }{\partial \bar{w}_j}\big ) (w)} {\|\gamma(w)\|^4}\Big )_{|w=w_0}\\
%&= \big ( \frac{\partial^2 \gamma}{\partial w_i \partial \bar{w}_j} \big )(w_0), \end{align*}
The following theorem was proved for $m=2$ in ({cf. \cite[Theorem
7]{douglas}}). However, for any natural number $m$, the proof is
evident from the preceding discussion.
\begin{theorem}
Two $m$-tuples of operators $T$ and $\tilde{T}$ in $\mathcal
P_1(\Omega)$ are unitarily equivalent if and only if $N_k(w)$ and
$\tilde{N}_k(w)$ are simultaneously unitarily equivalent for $w$
in some open subset of $\Omega$.
\end{theorem}
\begin{proof}
Let us fix an arbitrary point $w$ in $\Omega$. In what follows,
the dependence on this $w$ is implicit. Suppose that there exists
a unitary operator $U:\mathcal N\rightarrow \widetilde{\mathcal
N}$ such that $UN_i=\tilde{N_i}U$, $i= 1,\ldots,m.$
%If we can show that
%$${\rm tr}\big ( N_i N_j^* \big )={\rm tr}\big ( \tilde{N_i}
%\tilde{N_j}^* \big ), 1 \leq i, j \leq m$$ then we are done.
For $1 \leq i, j \leq m,$ we have
\begin{align*}{\rm tr}\big ( \tilde{N_i}
\tilde{N_j}^* \big )&={\rm tr}\big (\big(U N_iU^*\big)\big(U
N_jU^*\big)^* \big )\\&={\rm tr}\big ( UN_i N_j^* U^*\big
)\\&={\rm tr}\big ( N_i N_j^* U^*U\big )\\&={\rm tr}\big ( N_i
N_j^* \big ).
\end{align*}
Thus the curvature of the operators $T$ and $\tilde{T}$ coincide
making  them unitarily equivalent proving the Theorem in one
direction. In the other direction, we simply have to observe that
if the operators $T$ and $\tilde{T}$ are unitarily equivalent then
the unitary $U$ given in \eqref{CD unitary} evidently maps
$\mathcal N$ to
 $\tilde{\mathcal N}$. Thus the restriction of $U$ to the subspace $\mathcal N$ intertwines $N_k$
 and $\tilde{N}_k$ simultaneously for $k=1,\cdots ,m$.
\end{proof}

As is well-known (cf. \cite{curto} and \cite{cowen}), the m-tuple
$T$ in $\mathcal P_1(\Omega)$ can be represented as the adjoint of
the $m$-tuple of multiplications $M$ by the co-ordinate functions
on a Hilbert space $\mathcal H$ of holomorphic functions defined
on $\Omega^* = \{\bar{w}\in \mathbb C^m: w\in \Omega\}$ possessing
a reproducing kernel $K:\Omega^*\times \Omega^* \to \mathbb C$
which is holomorphic in the first variable and anti-holomorphic in
the second.

In this representation, if we set $\gamma(w) = K(\cdot, \bar{w})$,
$w\in \Omega$, then we obtain a natural non-vanishing
``holomorphic" map into the Hilbert space $\mathcal H$ defined on
$\Omega$.

The localization $N(w)$ obtained from the commuting tuple of
operators $T$
 defines a a homomorphism $\rho_{T}$ on the algebra $\mathcal O(\Omega)$ of functions,
 holomorphic in some neighborhood of the closed set $\bar{\Omega},$ by the rule
\begin{equation} \label{homN}
\rho_{T}(f)=\left(\begin{smallmatrix}f(w) & \nabla f(w) A(w)^{\rm t}\\
0 & f(w)I_m\end{smallmatrix}\right),\,\, f\in \mathcal O(\Omega).
\end{equation}
We recall from ({cf. \cite[Theorem5.2]{GM}}) that the
contractivity of the homomorphism
%$\rho_{N}:f\rightarrow f(N(w))$, $f \in \mathcal A(\Omega)$
implies the curvature inequality $\|\big(\mathcal K(w)^{\rm
t}\big)^{-1}\| \leq 1$. Here $\mathcal K(w)$ is thought of as a
linear transformation from the normed linear space $(\mathbb C^m,
\mathcal C_{\Omega,w})^*$ to the normed linear space $(\mathbb
C^m, \mathcal C_{\Omega,w})$. The operator norm is computed
accordingly with respect to these norms.
\subsection{Infinite divisibility}
Let $K$ be a positive definite kernel defined on the domain
$\Omega$ and let $\lambda > 0$ be arbitrary. Since $K^\lambda$ is
a real analytic function defined on $\Omega$, it admits a power
series representation of the form
$$
K^\lambda(w,w) = \sum_{I\,J} a_{I\,J}(\lambda) (w-w_0)^I
\overline{(w-w_0)}^J
$$
in a small neighborhood of a fixed but arbitrary $w_0\in \Omega$.
The polarization $K^\lambda(z,w)$ is the function represented by
the power series
$$
K^\lambda(z,w) = \sum_{I\,J} a_{I\,J}(\lambda) (z-w_0)^I
\overline{(w-w_0)}^J,\,\,w_0\in \Omega.
$$
It follows that the polarization $K^\lambda(z,w)$ of the function
$K(w,w)^\lambda$
 defines a Hermitian kernel, that is, $K^\lambda(z,w) = \overline{K(w,z)^\lambda}$. Schur's Lemma (cf. \cite{bhatia})
  ensures the positive definiteness of $K^\lambda$ whenever $\lambda$ is a natural number. However,
  it is not necessary that $K^\lambda$ must be positive definite for all real $\lambda > 0$.
  Indeed a positive definite kernel $K$ with the property that $K^\lambda$ is positive definite for all
  $\lambda >0$ is called infinitely divisible and plays an important role in studying curvature inequalities
  (cf. \cite[Theorem 3.3]{shibu}).

Although, $K^\lambda$ need not be positive definite for all
$\lambda >0$, in general, a related question raised here is
relevant to the study of localization of the Cowen-Douglas
operators.

Let $w_0$ in $\Omega$ be fixed but arbitrary. Also, fix a $\lambda
>0$. Define the mutual inner product  of the vectors
$$
\{(\bar{\boldsymbol{\partial}}^I K^\lambda)(\cdot, w_0): I= (i_1,
\ldots, i_m)\},
$$
by the formula
$$
\langle (\bar{\boldsymbol{\partial}}^J K^\lambda)(\cdot, w_0) ,
(\bar{\boldsymbol{\partial}}^I K^\lambda)(\cdot, w_0) \rangle =
\big ( \boldsymbol{\partial}^I\bar{\boldsymbol{\partial}}^J
K^\lambda \big ) (w_0, w_0).
$$
Now, if $K^\lambda$ were positive definite, for the $\lambda$ we
have picked, then this formula would extend to an inner product on
the linear span of these vectors. The completion of this inner
product space is then a Hilbert space, which we denote by
$\mathcal H^{(\lambda)}$. The reproducing kernel for the Hilbert
space $\mathcal H^{(\lambda)}$ is easily verified to be the
original kernel $K^\lambda$. The Hilbert space $\mathcal
H^{(\lambda)}$ is independent of the choice of $w_0$.

Now, even if $K^\lambda$ is not necessarily positive definite, we
may ask whether this formula defines an inner product on the
$(m+1)$ dimensional space $\mathcal N^{(\lambda)}(w)$ spanned by
the vectors
$$\{K^\lambda(\cdot, w), (\bar{\partial}_1 K^\lambda)(\cdot, w), \ldots , (\bar{\partial}_m K^\lambda)(\cdot, w)\}.$$
An affirmative answer to this question is equivalent to the
positive definiteness of the matrix
$$
\big ( \!\! \big (\, (\partial_i\bar{\partial}_j K^\lambda) (w,
w)\big ) \!\! \big )_{i\,j = 0}^m.
$$

Let $\bar{\boldsymbol{\partial}}_m^{\rm t} = \begin{pmatrix}{1,
\partial_1, \ldots, \partial_m}
\end{pmatrix}$ and $\boldsymbol{\partial}_m$ be its conjugate transpose.
Now,
$$\big ( \boldsymbol{\partial}_m \bar{\boldsymbol{\partial}}_m^{\rm t} K^\lambda)(w,w):=
\big ( \!\! \big (\, (\partial_j\bar{\partial}_i K^\lambda)
(w,w)\big ) \!\! \big )_{i\,j = 0}^m,\,\, w\in \Omega\subseteq
\mathbb C^m.$$

\begin{lemma}
For a fixed but arbitrary $w$ in $\Omega$,  the $(m+1) \times
(m+1)$ matrix $\big ( \boldsymbol{\partial}_m
\bar{\boldsymbol{\partial}}_m^{\rm t} K^\lambda)(w,w)$ is positive
definite.
\end{lemma}

\begin{proof}
The proof is by induction on $m$.  For $m=1$ and any positive
$\lambda$, a direct verification, which follows, shows that $$\big
(\boldsymbol{\partial}_1 \bar{\boldsymbol{\partial}}_1^{\rm t}
K^\lambda\big )(w,w):= \left(  \begin{smallmatrix}
K^{\lambda}(w, w) & \partial_{1}K^{\lambda}(w, w)  \\
\bar{\partial}_{1}K^{\lambda}(w, w) &
\partial_{1}\bar{\partial}_{1} K^{\lambda}(w, w)
\end{smallmatrix}\right)$$ is positive.

Since $K^\lambda (w,w) > 0$ for any $\lambda >0$, the verification
that $\big (\boldsymbol{\partial}_1
\bar{\boldsymbol{\partial}}_1^{\rm t} K^\lambda\big )(w,w)$ is
positive definite amounts to showing that $\det \big
(\boldsymbol{\partial}_1 \bar{\boldsymbol{\partial}}_1^{\rm t}
K^\lambda\big )(w,w) > 0$.  An easy computation gives
\begin{eqnarray*}
\det \big (\boldsymbol{\partial}_1
\bar{\boldsymbol{\partial}}_1^{\rm t} K^\lambda\big )(w,w)
 &=& \lambda K^{2 \lambda -2 }(w,w) \big \{ K(w,w) (\bar{\partial_1} \partial_1 K) (w,w) -
 |\partial_1 K(w,w)|^2 \big \}\\
&=& \lambda K^{2 \lambda  }(w,w)\frac{  \|K(\cdot, w)\|^2
\|(\bar{\partial_1}K) (\cdot,w)\|^2 - |\langle K(\cdot, w),
(\bar{\partial_1} K)(\cdot,w)\,\rangle|^2 } {\|K(\cdot, w)\|^{4}},
\end{eqnarray*}
which is clearly positive since $K(\cdot, w)$ and
$(\bar{\partial}_{1} K)(\cdot, w)$ are linearly independent.

Now assume that $\big ( \boldsymbol{\partial}_{m-1}
\bar{\boldsymbol{\partial}}_{m-1}^{\rm t} K^\lambda\big)(w,w)$ is
positive definite. We note that
$$
\big ( \boldsymbol{\partial}_m \bar{\boldsymbol{\partial}}_m^{\rm
t} K^\lambda\big)(w,w) = \left ( \begin{smallmatrix}\big (
\boldsymbol{\partial}_{m-1} \bar{\boldsymbol{\partial}}_{m-1}^{\rm
t} K^\lambda \big)(w,w) & \big (
\partial_{m} \bar{\boldsymbol{\partial}}_{m-1}^{\rm t}
K^\lambda\big)(w,w)
\\ \big (
\boldsymbol{\partial}_{m-1} \bar{\partial}_{m} K^\lambda\big)(w,w)
& ({\partial}_m \bar{\partial}_m
K^\lambda)(w,w)\end{smallmatrix}\right ).
$$
Since $\big ( \boldsymbol{\partial}_{m-1}
\bar{\boldsymbol{\partial}}_{m-1}^{\rm t} K^\lambda\big)(w,w)$ is
positive definite by the induction hypothesis and for
 $\lambda > 0$, we have $$({\partial}_m \bar{\partial}_m
K^\lambda)(w,w) = \lambda K(w,w)^{\lambda-2}\big \{ K(w,w)
(\partial_{m} \bar{\partial}_{m} K)w,w)
  +  (\lambda-1) |(\bar{\partial}_{m}K)(w,w)|^2 \big \} > 0,$$
it follows that $\big ( \boldsymbol{\partial}_m
\bar{\boldsymbol{\partial}}_m^{\rm t} K^\lambda\big)(w,w)$ is
positive definite if and only if $\det\big(\big (
\boldsymbol{\partial}_m \bar{\boldsymbol{\partial}}_m^{\rm t}
K^\lambda\big)(w,w)\big) > 0$ (cf. \cite{shibu}). To verify this
claim, we note
$$\big ( \boldsymbol{\partial}_m \bar{\boldsymbol{\partial}}_m^{\rm
t} K^\lambda\big)(w,w) = \left ( \begin{smallmatrix}  K^\lambda (w,w) & B(w,w) \\
B(w,w)^* & D(w,w) \end{smallmatrix}\right ),$$ where $D=\big (
\!\! \big (\, (\partial_j\bar{\partial}_i K^\lambda) (w, w)\big )
\!\! \big )_{i\,j = 1}^{m}$ and $B=\big(\partial_{1}
K^{\lambda}(w, w), \ldots,
\partial_{m} K^{\lambda}(w, w)\big).$ Recall that (cf. \cite{hal})
$$\det \big ( \boldsymbol{\partial}_m \bar{\boldsymbol{\partial}}_m^{\rm
t} K^\lambda\big)(w,w) = \det \Big (D(w,w) -\frac{B^*(w,w)B(w,w)}
{K^{\lambda}(w, w)} \Big )\det K^{\lambda}(w, w) .$$ Now,
following (cf. \cite[proposition 2.1(second proof)]{shibu}), we
see that
$$D(w,w) -\frac{B^*(w,w)B(w,w)}{K^{\lambda}(w, w)}
 = \lambda K^{2 \lambda - 2} (w,w) \big ( \!\! \big (\, K^{2}(w,w)({\partial_j\,\bar{\partial}_i}\log K) (w, w)\big )
 \!\! \big )_{i\,j = 1}^m,$$  which was shown to be a Grammian.
 Thus $D(w,w) -\frac{B^*(w,w)B(w,w)}{K^{\lambda}(w, w)}$ is
 a positive definite matrix and hence its determinant is positive.
\end{proof}
\section{Explicit formulae} For any bounded open connected subset
$\Omega$ of $\mathbb C^m$, let $\bSB_\Omega$ denote the Bergman
kernel of $\Omega$. This is the reproducing kernel of the Bergman
space $\mathbb A^2(\Omega)$ consisting of square integrable
holomorphic functions on $\Omega$ with respect to the volume
measure. Consequently, it has a representation of the form
\begin{equation} \label{Bexpanorth}\boldsymbol B_{\Omega}(z,w)=\sum_{k}\varphi_{k}(z)\overline{\varphi_{k}(w)},
\end{equation}
where $\{\varphi_k\}_{k=0}^\infty$ is any orthonormal basis of
$\mathbb A^2(\Omega)$. This series is uniformly convergent on
compact subsets of $\Omega \times \Omega.$

We now exclusively study the case of the Bergman kernel on the
unit ball $\mathcal D$ (with respect to the usual operator norm)
in the linear space of all $r \times s$ matrices $\mathcal
M_{rs}(\mathbb C)$. The unit ball $\mathcal D$ may be also
described as
$$\mathcal D=\{Z \in\mathcal M_{rs}(\mathbb C):
I-ZZ^* \geq 0\}.$$ The Bergman Kernel for the domain $\mathcal D$
is $\boldsymbol B_{\mathcal D}(Z,Z)=\det (I-ZZ^*)^{-p},$ where
$p=r+s.$ In what follows we give a simple proof of this.

As an immediate consequence of the change of variable formula for
integration, we have the transformation rule for the Bergman
kernel. We provide the straightforward proof.

\begin{lemma}\label{lemm:main2}Let $\Omega$ and $\tilde{\Omega}$ be two domains in
$\mathbb C^m$ and $\varphi:\Omega\rightarrow \tilde{\Omega}$ be a
bi-holomorphic map. Then $$\bSB_{\Omega}(z,w)=J_{\mathbb
C}\varphi(z)\overline{J_{\mathbb
C}\varphi(w)}\bSB_{\widetilde{\Omega}}(\varphi(z), \varphi(w))$$
for all $z, w \in \Omega,$ where $J_{\mathbb C}\varphi(w)$ is the
determinant of the derivative $D\varphi(w)$.
\end{lemma}
\begin{proof}
Suppose  $\{\tilde{\phi}_n\}$ be an orthonormal basis for $\mathbb
A^{2}(\tilde{\Omega}).$ By change of variable formula, it follows
easily  that $\phi_n=\{J_{\mathbb C}\varphi(w)\tilde{\phi}_n\circ
\varphi\}, $ form an orthonormal basis for  $\mathbb
A^{2}(\Omega).$ Hence,
\begin{align*}
\boldsymbol
B_{\Omega}(z,w)=\sum_{n=0}^{\infty}\phi_n(z)\overline{\phi_n(w)}&=\sum_{n=0}^{\infty}J_{\mathbb
C}\varphi(w)(\tilde{\phi}_n\circ \varphi)(z)\overline{J_{\mathbb
C}\varphi(w)(\tilde{\phi}_n\circ \varphi)(w)}\\&=J_{\mathbb
C}\varphi(w)\overline{J_{\mathbb
C}\varphi(w)}\sum_{n=0}^{\infty}\tilde{\phi}_n(\varphi(z))\overline{\tilde{\phi}_n(\varphi(w))}\\&=
J_{\mathbb C}\varphi(w)\overline{J_{\mathbb
C}\varphi(w)}\boldsymbol B_{\widetilde{\Omega}}(\varphi(z),
\varphi(w))
\end{align*} completing our proof.
\end{proof}
If $\Omega$ is a domain in $\mathbb C^m$ and the bi-holomorphic
automorphism group, ${\rm Aut}(\Omega)$ is transitive, then we can
determine the Bergman kernel as well as its curvature from its
value at $0$! A domain with this property is called homogeneous.
For instance, the unit ball $\mathcal D$ in the linear space of
$r\times s$ matrices are homogeneous. If $\Omega$ is homogeneous,
then for any $w\in \Omega$,
 there exists an bi-holomorphic automorphism $\varphi_w$ with the property $\varphi_w(w) = 0$.
 The following Corollary is an immediate consequence of Lemma \ref{lemm:main2}.

\begin{corollary}\label{corolo1}For any homogeneous domain $\Omega$ in $\mathbb C^m$, we have
$$\bSB_{\Omega}(w,w)=J_{\mathbb
C}\varphi_{w}(w)\overline{J_{\mathbb
C}\varphi_{w}(w)}\bSB_{\Omega}(0, 0), \,\,  w \in \Omega.$$
\end{corollary}

We recall from (cf. \cite[Theorem 2]{harris} ) that for $Z, W$ in
the matrix ball $\mathcal D$ (of size $r\times s$) and $\mathbf
u\in \mathbb C^{r\times s}$, we have $$D\varphi_W(Z) \cdot \mathbf
u = (I-W W^*)^{\frac{1}{2}}(I-ZW^*)^{-1}\mathbf u (I-
W^*Z)^{-1}(I-W^*W)^{\frac{1}{2}}.$$ In particular, $D\varphi_W(W)
\cdot u =  (I-WW^*)^{-\frac{1}{2}} \mathbf u (I-W^*
W)^{-\frac{1}{2}}.$ Thus $D\varphi_W(W) = (I-WW^*)^{-\frac{1}{2}}
\otimes (I-W^* W)^{-\frac{1}{2}}.$ We therefore (cf.
\cite[exercise 8]{halmos} \cite{hua}) have
\begin{align*} \det D\varphi_W(W) &=
\big (\det (I-WW^*)^{-\frac{1}{2}} \big )^s \big (\det (I-W^* W)^{-\frac{1}{2}} \big )^r\\
&=\big (\det (I-WW^*)^{-\frac{1}{2}} \big )^{r+s}.
\end{align*}
It then follows that
\begin{align*}
J_{\mathbb C}\varphi_{W}(W)\overline{J_{\mathbb C}\varphi_{W}(W)}
% &= \det (I - w w^*)^{-r} \det (I - w^* w)^{-s}\nonumber \\&
&= \det (I - WW^*)^{-(r+s)},\,\, W\in \mathcal D.
\end{align*}
With a suitable normalization of the volume measure, we may assume
that $\bSB_\Omega(0,0) = 1$. With this normalization, we have
\begin{equation} \label{BergKern}
\bSB_{\mathcal D}(W,W) = \det (I - WW^*)^{-(r+s)},\,\, W\in
\mathcal D.
\end{equation}

The Bergman kernel $\bSB_{\Omega},$ where $\Omega=\{(z_1, z_2):
|z_2|
 \leq (1-|z_1|^2)\}
\subset \mathbb C^2$ is known (cf. \cite[Example 6.1.6]{pflug}):
\begin{equation}\label{BergNil}
\bSB_\Omega(z, w)=\frac{3(1-z_1\bar w_1)^2+z_2\bar
w_2}{\{(1-z_1\bar w_1)^2-z_2\bar w_2\}^3},\,\,z,w\in \Omega.
\end{equation}
The domain $\Omega$ is not homogeneous. However, it is a Reinhadt
domain. Consequently, an orthonormal basis consisting of monomials
exists in the Bergman space of this domain. We give a very similar
example below to show that computing the Bergman kernel in a
closed form may not be easy even for very simple Reinhadt domains.
We take $\Omega$ to be the domain
 $$\{(z_1, z_2, z_3):
|z_2|^2
 \leq (1-|z_1|^2)(1-|z_3|^2), 1-|z_3|^2\geq 0\}
\subset \mathbb C^3.$$
\begin{lemma}\label{bergman}The Bergman kernel $\bSB_{\Omega}(z,w)$ for the domain $\Omega$ is given by the formula
$$\sum_{p, m, n=0}^{\infty}\frac{m+1}{4\beta(n+1, m+2)\beta(p+1,
m+2)}(z_1\bar{w}_1)^n(z_2\bar{w}_2)^m(z_3\bar{w}_3)^p,$$ where
$\beta(m,n)$ is the Beta function.
\end{lemma}
\begin{proof}
Let $\{(z_1)^n(z_2)^m(z_3)^p\}_{n, m, p=1}^{\infty}$ be the
orthonormal basis for the Bergman space $\mathbb A^2(\Omega).$
Now, {\small
\begin{align}\label{eqann:main1}\|(z_1)^n(z_2)^m(z_3)^p\|^2\nonumber&
=\int_{0}^{2\pi}d\theta_1 d\theta_2
d\theta_3\int_{0}^{1}r_{1}^{(2n+1)}dr_1
\int_{0}^{1}r_{3}^{(2p+1)}dr_3\int_{0}^{\sqrt{(1-r_{1}^{2})(1-r_{2}^{2})}}r_{2}^{(2m+1)}dr_2\\\nonumber&=
8\pi^3\int_{0}^{1}r_{1}^{(2n+1)}dr_1
\int_{0}^{1}r_{3}^{(2p+1)}dr_3\frac{(1-r_{1}^{2})^{(m+1)}(1-r_{2}^{2})^{(m+1)}}{2m+2}\\&=\frac{\pi^3}{m+1}
\int_{0}^{1}s_1^{n}(1-s_1)^{(m+1)}ds_1\int_{0}^{1}s_2^{p}(1-s_2)^{(m+1)}ds_2
\end{align}}where $r_1^{2}=s_1$ and $r_2^{2}=s_2.$
Since $\beta(n, m)=\int_{0}^{1}r^{(n-1)}(1-r)^{(m-1)}dr,$
therefore  equation (\ref{eqann:main1})is equal to{
\begin{align}\|(z_1)^n(z_2)^m(z_3)^p\|^2\nonumber&=\frac{\pi^3}{m+1}\beta(n+1, m+2)\beta(p+1, m+2).
\end{align}}
From equation (\ref{eqann:main1}), it follows that
%\begin{align*}
$\|1\|^2 \pi^3\beta(1, 2)\beta(1, 2)=\frac{\pi^3}{4}$.
%\end{align*}
We normalize the volume measure in an appropriate manner to ensure
{
\begin{align}\|(z_1)^n(z_2)^m(z_3)^p\|^2\nonumber&=\frac{4}{m+1}\beta(n+1, m+2)\beta(p+1,
m+2).\end{align}} Having computed an orthonormal basis for the
Bergman space, we can complete the the computation of the Bergman
kernel using the infinite expansion \eqref{Bexpanorth}.
\end{proof}

The Proposition following the Lemma (a change of variable formula
from (cf. \cite[The chain rule 1.3.3]{rudin}) given below provides
the transformation rule for the Bergman metric (cf.
\cite[proposition 1.4.12]{krantz}).

\begin{lemma}\label{chain rule}
Suppose $\Omega$ is in $\mathbb C^m, F=(f_1, \ldots, f_n)$ maps
$\Omega$ into $\mathbb C^n, g$ maps the range of $F$ into $\mathbb
C,$ and $f_1, \ldots, f_n, g$ are of class $\mathcal C^2.$ If
$$h=g\circ F=g(f_1, \ldots, f_n)$$ then, for $1 \leq i,j\leq m$ and
$z \in \Omega,$ $$\big(\overline{D}_jD_ih
\big)(z)=\sum_{k=1}^{n}\sum_{l=1}^{n}\big(\overline{D}_lD_kh
\big)(w)\overline{D_jf_l}(z)D_if_k(z),$$ where
$\overline{D}_j\bar{f}_l=\overline{D_jf_l}(z).$
\end{lemma}
\begin{proposition}\label{lemm:main3}
Let $\Omega$ and  $\tilde{\Omega}$ be two domain in $\mathbb C^m$
and $\varphi:\Omega\rightarrow \tilde{\Omega}$ is bi-holomorphic
map. Then $$\mathcal K_{\bSB_{\Omega}}(w)= {\big (D\varphi\big
)(w)}^{\rm t}\mathcal
K_{\bSB_{\widetilde{\Omega}}}(\varphi(w))\overline{\big
(D\varphi\big )(w)}$$ for all $ w \in \Omega.$
\end{proposition}
\begin{proof}
For any holomorphic function $\varphi$ defined on $\Omega$, we
have $\frac{\partial}{\partial w_i \partial \bar{w}_j}
\log|J_{\mathbb C}\varphi(w)|^{2}=0.$ Combining this with Lemma
\ref{lemm:main2}, we get
\begin{align*}
\frac{\partial^2}{\partial w_i \partial
\bar{w}_j}\log\bSB_{\widetilde{\Omega}}(\varphi(w),
\varphi(w))&=\frac{\partial^2}{\partial w_i \partial
\bar{w}_j}\log|J_{\mathbb
C}\varphi(w)|^{-2}\bSB_{\Omega}(w,w)\\&=-\frac{\partial^2}{\partial
w_i \partial \bar{w}_j}\log|J_{\mathbb
C}\varphi(w)|^{2}+\frac{\partial^2}{\partial w_i \partial
\bar{w}_j}\bSB_{\Omega}(w,w)\\&=\frac{\partial^2}{\partial w_i
\partial \bar{w}_j}\bSB_{\Omega}(w,w).
\end{align*} Also by Lemma \ref{chain rule} with $g(z)=\log\bSB_{\widetilde{\Omega}}(z, z)$ and $F=f$
we have, $$\frac{\partial^2}{\partial w_i \partial
\bar{w}_j}\log\bSB_{\widetilde{\Omega}}(\varphi(w),
\varphi(w))=\sum_{k, l=1}^{n}\frac{\partial\varphi_{k} }{\partial
w_i}(w)\frac{\partial^2}{\partial w_k \partial
\bar{z}_l}\log\bSB_{\widetilde{\Omega}}(z, z)(\varphi(w),
\varphi(w))\frac{\partial\varphi_{l} }{\partial w_j}(w).$$ Hence
\begin{eqnarray*}
\lefteqn{ \left(\!\!\left (\frac{\partial^2}{\partial w_i \partial
\bar{w}_j}\log\bSB_{\widetilde{\Omega}}(\varphi(w),
\varphi(w)\right)\!\!\right)_{ij}}\\
&\phantom{~~~~~~~~~~~~~~~~~~~~}=&\left(\!\!\left(\frac{\partial\varphi_{k}
}{\partial w_i}(w)\right)\!\!\right)_{ik}\left(\!\!\left
(\frac{\partial^2}{\partial z_k \partial
\bar{z}_l}\log\bSB_{\widetilde{\Omega}}(z, z)(\varphi(w),
\varphi(w)\right)\!\!\right)_{kl}\left(\!\!\left(\frac{\partial\varphi_{l}
}{\partial w_j}(w)\right)\!\!\right)_{lj} \\
&\phantom{~~~~~~~~~~~~~~~~~~~~}=&{\big (D\varphi\big )(w)}^{\rm
t}\mathcal K_{\bSB_{\widetilde{\Omega}}}(\varphi(w))\overline{\big
(D\varphi\big )(w)}.
\end{eqnarray*}
Therefore we have the desired transformation rule for the Bergman
metric, namely,
$$\mathcal K_{\bSB_{\Omega}}(w)= {\big (D\varphi\big
)(w)}^{\rm t}\mathcal
K_{\bSB_{\widetilde{\Omega}}}(\varphi(w))\overline{\big
(D\varphi\big )(w)},\,\, w \in \Omega.$$\end{proof} As a
consequence of this transformation rule, a formula for the Bergman
metric at an arbitrary $w$ in $\Omega$ is obtained from its value
at $0$. The proof follows from the transitivity of the
automorphism group.

\begin{corollary}\label{transK} For a homogeneous domain $\Omega$, pick a  a  bi-holomorphic  automorphism
$\varphi_w$ of $\Omega$ with $\varphi_w(w) = 0$, $w\in \Omega$, we
have
$$\mathcal K_{\bSB_{\Omega}}(w)=
\big(D \varphi_{w}(w)\big)^{\rm t}\mathcal
K_{\bSB_{\Omega}}(0)\overline{D\varphi_{w}(w)}$$ for all $ w \in
\Omega.$
\end{corollary}
For the matrix ball $\mathcal D$, as is well-known, $\bSB_\mathcal
D^\lambda$ is not necessarily positive definite for all $\lambda >
0$. However, as we have pointed out the space $\mathcal
N^{(\lambda)}(w)$ has a natural inner product induced by
$\bSB_\mathcal D^\lambda$. Thus we explore properties of
$\bSB_\mathcal D^\lambda$ for all $\lambda >0$. In what follows,
we will repeatedly use the transformation rule for
$\bSB_\Omega^\lambda$ which is an immediate consequence of the
transformation rule for $\bSB_\Omega,$ namely,
\begin{equation}\label{transK^l}
\mathcal K_{\bSB^{\lambda}_{\Omega}}(w)=\lambda\mathcal
K_{\bSB_{\Omega}}(w)=\lambda{D\varphi_{w}(w)}^{\rm t}\mathcal
K_{\bSB_{\Omega}}(0)D\varphi_{w}(w)
\end{equation} for $w \in \Omega$ and
$\lambda >0$.

To compute the Bergman metric, we begin with a Lemma on the Taylor
expansion of the determinant. To facilitate its proof, for $Z$  in
$\mathcal M_{rs}(\mathbb C),$ we write  $Z=\left (
\begin{smallmatrix} Z_1 \\ \vdots\\ Z_r
\end{smallmatrix}\right ),$ with $Z_i=\left(z_{i1}, \ldots,
z_{is}\right),$   $i=1 , \ldots, r.$ In this notation,
$$I-ZZ^*=\left (
\begin{smallmatrix} 1-\|Z_1\|^2 & -\langle Z_1, Z_2\rangle &
\cdots
 & -\langle Z_1, Z_r\rangle  \\
 \vdots & \vdots & \vdots & \vdots\\
 -\langle Z_r, Z_1\rangle &  -\langle Z_r, Z_2\rangle & \cdots & 1-\|Z_r\|^2
\end{smallmatrix}\right ),$$ where
$\|Z_i\|^2=\sum_{j=1}^{s}|z_{ij}|^2, \langle Z_i, Z_j\rangle
=\sum_{k=1}^{s}z_{ik}\bar{z}_{jk}.$ Set $X_{ij}=\langle Z_i,
Z_j\rangle, 1\leq i, j \leq r.$

The curvature $\mathcal K_{\boldsymbol B_{\mathcal D}}(0)$ of the
Bergman kernel, which is often called the Bergman metric, is
easily seen to be $p$ times the $rs \times rs$ identity as a
consequence of the following Lemma. The value of the curvature
$\mathcal K_{\boldsymbol B_{\mathcal D}}(W)$ at an arbitrary point
$W$ is then easy to write down using the homogeneity of the unit
ball $\mathcal D$.

\begin{lemma}\label{lem:det}
The determinant $\det (I-ZZ^*)=1-\sum_{i=1}^{r}\|Z_i\|^2+ P(X),$
where $P(X)=\sum_{|{\ell}|\geq2}p_{\ell}X^{\ell}$ with
$$X^{\ell}:=X_{11}^{\ell_{11}}\ldots X_{1r}^{\ell_{1r}}\ldots
X_{r1}^{\ell_{r1}}\ldots X_{rr}^{\ell_{rr}}.$$
\end{lemma}
\begin{proof}
The proof is by induction on $r$.  For $r=1$ we have $\det
(I-ZZ^*)=1-\|Z\|^2.$ Therefore in this case , $P=0$ and we are
done. For $r=2,$ we have
$$\det (I-ZZ^*)=\det\left (
\begin{smallmatrix} 1-\|Z_1\|^2 & -\langle Z_1,
Z_2\rangle \\
 -\langle Z_2, Z_1\rangle  & 1-\|Z_2\|^2
\end{smallmatrix}\right ).$$ For $r=2,$  a
direct verification shows that the $\det (I-ZZ^*)$  is equal to
$1-\sum_{i=1}^{2}\|Z_i\|^2+ P(X),$ where
$P(X)=X_{11}X_{22}-|X_{12}|^2.$ The decomposition
$$I-ZZ^*= \left ( \begin{array}{cccc|c}
%\begin{smallmatrix}
1-\|Z_1\|^2 & -\langle Z_1, Z_2\rangle & \cdots
 & -\langle Z_1, Z_{r-1}\rangle &-\langle Z_1,Z_r \rangle \\
 \vdots & \vdots & \vdots & \vdots & \vdots \\
 -\langle Z_{r-1}, Z_1\rangle &  -\langle Z_{r-1}, Z_2\rangle & \cdots & 1-\|Z_{r-1}\|^2 &
  - \langle Z_{r-1} , Z_r \rangle \\ \hline
%\end{smallmatrix}
%\begin{smallmatrix}
%-\langle Z_1,Z_r \rangle \vdots
 %- \langle Z_{r-1} , Z_r \rangle
%\end{smallmatrix}
%\\ \hline
%\begin{smallmatrix} \phantom{~~}

-\langle Z_r,Z_1 \rangle &  -\langle Z_r,Z_2 \rangle & \cdots &
 - \langle Z_{r} , Z_{r-1} \rangle & 1- \|Z_r\|^2\\
\end{array} \right )$$
is crucial to our induction argument. Let $A_{ij}$, $i,j=1,2$,
denote the blocks in this decomposition.
 By induction hypothesis, we have
$$\det A_{11}=1-\sum_{i=2}^{r}\|Z_i\|^2+ Q(X),$$ where
$Q(X)=\sum_{|\ell|\geq2}q_{\ell}X^{\ell}.$  Since  $\det
(A_{22}-A_{21}A_{11}^{-1}A_{12})$ is a scalar, it follows that
\begin{eqnarray*}
\det (I-ZZ^*) &=& (A_{22}-A_{21}A_{11}^{-1}A_{12})\,\det A_{11}\\
&=& A_{22} \det A_{11} - A_{21}\big (\det A_{11}\big ) A_{11}^{-1} A_{12}\\
&=&  A_{22} \det A_{11} - A_{21}\big ( {\rm Adj} (A_{11}) \big
)A_{12},
\end{eqnarray*}
where, as usual, ${\rm Adj}(A_{11})$ denotes the transpose of the
matrix of co-factors of $A_{11}$. Clearly, $A_{21}\big ( {\rm Adj}
(A_{11}) \big )A_{12}$ is a sum of $(r-1)^2$ terms. Each of these
is of the form $ X_{k 1} a_{j k} X_{1 j}$, where $a_{jk}$ denotes
the $(j,k)$ entry of ${\rm Adj}(A_{11})$. It follows that any one
term in the sum $A_{21}\big ( {\rm Adj} (A_{11}) \big )A_{12}$ is
some constant multiple of $X^{\ell}$ with $|\ell|\geq2.$
Furthermore, $$A_{22} \det A_{11} =
1-\sum_{i=1}^{r}\|Z_i\|^2+\|Z_r\|^2\sum_{i=1}^{r-1}\|Z_i\|^2+Q(X)(1-\|Z_r\|^2).$$
Putting these together, we see that   $$\det
(I-ZZ^*)=1-\sum_{i=1}^{r}\|Z_i\|^2+ P(X),$$ where
$P(X)=X_{rr}\sum_{i=1}^{r-1}X_{ii}+Q(X)(1-X_{rr})-A_{21} \big
({\rm Adj}(A_{11}) A_{12}$ completing the proof.
\end{proof}

 Let $\mathcal K_{\boldsymbol B_{\mathcal D}}(Z)$ be the curvature (some times also called the Bergman metric) of
 the Bergman Kernel $\boldsymbol B_{\mathcal D}(Z,Z).$
 Set $w_1=z_{11}, \ldots, w_{s}=z_{1s}, \ldots, w_{rs-s+1}=z_{r1},\ldots ,w_{rs}=z_{rs}.$
The formula for the Bergman metric given below is due to Koranyi
(cf. \cite{Adam}).
\begin{theorem}\label{curva}
$\mathcal K_{\boldsymbol B_{\mathcal D}}(0)=pI,$ where $I$ is the
$rs \times rs$ identity matrix.
\end{theorem}
\begin{proof}
%We have $$\mathcal K_{\boldsymbol B_{\mathbb
%B}}(0,0)=\left(\!\! \left (\, \Big ( \frac{\partial ^{2}}{\partial
%w_{i}{\partial}\bar{w}_{j}}\log \boldsymbol B_{\mathcal D}\Big )(0, 0)
%\right)\!\!\right )_{ij}.$$
Lemma \ref{lem:det} says that
$$\log \boldsymbol B_{\mathcal D}(Z)=-p\log\big(1-\sum_{i=1}^{r}\|Z_i\|^2+ P(X)\big).$$ It now
follows that $\big (\frac{\partial ^{2}}{\partial
w_{i}{\partial}\bar{w}_{j}}\log \boldsymbol B_{\mathcal D}\big )
(0) = 0,$  $i \neq j$. On the other hand, $\big (\frac{\partial
^{2}}{\partial w_{i}{\partial}\bar{w}_{i}}\log \boldsymbol
B_{\mathcal D}\big ) (0)=p,$  $i=1,\ldots , rs.$
%completing the verification of the formula for the Bergman metric at $0$.
\end{proof}

In consequence, for the matrix ball $\mathcal D$, which is a
homogeneous domain in $\mathbb C^{r\times s}$, we record
separately, the transformation rule:
\begin{align}\label{transinvtranspose}
\big (\mathcal K_{\bSB_{\mathcal D}}(W)^{\rm t}\big )^{-1} &=
\big(D \varphi_{W}(W)\big)^{-1} \big (\mathcal
K_{\bSB_{\Omega}}(0)^{\rm t}\big )^{-1}
 \big (\overline{D\varphi_{W}(W)}^{\rm \,\,t}\big )^{-1}\nonumber\\
&=\frac{1}{p}\big ( \overline{D\varphi_{W}(W)}^{\rm \,\,t}D
\varphi_{W}(W)\big ) ^{-1}, \,\,W\in \mathcal D,
\end{align}
where $p=r+s$.

\section{Curvature inequalities}
\subsection{The Euclidean Ball}
Let $\Omega$ be a homogeneous domain and
$\theta_{w}:\Omega\rightarrow \Omega$ be a bi-holomorphic
automorphism of $\Omega$ with $\theta_{w}(w)=0.$  The linear map
$D\theta_{w}(w): (\mathbb C^m, \mathcal C_{\Omega, w}) \to
(\mathbb C^m, \mathcal C_{\Omega,0})$ is a contraction by
definition. Since $\theta_w$ is invertible, $D\theta_w^{-1}(0):
(\mathbb C^m, \mathcal C_{\Omega,0}) \to (\mathbb C^m, \mathcal
C_{\Omega, w})$ is also a contraction. However, since
$D\theta_w^{-1}(0) = D\theta_w(w)^{-1}$, it follows that
$D\theta_w(w)$ must be an isometry.  We paraphrase the Theorem
from (cf. \cite[Theorem 5.2]{GM}) slightly.
\begin{lemma} \label{A(w)} If $\Omega$ is a homogeneous domain and $\theta_w$ is a bi-holomorphic
automorphism with $\theta_w(w) =0$, then  $\|A(w)^{\rm t}
\|_{\ell^2 \rightarrow \mathcal C_{\Omega, w}}\leq 1$ if and only
if $\|A(0)^{\rm t}\|_{\ell^2 \rightarrow \mathcal C_{\Omega,
0}}\leq 1.$
\end{lemma}
\begin{proof}
As before, let $\boldsymbol D_{w}\Omega:=\{D f(w):f\in {\rm
Hol}_w(\Omega, \mathbb D)\}.$ The map $\varphi\mapsto
\varphi\circ\theta_w(w)$ is  injective from ${\rm Hol}_0(\Omega,
\mathbb D)$ onto ${\rm Hol}_w(\Omega, \mathbb D).$ Therefore,
\begin{align*}
\boldsymbol D_{w}\Omega &=\{D(f\circ\theta_{w})(w): f \in {\rm
Hol}_0(\Omega, \mathbb D)\}\\&=\{Df(0)D\theta_{w}(w): f \in {\rm
Hol}_0(\Omega, \mathbb D)\}\\&=\{u \cdot D\theta_w(w): u \in
\boldsymbol D_{0}\Omega\}
\end{align*}
This is another way of saying that $D\theta_w(w)$ is an isometry.
\begin{align*}
\sup_{v \in \boldsymbol D_{w}\Omega}\| A(w)^{\rm t} v\|&=\sup_{u
\in \boldsymbol D_{0}\Omega}\|A(w)^{\rm t} D\theta_w(w)
u\|\\&=\sup_{u \in \boldsymbol D_{0}\Omega}\| A(0)^{\rm t} u\|,
\end{align*} where we have set $A(0)^{\rm t}:=A(w)^{\rm
t}D\theta_w(w).$ Thus we have shown
\begin{align*} \label{Dw0}
\{ A(w)^{\rm t}: \|A(w)^{\rm t} \|_{\ell^2 \rightarrow \mathcal
C_{\Omega, w}}\leq 1\}&= \{A(0)^{\rm t}D\theta_w(w)^{-1}:
\|A(0)^{\rm t} \|_{\ell^2 \rightarrow \mathcal C_{\Omega,
w}}\}\\&=\{A(0)^{\rm t}D\theta^{-1}_{w}(0): \|A(0)^{\rm t}
\|_{\ell^2 \rightarrow \mathcal C_{\Omega, w}}\}.
\end{align*}
The proof is now complete since $D\theta_w(w)$ is an isometry.
\end{proof}

%From this Lemma it follows that
%$$D\theta_w(w).A(w)^t\overline{A(w)}(\overline{D\theta_w(w)})^{\rm
%t}=A(0)^t\overline{A(0)}.$$

%Since $M_{\Omega}(A(w)^t,
%w)=M_{\Omega}(D\theta_w(w).A(w)^t, 0),$ therefore we have
%$$\|A(w)^{\rm t}\overline{A(w)}\|_{\mathcal C_{\Omega,
%w}^*\rightarrow \mathcal C_{\Omega, w}}=\|A(0)^{\rm
%t}\overline{A(0)}\|_{\mathcal C_{\Omega, 0}^*\rightarrow\mathcal
%C_{\Omega, 0}}.$$

We  note that if $\|A(w)^{\rm t}\|_{\ell^2 \rightarrow \mathcal
C_{\Omega, w}}\leq 1,$ then
\begin{align} \|\big
(\mathcal K (w)^{\rm t} \big )^{-1}\|_{\mathcal C_{\Omega,
w}^*\rightarrow \mathcal C_{\Omega, w}}\nonumber &= \|A(w)^{\rm
t}\overline{A(w)}\|_{\mathcal C_{\Omega, w}^* \rightarrow \mathcal
C_{\Omega, w}}\\\nonumber &\leq \|A(w)^{\rm t}\|_{\ell^2
\rightarrow \mathcal C_{\Omega, w}} \|\overline{A(w)}\|_{\mathcal
C_{\Omega, w}^* \to \ell^2}\\ & = \|A(w)^{\rm t}\|^2_{\ell^2
\rightarrow \mathcal C_{\Omega, w}} \leq 1,\end{align} which is
the curvature inequality of (cf. \cite[Theorem 5.2]{GM}). For a
homogeneous domain $\Omega$, using the transformation rules in
Corollary \ref{transK} and the equation \eqref{transinvtranspose},
for the curvature $\mathcal K$ of the Bergman kernel
$\bSB_\Omega$, we have
\begin{align}
\|\big (\mathcal K (w)^{\rm t} \big )^{-1}\|_{\mathcal C_{\Omega,
w}^*\rightarrow \mathcal C_{\Omega, w}} \nonumber&=\big \|  {\big
( D\theta_w(w)^{\rm t} \mathcal K(0) \overline{D\theta_w(w)} \big
)^{\rm t}}^{-1} \big \|_{\mathcal C_{\Omega, w}^*\rightarrow
\mathcal C_{\Omega, w}}\\\nonumber &=\big \| D\theta_w(w)^{-1}\big
(\mathcal K(0)^{\rm t}\big )^{-1} \overline{
D\theta_w(w)^{-1}}^{\rm t} \|_{\mathcal C_{\Omega, w}^*\rightarrow
\mathcal C_{\Omega, w}}\\\nonumber &=\big \| D\theta_w(w)^{-1}
A(0)^{\rm t} \overline{A(0)} \overline{ D\theta_w(w)^{-1}}^{\rm t}
\|_ {\mathcal C_{\Omega, w}^*\rightarrow \mathcal C_{\Omega, w}}\\
&\leq \big \| D\theta_w(w)^{-1} A(0)^{\rm t}\|_{\ell^2 \to
\rightarrow \mathcal C_{\Omega, w}}^2 =  \big \| A(0)^{\rm
t}\|_{\ell^2 \to \rightarrow \mathcal C_{\Omega, 0}}^2
\end{align}
since $ D\theta_w(w)^{-1}$ is an isometry.  For the Euclidean ball
$\mathbb B:= \mathbb B^n$, the inequality for the curvature is
more explicit. In the following, we set $\mathfrak B(w,w):= \big (
\bSB_\mathbb B(w,w)\big )^{-\frac{1}{n+1}}$. Thus polarizing
$\mathfrak B$, we have $\mathfrak B(z,w) = \big ( 1 - \langle z,
w\rangle)^{-1}$, $z,w \in \mathbb B$.
 The inequality appearing below (cf. \cite{GM}) is a point-wise inequality with respect to
 the usual ordering of Hermitian matrices.
\begin{theorem}\label{lemm:main4}Let $\theta_w$ is a
bi-holomorphic automorphism of $\mathbb B$ such that
$\theta_w(w)=0.$ If $\rho$ is contractive homomorphism of
$\mathcal O(\mathbb B)$ induced by the localization $N(w)$,  $T
\in \mathcal P_1(\mathbb B),$ then $$\mathcal K(w)\leq
-\overline{D\theta_w(w)}^tD\theta_w(w)= \mathcal K_{\mathfrak
B}(w),\,\, w\in \mathbb B$$
\end{theorem}

\begin{proof}The equation \eqref{Dw0} combined with the equality $\mathcal C_{\mathbb B,0}=\|\cdot \|_{\ell^2}$
and the contractivity of $\rho_{T}$ implies that $\|D\theta_w(w)
A(w)^t\|_{\ell_2\rightarrow\ell_2}\leq 1.$  Hence
\begin{eqnarray*}
I-D\theta_w(w)A(w)^t\overline{A(w)}\,\overline{D\theta_w(w)}^{\rm
\,\,t}\geq
0&\Leftrightarrow&(D\theta_w(w))^{-1}\big(\overline{D\theta_w(w)}^{\rm
\,\,t}\big)^{-1}-A(w)^t\overline{A(w)}\geq
0\\&\Leftrightarrow&A(w)^t\overline{A(w)}\leq(D\theta_w(w))^{-1}\big(\overline{D\theta_w(w)}^{\rm
\,\,t}\big)^{-1}\\&\Leftrightarrow& \big(- \mathcal K(w)^{\rm
t}\big)^{-1}\leq \big ( \overline{D\varphi_{w}(w)}^{\rm \,\,t}D
\varphi_{w}(w)\big ) ^{-1}.\\
%& \Leftrightarrow& \big(-\mathcal
%K(w)^{\rm t}\big)^{-1}\leq \frac{1}{n+1}\big (-\mathcal K_{
%\bSB_\mathbb B}(w) \big )^{-1}.
\end{eqnarray*}
Since $-\big(\mathcal K(w)^{\rm t}\big)^{-1}$ and $\big (
\overline{D\theta_w(w)}^{\rm t}D\theta_w(w)\big )^{-1}$ are
positive definite matrices, it follows (cf. \cite{rajendra}) that
$\mathcal K(w)\leq -\overline{D\theta_w(w)}^{\rm t}D\theta_w(w)
=\mathcal K_{\mathfrak B}(w).$
\end{proof}
This inequality generalizes the curvature inequality obtained in
(cf. \cite{misra}) for the unit disc. However, assuming that
$\mathcal K_{\mathfrak B^{-1}\!\!\,K}(w)$ is a non-negative Kernel
defined on the ball $\mathbb B$ implies $(\mathfrak
B(w))^{-1}K(w)$ is a non-negative kernel on $\mathbb B$ (cf.
\cite[Theorem 4.1]{shibu}), indeed, it must be infinitely
divisible. This stronger assumption on the curvature amounts to
the factorization of the kernel $K(z,w) = \mathfrak
B(z,w)\tilde{K}(z,w)$ for some positive definite kernel
$\tilde{K}$ on the ball $\mathbb B$ with the property: $\big
(\mathfrak B(z,w)\tilde{K}(z,w)\big )^\lambda$ is non-negative
definite for all $\lambda >0$.

%This factorization does not ensure the contractivity of the  homomorphism induced by
%the commuting tuple of multiplication operators on the Hilbert space determined by the positive definite kernel is $K$.
%So, it is not surprising that the curvature inequality alone does not imply the contractivity of this homomorphism as shown below.

For $\lambda >0$, the polarization of the function
$\bSB(w,w)^\lambda$ defines a positive definite kernel
$\bSB^\lambda(z,w)$ on the ball $\mathbb B$  (cf.
\cite[Proposition 5.5]{arazy}).
%The corresponding Hilbert space is called the ``twisted Bergman space''.
We note that $\mathcal K_{\bSB^{\lambda}}(w)\leq \mathcal
K_{\mathfrak B}(w)$ if and only if $\mathcal
K_{\bSB^{\lambda}}(0)\leq \mathcal K_{\mathfrak B}(0)=-I.$ Since
$\mathcal K_{\bSB^{\lambda}}(0)=-\lambda(n+1)I,$ it follows that
$\mathcal K_{\bSB^{\lambda}}(w)\leq \mathcal K_{\mathfrak B}(w)$
if and only if $\lambda \geq \frac{1}{n+1}.$
%where $\nu = \lambda(n+1).$
Thus whenever $\lambda \geq \frac{1}{n+1}$, we have the point-wise
curvature inequality for $\bSB^\lambda(w, w)$. However, since the
operator of multiplication by the co-ordinate functions on the
Hilbert space corresponding to the kernel $\bSB^{\lambda}(w, w),$
is not even a contraction for $\frac{1}{n+1} \leq \lambda <
\frac{n}{n+1},$ the induced homomorphism can't be contractive.  We
therefore conclude that the curvature inequality does not imply
the contractivity of $\rho_{T}$ whenever $n>1$. For $n=1$, an
example illustrating this (for the unit disc) was given in (cf.
\cite[page2]{shibu}). Thus the contractivity of the homomorphism
induced by the commuting tuple of the local operators $N(W),$ for
$T\in \mathcal P_1(\mathbb B)$ does not imply the contractivity of
the homomorphism induced by the commuting tuple of operators $T$.

\subsection{The matrix ball}
We recall that the positive function $\bSB_\mathcal D^\lambda,
\lambda >0,$ defines an inner product on the finite dimensional
space $\mathcal N^{(\lambda)}(w)$ for all $\lambda >0$
irrespective of whether $\bSB_\mathcal D^\lambda$ is positive
definite on the matrix ball $\mathcal D$ or not.  In this section,
we exclusively study the curvature inequality and contrativity of
the homomorphism $\rho^{(\lambda)}_{N^{(\lambda)}(w)}(w)$ induced
by the commuting tuples $N^{(\lambda)}(w)$ on the finite
dimensional Hilbert subspace $\mathcal N^{(\lambda)}(w),$ $\lambda
>0.$ We set $\mathcal K^{(\lambda)}(w):=\mathcal K_{\bSB_\mathcal D^\lambda}(w),$ $w\in \mathcal D.$
 %of the ``twisted Bergman spaces'' $\mathbb  A^{(\lambda)}(\mathcal %D)$ %corresponding to $\bSB_\mathcal D^\lambda$, $\lambda >0$.
If the homomorphism $\rho^{(\lambda)}_{N^{(\lambda)}(w)}(w)$ is
contractive for some $\lambda >0$, then for this $\lambda$, we
have:$\|\big({\mathcal K^{(\lambda)}}^{\rm t}\big)^{-1}(0)\|\leq
1.$ Like the Euclidian Ball, we study several implications of the
curvature inequality in this case.
%the commuting tuple of multiplication operators on $\mathbbb A^{(\lambda)} set of
%$\lambda$ for which it satisfies  the curvature inequality for the
%domain $\mathcal D$ but does not imply the contractivity of the
%homomorphism $\rho$ induced by the commuting tuples of
%multiplication operators.
%The norm  of $\big(\mathcal K^{\rm
%t}\big)^{-1}$ is easily seen to be $\frac{1}{p}$ as a consequence
%of the following Theorem.

\begin{theorem}\label{pmat}
For $\lambda >0,$  we have $\|\big({\mathcal K^{(\lambda)}}^{\rm
t}\big)^{-1}(0)\|_{\mathcal C_{\mathcal D, 0}^* \to \mathcal
C_{\mathcal D, 0}} = \frac{1}{\lambda p},$ $p=r+s.$
\end{theorem}
\begin{proof}
We have shown that $\big (\mathcal K^{\rm t}\big )^{-1}(0)=
\frac{1}{p} I_{rs}.$ Since $\mathcal C_{\mathcal D, 0}$ is the
operator norm on $(\mathcal M)_{rs}$ and consequently $\mathcal
C_{\mathcal D, 0}^*$ is the trace norm, it follows that
$\|I_{rs}\|_{\mathcal C_{\mathcal D, 0}^* \to \mathcal C_{\mathcal
D, 0}} \leq 1.$ This completes the proof.
\end{proof}
The following Theorem provides a necessary condition for the
contractivity of the homomorphism induced by the commuting tuple
of the local operators $N^{(\lambda)}(w).$

\begin{theorem} \label{themm2}
If the homomorphism $\rho^{(\lambda)}_{N^{(\lambda)}(w)}$ is
contractive, then $\nu \geq 1,$ where $\nu= \lambda p.$
\end{theorem}
\begin{proof}
The matrix unit ball $\mathcal D$ is homogenous. Let
$\theta_{w}(w)$ be the bi-holomorphic automorphism of $\mathcal D$
with $\theta_{w}(w)=0.$ We have seen that $A(w)^{\rm t}=A(0)^{\rm
t}D\theta_w^{-1}(0).$ Since $D\theta_w^{-1}(0)$ is an isometry,
therefore the contractivity of
$\rho^{(\lambda)}_{N^{(\lambda)}(w)}(0)$  implies that
contractivity of $\rho^{(\lambda)}_{N^{(\lambda)}(w)}(w),$  $w \in
\Omega,$ see Lemma \ref{A(w)}. The contractivity of
$\rho^{(\lambda)}_{N^{(\lambda)}(w)}(w)$ is equivalent to
$\|A(0)^{\rm t}\|_{\ell^2 \rightarrow \mathcal C_{\mathcal D,
0}}\leq 1.$ Therefore the contractivity of
$\rho^{(\lambda)}_{N^{(\lambda)}(w)}(w),$ for some $w \in \mathcal
D,$ implies $\|\big({\mathcal K^{(\lambda)}}^{\rm
t}\big)^{-\frac{1}{2}}(0)\|_{\mathcal C_{\mathcal D, 0}^* \to
\mathcal C_{\mathcal D, 0}}\leq 1.$ Theorem \ref{pmat} shows that
$\nu \geq 1.$

\end{proof}

If $\lambda >0$ is picked such that $\bSB^\lambda_\mathcal D$ is
positive definite,
 then Arazy and Zhang (cf. \cite[Proposition 5.5]{arazy})
 prove that the homomorphism induced by the commuting tuple of multiplication operators on
 the twisted Bergman space $\mathbb A^{(\lambda)}(\mathcal D)$ is bounded (k-spectral) if and only if $\nu \geq s.$

%the relationship
%between the $K$- spectral of $d$-tuple of multiplication operators
%and the range of $\nu.$
%\begin{theorem}(cf.\cite{arazy})\label{AR} Set $\Omega:= (\mathcal M_{rs})_1, rs=d$ and  $M^{(\lambda)}=
%(M_1^{(\lambda)},\ldots, M_d^{(\lambda)}): \mathbb
%A^{(\lambda)}(\Omega)\rightarrow\mathbb A^{(\lambda)}(\Omega)$ be
%the multiplication operator on the Bergman space $ \mathbb
%A^{(\lambda)}(\Omega).$ Then $M^{(\lambda)}$ defines a bounded
%homomorphism if and only if $\nu\geq s.$
%\end{theorem}

It follows that if $1\leq \nu<s$, then the homomorphism induced by
the commuting tuple of multiplication operators is not contractive
on twisted Bergman space $\mathbb A^{(\lambda)}(\mathcal D).$
While the homomorphism $\rho^{(\lambda)}_{N^{(\lambda)}(w)}(w),$
$w \in \Omega,$ is contractive on the finite dimensional Hilbert
space $\mathcal N^{\lambda}(w)$. This is equivalent to the
curvature inequality for $\nu \geq 1.$ However, for $1\leq \nu<s,$
the $rs$-tuple of multiplication operators on twisted Bergman
space $\mathbb A^{(\lambda)}(\mathcal D)$ is not contractive .

The localization of $N^{(\lambda)}(w)$ of any commuting tuple of
operators $T$ in $\mathcal P_1(\mathcal D)$  induces a
homomorphism $\rho^{(\lambda)}_{N^{(\lambda)}(w)}(w):\mathcal
A(\mathcal D) \rightarrow \mathcal L(\mathbb C^{rs+1})$ as
described in the equation \eqref{homN}. Therefore
$\rho^{(\lambda)}_{N^{(\lambda)}(w)}(w) \otimes I_{rs}: \mathcal
A(\mathcal D) \otimes \mathcal M_{rs} \to \mathcal L(\mathcal
N(w)) \otimes \mathcal M_{rs}$ is given by the formula
\begin{equation}\label{homNmat}
\rho^{(\lambda)}_{N^{(\lambda)}(w)}(w) \otimes
I_{rs}(P):=\left(\begin{smallmatrix}P(w) \otimes
I_{rs} & DP(w) \cdot N(w)\\
0 & P(w) \otimes I_{rs} \end{smallmatrix}\right),$$ where $$D P(w)
\cdot N(w)=\partial_1 P(w) \otimes N_1(w)+ \ldots
+\partial_{d}P(w)\otimes N_{rs}(w).
\end{equation}
The contractivity of
$\rho^{(\lambda)}_{N^{(\lambda)}(w)}(w)\otimes I_{rs},$ as shown
in (cf. \cite[Theorem 1.7]{sastry}, \cite[Theorm 4.2]{vern}) is
equivalent to the contractivity of the operator
$$\|\partial_1 P(w) \otimes N_1(w)+ \ldots
+\partial_{d}P(w)\otimes N_{rs}(w)\|_{\rm op}\leq 1.$$ Let
$P_\mathbf A$ be the matrix valued polynomial in $rs$ variables:
$$P_{\mathbf
A}(z)=\sum_{i=1}^{r}\sum_{j=1}^{s}z_{ij}E_{ij},$$ where $E_{ij}$
be the $r \times s$ matrices whose $(i, j)$ entries are $1$ and
other entries are $0$. Let $V=\left(V_1^{\rm t},
\ldots,V_{rs}^{\rm t}\right)$ be the $rs \times rs$ matrix, where
\begin{eqnarray*}V_1=(v_{11}, 0, \ldots,0_{sr}),\ldots ,V_{sr}=(
0, \ldots ,0, \ldots, v_{sr}).
\end{eqnarray*} We compute  the norm of $\rho^{(\lambda)}_{N^{(\lambda)}(w)}(w) \otimes
I_{rs}(P_{\mathbf A}).$
\begin{theorem}\label{P_A} For $\rho^{(\lambda)}_{N^{(\lambda)}(w)}(w)\otimes I_{rs}$ as above,
we have
$$\|\rho^{(\lambda)}_{N^{(\lambda)}(w)}(w)\otimes I_{rs}(P_{\mathbf
A})\|^2=\max\{\sum_{i=1}^{s}|v_{1i}|^2, \ldots,
\sum_{i=1}^{s}|v_{ri}|^2\}.$$
\end{theorem}
\begin{proof}
We have
\begin{eqnarray*}\lefteqn{ \|\big(\rho^{(\lambda)}_{N^{(\lambda)}(w)}(w) \otimes I_{rs}\big)(P_{\mathbf A})\|^2}
&\phantom{Avijit Pal Avijit pal Avijit pal}=&\|V_1\otimes
E_{11}+\ldots+V_s\otimes E_{1s}+V_{s+1}\otimes E_{21}+\ldots +V_{rs}\otimes E_{rs}\|^2\\
&\phantom{Avijit Pal Avijit pal Avijit
pal}=&\left\|\left(\begin{smallmatrix}V_1& \ldots & V_s\\\vdots &
\vdots & \vdots\\V_{rs-s+1}& \ldots
& V_{rs}\end{smallmatrix}\right)\right\|^2 =\left\|\left(\begin{smallmatrix}W_1\\
\vdots\\W_{r}\end{smallmatrix}\right)\right\|^2,
\end{eqnarray*}
where $W_i=\big(V_{is-s+1},\ldots, V_{is}\big).$ It is easy to see
that $W_iW_{j}^*=0$ for $i\neq j.$ Furthermore,
$W_iW_{i}^*=\sum_{j=1}^{s}|v_{ij}|^2.$ Hence we have
$$\|\rho^{(\lambda)}_{N^{(\lambda)}(w)}(w) \otimes I_{rs}(P_{\mathbf A})\|^2=\max\{\sum_{i=1}^{s}|v_{1i}|^2,
\ldots, \sum_{i=1}^{s}|v_{ri}|^2\}$$ completing the proof of the
theorem.
\end{proof}
Even for the small class of the form discussed here, homomorphisms
finding the cb norm of $\rho^{(\lambda)}_{N^{(\lambda)}(w)}(w)$ is
not easy. However, we determine when
$\|\rho^{(\lambda)}_{N^{(\lambda)}(w)}(w) \otimes
I_{rs}(P_{\mathbf A})\|^2 \leq 1.$ This gives a necessary
condition for the complete contractivity of
$\rho^{(\lambda)}_{N^{(\lambda)}(w)}(w).$
\begin{theorem}\label{complete}If $\|\rho ^{(\lambda)}_{N^{(\lambda)}(w)}(w)\otimes I_{rs}(P_{\mathbf
A})\|^2 \leq 1,$ then  $\nu \geq s.$
\end{theorem}
\begin{proof}By Theorem \ref{P_A} we have
$$\|\rho^{(\lambda)}_{N^{(\lambda)}(w)}(w) \otimes I_{rs}(P_{\mathbf
A})\|^2=\max\{\sum_{i=1}^{s}|v_{1i}|^2, \ldots,
\sum_{i=1}^{s}|v_{ri}|^2\}.$$ Since $|v_{ij}|^2=\frac{1}{\nu}, 1
\leq i\leq r, 1\leq j\leq s,$ it is immediate that
$\|\rho^{(\lambda)}_{N^{(\lambda)}(w)}(w) \otimes
I_{rs}(P_{\mathbf A})\|^2\leq 1$ implies $\nu \geq s$ completing
the proof of the theorem.
\end{proof}

 As a consequence, it follows  that if $1\leq \nu<s$, then
the homomorphism induced by the commuting tuple of the local
operators $N^{(\lambda)}(w)$ is not completely contractive.

\subsection{More examples}
We have  discussed the Bergman kernel $\bSB_{\Omega}(w, w)$ for
the domain $\Omega=\{(z_1, z_2): |z_2|
 \leq (1-|z_1|^2)\}
\subset \mathbb C^2.$ The curvature $\mathcal
K_{\bSB_{\Omega}}(w)=\sum_{i, j=1}^{2}T_{ij}(w)dw_i \wedge
d\bar{w}_j$ of the Bergman Kernel $\bSB_{\Omega}(w, w)$ is
(cf.\cite[Example 6.2.1]{pflug}):
$$
T_{11}(w)=6\big(\frac{1}{C(w)}-\frac{1}{D(w)}\big)+12|w_1|^2|w_2|^2\big(\frac{1}{C^2(w)}+\frac{1}{D^2(w)}\big),$$
$$T_{12}(w)=\bar{T}_{21}(w)=6w_1\bar{w}_2(1-|w_1|^2)\big(\frac{1}{C^2(w)}+\frac{1}{D^2(w)}\big),$$
$$T_{22}(w)=3(1-|w_1|^2)^2\big(\frac{1}{C^2(w)}+\frac{1}{D^2(w)}\big),$$
where $ C(w):=(1-|w_1|^2)^2-|w_2|^2$ and
$D(w):=3(1-|w_1|^2)^2+|w_2|^2.$ We have  seen that the
polarization $\bSB_{\Omega}^\lambda(z,w)$ of the function
$\bSB_{\Omega}(w,w)^\lambda$
 defines a Hermitian structure for $\mathcal N^{(\lambda)}(w).$
 Specializing to $w=0,$  since $-\big (\mathcal K(0)^{\rm t}\big )^{-1}= A(0)^{\rm
 t}\overline{A(0)}, $ we have
 $a_{11}^\lambda(0)=\frac{1}{\sqrt{T_{11}(0)}}$ and
 $a_{22}^\lambda(0)=\frac{1}{\sqrt{T_{22}(0)}},$  where
 $(A^{\lambda}(0))^{\rm t}=\left(\begin{smallmatrix}
a_{11}^\lambda(0) & 0\\
0 & a_{22}^\lambda(0)\\
\end{smallmatrix}\right).$
\begin{theorem} \label{themm5}The contractivity of the
homomorphism  $\rho^{(\lambda)}_{N^{(\lambda)}}(0)$ implies
$16\lambda\geq 5.$
\end{theorem}
\begin{proof}
We have $a_{11}^\lambda(0)=\frac{1}{2\sqrt{\lambda}},
a_{12}^\lambda(0)=0, a_{22}^\lambda(0)=\frac{3}{\sqrt{10\lambda}}.
$ Contractivity of homomorphism
$\rho^{(\lambda)}_{N^{(\lambda)(0)}}$ is equivalent to
$\|(A^{\lambda}(0))^{\rm t}\|_{\ell^2 \rightarrow \mathcal
C_{\Omega, 0}}\leq 1 .$ This is equivalent to
$(2(a_{11}^\lambda(0))^2-1)^2\leq (1-(a_{22}^\lambda(0))^2).$
Hence  $16\lambda\geq 5$ completing our proof.
\end{proof}
The bi-holomorphic automorphism group of $\Omega$ is not
transitive. So the contractivity of the homomorphism
$\rho^{(\lambda)}_{N^{(\lambda)}}(0)$   does not necessarily imply
the contractivity of the homomorphism
$\rho^{(\lambda)}_{N^{(\lambda)}}(w), w \in \Omega.$  Determining
which of the homomorphism $\rho^{(\lambda)}_{N^{(\lambda)}}(w)$ is
contractive, appears to be a hard problem.

Let $P_{\mathbf A}:\Omega \rightarrow(\mathcal M_2)_1$ be the
matrix valued polynomial on $\Omega$ defined by $P_{\mathbf
A}(z)=z_1A_1+z_2A_2$ where $A_1= I_2$ and
$A_2=\left(\begin{smallmatrix}
0 & 1\\
    0    & 0
\end{smallmatrix}\right).$ It is natural to ask when $\rho^{(\lambda)}_{N^{(\lambda)}}(w)$ is completely contractive.
As before, we only obtain a necessary condition using the
polynomial $P_{\mathbf A}.$

\begin{theorem}\label{themm22}
$\|\rho^{(2)}_{N^{(\lambda)}}(0)(P_{\mathbf A})\|\leq 1$ if and
only if $\lambda \geq \frac{11}{20}.$
\end{theorem}
\begin{proof}
Suppose that  $\|\rho^{(2)}_{N^{(\lambda)}}(0)(P_{\mathbf
A})\|\leq 1.$  Then we have
$(a_{11}^\lambda(0))^2+(a_{22}^\lambda(0))^2\leq 1.$ Hence
$\lambda \geq \frac{11}{20}.$ The converse verification is also
equally easy.
\end{proof}

We conclude that if $\frac{5}{16}\leq \lambda < \frac{11}{20},$
the homomorphism $\rho^{(\lambda)}_{N^{(\lambda)}}(0)$ is
contractive but not completely contractive. An explicit
description of the set
$$\{\lambda: \|\rho^{(\lambda)}_{N^{(\lambda)}}(w)(P_{\mathbf A})
\|_{\rm op}\leq 1, w \in \Omega\}$$ would certainly  provide
greater insight. However, it appears to be quite intractable, at
least for now.

The formula for the Bergman kernel for the domain
$$\Omega:=\{(z_1, z_2, z_3): |z_2|^2
 \leq (1-|z_1|^2)(1-|z_3|^2), 1-|z_3|^2\geq 0\}
\subset \mathbb C^3.$$ is given in Lemma \ref{bergman}. From Lemma
\ref{bergman} we have  $\bSB_{\Omega}^\lambda(z,0)=1$ and
$\partial_{i}\bSB_{\Omega}^\lambda(z,0)=0$ for $i=1, 2, 3.$ Hence
the desired curvature matrix is of the form $$ \big ( \!\! \big
(\, (\partial_i\bar{\partial}_j \log\bSB_{\Omega}^\lambda) (0,
0)\big ) \!\! \big )_{i\,j = 1}^m.$$ Let
$T_{ij}(0)=\partial_i\bar{\partial}_j \log\bSB_{\Omega}^\lambda
(0, 0),$ that is, $\mathcal K_{\bSB_{\Omega}}(0)=\sum_{i,
j=1}^{3}T_{ij}(0)dw_i \wedge d\bar{w}_j.$ An easy computation
shows that $T_{11}(0)=3\lambda=T_{33}(0),
T_{22}(0)=\frac{9\lambda}{2}$ and $T_{ij}(0)=0$ for $i \neq j.$ As
before, we have $a_{11}^\lambda(0)=\frac{1}{\sqrt{T_{11}(0)}},
 a_{22}^\lambda(0)=\frac{1}{\sqrt{T_{22}(0)}}$ and
 $a_{33}^\lambda(0)=\frac{1}{\sqrt{T_{33}(0)}},$ where
$A(0)^{\rm t}=\left(\begin{smallmatrix}
a_{11}^{\lambda}(0) & 0 &0\\
0 & a_{22}^{\lambda}(0) &0 \\
0 & 0& a_{33}^{\lambda}(0)\\
\end{smallmatrix}\right).$

\begin{theorem}\label{themm10} The contractivity of the
homomorphism  $\rho^{(\lambda)}_{N^{(\lambda)}}(0)$  implies
$\lambda\geq \frac{1}{4}.$
\end{theorem}
\begin{proof}
From Lemma (\ref{lemm:main4}) we have
$a_{11}^{\lambda}(0)=\frac{1}{\sqrt{3\lambda}},
a_{12}^{\lambda}(0)=a_{13}^{\lambda}(0)=0,
a_{22}^{\lambda}(0)=\frac{\sqrt{2}}{3\sqrt{\lambda}},
a_{23}^{\lambda}(0)=0$ and
$a_{33}^{\lambda}(0)=\frac{1}{\sqrt{3\lambda}}.$ The contractivity
of the homomorphism  $\rho^{(\lambda)}_{N^{(\lambda)}(w)}(w)(0)$
is equivalent to $\big \| A(0)^{\rm t}\|_{\ell^2 \to
 \mathcal C_{\Omega, 0}}^2 \leq 1.$ This is equivalent to
$|a_{11}^{\lambda}(0)|^2(1-| a_{33}^{\lambda}(0)|^2)\geq
(|a_{22}^{\lambda}(0)|^2-| a_{33}^{\lambda}(0)|^2).$  Hence we
have $\lambda\geq \frac{1}{4}.$
\end{proof}
For our final example, let $P_{\mathbf A}:\Omega
\rightarrow(\mathcal M_2)_1$ be also the matrix valued polynomial
on $\Omega$ defined by $P_{\mathbf A}(z)=z_1A_1+z_2A_2+z_3A_3$
where $A_1=\left(\begin{smallmatrix}
1 & 0\\
    0    & 0
    \end{smallmatrix}\right),
A_2=\left(\begin{smallmatrix}
0 & 1\\
    0    & 0
    \end{smallmatrix}\right), A_3=\left(\begin{smallmatrix}
0 & 0\\
    0    & 1
    \end{smallmatrix}\right).$
\begin{theorem}\label{themm9}
$\|\rho^{(2)}_{N^{(\lambda)}}(0)(P_{\mathbf A})\|\leq 1$ if and
only if $\lambda \geq \frac{5}{9}.$
\end{theorem}
\begin{proof}
Suppose that  $\|\rho^{(2)}_{N^{(\lambda)}}(0)(P_{\mathbf
A})\|\leq 1.$  Then we have
$$\max\{(a_{11}^\lambda(0))^2+(a_{22}^\lambda(0))^2,
(a_{33}^\lambda(0))^2\} \leq 1.$$ Hence $\lambda \geq
\frac{5}{9}.$ The converse statement is easily verified.
\end{proof}

Thus if $\frac{1}{4}\leq \lambda < \frac{5}{9},$ the homomorphism
$\rho^{(\lambda)}_{N^{(\lambda)}}(0)$ is contractive but not
completely contractive.

\chapter{Contractivity vs. complete contractivity}
\phantom{}
\section{Homomorphisms  induced by $m$ vectors}
 We now assume that  $\mathbf v_{i}=\left(v_{i1}, \ldots, v_{im}\right), 1\leq i \leq m,$ is a vector in $\mathbb C^m.$
 The commuting tuple
$$N(V, w):=\left(
\left ( \begin{smallmatrix}
w_1 & \mathbf v_1   \\
0  & w_1I_m
\end{smallmatrix}\right ),\cdots , \left ( \begin{smallmatrix}
w_m & \mathbf v_m   \\
0  & w_mI_m
\end{smallmatrix}\right )\right),$$ $w=(w_1,\ldots,w_m)\in \Omega_{\mathbf A},$ defines a homomorphism
$\rho_{V}:\mathcal O(\Omega_{\mathbf A})\rightarrow \mathcal
M_{m+1}$ which is given by the formula $$\rho_{V}(f)= \left (
\begin{smallmatrix}
f(w)& \nabla f(w) V \\
0  & f(w)I
\end{smallmatrix}\right ),\, f \in \mathcal O(\Omega_{\mathbf A}) ,$$ where
$\nabla f(w) V =\partial_{1}f(w)\mathbf v_1+ \cdots+
\partial_{m}f(w)\mathbf v_{m}.$
We derive a criterion for  contractivity of the homomorphism
$\rho_{V}.$ We also compute $\|\rho_{V}^{(n)}(P_{\mathbf A})\|,$
where \mbox{$P_{\mathbf A}(z_{1}, \ldots,
z_{m})=z_{1}A_{1}+\cdots+z_{m}A_{m}.$} If $f : \Omega_{\mathbf
A}\longmapsto \mathbb D$ is a holomorphic function with  $f(0)=0$
and $||f||_{\infty,\mathbb D} \leq 1,$ then the vector
$(\partial_{1}f(0),\ldots,
\partial_{m}f(0))$ is in the dual unit ball $(\mathbb C^m,
\|\,\cdot\,\|^{*}_{\Omega_{\mathbf A}})_1$(see Corollary
\ref{dualballsc}). Now, {\small\begin{align*} \sup \{ \|
\rho_{V}(f) \| : \|f \|_{\infty, \mathbb{D}} \leq 1 \} & = \sup
\{\|
\rho_{V}(f) \| : \|f \|_{\infty, \mathbb{D}} \leq 1,f(0)=0 \}\\
& = \sup \{\|\,\partial_{1}f(0)\mathbf v_1+ \cdots+
\partial_{m}f(0)\mathbf v_{m}\,\|: \|f \|_{\infty,
\mathbb{D}} \leq 1,f(0)=0 \}\\
& = \sup \{\|\,\lambda_1\mathbf v_{1}+\cdots+\lambda_m\mathbf
v_{m}\,\|:(\lambda_1, \ldots, \lambda_m)\in (\mathbb C^m,
\|\,\cdot\,\|^{*}_{\Omega_{\mathbf A}})_1\}.
\end{align*}}
Let $L_{V}:(\mathbb C^m,\|\,\cdot\,\|^{*}_{\Omega _{\mathbf A}})
\rightarrow (\mathbb C^m,\|\,\cdot\,\|_2)$ be the linear map
induced by the  matrix $ \left(\mathbf v_1^{\rm t}, \ldots,
\mathbf v_m^{\rm t} \right)$. The matrix representing $L_{V}$ also
gives a matrix representation of the adjoint
$$L_{V}^*:(\mathbb C^m,\|\,\cdot\,\|_2)\rightarrow
(\mathbb C^m,\|\,\cdot\,\|_{\Omega_{\mathbf A}}).$$
%which is of the
%form $V= \left(V_1^{\rm t}, \ldots, V_m^{\rm t} \right)^{\rm t}$.

Clearly, $\|\,L_{V}\,\|_{(\mathbb C^m,
\|\,\cdot\,\|_{\Omega_{\mathbf A}}^*)\rightarrow (\mathbb
C^m,\|\,\cdot\,\|_2)} \leq 1$ if and only if
$\|L_{V}^*\|_{(\mathbb C^m, \|\,\cdot\,\|_2)\rightarrow (\mathbb
C^m, \|\,\cdot\,\|_{\Omega_{\mathbf A}})} \leq 1$ if and only if
$\|\,\rho_{V}\,\|_{\mathcal O(\Omega_{\mathbf A})\rightarrow
\mathcal M(\mathbb C^{m+1})}\leq 1.$ In characterizing the
contractivity of $\rho_{V},$ we will often determine if
$\|L_{V}^*\|_{(\mathbb C^m, \|\,\cdot\,\|_2)\rightarrow (\mathbb
C^m, \|\,\cdot\,\|_{\Omega_{\mathbf A}})} \leq 1.$ The following
proposition gives a criterion for the contractivity of $\rho_{V}.$
\begin{proposition}
%Suppose that $B_{j}=\sum_{i=1}^{m}v_{ij}A_{i}$ then
\begin{enumerate}
The following conditions are equivalent:

 \item[(i)] $\rho_{V}$ is
contractive,

\item [(ii)] $\sup_{\sum_{j=1}^{m}|x_j|^2 \leq 1}
\|\sum_{j=1}^{m}x_jB_{j} \|^2 \leq 1,$ where
$B_{j}=\sum_{i=1}^{m}v_{ij}A_{i},$

\item[(iii)] \begin{equation}\label{mainmain} B(\beta,
\beta)=\left (
\begin{smallmatrix}
1-\left\langle B_{1}B_{1}^{*}\beta, \beta\right\rangle & -\left\langle B_{1}B_{2}^{*}\beta, \beta\right\rangle &\ldots & -\left\langle B_{1}B_{m}^{*}\beta, \beta\right\rangle \\
-\left\langle B_{2}B_{1}^{*}\beta, \beta\right\rangle & 1-\left\langle B_{2}B_{2}^{*}\beta, \beta\right\rangle & \ldots & -\left\langle B_{2}B_{m}^{*}\beta, \beta\right\rangle \\
\vdots & \vdots & \vdots & \vdots \\
-\left\langle B_{m}B_{1}^{*}\beta, \beta\right\rangle  &
-\left\langle B_{m}B_{2}^{*}\beta, \beta\right\rangle  & \ldots &
1 -\left\langle B_{m}B_{m}^{*}\beta, \beta\right\rangle
\end{smallmatrix}\right )\geq 0,
\end{equation}
where $\sum_{i=1}^{n}|\beta_i|^2=1.$
\end{enumerate}
\end{proposition}

\begin{proof}
First we will prove that  $(i)$ and  $(ii)$ are equivalent. We
have earlier seen that \mbox{$\|L_{V}^*\|_{(\mathbb C^m,
\|\,\cdot\,\|_2)\rightarrow (\mathbb C^m,
\|\,\cdot\,\|_{\Omega_{\mathbf A}})} \leq 1$} if and only if
$\|\,\rho_{V}\,\|_{\mathcal O(\Omega_{\mathbf A})\rightarrow
\mathcal M(\mathbb C^{m+1})}\leq 1.$ The matrix representation of
$L_{V}^*$ is of the form $\left (
\begin{smallmatrix}
v_{11} & \ldots & v_{1m}  \\
\vdots & \vdots &\vdots\\
 v_{m1} &\ldots   & v_{mm}
\end{smallmatrix} \right ).$ Since $L_{V}^*$ maps $(\mathbb
C^m,\|\,\cdot\,\|_2)$ into $(\mathbb
C^m,\|\,\cdot\,\|_{\Omega_{\mathbf A}}),$ we have \mbox{$\left (
\begin{smallmatrix}
v_{11} & \ldots & v_{1m}  \\
\vdots & \vdots &\vdots\\
 v_{m1} &\ldots   & v_{mm}
\end{smallmatrix} \right )\left ( \begin{smallmatrix}
x_1 \\
 \vdots\\
 x_m
\end{smallmatrix} \right )\in (\mathbb C^m,\|\,\cdot\,\|_{\Omega_{\mathbf
A}}).$}
%with $\sum_{j=1}^{m}|x_j|^2 \leq 1.$
Thus
\mbox{$\|L_{V}^*\|_{(\mathbb C^m, \|\,\cdot\,\|_2)\rightarrow
(\mathbb C^m, \|\,\cdot\,\|_{\Omega_{\mathbf A}})} \leq 1$} if and
only if $\sup_{\sum_{j=1}^{m}|x_j|^2 \leq 1}
\|\sum_{j=1}^{m}x_jB_{j} \|^2 \leq 1,$ where
$B_{j}=\sum_{i=1}^{m}v_{ij}A_{i}.$ Hence $(i)$ and $(ii)$ are
equivalent.

To see that $(ii)$ and $(iii)$ are equivalent note that
\mbox{$\sup_{\sum_{j=1}^{m}|x_j|^2 \leq 1}
\|\sum_{j=1}^{m}x_jB_{j} \|^2 \leq 1 $}  if and only if

{\small \begin{align}\nonumber
I_{n}-\sum_{j=1}^{m}|x_j|^2B_{j}B_{j}^*-\sum_{i=1}^{m}\sum_{j,
i<j}^{m}x_{i}\bar{x_{j}}B_{i}B_{j}^*-\sum_{j=1}^{m}\sum_{i,
j<i}^{m}\bar{x_{j}}{x_{i}}B_{i}B_{j}^* \geq 0
\end{align}}which is clearly equivalent
to $$\sum_{j=1}^{m}|x_j|^2\left\langle (I_{n}-B_{j}B_{j}^*)\alpha,
\alpha\right\rangle-\sum_{i=1}^{m}\sum_{j,
i<j}^{m}x_{i}\bar{x_{j}}\left\langle B_{i}B_{j}^*\alpha,
\alpha\right\rangle-\sum_{j=1}^{m}\sum_{i,
j<i}^{m}\bar{x_{j}}{x_{i}}\left\langle B_{i}B_{j}^*\alpha,
\alpha\right\rangle \geq 0.$$ Or, equivalently,

\begin{equation}\label{mainmain1}\left\langle \left (
\begin{smallmatrix}
\left\langle (I_{n}-B_{1}B_{1}^{*})\alpha, \alpha\right\rangle &
-\left\langle B_{1}B_{2}^{*}\alpha, \alpha\right\rangle &\ldots &
-\left\langle B_{1}B_{m}^{*}\alpha, \alpha\right\rangle \\
-\left\langle B_{2}B_{1}^{*}\alpha, \alpha\right\rangle &
\left\langle (I_{n}-B_{2}B_{2}^{*})\alpha, \alpha\right\rangle & \ldots &
 -\left\langle B_{2}B_{m}^{*}\alpha, \alpha\right\rangle \\
\vdots & \vdots & \vdots & \vdots \\
-\left\langle B_{m}B_{1}^{*}\alpha, \alpha\right\rangle  &
-\left\langle B_{m}B_{2}^{*}\alpha, \alpha\right\rangle  & \ldots
& \left\langle (I_{n} -B_{m}B_{m}^{*})\alpha, \alpha\right\rangle
\end{smallmatrix}\right )\left ( \begin{smallmatrix}x_1\\ \vdots\\ \vdots\\x_m
\end{smallmatrix}\right ), \left ( \begin{smallmatrix}x_1\\ \vdots\\ \vdots\\x_m
\end{smallmatrix}\right )\right\rangle\geq 0.
\end{equation}
Putting $\beta= \frac{\alpha}{\|\alpha\|}$ in Equation
(\ref{mainmain1}) we verify Equation (\ref{mainmain}).
%Conversely, if $B(\beta, \beta)\geq 0,$ then
%$\sup_{\sum_{j=1}^{m}|x_j|^2 \leq 1} \|\sum_{j=1}^{m}x_jB_{j} \|^2
%\leq 1. $
This completes the proof of the proposition.
\end{proof}
%\begin{corollary}
%$\rho_{V}$ is contractive if and only if
%\begin{equation}\label{mainmain22} B(\beta, \beta)=\left (
%\begin{smallmatrix}
%1-\left\langle B_{1}B_{1}^{*}\beta, \beta\right\rangle & -\left\langle B_{1}B_{2}^{*}\beta, \beta\right\rangle &\ldots & -\left\langle B_{1}B_{m}^{*}\beta, \beta\right\rangle \\
%-\left\langle B_{2}B_{1}^{*}\beta, \beta\right\rangle & 1-\left\langle B_{2}B_{2}^{*}\beta, \beta\right\rangle & \ldots & -\left\langle B_{2}B_{m}^{*}\beta, \beta\right\rangle \\
%\vdots & \vdots & \vdots & \vdots \\
%-\left\langle B_{m}B_{1}^{*}\beta, \beta\right\rangle  &
%-\left\langle B_{m}B_{2}^{*}\beta, \beta\right\rangle  & \ldots &
%1 -\left\langle B_{m}B_{m}^{*}\beta, \beta\right\rangle
%\end{smallmatrix}\right )\geq 0,
%\end{equation}
%where  $\sum_{i=1}^{n}|\beta_i|^2=1.$
%\end{corollary}
In particular, if  $ B_1= v_{11}A_1+v_{21}A_2 $ and $B_2=
v_{12}A_1+v_{22}A_2 ,$ then we have $$\sup_{|x|^2+|y|^2 \leq 1}
\|xB_1+yB_2 \|^2 \leq 1, $$ which is equivalent to one of
$1-\left\langle B_1B_{1}^{*}\beta, \beta\right\rangle \geq 0$ or
$1-\left\langle B_2B_{2}^{*}\beta, \beta\right\rangle \geq 0$ and
{\small\begin{equation} \inf_{\beta}\{ 1-\left\langle
B_1B_{1}^{*}\beta, \beta\right\rangle -\left\langle
B_2B_{2}^{*}\beta, \beta\right\rangle +\left\langle
B_1B_{1}^{*}\beta, \beta\right\rangle \left\langle
B_2B_{2}^{*}\beta, \beta\right\rangle - |\left\langle
B_1B_{2}^{*}\beta, \beta\right\rangle |^2\} \geq 0,
\end{equation}}
for all $\beta \in \mathbb C^2$ with $\|\beta\|_2=1.$ Hence
$\|\,\rho_{V}\,\|_{\mathcal O(\Omega_{\mathbf A})\rightarrow
\mathcal M(\mathbb C^{3})}\leq 1$ if and only if $1-\left\langle
B_1B_{1}^{*}\beta, \beta\right\rangle \geq 0$ and
{\small\begin{equation}\label{main} \inf_{\beta}\{ 1-\left\langle
B_1B_{1}^{*}\beta, \beta\right\rangle -\left\langle
B_2B_{2}^{*}\beta, \beta\right\rangle +\left\langle
B_1B_{1}^{*}\beta, \beta\right\rangle \left\langle
B_2B_{2}^{*}\beta, \beta\right\rangle - |\left\langle
B_1B_{2}^{*}\beta, \beta\right\rangle |^2 \}\geq 0.
\end{equation}}
The following proposition gives a criterion for the contractivity
of $\rho_{V}^{(n)}(P_{\mathbf A}).$
\begin{proposition}
The following conditions are equivalent:
\begin{enumerate}
\item[(i)]$\|\rho_{V}^{(n)}(P_{\mathbf A})\|\leq 1,$

\item[(ii)] $\| (B_1, \ldots , B_m)
 \|\leq 1,$ where $B_i=\sum_{j=1}^{m}v_{ij}A_j,$

\item[(iii)] $1-\sum_{i=1}^{m}\sum_{j=1}^{m}\langle
B_iB_{j}^{*}\beta, \beta\rangle \geq 0,$ where
$\beta=(\beta_1,\ldots, \beta_m)\in \mathbb C^m $ and
$\sum_{i=1}^{m}|\beta_i|^2=1.$

\end{enumerate}
\end{proposition}

\begin{proof}
First we will prove that  $(i)$ and  $(ii)$ are equivalent. Since
$P_{\mathbf A}(0)=0,$ it follows from Proposition \ref{prop 1}
that $\|\rho_{V}^{(n)}(P_{\mathbf A})\|\leq 1$ if and only if
$$\|A_1\otimes \mathbf v_1+\ldots+A_m\otimes \mathbf v_m\|\leq 1.$$
For $\mathbf v_{i}=\left(v_{i1}, \ldots, v_{im}\right),$ we have
\begin{align*}
A_1\otimes \mathbf v_1+\ldots+A_m\otimes \mathbf v_m &=(
B_1,\ldots, B_m ).
\end{align*}

Thus $\rho_{V}^{(n)}(P_{\mathbf A})$ is contractive if and only if
$\| (B_1, \ldots , B_m)
 \|\leq 1.$
Hence $(i)$ is equivalent to $(ii).$

 Now, we will prove that
$(ii)$ implies $(iii).$ Let $T=( B_1, \ldots,  B_m).$ The
contractivity of  $T$ is equivalent to the positivity of $I-TT^*,$
which is equivalent to $\left\langle (I-TT^*)\alpha, \alpha
\right\rangle \geq 0 $ for all $\alpha.$ In particular, putting
$\beta=\frac{\alpha}{\|\alpha\|},$ we have $(iii).$ Clearly,
$(iii)$ implies $(i)$  completing the proof.
\end{proof}
\begin{example}
If $A_1=I_2$ and  $A_2=\left(\begin{smallmatrix}
0 & 1   \\
0  & 0 \end{smallmatrix}\right),$ then the homomorphism $\rho_{V}$
is contractive if and only if $|v|^2\leq 1$ and
$$\inf_{\beta}\{1-|v |^2 -|w|^2|\beta_1|^2
 + |vw|^2|\beta_1|^4\}\geq 0.$$ Also, $\|P_{\mathbf A}(T_1, T_2)\|\leq 1$
 if and only if
$$\inf_{\beta}\{1-|v|^2 -|w|^2|\beta_1|^2\}\geq 0.$$
If $\mathbf v_1=(\frac{1}{\sqrt{2}}, 0)$ and $\mathbf v_2=(0, 1),$
then it is easy to see that the homomorphism $\rho_{V}$ is
contractive. But for this choice of $\mathbf v_1, \mathbf v_2$ we
have $\|P_{\mathbf A}(T_1, T_2)\|> 1.$ Hence this contractive
homomorphism $\rho_{V}$ is not completely contractive.
\end{example}
Let $\Omega_\mathbf A=\{(z_1,
z_2):\left\|z_1A_1+z_2A_2\right\|_{\rm op} < 1\}$ be a domain in
$\mathbb C^2,$ where $A_1=\left (
\begin{smallmatrix}
1 & 0   \\
0  & d_{2}
\end{smallmatrix}\right )$ or $\left ( \begin{smallmatrix}
d_{1} & 0   \\
0  & 1
\end{smallmatrix}\right )$ and  $A_2= \left ( \begin{smallmatrix}
 0 &  b \\
c &  0
\end{smallmatrix}\right ) ,\left ( \begin{smallmatrix}
1 &  b \\
c &  0
\end{smallmatrix}\right ) $ or $\left ( \begin{smallmatrix}
 0 &  b \\
 c &  1
\end{smallmatrix}\right ) $ with $b \in \mathbb R^{+}.$
Let  $P_{\mathbf A}:\Omega_\mathbf A\to (\mathcal M_2)_1$ be the
matrix valued polynomial of the form $P_{\mathbf
A}(z)=z_1A_1+z_2A_2.$ In particular, if  $ B_1=
v_{11}A_1+v_{21}A_2 $ and $B_2= v_{12}A_1+v_{22}A_2 ,$ then
$\|\rho_{V}^{(2)}(P_{\mathbf A})\|\leq 1$ if and only if
{\small\begin{equation}\label{com connt} \inf_{\beta}\{
1-\left\langle B_1B_{1}^{*}\beta, \beta\right\rangle -\left\langle
B_2B_{2}^{*}\beta, \beta\right\rangle\} \geq 0,
\end{equation}} where $\|\beta\|_2=1.$ Finding a $V$ such that
$\|L_{V}\|_{(\mathbb C^2,\|\,\cdot\,\|^{*}_{\Omega _{\mathbf A}})
\rightarrow (\mathbb C^2,\|\,\cdot\,\|_2) }\leq 1$ for which
\mbox{$\|\rho_{V}^{(2)}(P_{\mathbf A})\|_{\rm op}>1$} produces an
example of a contractive homomorphism of $\mathcal
O(\Omega_\mathbf A)$ which is not completely contractive. Thus it
is enough to find $\mathbf v_1$ and $\mathbf v_2$ for which
inequality (\ref{main}) is valid while inequalities (\ref{com
connt}) fails. Luckily for us, to find such an example, it is
enough to choose $\mathbf v_1=(v,0)$ and $\mathbf v_2=(0, w).$ In
this case, the inequality (\ref{main}) is equivalent to
$\inf_{\beta}\{1-|v|^2\|A_1^{*}\beta\|^2\}\geq 0$ and
{\small\begin{equation}\label{mainconnn} \inf_{\beta}\{
1-|v|^2\|A_1^{*}\beta\|^2-|w|^2\|A_2^{*}\beta\|^2
+|vw|^2(\|A_1^{*}\beta\|^2\|A_2^{*}\beta\|^2 - |\left\langle
A_1A_2^*\beta, \beta\right\rangle |^2 )\}\geq 0.
\end{equation}}
and the inequality (\ref{com connt}) is equivalent to
{\small\begin{equation}\label{maincon} \inf_{\beta}\{
1-|v|^2\|A_1^{*}\beta\|^2-|w|^2\|A_2^{*}\beta\|^2\}\geq 0.
\end{equation}}
Define $g_{(v,w)}:\partial \mathbb B^2\rightarrow\mathbb
R\cup\{0\}$ by
{\small\begin{align*}g_{(v,w)}(\beta)&=1-|v|^2\|A_{1}^*\beta\|^2
-|w|^2\|A_{2}^*\beta\|^2+
|vw|^2(\|A_{1}^*\beta\|^2\|A_{2}^*\beta\|^2 - |\left\langle
A_1A_{2}^*\beta, \beta\right\rangle |^2),\end{align*}}where
$\mathbb B^2$ is the closed unit ball in $\mathbb C^2$ with
respect to the $\ell^2$ norm. Since $g$ is continuous and
$\partial \mathbb B^2$ is compact, it follows that
$\inf_{\beta}g_{(v,w)}(\beta)$ exists. Hence $\|\rho_{V}\|\leq 1$
is equivalent to $|v|^2\leq \frac{1}{\|A_1^{*}|^2}$ and
$\inf_{\beta}g_{(v,w)}(\beta)\geq 0.$

%If we consider
%$$\|A_1^*\beta\|^2\|A_2^*\beta\|^2 - |\left\langle A_1A_2^*\beta,
%\beta\right\rangle |^2=0,$$ then $A_2^*\beta=\mu(\beta)
%A_1^*\beta.$  For this case, we have three possibilities, namely,
%$$\dim\ker(A_2^*-\mu(\beta) A_1^*)=0,1,2.$$ Let $\mu_1(\beta{'}),\mu_2(\beta{''})$ be two
%roots of $\det(A_{2}^*-\mu(\beta) A_{1}^*)=0$ and $\beta{'}$ and
%$\beta^{''}$ be the  vectors corresponding to the roots
%$\mu_1(\beta{'}), \mu_2(\beta{''})$ respectively.
%\begin{equation}
%\dim\ker(A_2^*-\mu A_1^*)=\begin{cases}

%1\\

%{0}   & \text{}\end{cases}
%\end{equation}
\begin{theorem}\label{mainresult}
If $A_1$ and $A_2$ are not simultaneously diagonalizable, then
there exists  a contractive linear map from $(\mathbb C^2,
\|\cdot\|_{\Omega_\mathbf A})$ to $\mathcal M_n(\mathbb C)$ which
is not completely contractive.
\end{theorem}
\begin{proof}
 Fix $\mathbf v_1=(v,0), \mathbf v_2=(0,w).$ Let
$L_{(\mathbf v_1, \mathbf v_2)}:(\mathbb C^2,\|\,\cdot\,\|^{*}
_{\Omega _{\mathbf A}})\rightarrow (\mathbb C^2,\|\,\cdot\,\|_2)$
be the linear map $(z_1,z_2)\mapsto(z_1v,z_2w).$
%For $\beta$ in
%$\mathbb C^2,$ let
%$$g_{(v,w)}(\beta):=\{1-|v|^2\|A_{1}^*\beta\|^2
%-|w|^2\|A_{2}^*\beta\|^2+
%|vw|^2(\|A_{1}^*\beta\|^2\|A_{2}^*\beta\|^2 - |\left\langle
%A_1A_{2}^*\beta, \beta\right\rangle |^2)\}.$$
The contractivity of $L_{(\mathbf v_1, \mathbf v_2)}:(\mathbb
C^2,\|\,\cdot\,\|^{*} _{\Omega _{\mathbf A}})\rightarrow (\mathbb
C^2,\|\,\cdot\,\|_2)$ is equivalent to $|v|^2\leq
\frac{1}{\|A_1^{*}\|^2}$ and $(v,w)\in \mathcal{E}:=\{(v,
w):\inf_{\beta, \|\beta\|_2=1}g_{(v, w)}(\beta)\geq 0\}$  and the
contractivity of $L_{(\mathbf v_1, \mathbf v_2)}^{(2)}(P_{\mathbf
A})$ is shown to be equivalent to the condition
$$\inf_{\beta}\{
1-|v|^2\|A_1^{*}\beta\|^2-|w|^2\|A_2^{*}\beta\|^2:\|\beta\|_2 =
1\}\geq 0.$$  Pick $(v, w)$ such that $w=\lambda v, \lambda >0.$
There exists $\beta$ in $\mathbb C^2$  such that either
$(A_2^*-\mu A_1^*)\beta=0$ or $(A_1^*-\nu A_2^*)\beta=0$ for some
$\mu, \nu \in \mathbb C.$ The set
$$\mathcal B:=\{ \beta: \|\beta\|_2=1,\,(A_2^*-\mu A_1^*)\beta = 0\, \mbox{\rm or}
\,(A_1^*-\nu A_2^*) \beta=0 \,\mbox{\rm for some } \mu,\, \nu\in
\mathbb C\}$$ of these vectors is non-empty. The proof of the
theorem involves two steps:

{\sf Claim 1:}  We show that there exists a $\lambda > 0,$ say
$\lambda_0,$ such that $(v, \lambda_0 v)$ is in $\mathcal E$ with
the property:

$g_{(v,\lambda_0v)}(\beta^{\prime\prime})>g_{(v,\lambda_0v)}(\beta^{\prime})>g_{(v,\lambda_0v)}(\beta)$
or
$g_{(v,\lambda_0v)}(\beta^{\prime})>g_{(v,\lambda_0v)}(\beta^{\prime\prime})>g_{(v,\lambda_0v)}(\beta)$
whenever $\beta^{\prime} ,\beta^{\prime\prime}\in \mathcal B.$

{\sf Claim 2:}   We then prove that there exists a $v$
($|v|<\frac{1}{\|A_1^*\|},$ this is necessary for contractivity),
say $v_0,$ such that
%$(v_0,\lambda_0 v_0)$ is in $\mathcal
%E_0:=\{(v,w):
$\inf_{\beta}g_{(v_0,\,\lambda_0 v_0)}(\beta)=0,$
%\},$
%From above
%discussion we conclude that there exists a $\lambda_0$ and $v_0$
%such that \mbox{$\inf_{\beta}g_{(v_0,\,\lambda_0 v_0)}(\beta)=0,$}
that is,
$$ \inf_{\beta}\{1-|v_0|^2\|A_{1}^*\beta\|^2 -|\lambda_0
v_0|^2\|A_{2}^*\beta\|^2+
\lambda_0^2|v_0|^4(\|A_{1}^*\beta\|^2\|A_{2}^*\beta\|^2 -
|\left\langle A_1A_{2}^*\beta, \beta\right\rangle |^2)\}=0.$$
Hence there exists a $\beta_0$ such that
$$
1-|v_0|^2\|A_{1}^*\beta_0\|^2 -|\lambda_0
v_0|^2\|A_{2}^*\beta_0\|^2+
\lambda_0^2|v_0|^4(\|A_{1}^*\beta_0\|^2\|A_{2}^*\beta_0\|^2 -
|\left\langle A_1A_{2}^*\beta_0, \beta_0\right\rangle |^2)=0$$
which is equivalent to $\|L^{(2)}_{(\mathbf v_1, \mathbf
v_2)}(P_{\mathbf A})\|>1.$ This completes the proof subject to the
verification of Claims $1$ and $2.$
\end{proof}
In the remaining part of this chapter, we will carry out this
verification on a case by case basis.
%To verify these steps in the
%above Theorem \ref{mainresult}, we need to consider
This involves four cases, namely, \mbox{$(i)\, b\neq |c|$ and
$|d_2|= 1,$} \mbox{$(ii)\,b= |c|$ and $|d_2|\neq 1,$}
\mbox{$(iii)\,b= |c|$ and $|d_2|= 1$} and \mbox{$(iv)\,b\neq |c|$
and $|d_2|\neq 1.$} For case $(iv)$ we have shown in Chapter $3$
that there exists a contractive linear map from $(\mathbb C^2,
\|\cdot\|_{\Omega_\mathbf A})$ to $\mathcal M_n(\mathbb C)$ which
is not completely contractive. We will prove Theorem
\ref{mainresult} for the remaining cases. The existence of a
$\lambda
> 0,$ say $\lambda_0,$ such that $(v, \lambda_0 v)$ is in
$\mathcal E$ with the property:

$g_{(v,\lambda_0v)}(\beta^{\prime\prime})>g_{(v,\lambda_0v)}(\beta^{\prime})>g_{(v,\lambda_0v)}(\beta)$
or
$g_{(v,\lambda_0v)}(\beta^{\prime})>g_{(v,\lambda_0v)}(\beta^{\prime\prime})>g_{(v,\lambda_0v)}(\beta)$
whenever $\beta^{\prime} ,\beta^{\prime\prime}\in \mathcal B$
follows from Theorems \ref{Thm:main1} and \ref{Thm:main2}.
%\begin{enumerate}
%\item [(i)] $b\neq |c|$ and $|d_2|\neq 1,$

%\item [(ii)] $b\neq |c|$ and $|d_2|= 1,$

%\item [(ii)] $b= |c|$ and $|d_2|= 1,$

%\item [(ii)] $b= |c|$ and $|d_2|\neq 1.$
%\end{enumerate}

\begin{theorem}\label{Thm:main1}
Let  $A_1$ be of the form $\left (
\begin{smallmatrix}
1 & 0   \\
0  & d_{2}
\end{smallmatrix}\right )$ or  $\left (
\begin{smallmatrix}
d_{1} & 0   \\
0  & 1
\end{smallmatrix}\right )$ and $A_2$ be of the form $\left ( \begin{smallmatrix}
 1 &  b \\
 c &  0
\end{smallmatrix}\right ) $ or $\left ( \begin{smallmatrix}
 0 &  b \\
 c &  1
\end{smallmatrix}\right ) $ and
 assume that they are not simultaneously
diagonalizable. Then there exists $(v, \,\lambda_0 v)$ in
$\mathcal{E}$ such that neither $g_{(v, \,\lambda_0
v)}(\beta^{'})$ nor $g_{(v, \,\lambda_0 v)}(\beta^{\prime
\prime})$ is equal to $\inf_{\beta}g_{(v,\,\lambda_0 v)}(\beta).$

\end{theorem}
\begin{proof}Suppose  $A_1=\left(\begin{smallmatrix}
1 & 0   \\
0  & d_2
\end{smallmatrix}\right),$
$A_2=\left(\begin{smallmatrix}
 1 & b \\
c &  0
\end{smallmatrix} \right).$ The homomorphism $\rho_{V}$ is
contractive, that is, $\|\rho_{V}\| \leq 1$ if and only if
$\|L_{(\mathbf v_1, \mathbf v_2)}\|_{(\mathbb
C^2,\|\,\cdot\,\|^{*} _{\Omega _{\mathbf A}})\rightarrow (\mathbb
C^2,\|\,\cdot\,\|_2)}\leq 1$ if and only if $|v|^2\leq
\frac{1}{\|A_1^{*}\|^2}$ and $(v, w)\in \mathcal E.$
%Now, {\small\begin{align}\label{eqa1}
%\inf_{\beta}g_{(v,w)}(\beta)\nonumber&=\inf_{\beta}\{1-|v|^2\|A_{1}^*\beta\|^2
%-|w|^2\|A_{2}^*\beta\|^2 +
%|v|^2|w|^2(\|A_{1}^*\beta\|^2\|A_{2}^*\beta\|^2 - |\left\langle
%A_{1}A_{2}^*\beta, \beta\right\rangle
%|^2)\}\\\nonumber&=\inf_{\beta}\{1-|v|^2-|w|^2(1+b^2)\}|\beta_{1}|^2+\{1-|v|^2|d_2|^2-
%|w|^2(|c|^2)\}|\beta_{2}|^2\\
%&-2|w|^{2}\Re
%c\beta_{1}\overline{\beta}_{2}+|w|^2|v|^2|b\bar{\beta_1}^2-
%d_2\bar{\beta_2}(\bar{\beta_1}+c\bar{\beta_2})|^2,
%\end{align}}  where
%$\|\beta\|_2=|\beta_1|^2+|\beta_2|^2=1.$ Pick $(v, w)$ such that
%$w=\lambda v, \lambda >0.$ Substituting $w=\lambda v$ in Equation
%(\ref{eqa1}),
Now, {\small\begin{eqnarray}\label{eqa2}
\nonumber\inf_{\beta}g_{(v, \,\lambda v)}(\beta
)&=&\inf_{\beta}\{1-|v|^2-\lambda^2|v|^2(1+b^2)\}|\beta_{1}|^2+\{1-|v|^2|d_2|^2-\lambda^2|v|^2|c|^2\}|\beta_{2}|^2
\\&& -2\lambda^2|v|^2\Re
 c\beta_{1}\bar{\beta}_{2}+\lambda^2|v|^4|b\bar{\beta_1}^2-
d_2\bar{\beta_2}(\bar{\beta_1}+c\bar{\beta_2})|^2 .
\end{eqnarray}}
%As earlier we have mentioned that
The proof of the theorem involves three distinct cases as
indicated above. For each case we will follow the following steps:

{\sf Step 1:} First we show there exists a $\lambda>0$ say
$\lambda_0,$ such that  either $g_{(v,\,\lambda_0
v)}(\beta^{\prime\prime})>g_{(v,\,\lambda_0 v)}(\beta^{\prime})$
or $g_{(v,\,\lambda_0 v)}(\beta^{\prime\prime})<g_{(v,\,\lambda_0
v)}(\beta^{\prime}).$

{\sf Step 2:}  If $g_{(v,\,\lambda_0
v)}(\beta^{\prime\prime})>g_{(v,\,\lambda_0 v)}(\beta^{\prime})$
(resp. $g_{(v,\,\lambda_0 v)}(\beta^{\prime})>g_{(v,\,\lambda_0
v)}(\beta^{\prime\prime})$), then we show that there exists a
$\beta$  such that $g_{(v,\,\lambda_0
v)}(\beta^{\prime})>g_{(v,\,\lambda_0 v)}(\beta)$ (resp.
$g_{(v,\,\lambda_0 v)}(\beta^{\prime\prime})>g_{(v,\,\lambda_0
v)}(\beta)$).

{\sf Case (i):} Here $b\neq |c|$ and $|d_2|= 1,$ that is,
$A_1=\left (
\begin{smallmatrix}
 1 &  0 \\
 0 &  \exp(i\theta)
\end{smallmatrix}\right )$ and $A_2=\left ( \begin{smallmatrix}
 1 &  b \\
 c &  0
\end{smallmatrix}\right )$ with $b \neq |c|.$ Let  $U=\left (
\begin{smallmatrix}
 1 &  0 \\
 0 &  \exp(-i\theta)
\end{smallmatrix}\right ).$ Then $U$ is a unitary and the pair $(A_1U, A_2U)$ determines the same set
$\Omega _{\mathbf A}.$ So, we may assume without loss of
generality that $\mathbf A$ is of the form  $(I_2, \left (
\begin{smallmatrix}
 1 &  b \\
 c &  0
\end{smallmatrix}\right ) )$  with $b, c\in \mathbb C, |b| \neq |c|.$ Let $W$ be a unitary such that
$WA_2W^*= \left ( \begin{smallmatrix}
 \alpha &  \gamma \\
 0 &  \delta
\end{smallmatrix}\right ) ,$ where $\alpha, \beta$ are the eigenvalue of $A_2$ with $|\alpha|^2\geq|\delta|^2.$
Therefore, without loss of generality we may also assume that
$A_1=I_2$ and $A_2=\left ( \begin{smallmatrix}
 \alpha &  \gamma \\
 0 &  \delta
\end{smallmatrix}\right ) $ with $|\alpha|^2\geq|\delta|^2.$
%The eigenvalues of $A_2$ are either equal or distinct.
%If $A_1=I_2$ and $A_2=\left (
%\begin{smallmatrix}
 %\alpha &  \gamma \\
% 0 &  \delta
%\end{smallmatrix}\right ),$
Then Equation (\ref{eqa2}) is equivalent to the condition $|v|^2
\leq 1$ and
{\small\begin{align}\label{eqa3}\inf_{\beta}g_{(v,\,\lambda
v)}(\beta)\nonumber&=\inf_{\beta}\{1-|v|^2-\lambda^2|v|^2(|\alpha|^2+|\gamma|^2)\}
|\beta_{1}|^2+\{1-|v|^2-\lambda^2|v|^2|\delta|^2\}|\beta_{2}|^2\\
&-2\lambda^2|v|^2\Re
\bar{\gamma}\beta_{1}\bar{\beta}_{2}\delta+\lambda^2|v|^4|\bar{\beta_1}(\gamma\bar{\beta_1}+\delta\bar{\beta_2})-
 \bar{\beta_2}\alpha\bar{\beta_1}|^2.
\end{align}}
The roots of $\det( A_{2}^*-\mu A_{1}^*)=0$ are
$\mu_1=\bar{\alpha}, \mu_2=\bar{\delta}.$ The vectors
$\beta^{\prime}, \beta^{\prime \prime}$ satisfying $(
A_{2}^*-\mu_1 A_{1}^*)\beta^{\prime}=0$ and $( A_{2}^*-\mu_2
A_{1}^*)\beta^{\prime \prime}=0$ are
$$\beta^{\prime}=\Big(\frac{|\delta-\alpha|\exp({i\theta})}{\sqrt{|\delta-\alpha|^2+|\gamma|^2}},
\frac{-\bar{\gamma}\exp{i(\theta-\phi)}}{\sqrt{|\delta-\alpha|^2+|\gamma|^2}}\Big),$$
$\beta^{\prime \prime}=(0, \exp({i\psi}))$ respectively, where
$\bar{\delta}-\bar{\alpha}=|\delta-\alpha|\exp({i\phi}).$ From
Equation (\ref{eqa3}) it is easy to see that
$$g_{(v,\,\lambda v)}(\beta^{\prime})=\{1-|v|^2-\lambda^2|v|^2(|\alpha|^2+|\gamma|^2)\}
|\beta_{1}^{\prime}|^2+\{1-|v|^2-\lambda^2|v|^2|\delta|^2\}|\beta_{2}^{'}|^2\\
-2\lambda^2|v|^2\Re
\bar{\gamma}\beta_{1}^{\prime}\bar{\beta}_{2}^{'}\delta$$ and
$g_{(v,\,\lambda v)}(\beta^{\prime
\prime})=1-|v|^2-\lambda^2|v|^2|\delta|^2,$ where
$\beta^{\prime}=(\beta_{1}^{\prime}, \beta_{2}^{\prime}).$ Note
that
\begin{align}\label{beta}
g_{(v,\,\lambda v)}(\beta^{\prime\prime})-g_{(v,\,\lambda
v)}(\beta^{\prime})\nonumber&=\lambda^2|v|^2(|\alpha|^2+|\gamma|^2-
|\delta|^2)|\beta_{1}^{\prime}|^2+ 2\lambda^2|v|^2\Re
\bar{\gamma}\beta_{1}^{\prime}\bar{\beta}_{2}^{\prime}\delta
\\&=\lambda^2|v|^2(|\alpha|^2-|\delta|^2).
\end{align}
\begin{enumerate}
\item[(a)] We assume that $|\alpha|^2>|\delta|^2.$ Since
$|\alpha|^2>|\delta|^2,$ from Equation (\ref{beta}) we have
$g_{(v,\,\lambda v)}(\beta^{\prime\prime})>g_{(v,\,\lambda
v)}(\beta^{\prime}).$ Hence we conclude that
\mbox{$\inf_{\beta}g_{(v,\,\lambda v)}(\beta)\neq g_{(v,\,\lambda
v)}(\beta^{\prime \prime}).$}
%infimum does not attain at
%$\beta^{\prime \prime}.$

%Let $\beta=(1,0).$ Then $g_{(v,\,\lambda
%v)}((1,0))=\{1-|v|^2-\lambda^2|v|^2(|\alpha|^2+|\gamma|^2)\}+\lambda^2|v|^4|\gamma|^2.$
In order to prove {\sf Step 2}, it is sufficient to observe that
{\small\begin{align}\label{eqa4} g_{(v,\,\lambda
v)}(\beta^{\prime})-g_{(v,\,\lambda v)}((1,0))
\nonumber&=\lambda^2|v|^2(|\alpha|^2+|\gamma|^2-|\delta|^2)|\beta_{2}^{\prime}|^2-
2\lambda^2|v|^2\Re
\bar{\gamma}\beta_{1}^{\prime}\bar{\beta}_{2}^{\prime}\delta-\lambda^2|v|^4|\gamma|^2
\\\nonumber&=\lambda^2|v|^2(|\alpha|^2+|\gamma|^2-|\delta|^2)\frac{|\gamma|^2}{|\delta-\alpha|^2+|\gamma|^2}
\\\nonumber&+\frac{2\lambda^2|v|^2|\gamma|^2\Re\delta(\bar{\delta}-\bar{\alpha})}{|\delta-\alpha|^2+|\gamma|^2}
-\lambda^2|v|^4|\gamma|^2\\&=\lambda^2|v|^2|\gamma|^2(1-|v|^2).
\end{align}}
The Equation (\ref{eqa4}) shows that for all $|v|\in[0, 1),$ for
all $\lambda,$ there exists a \mbox{$\beta=(1,0)$} such that
$g_{(v,\,\lambda v)}(\beta^{\prime})>g_{(v,\,\lambda v)}((1,0)).$
We therefore conclude that neither $g_{(v,\,\lambda_0
v)}(\beta^{\prime})$ nor $g_{(v,\,\lambda_0 v)}(\beta^{\prime
\prime})$ is equal to $\inf_{\beta}g_{(v,\,\lambda_0 v)}(\beta)$
for any $v$ with $ |v|< 1.$

\item[(b)]
%We consider the case $\alpha=\delta,$ that is,
%$\bar{b}=\frac{-1}{4\bar{c}}.$ Then  $A_2$  is of the form $\left
%( \begin{smallmatrix}
% \frac{1}{2} &  \gamma \\
 %0 &  \frac{1}{2}
%\end{smallmatrix}\right ).$
Suppose $|\alpha|^2=|\delta|^2.$ Then from Equation (\ref{beta})
%We can easily see that $\beta^{\prime}=\beta^{\prime \prime}=(0,
%1).$ Hence from Equation (\ref{eqa3})
we have $g_{(v,\,\lambda v)}(\beta^{\prime
\prime})=g_{(v,\,\lambda v)}(\beta^{\prime }).$ From Equation
(\ref{eqa4}) we see that $g_{(v,\,\lambda
v)}(\beta^{\prime})-g_{(v,\,\lambda v)}((1, 0)
=\lambda^2|v|^2|\gamma|^2(1-|v|^2).$
%\end{align}}
%The Equation (\ref{eqa44}) shows that
Therefore it follows that for all $|v|\in[0, 1)$ and for all
$\lambda,$ there exists a $\beta=(1, 0)$ such that
$g_{(v,\,\lambda v)}(\beta^{\prime})>g_{(v,\,\lambda v)}((1, 0)).$
Hence we see that neither $g_{(v,\,\lambda_0 v)}(\beta^{\prime})$
nor $g_{(v,\,\lambda_0 v)}(\beta^{\prime \prime})$ is equal to
$\inf_{\beta}g_{(v,\, \lambda_0 v)}(\beta)$  for any $v$ with $
|v|< 1.$
\end{enumerate}
%Suppose $\det(\nu A_{2}^*-A_{1}^*)=0.$ Then we see that
The roots of $\det(\nu A_{2}^*-A_{1}^*)=0$ are
$\nu_1=\frac{1}{\bar{\alpha}}, \nu_2=\frac{1}{\bar{\beta}}.$
%\mbox{$\beta^{\prime}=(\frac{\overline{\delta-\alpha}}{\sqrt{|\delta-\alpha|^2+|\gamma|^2}},
%\frac{-\bar{\gamma}}{\sqrt{|\delta-\alpha|^2+|\gamma|^2}}),$}
%$\beta^{\prime \prime}=(0, 1).$
The vectors $\beta^{\prime}, \beta^{\prime \prime}$ satisfying $(
\nu_1A_{2}^*- A_{1}^*)\beta^{\prime}=0$ and $(\nu_2 A_{2}^*-
A_{1}^*)\beta^{\prime \prime}=0$ are
$$\beta^{\prime}=\Big(\frac{|\delta-\alpha|\exp({i\theta})}{\sqrt{|\delta-\alpha|^2+|\gamma|^2}},
\frac{-\bar{\gamma}\exp{i(\theta-\phi)}}{\sqrt{|\delta-\alpha|^2+|\gamma|^2}}\Big),$$
$\beta^{\prime \prime}=(0, \exp({i\psi}))$ respectively, where
$\bar{\delta}-\bar{\alpha}=|\delta-\alpha|\exp({i\phi}).$
Proceeding the same way, as above, we also find that neither
$g_{(v, \,\lambda_0 v)}(\beta^{\prime})$ nor $g_{(v, \,\lambda_0
v)}(\beta^{\prime \prime})$ is equal to $\inf_{\beta}g_{(v,
\,\lambda_0 v)}(\beta)$
 for any  $v$ with $ |v|< 1.$
%1].$ %If $\det(\nu A_{2}^*-A_{1}^*)=0,$ then we can easily see that
%$\nu_1= 2=\nu_2$ are the roots of $\det(\nu A_{2}^*-A_{1}^*)=0$
%and the corresponding vectors are $\beta^{\prime}= \beta^{\prime
%\prime}=(0, 1).$ As before we can prove that neither
%$g_{(v,\,\lambda_0 v)}(\beta^{\prime})$ nor $g_{(v, \,\lambda_0
%v)}(\beta^{\prime \prime})$ is equal to $\inf_{\beta}g_{(v,\,
%\lambda_0 v)}(\beta)$ for all $|v|\in[0, 1].$

%Similarly if either $A_1=\left (
%\begin{smallmatrix}
%d_{1} & 0   \\
%0  & 1
%\end{smallmatrix}\right ),\left (
%\begin{smallmatrix}
%1 & 0   \\
%0  & d_{2}
%\end{smallmatrix}\right )$ or $ A_2=\left ( \begin{smallmatrix}
% 0 &  b \\
 %c &  1
%\end{smallmatrix}\right ) $ with either $b= |c|$ and $|d_1|= 1$ or $b= |c|$ and $|d_2|= 1$ and $A_1=\left (
%\begin{smallmatrix}
%d_{1} & 0   \\
%0  &  1
%\end{smallmatrix}\right ),  A_2=\left ( \begin{smallmatrix}
% 1 &  b \\
% c &  0
%\end{smallmatrix}\right ) $ with $b= |c|$ and $|d_1|= 1,$ then we conclude that
%neither $g_{(v,\,\lambda_0 v)}(\beta^{\prime})$ nor
%$g_{(v,\,\lambda_0 v)}(\beta^{\prime \prime})$ is equal to
%$\inf_{\beta}g_{(v,\, \lambda_0 v)}(\beta)$  for any  $v$ with $
%|v|< 1.$

{\sf Case (ii):} In this case, $A_1=\left(\begin{smallmatrix}
1 & 0   \\
0  & d_2
\end{smallmatrix}\right),$
$A_2=\left(\begin{smallmatrix}
 1 & |c| \\
c &  0
\end{smallmatrix} \right)$ with $|d_2|\neq 1.$ The roots of
$\det(A_{2}^*-\mu A_{1}^*)=0$ are
$$\mu_1=\frac{\sqrt{\bar{d}_2}+\sqrt{\bar{d}_2+4|c|\bar{c}}}
{2\sqrt{\bar{d}_2}},
\mu_2=\frac{\sqrt{\bar{d}_2}-\sqrt{\bar{d}_2+4|c|\bar{c}}}
{2\sqrt{\bar{d}_2}}.$$  The vectors $\beta^{\prime}, \beta^{\prime
\prime}$ satisfying $( A_{2}^*-\mu_1 A_{1}^*)\beta^{\prime}=0$ and
$( A_{2}^*-\mu_2 A_{1}^*)\beta^{\prime \prime}=0$ are
$$\beta^{\prime}=\Big(\frac{|c|\exp({i\theta_1})}{\sqrt{|\mu_2|^2+|c|^2}},
\frac{-\mu_2\exp
i(\theta_1-\phi_1)}{\sqrt{|\mu_2|^2+|c|^2}}\Big)$$ and
$$\beta^{\prime \prime}=\Big(\frac{|c|\exp({i\theta_2})}{\sqrt{|\mu_1|^2+|c|^2}},
\frac{-\mu_1\exp{i(\theta_2-\phi_1)}}{\sqrt{|\mu_1|^2+|c|^2}}\Big)$$
respectively, where $\bar{c}=|c|\exp({i\phi_1}).$ Substituting
$\beta=\beta^{\prime}$ and $\beta=\beta^{\prime \prime}$ in
Equation (\ref{eqa2}) we have {\small\begin{eqnarray}
g_{(v,\,\lambda
v)}(\beta^{\prime})\nonumber&=&\{1-|v|^2-\lambda^2|v|^2(1+|c|^2)\}|\beta^{\prime}_{1}|^2
+\{1-|v|^2|d_2|^2-\lambda^2|v|^2|c|^2\}
|\beta^{\prime}_{2}|^2-2\lambda^2|v|^2\Re
c\beta^{\prime}_{1}\bar{\beta}^{\prime}_{2}
\end{eqnarray}} and
{\small\begin{eqnarray} g_{(v,\,\lambda v)}(\beta^{\prime
\prime})\nonumber&=&\{1-|v|^2-\lambda^2|v|^2(1+|c|^2)\}|\beta^{\prime
\prime}_{1}|^2+\{1-|v|^2|d_2|^2- \lambda^2|v|^2|c|^2\} |\beta^{
\prime \prime}_{2}|^2-2\lambda^2|v|^2\Re c\beta^{\prime
\prime}_{1}\bar{\beta}^{\prime \prime}_{2},
\end{eqnarray}} where $\beta^{\prime}=\left(\beta^{\prime}_{1}, \beta^{\prime}_{2}\right)$ and
$\beta^{\prime \prime}=\left(\beta^{\prime \prime}_{1},
\beta^{\prime \prime}_{2}\right).$ Now,
\begin{eqnarray}\label{eqa5}\nonumber\lefteqn{
g_{(v,\,\lambda v)}(\beta^{\prime})-g_{(v,\,\lambda
v)}(\beta^{\prime \prime})
%\nonumber&=&\{1-|v|^2-\lambda^2|v|^2(1+|c|^2)\}(|\beta^{\prime}_{1}|^2-|\beta^{\prime
%\prime}_{1}|^2)\\\nonumber
%&&+\{1-|v|^2|d_2|^2-\lambda^2|v|^2|c|^2\}(|\beta^{\prime}_{2}|^2-|\beta^{\prime
%\prime}_{2}|^2)
%+2\lambda^2|v|^2\Re c(\beta^{\prime \prime}_{1}\bar{\beta}^{\prime \prime}_{2}-\beta^{\prime}_{1}\bar{\beta}^{\prime}_{2})
%\nonumber&\phantom{Avijit Pal Avijit pal}=&
=\{1-|v|^2-\lambda^2|v|^2(1+|c|^2)\}(|\beta^{\prime}_{1}|^2-|\beta^{\prime \prime}_{1}|^2)}\\
&\phantom{Avijit Pal Avijit pal
}+&\{1-|v|^2|d_2|^2-\lambda^2|v|^2|c|^2\}(|\beta^{\prime}_{2}|^2-|\beta^{\prime
\prime}_{2}|^2) -2\lambda^2|v|^2 r,
\end{eqnarray} where
$r=\Re c(\beta^{\prime }_{1}\bar{\beta}^{\prime
}_{2}-\beta^{\prime \prime}_{1}\bar{\beta}^{\prime \prime}_{2}),$
depends only on $A_1, A_2.$

%If $|\beta^{'}_{1}|= 0,$  then we have $c=0,$ that is, $A_1=\left
%(
%\begin{smallmatrix}
% 1 &  0 \\
% 0 &  \exp(i\theta)
%\end{smallmatrix}\right )$ and $A_2=\left ( \begin{smallmatrix}
% 1 &  0 \\
% 0 &  0
%\end{smallmatrix}\right ).$ This contradicts the fact that $A_1$ and $A_2$ are not simultaneously
%diagonalizable. Therefore, we have $|\beta^{'}_{1}| \neq 0.$
% As we have assumed that
Since $A_1$ and $A_2$ are not simultaneously diagonalizable by
hypothesis, it follows that $c\neq 0,$ or equivalently,
$\beta^{\prime}_{1} \neq 0.$ Similarly we show that $\beta^{\prime
\prime}_{2} \neq 0.$
%Hence there are three possibilities, namely,
%either $|\beta^{\prime}_{1}|^2<|\beta^{\prime \prime}_{1}|^2,$
%$|\beta^{\prime}_{1}|^2>|\beta^{\prime \prime}_{1}|^2 $ or
%$|\beta^{\prime}_{1}|^2=|\beta^{\prime \prime}_{1}|^2 .$
Without loss of generality we can assume that
$|\beta^{\prime}_{1}|^2\geq|\beta^{\prime \prime}_{1}|^2.$
%{\sf \bf{$1<|d_2|,\,-1 + 2\Re \bar{\mu}_{1}>0,\,r>0$}:}
This splits into two cases, namely, $(a)
|\beta^{\prime}_{1}|^2>|\beta^{\prime \prime}_{1}|^2,$ which is
equivalent to $r\neq 0.$ and $(b)
|\beta^{\prime}_{1}|^2=|\beta^{\prime \prime}_{1}|^2$ which is
equivalent to $r=0.$ We now consider these two cases separately.
\begin{enumerate}

\item[(a)]Suppose $|\beta^{\prime}_{1}|^2>|\beta^{\prime
\prime}_{1}|^2$ which is equivalent to $r\neq 0.$ Since
$|\beta^{\prime}_{1}|^2>|\beta^{\prime \prime}_{1}|^2$ we have
$|\beta^{\prime}_{1}|^2-|\beta^{\prime \prime}_{1}|^2=\delta_1>0.
$
%Since
%$|\beta^{\prime}_{1}|^2>|\beta^{\prime \prime}_{1}|^2 ,$ therefore
%we have $(|\beta^{\prime}_{1}|^2-|\beta^{\prime
%\prime}_{1}|^2)=\delta_1,$ where $\delta _1>0.$
Also, from above relation it follows that
$(|\beta^{\prime}_{2}|^2-|\beta^{\prime
\prime}_{2}|^2)=-\delta_1.$ Substituting
$(|\beta^{\prime}_{1}|^2-|\beta^{\prime \prime}_{1}|^2)=\delta_1$
in Equation (\ref{eqa5}) we have
{\small\begin{eqnarray}\label{eqa6} g_{(v,\,\lambda
v)}(\beta^{\prime})-g_{(v,\,\lambda v)}(\beta^{\prime
\prime})&=&\{|v|^2(|d_2|^2-1)-\lambda^2|v|^2)\}\delta
_1-2\lambda^2|v|^2r.
\end{eqnarray}}
Now, we have several possibilities which are listed below.
\begin{itemize}
\item {\sf \bf{$1<|d_2|,\,-1 + 2\Re \bar{\mu}_{1}>0,\,r>0$}:}

%Suppose  $1<|d_2|$ and $-1 + 2\Re \bar{\mu}_{1}>0.$ Then
From Equation (\ref{eqa6}) we observe that $g_{(v ,\,\lambda
v)}(\beta^{\prime})>g_{(v,\,\lambda v)}(\beta^{\prime \prime})$ if
$\lambda^2<\frac{(|d_2|^2-1)\delta_1}{(\delta_1+2r)}.$ Hence
\mbox{$\inf_{\beta}g_{(v,\,\lambda v)}(\beta)\neq g_{(v,\,\lambda
v)}(\beta^{\prime })$} for all $\lambda,$
$\lambda^2<\frac{(|d_2|^2-1)\delta_1}{(\delta_1+2r)}.$

Evaluating $g_{(v ,\,\lambda v)}$ at $(0,1),$ we have
$$g_{(v ,\,\lambda
v)}((0,1))=\{1-|v|^2|d_2|^2-\lambda^2|v|^2|c|^2\}+\lambda^2|v|^4|cd_2|^2$$
which gives
%Now, {\small\begin{eqnarray}\label{eqa7} \nonumber&&
%g_{\lambda}(\beta^{''}, v)-g_{\lambda}(\beta,
%v)\\\nonumber&=&\{1-|v|^2-\lambda^2|v|^2(1+|c|^2)\}(|\beta^{''}_{1}|^2-|\beta_1|^2)
%+\{1-|v|^2|d_2|^2-\lambda^2|v|^2|c|^2\}(|\beta^{''}_{2}|^2-|\beta_2|^2)\\
%&&+2\lambda^2|v|^2\Re c
%(\beta_1\bar{\beta}_{2}-\beta^{''}_{1}\bar{\beta}^{''}_{2})
%-\lambda^2|v_1|^4||c|\bar{\beta_1}^2-d_2\bar{\beta_2}(\bar{\beta_1}+c\bar{\beta_2})|^2.
%\end{eqnarray}}
%Then from Equation (\ref{eqa7}) we have
{\small\begin{eqnarray}\label{eqa8}\nonumber\lefteqn{g_{(v,\,\lambda
v)}(\beta^{\prime \prime})-g_{(v,\,\lambda v)}((0,1))}
%\\\nonumber&
%=&\{|v|^2(|d_2|^2-1)-\lambda^2|v|^2\}\frac{|c|^2}{(|\mu_1|^2+|c|^2)}
%+\frac{2\lambda^2|v|^2|c|^2}{(|\mu_1|^2+|c|^2)}\Re \bar{\mu}_{1}
%-\lambda^2|v|^4|cd_2|^2
\\&
=&\frac{|c|^2|v|^2}{(|\mu_1|^2+|c|^2)}\{(|d_2|^2-1) -\lambda^2 +
2\lambda^2\Re \bar{\mu}_{1}
-\lambda^2|v|^2|d_2|^2(|\mu_1|^2+|c|^2)\}.
\end{eqnarray}}
Since $|d_2|>1,$ from Equation (\ref{eqa8}), we have
$g_{(v,\,\lambda v)}(\beta^{\prime \prime})>g_{(v,\, \lambda
v)}((0,1))$ for all $\lambda$ and for all $v,$ $|v|^2\leq \frac{-1
+ 2\Re \bar{\mu}_{1}}{|d_2|^2(|\mu_1|^2+|c|^2)}.$

Also, from Equation (\ref{eqa8}), we obtain $g_{(v,\,\lambda
v)}(\beta^{\prime \prime})>g_{(v,\, \lambda v)}((0,1))$ for all
$\lambda,$
$$\lambda ^2<\frac{(|d_2|^2-1)}{1-2\Re \bar{\mu}_{1}
+|v|^2|d_2|^2(|\mu_1|^2+|c|^2)}$$ and for all $v,$ $|v|^2>
\frac{-1 + 2\Re \bar{\mu}_{1}}{|d_2|^2(|\mu_1|^2+|c|^2)}.$ Thus we
have $g_{(v,\,\lambda v)}(\beta^{\prime \prime})>g_{(v,\, \lambda
v)}((0,1))$ for all $\lambda,$ $$\lambda^2
<\frac{(|d_2|^2-1)}{1-2\Re \bar{\mu}_{1}
+|v|^2|d_2|^2(|\mu_1|^2+|c|^2)})$$

%for all
%$$$\lambda ^2<\frac{(|d_2|^2-1)}{1-\Re \bar{\mu}_{1}
%+|v|^2|d_2|^2(|\mu_1|^2+|c|^2)}$$ and
%for all $|v|$ in particular,
%$g_{(v,\,\lambda
%v)}(\beta^{\prime \prime})>g_{(v,\, \lambda v)}((0,1))$ for all
%$$\lambda ^2<\frac{(|d_2|^2-1)}{1-\Re \bar{\mu}_{1}
%+|v|^2|d_2|^2(|\mu_1|^2+|c|^2)}$$ and
for all $v,v>0.$
%with $|v|\in (0, \frac{1}{\|A_1^*\|}].$

We therefore conclude that neither $g_{(v,\,\lambda
v)}(\beta^{\prime})$ nor $g_{(v,\,\lambda v)}(\beta^{\prime
\prime})$ is equal to $\inf_{\beta}g_{(v,\, \lambda v)}(\beta)$
for any $v$ with $|v|$ in $(0, \frac{1}{\|A_1^*\|}]$ and for any
$\lambda$ with
$$\lambda ^2<\min\Big\{\frac{(|d_2|^2-1)}{1-2\Re \bar{\mu}_{1}
+|v|^2|d_2|^2(|\mu_1|^2+|c|^2)},
\frac{(|d_2|^2-1)\delta_1}{(\delta_1+2r)}\Big\}.$$

\item {\sf \bf{$1<|d_2|,\,-1 + 2\Re \bar{\mu}_{1}<0,\,r>0$}:}

%Suppose  $1<|d_2|$ and $-1 + 2\Re \bar{\mu}_{1}<0.$ Then
From Equation (\ref{eqa6}) we observe that $g_{(v ,\,\lambda
v)}(\beta^{\prime})>g_{(v,\,\lambda v)}(\beta^{\prime \prime})$ if
$\lambda^2<\frac{(|d_2|^2-1)\delta_1}{(\delta_1+2r)}.$ Hence
\mbox{$\inf_{\beta}g_{(v,\,\lambda v)}(\beta)\neq g_{(v,\,\lambda
v)}(\beta^{\prime })$} for all $\lambda,$
$\lambda^2<\frac{(|d_2|^2-1)\delta_1}{(\delta_1+2r)}.$

Since $|d_2|>1,$ from Equation (\ref{eqa8}), we have $g_{(v,\,
\lambda v)}(\beta^{\prime \prime})>g_{(v, \, \lambda v)}((0,1))$
for all $\lambda,$
$$\lambda ^2<\frac{(|d_2|^2-1)}{1-2\Re \bar{\mu}_{1}
+(|\mu_1|^2+|c|^2)}$$ and for all $v,$  $|v|$ in $(0,
\frac{1}{\|A_1^*\|}].$

Thus we conclude that neither $g_{(v,\,\lambda
v)}(\beta^{\prime})$ nor $g_{(v,\,\lambda v)}(\beta^{\prime
\prime})$ is equal to $\inf_{\beta}g_{(v,\, \lambda v)}(\beta)$
for any $v$ with $|v|$ in $(0, \frac{1}{\|A_1^*\|}]$ and for any
$\lambda$ with
$$\lambda ^2<\min\Big\{\frac{(|d_2|^2-1)}{1-2\Re \bar{\mu}_{1}
+(|\mu_1|^2+|c|^2)},
\frac{(|d_2|^2-1)\delta_1}{(\delta_1+2r)}\Big\}.$$

%\end{itemize}

%\item[(b)] Suppose $r$ is negative  and
%$|\beta^{\prime}_{1}|^2>|\beta^{\prime \prime}_{1}|^2. $
%Then from
%Equation (\ref{eqa6}) we have {\small\begin{eqnarray}\label{eqa11}
%g_{(v,\,\lambda v)}(\beta^{\prime})-g_{(v,\,\lambda
%v)}(\beta^{\prime
%\prime})&=&\{|v|^2(|d_2|^2-1)-\lambda^2|v|^2)\}\delta
%_1-2\lambda^2|v|^2r.
%\end{eqnarray}}
%\begin{itemize}
\item
{\sf\bf{$|d_2|>1,\,1-2\Re\bar{\mu}_1>0,\,r<0,\,2r+\delta_1<0$}:}

%Suppose  $1<|d_2|,\,1-2\Re\bar{\mu}_1>0$ and $2r+\delta_1$ is
%negative. Then
From Equation (\ref{eqa6}) we have $g_{(v,\,\lambda
v)}(\beta^{\prime \prime}) <g_{(v, \, \lambda v)}(\beta^{\prime})$
for all $\lambda.$ Hence we have
\mbox{$\inf_{\beta}g_{(v,\,\lambda v)}(\beta)\neq g_{(v,\,\lambda
v)}(\beta^{\prime })$} for all $\lambda.$

Since $|d_2|>1,$ from Equation (\ref{eqa8}) we have
$g_{(v,\,\lambda v)}(\beta^{\prime \prime})>g_{(v,\, \lambda
v)}((0,1))$ for all $\lambda$ and for all $v,$ $|v|^2\leq \frac{-1
+ \Re \bar{\mu}_{1}}{|d_2|^2(|\mu_1|^2+|c|^2)}.$

Also, from Equation (\ref{eqa8}) we obtain $g_{(v,\, \lambda
v)}(\beta^{\prime \prime})>g_{(v, \, \lambda v)}((0,1))$ for all
$\lambda,$
$$\lambda ^2<\frac{(|d_2|^2-1)}{1-2\Re \bar{\mu}_{1}
+|v|^2|d_2|^2(|\mu_1|^2+|c|^2)}$$ and for all $v,$ $|v|^2>
\frac{-1 + \Re \bar{\mu}_{1}}{|d_2|^2(|\mu_1|^2+|c|^2)}.$ Thus we
have $g_{(v,\,\lambda v)}(\beta^{\prime \prime})>g_{(v,\, \lambda
v)}((0,1))$ for all $\lambda,$  $$\lambda^2
<\frac{(|d_2|^2-1)}{1-2\Re \bar{\mu}_{1}
+|v|^2|d_2|^2(|\mu_1|^2+|c|^2)}$$ and for all $v,v>0.$

We therefore conclude that neither $g_{(v,\,\lambda
v)}(\beta^{\prime})$  nor $g_{(v, \,\lambda v)}(\beta^{\prime
\prime})$ is equal to $\inf_{\beta}g_{(v,\, \lambda v)}(\beta)$
for any  $v$ with  $|v|$ in $(0, \frac{1}{\|A_1^*\|}]$ and for any
$\lambda$ with $$\lambda ^2<\frac{(|d_2|^2-1)}{1-2\Re
\bar{\mu}_{1} +|d_2|^2|v|^2(|\mu_1|^2+|c|^2)},$$

\item {\sf
\bf{$|d_2|>1,\,1-2\Re\bar{\mu}_1<0,\,r<0,\,2r+\delta_1<0$}:}

%Suppose  $1<|d_2|,\,1-2\Re\bar{\mu}_1<0$ and $2r+\delta_1$ is
%negative. Then
From Equation (\ref{eqa6}) we have $g_{(v,\,\lambda
v)}(\beta^{\prime \prime}) <g_{(v, \, \lambda v)}(\beta^{\prime})$
for all $\lambda.$ Hence we conclude that
\mbox{$\inf_{\beta}g_{(v,\,\lambda v)}(\beta)\neq g_{(v,\,\lambda
v)}(\beta^{\prime })$} for all $\lambda.$

%Since $|d_2|>1,$ from Equation (\ref{eqa8}) it is easy to see that
%$g_{(v,\,\lambda v)}(\beta^{\prime \prime})>g_{(v,\, \lambda
%v)}((0,1))$ for all $\lambda$ and for all $|v|^2\leq \frac{-1 +
%\Re \bar{\mu}_{1}}{|d_2|^2(|\mu_1|^2+|c|^2)}.$

Since $|d_2|>1,$ from Equation (\ref{eqa8}), we have $g_{(v,\,
\lambda v)}(\beta^{\prime \prime})>g_{(v, \, \lambda v)}((0,1))$
for all $\lambda$ with
$$\lambda ^2<\frac{(|d_2|^2-1)}{1-2\Re \bar{\mu}_{1}
+(|\mu_1|^2+|c|^2)}$$ and for all $v$ with $|v|$ in $(0,
\frac{1}{\|A^*\|}].$

%$|v|^2> \frac{-1 + \Re \bar{\mu}_{1}}{|d_2|^2(|\mu_1|^2+|c|^2)}.$
%Thus if we choose $\lambda^2 \in (0,\frac{(|d_2|^2-1)}{1-\Re
%\bar{\mu}_{1} +|v|^2|d_2|^2(|\mu_1|^2+|c|^2)}),$ then one can
%easily see that $g_{(v,\,\lambda v)}(\beta^{\prime
%\prime})>g_{(v,\, \lambda v)}((0,1))$ for all $|v|\in \mathbb
%R^{+},$ in particular, for all $|v|\in (0, \frac{1}{\|A^*\|}].$

Hence if $\lambda$ is chosen with $$\lambda
^2<\frac{(|d_2|^2-1)}{1-2\Re \bar{\mu}_{1} +(|\mu_1|^2+|c|^2)},$$
then it follows that neither $g_{(v,\,\lambda v)}(\beta^{\prime})$
nor $g_{(v, \,\lambda v)}(\beta^{\prime \prime})$ is equal to
$\inf_{\beta}g_{(v,\, \lambda v)}(\beta)$  for any  $v$ with $|v|$
in $(0, \frac{1}{\|A_1^*\|}].$

\item {\sf
\bf{$|d_2|>1,\,1-2\Re\bar{\mu}_1>0,\,r<0,\,2r+\delta_1>0$}:}

%Suppose $1<|d_2|,\,1-2\Re\bar{\mu}_1>0$ and $2r+\delta_1$ is
%positive. Then
From Equation (\ref{eqa6}), we have $g_{(v,\,\lambda
v)}(\beta^{\prime \prime}) <g_{(v,\,\lambda v)}(\beta^{\prime})$
if $\lambda^2<\frac{(|d_2|^2-1)\delta_1}{(\delta_1+2r)}.$ Hence we
have \mbox{$\inf_{\beta}g_{(v,\,\lambda v)}(\beta)\neq
g_{(v,\,\lambda v)}(\beta^{\prime })$}  for all $\lambda,$
$\lambda^2<\frac{(|d_2|^2-1)\delta_1}{(\delta_1+2r)}.$

From Equation (\ref{eqa8}) we have $g_{(v,\,\lambda
v)}(\beta^{\prime \prime})>g_{(v,\,\lambda v)}((0,1))$ for all
$\lambda$ and for all $v,$ $|v|^2\leq \frac{-1 + \Re
\bar{\mu}_{1}}{|d_2|^2(|\mu_1|^2+|c|^2)}.$

Also, from Equation (\ref{eqa8}) we obtain $g_{(v,\, \lambda
v)}(\beta^{\prime \prime})>g_{(v,\,\lambda v)}((0,1))$ for all
$\lambda,$
$$\lambda ^2<\frac{(|d_2|^2-1)}{1-2\Re \bar{\mu}_{1}
+|v|^2|d_2|^2(|\mu_1|^2+|c|^2)}$$ and for all $v,$ $|v|^2>\frac{-1
+ \Re \bar{\mu}_{1}}{|d_2|^2(|\mu_1|^2+|c|^2)}.$

Thus we have $g_{(v,\, \lambda v)}(\beta^{\prime
\prime})>g_{(v,\,\lambda v)}((0,1))$ for all $\lambda,$
$$\lambda ^2<\frac{(|d_2|^2-1)}{1-2\Re \bar{\mu}_{1}
+|v|^2|d_2|^2(|\mu_1|^2+|c|^2)}$$ and for all $v, v>0.$

If $\lambda$ is chosen with
$$\lambda ^2<\min\Big\{\frac{(|d_2|^2-1)}{1-2\Re \bar{\mu}_{1}
+|v|^2|d_2|^2(|\mu_1|^2+|c|^2)},
\frac{(|d_2|^2-1)\delta_1}{(\delta_1+2r)}\Big\},$$ then neither
$g_{(v,\,\lambda v)}(\beta^{\prime})$  nor $g_{(v,\,\lambda
v)}(\beta^{\prime \prime})$ is equal to
$\inf_{\beta}g_{(v,\,\lambda v)}(\beta)$ for any $v$ with $|v|$ in
$(0, \frac{1}{\|A_1^*\|}].$

\item {\sf
\bf{$|d_2|>1,\,1-2\Re\bar{\mu}_1<0,\,r<0,\,2r+\delta_1>0$}:}

%Suppose $1<|d_2|,\,1-2\Re\bar{\mu}_1<0$ and $2r+\delta_1$ is
%positive. Then
From Equation (\ref{eqa6}), we have $g_{(v,\,\lambda
v)}(\beta^{\prime \prime}) <g_{(v,\,\lambda v)}(\beta^{\prime})$
if $\lambda^2<\frac{(|d_2|^2-1)\delta_1}{(\delta_1+2r)}.$ Hence we
have \mbox{$\inf_{\beta}g_{(v,\,\lambda v)}(\beta)\neq
g_{(v,\,\lambda v)}(\beta^{\prime })$} for all $\lambda,$
$\lambda^2<\frac{(|d_2|^2-1)\delta_1}{(\delta_1+2r)}.$

%From Equation (\ref{eqa8}) it is easy to see that $g_{(v,\,\lambda
%v)}(\beta^{\prime \prime})>g_{(v,\,\lambda v)}((0,1))$ for all
%$\lambda$ and for all $|v|^2\leq \frac{-1 + \Re
%\bar{\mu}_{1}}{|d_2|^2(|\mu_1|^2+|c|^2)}.$
%Also, from Equation (\ref{eqa8}) we have $g_{(v,\, \lambda
%v)}(\beta^{\prime \prime})>g_{(v,\,\lambda v)}((0,1))$ for all
%$$\lambda ^2<\frac{(|d_2|^2-1)}{1-\Re \bar{\mu}_{1}
%+|v|^2|d_2|^2(|\mu_1|^2+|c|^2)}$$ and for all $|v|^2>\frac{-1 +
%\Re
%\bar{\mu}_{1}(\beta^{'})}{|d_2|^2(|\mu_1(\beta^{'})|^2+|c|^2)}.$
From Equation (\ref{eqa8}), we also have $g_{(v,\, \lambda
v)}(\beta^{\prime \prime})>g_{(v, \, \lambda v)}((0,1))$ for all
$\lambda,$
$$\lambda ^2<\frac{(|d_2|^2-1)}{1-2\Re \bar{\mu}_{1}
+(|\mu_1|^2+|c|^2)}$$ and for all $v,$ $|v|$ in $(0,
\frac{1}{\|A_1^*\|}].$
%Thus we conclude that $g_{(v,\, \lambda
%v)}(\beta^{\prime \prime})>g_{(v,\,\lambda v)}((0,1))$ for all
%$$\lambda ^2<\frac{(|d_2|^2-1)}{1-\Re \bar{\mu}_{1}
%+|v|^2|d_2|^2(|\mu_1|^2+|c|^2)}$$ and for all $|v|\in \mathbb
%R^{+},$ in particular, for all $|v|\in (0, \frac{1}{\|A^*\|}].$

If $\lambda$ is chosen with
$$\lambda ^2<\min\Big\{\frac{(|d_2|^2-1)}{1-2\Re \bar{\mu}_{1}
+(|\mu_1|^2+|c|^2)},
\frac{(|d_2|^2-1)\delta_1}{(\delta_1+2r)}\Big\},$$ then neither
$g_{(v,\,\lambda v)}(\beta^{\prime})$  nor $g_{(v,\,\lambda
v)}(\beta^{\prime \prime})$ is equal to
$\inf_{\beta}g_{(v,\,\lambda v)}(\beta)$ for any $v$ with $|v|$ in
$(0, \frac{1}{\|A_1^*\|}].$
%\end{itemize}
%\begin{itemize}
\item {\sf \bf{$|d_2|<1,\,|\mu_2|^2 + 2|c|^2\Re
\bar{\mu}_{2}>0,\,r>0$}:}

%Let $1>|d_2|$ and $|\mu_2|^2 + 2|c|^2\Re \bar{\mu}_{2}>0.$ Then
From Equation (\ref{eqa6}), it is easy to see that
$g_{(v,\,\lambda v)}(\beta^{\prime \prime})
>g_{(v,\,\lambda v)}(\beta^{\prime})$ for all $\lambda.$ Hence \mbox{$\inf_{\beta}g_{(v,\,\lambda v)}(\beta)\neq
g_{(v,\,\lambda v)}(\beta^{\prime \prime})$} for all $\lambda.$
%Now,{\small\begin{eqnarray}\label{eqa9} \nonumber&&
%g_{\lambda}(\beta^{'}, v)-g_{\lambda}(\beta,
%v)\\\nonumber&=&\{1-|v|^2-\lambda^2|v|^2(1+|c|^2)\}(|\beta^{'}_{1}|^2-|\beta_1|^2)
%+\{1-|v|^2|d_2|^2-\lambda^2|v|^2|c|^2\}(|\beta^{'}_{2}|^2-|\beta_2|^2)\\
%&&+2\lambda^2|v|^2\Re c
%(\beta_1\bar{\beta}_{2}-\beta^{'}_{1}\bar{\beta}^{'}_{2})
%-\lambda^2|v_1|^4||c|\bar{\beta_1}^2-d_2\bar{\beta_2}(\bar{\beta_1}+c\bar{\beta_2})|^2.
%\end{eqnarray}}
%Let $\beta=(1, 0).$ Then $g_{(v ,\,\lambda
%v)}((1,0))=\{1-|v|^2-\lambda^2|v|^2(1+|c|^2)\}+\lambda^2|v|^4|c|^2.$

Also, note that {\small\begin{eqnarray}\label{eqa10}
\nonumber\lefteqn{g_{(v,\,\lambda
v)}(\beta^{\prime})-g_{(v,\,\lambda v)}((1,0))
 %\\\nonumber&=&\{|v|^2(1-|d_2|^2)+\lambda^2|v|^2\}\frac{|\mu_2|^2}{(|\mu_2|^2+|c|^2)}
%+\frac{2\lambda^2|v|^2|c|^2}{(|\mu_2|^2+|c|^2)}\Re \bar{\mu}_{2}
%-\lambda^2|v|^4|c|^2
=\frac{|v|^2}{(|\mu_2|^2+|c|^2)}\{(1-|d_2|^2)|\mu_2|^2
+\lambda^2|\mu_2|^2}
\\&\phantom{Avijit Pal Avijit pal }+& 2\lambda^2|c|^2\Re
\bar{\mu}_{2} -\lambda^2|v|^2(|\mu_2|^2+|c|^2)|c|^2\}.
\end{eqnarray}}
From Equation (\ref{eqa10}), We have $g_{(v,\,\lambda
v)}(\beta^{\prime })>g_{(v,\,\lambda v)}((1,0))$ for all $\lambda$
and for all $v,$ $|v|^2 \leq \frac{(|\mu_2|^2 + 2|c|^2\Re
\bar{\mu}_{2})}{|c|^2(|\mu_2|^2+|c|^2)}.$

Also, from Equation (\ref{eqa10}), we obtain $g_{(v,\,\lambda
v)}(\beta^{\prime })>g_{(v,\,\lambda v)}((1,0))$ for all
$\lambda,$
$$\lambda
^2<\frac{(1-|d_2|^2)|\mu_2|^2}{|v|^2|c|^2(|\mu_2|^2+|c|^2)-|\mu_2|^2
-2|c|^2\Re \bar{\mu}_{2} }$$ and for all $v,$
$|v|^2>\frac{(|\mu_2|^2 + 2|c|^2\Re
\bar{\mu}_{2})}{|c|^2(|\mu_2|^2+|c|^2)}.$ Thus we have
$g_{(v,\,\lambda v)}(\beta^{\prime })>g_{(v,\,\lambda v)}((1,0))$
for all
$$\lambda
^2<\frac{(1-|d_2|^2)|\mu_2|^2}{|v|^2|c|^2(|\mu_2|^2+|c|^2)-|\mu_2|^2
-2|c|^2\Re \bar{\mu}_{2} }$$ and for all $v, v>0.$ %with %$|v|\in (0,
%\frac{1}{\|A^*\|}].$

We therefore conclude that neither $g_{(v,\,\lambda
v)}(\beta^{\prime})$ nor $g_{(v,\,\lambda v)}(\beta^{\prime
\prime})$ is equal to $\inf_{\beta}g_{(v,\,\lambda v)}(\beta)$ for
any $v$ with $|v|$ in $(0, \frac{1}{\|A_1^*\|}]$ and for any
$\lambda$ with
$$\lambda
^2<\frac{(1-|d_2|^2)|\mu_2|^2}{|v|^2|c|^2(|\mu_2|^2+|c|^2)-|\mu_2|^2
-2|c|^2\Re \bar{\mu}_{2} }.$$

\item {\sf \bf{$|d_2|<1,\,|\mu_2|^2 + 2|c|^2\Re
\bar{\mu}_{2}<0,\,r>0$}:}

%Let  $1>|d_2|$ and $|\mu_2|^2 + 2|c|^2\Re \bar{\mu}_{2}<0.$ Then
From Equation (\ref{eqa6}), it is easy to see that
$g_{(v,\,\lambda v)}(\beta^{\prime \prime})
>g_{(v,\,\lambda v)}(\beta^{\prime})$ for all $\lambda.$ Hence \mbox{$\inf_{\beta}g_{(v,\,\lambda v)}(\beta)\neq
g_{(v,\,\lambda v)}(\beta^{\prime \prime})$} for all $\lambda.$

If $\lambda$ is chosen with $$\lambda
^2<\frac{(1-|d_2|^2)|\mu_2|^2}{|c|^2(|\mu_2|^2+|c|^2)-|\mu_2|^2
-2|c|^2\Re \bar{\mu}_{2} },$$ then from Equation (\ref{eqa10}), we
have $g_{(v,\,\lambda v)}(\beta^{\prime })>g_{(v,\,\lambda
v)}((1,0))$ for any $v$ with $|v|$ in $(0, 1].$

Thus we conclude that neither $g_{(v,\,\lambda
v)}(\beta^{\prime})$ nor $g_{(v,\,\lambda v)}(\beta^{\prime
\prime})$ is equal to $\inf_{\beta}g_{(v,\, \lambda v)}(\beta)$
for any  $v$ with $|v|$ in $(0,1]$ and for any $\lambda$ with
$$\lambda
^2<\frac{(1-|d_2|^2)|\mu_2|^2}{|c|^2(|\mu_2|^2+|c|^2)-|\mu_2|^2
-2|c|^2\Re \bar{\mu}_{2} }.$$

\item {\sf \bf{$|d_2|<1,\,|\mu_2|^2 + 2|c|^2\Re
\bar{\mu}_{2}>0,\,r<0,\,2r+\delta_1<0$}:}

%Suppose  $1>|d_2|,\, |\mu_1|^2 + 2|c|^2\Re \bar{\mu}_{1}>0$ and $
%2r+\delta_1$ is negative. Then
From Equation (\ref{eqa6}), we have $g_{(v,\,\lambda
v)}(\beta^{\prime \prime}) >g_{(v,\,\lambda v)}(\beta^{\prime})$
if $\lambda^2<\frac{(1-|d_2|^2)\delta_1}{-(\delta_1+2r)}.$ Hence
\mbox{$\inf_{\beta}g_{(v,\,\lambda v)}(\beta)\neq g_{(v,\,\lambda
v)}(\beta^{\prime \prime })$} for all $\lambda,$
$\lambda^2<\frac{(1-|d_2|^2)\delta_1}{-(\delta_1+2r)}.$

%Note that {\small\begin{eqnarray}\label{eqaeqa}\nonumber\lefteqn{
%g_{(v,\,\lambda v)}(\beta^{\prime \prime})-g_{(v,\,\lambda
%v)}((1,0))
 %\\\nonumber&=&\{|v|^2(1-|d_2|^2)+\lambda^2|v|^2\}\frac{|\mu_2|^2}{(|\mu_2|^2+|c|^2)}
%+\frac{2\lambda^2|v|^2|c|^2}{(|\mu_2|^2+|c|^2)}\Re \bar{\mu}_{2}
%-\lambda^2|v|^4|c|^2
%=\frac{|v|^2}{(|\mu_1|^2+|c|^2)}\{(1-|d_2|^2)|\mu_1|^2
%+\lambda^2|\mu_1|^2}
%\\&\phantom{Avijit Pal Avijit pal}+& 2\lambda^2|c|^2\Re \bar{\mu}_{1}
%-\lambda^2|v|^2(|\mu_1|^2+|c|^2)|c|^2\}.
%\end{eqnarray}}

From Equation (\ref{eqa10}), we have $g_{(v,\,\lambda
v)}(\beta^{\prime })>g_{(v,\,\lambda v)}((1,0))$ for all $\lambda$
and for all $v,$  $|v|^2\leq \frac{|\mu_2|^2 + 2|c|^2\Re
\bar{\mu}_{2}} {(|\mu_2|^2+|c|^2)|c|^2}.$

Also, from Equation (\ref{eqa10}), we obtain $g_{(v,\,\lambda
v)}(\beta^{\prime })>g_{(v,\,\lambda v)}((1,0))$ for all
$\lambda,$
$$\lambda
^2<\frac{(1-|d_2|^2)|\mu_2|^2}{|v|^2(|\mu_2|^2+|c|^2)|c|^2-|\mu_2|^2
-2|c|^2\Re \bar{\mu}_{2} }$$ and for all $v,$ $|v|^2>
\frac{|\mu_2|^2 + 2|c|^2\Re
\bar{\mu}_{2}}{(|\mu_2|^2+|c|^2)|c|^2}.$ Thus we have
$g_{(v,\,\lambda v)}(\beta^{\prime })>g_{(v,\,\lambda v)}((1,0))$
for all $\lambda,$ $$\lambda
^2<\frac{(1-|d_2|^2)|\mu_2|^2}{|v|^2(|\mu_2|^2+|c|^2)|c|^2-|\mu_2|^2
-2|c|^2\Re \bar{\mu}_{2} }$$ and for all $v,v>0.$

We therefore conclude that neither $g_{(v,\,\lambda
v)}(\beta^{\prime})$ nor $g_{(v,\,\lambda v)}(\beta^{\prime
\prime})$ is equal to $\inf_{\beta}g_{(v,\,\lambda v)}(\beta)$ for
any $v$ with $|v|$ in $(0,\frac{1}{\|A_1^*\|}]$ and for any
$\lambda$ with
$$\lambda
^2<\min\Big\{\frac{(1-|d_2|^2)|\mu_2|^2}{|v|^2(|\mu_2|^2+|c|^2)|c|^2-|\mu_2|^2
-2|c|^2\Re \bar{\mu}_{2} },
\frac{(1-|d_2|^2)\delta_1}{-(2r+\delta_1)}\Big\}.$$

\item {\sf \bf{$|d_2|<1,\,|\mu_2|^2 + 2|c|^2\Re
\bar{\mu}_{2}<0,\,r<0,\,2r+\delta_1<0$}:}

%Suppose  $1>|d_2|,\, |\mu_1|^2 + 2|c|^2\Re \bar{\mu}_{1}<0$ and $
%2r+\delta_1$ is negative. Then
From Equation (\ref{eqa6}), we have $g_{(v,\,\lambda
v)}(\beta^{\prime \prime}) >g_{(v,\,\lambda v)}(\beta^{\prime})$
if $\lambda^2<\frac{(1-|d_2|^2)\delta_1}{-(\delta_1+2r)}.$ Hence
\mbox{$\inf_{\beta}g_{(v,\,\lambda v)}(\beta)\neq g_{(v,\,\lambda
v)}(\beta^{\prime \prime})$} for all $\lambda,$
$\lambda^2<\frac{(1-|d_2|^2)\delta_1}{-(\delta_1+2r)}.$

From Equation (\ref{eqa10}), we have $g_{(v,\,\lambda
v)}(\beta^{\prime })>g_{(v,\,\lambda v)}((1,0))$ for all
$\lambda,$
$$\lambda
^2<\frac{(1-|d_2|^2)|\mu_2|^2}{(|\mu_2|^2+|c|^2)|c|^2-|\mu_2|^2
-2|c|^2\Re \bar{\mu}_{2} }$$ for all $v,$ $|v|$ in $(0,
\frac{1}{\|A_1^*\|}].$

We therefore conclude that neither $g_{(v,\,\lambda
v)}(\beta^{\prime})$ nor $g_{(v,\,\lambda v)}(\beta^{\prime
\prime})$ is equal to $\inf_{\beta}g_{(v,\,\lambda v)}(\beta)$ for
any
 $v,$ with $|v|$ in $(0,\frac{1}{\|A_1^*\|}]$ and
for any $\lambda$ with
$$\lambda
^2<\min\Big\{\frac{(1-|d_2|^2)|\mu_2|^2}{(|\mu_2|^2+|c|^2)|c|^2-|\mu_2|^2
-2|c|^2\Re \bar{\mu}_{2} },
\frac{(1-|d_2|^2)\delta_1}{-(2r+\delta_1)}\Big\}.$$

\item {\sf \bf{$|d_2|<1,\,|\mu_2|^2 + 2|c|^2\Re
\bar{\mu}_{2}>0,\,r<0,\,2r+\delta_1>0$}:}

%Suppose  $1>|d_2|,\,|\mu_2|^2 + 2|c|^2\Re \bar{\mu}_{2}>0$ and $
%2r+\delta_1$ is positive. Then
From Equation (\ref{eqa6}) it is easy to see that $g_{(v,\,\lambda
v)}(\beta^{\prime \prime})
>g_{(v,\,\lambda v)}(\beta^{\prime})$ for all $\lambda.$ Hence \mbox{$\inf_{\beta}g_{(v,\,\lambda v)}(\beta)\neq
g_{(v,\,\lambda v)}(\beta^{\prime \prime})$} for all $\lambda.$

From Equation (\ref{eqa10}) we have $g_{(v,\,\lambda
v)}(\beta^{\prime })>g_{(v,\,\lambda v)}((1,0))$ for all $\lambda$
and for all $v,$ $|v|^2 \leq \frac{(|\mu_2|^2 + 2|c|^2\Re
\bar{\mu}_{2})}{|c|^2(|\mu_2|^2+|c|^2)}.$

From Equation (\ref{eqa10}) we obtain $g_{(v,\,\lambda
v)}(\beta^{\prime })>g_{(v,\,\lambda v)}((1,0))$ for all
$\lambda,$
$$\lambda
^2<\frac{(1-|d_2|^2)|\mu_2|^2}{|v|^2|c|^2(|\mu_2|^2+|c|^2)-|\mu_2|^2
-2|c|^2\Re \bar{\mu}_{2} }$$ and for all $v,$
$|v|^2>\frac{(|\mu_2|^2 + 2|c|^2\Re
\bar{\mu}_{2})}{|c|^2(|\mu_2|^2+|c|^2)}.$ Thus we have
$g_{(v,\,\lambda v)}(\beta^{\prime })>g_{(v,\,\lambda v)}((1,0))$
for all $\lambda,$
$$\lambda
^2<\frac{(1-|d_2|^2)|\mu_2|^2}{|v|^2|c|^2(|\mu_2|^2+|c|^2)-|\mu_2|^2
-2|c|^2\Re \bar{\mu}_{2} }$$ and for all $v,v>0.$

We therefore conclude that neither $g_{(v,\,\lambda
v)}(\beta^{\prime})$ nor $g_{(v,\,\lambda v)}(\beta^{\prime
\prime})$ is equal to $\inf_{\beta}g_{(v,\,\lambda v)}(\beta)$ for
any $v$ with  $|v|$ in $(0, \frac{1}{\|A_1^*\|}]$ and for any
$\lambda$ with
$$\lambda
^2<\frac{(1-|d_2|^2)|\mu_2|^2}{|v|^2|c|^2(|\mu_2|^2+|c|^2)-|\mu_2|^2
-2|c|^2\Re \bar{\mu}_{2} }.$$

\item {\sf \bf{$|d_2|<1,\,|\mu_2|^2 + 2|c|^2\Re
\bar{\mu}_{2}<0,\,r<0,\,2r+\delta_1>0$}:}

%Suppose  $1>|d_2|,\,|\mu_2|^2 + 2|c|^2\Re \bar{\mu}_{2}<0$ and $
%2r+\delta_1$ is positive. Then
From Equation (\ref{eqa6}) it is easy to see that $g_{(v,\,\lambda
v)}(\beta^{\prime \prime})
>g_{(v,\,\lambda v)}(\beta^{\prime})$ for all $\lambda.$ Hence \mbox{$\inf_{\beta}g_{(v,\,\lambda v)}(\beta)\neq
g_{(v,\,\lambda v)}(\beta^{\prime \prime})$} for all $\lambda.$

If $\lambda$ is chosen with $$\lambda
^2<\frac{(1-|d_2|^2)|\mu_2|^2}{|c|^2(|\mu_2|^2+|c|^2)-|\mu_2|^2
-2|c|^2\Re \bar{\mu}_{2} },$$ then from Equation (\ref{eqa10}), we
have $g_{(v,\,\lambda v)}(\beta^{\prime })>g_{(v,\,\lambda
v)}((1,0))$ for any $v$ with $|v|$ in $(0, 1].$

Thus we conclude that neither $g_{(v,\,\lambda
v)}(\beta^{\prime})$ nor $g_{(v,\,\lambda v)}(\beta^{\prime
\prime})$ is equal to $\inf_{\beta}g_{(v,\, \lambda v)}(\beta)$
for any  $v$ with $|v|$ in $(0,1]$ and for any $\lambda$ with
$$\lambda
^2<\frac{(1-|d_2|^2)|\mu_2|^2}{|c|^2(|\mu_2|^2+|c|^2)-|\mu_2|^2
-2|c|^2\Re \bar{\mu}_{2} }.$$

\end{itemize}
\item[(b)] If $|\beta^{\prime}_{1}|^2=|\beta^{\prime
\prime}_{1}|^2,$ then $|\mu_1|^2=|\mu_2|^2.$ Thus $r=0.$
Therefore, from Equation (\ref{eqa6}) we have $g_{(v,\,\lambda
v)}(\beta^{\prime \prime}) =g_{(v, \, \lambda
v)}(\beta^{\prime}).$ Here we have  two possibilities, namely,
$|d_2|>1$ and $|d_2|<1.$
\begin{itemize}
\item {\sf \bf{$|d_2|>1,\,r=0$}:}

Since $r=0,$ we have $1-\Re \bar{\mu}_{1}=0.$ From Equation
(\ref{eqa8}) we can easily see that $g_{(v,\,\lambda
v)}(\beta^{\prime \prime})>g_{(v,\, \lambda v)}((0,1))$ for any
$\lambda$ with $\lambda ^2<\frac{(|d_2|^2-1)}{(|\mu_1|^2+|c|^2)}$
and for any  $v$ with $|v|$ in $(0, \frac{1}{\|A_1^*\|}].$

Hence we conclude that neither $g_{(v,\,\lambda
v)}(\beta^{\prime})$ nor $g_{(v,\,\lambda v)}(\beta^{\prime
\prime})$ is equal to $\inf_{\beta}g_{(v,\,\lambda v)}(\beta)$ for
any  $v$ with $|v|$ in $(0, \frac{1}{\|A_1^*\|}]$ and for any
$\lambda$ with $\lambda ^2<\frac{(|d_2|^2-1)}{(|\mu_1|^2+|c|^2)}.$

\item {\sf \bf{$|d_2|<1,\,r=0,\,|\mu_2|^2 + 2|c|^2\Re
\bar{\mu}_{2}>0$}:}

%If $|d_2|<1$ and $|\mu_2|^2 + 2|c|^2\Re \bar{\mu}_{2}>0,$ then
From Equation (\ref{eqa10})we also see that neither
$g_{(v,\,\lambda v)}(\beta^{\prime})$ nor $g_{(v,\,\lambda
v)}(\beta^{\prime \prime})$ is equal to
$\inf_{\beta}g_{(v,\,\lambda v)}(\beta)$  for any $v$ with $|v|$
in $(0, \frac{1}{\|A_1^*\|}]$ and for any $\lambda$ with
$$\lambda
^2<\frac{(1-|d_2|^2)|\mu_2|^2}{|v|^2|c|^2(|\mu_2|^2+|c|^2)-|\mu_2|^2
-2|c|^2\Re \bar{\mu}_{2} }.$$

\item {\sf \bf{$|d_2|<1,\,r=0,\,|\mu_2|^2 + 2|c|^2\Re
\bar{\mu}_{2}<0$}:}

%If $|d_2|<1$ and $|\mu_2|^2 + 2|c|^2\Re \bar{\mu}_{2}<0,$ then
From Equation (\ref{eqa10})we also see that neither
$g_{(v,\,\lambda v)}(\beta^{\prime})$ nor $g_{(v,\,\lambda
v)}(\beta^{\prime \prime})$ is equal to
$\inf_{\beta}g_{(v,\,\lambda v)}(\beta)$  for any $v$ with $|v|$
in $(0, \frac{1}{\|A_1^*\|}]$ and for any $\lambda$ with
$$\lambda
^2<\frac{(1-|d_2|^2)|\mu_2|^2}{|c|^2(|\mu_2|^2+|c|^2)-|\mu_2|^2
-2|c|^2\Re \bar{\mu}_{2} }.$$
\end{itemize}
\end{enumerate}
We note that $\nu_1=\frac{1}{\mu_1}$ and  $ \nu_2=\frac{1}{\mu_2}$
are the roots of $\det(\nu A_{2}^*-A_{1}^*)=0.$
%and corresponding
%vectors are
%$\beta{'}=(\frac{\bar{c}}{\sqrt{|\mu_2|^2+|c|^2}},
%\frac{-\mu_2}{\sqrt{|\mu_2|^2+|c|^2}})$ and
%$\beta{''}=(\frac{\bar{c}}{\sqrt{|\mu_1|^2+|c|^2}},
%\frac{-\mu_1}{\sqrt{|\mu_1|^2+|c|^2}})$
%$$\beta^{\prime}=\Big(\frac{|\bar{c}|\exp{i\theta_1}}{\sqrt{|\mu_2|^2+|c|^2}},
%\frac{-\mu_2\exp(\theta_1-\phi_1)}{\sqrt{|\mu_2|^2+|c|^2}}\Big)$$
%and
%$$\beta^{\prime \prime}=\Big(\frac{|\bar{c}|\exp{i\theta_2}}{\sqrt{|\mu_1|^2+|c|^2}},
%\frac{-\mu_1\exp{i(\theta_2-\phi_1)}}{\sqrt{|\mu_1|^2+|c|^2}}\Big)$$
%respectively.
The vectors $\beta^{\prime}, \beta^{\prime \prime}$ satisfying $(
A_{2}^*-\mu_1 A_{1}^*)\beta^{\prime}=0$ and $( A_{2}^*-\mu_2
A_{1}^*)\beta^{\prime \prime}=0$ are
$$\beta^{\prime}=\Big(\frac{|c|\exp({i\theta_1})}{\sqrt{|\mu_2|^2+|c|^2}},
\frac{-\mu_2\exp
i(\theta_1-\phi_1)}{\sqrt{|\mu_2|^2+|c|^2}}\Big)$$ and
$$\beta^{\prime \prime}=\Big(\frac{|c|\exp({i\theta_2})}{\sqrt{|\mu_1|^2+|c|^2}},
\frac{-\mu_1\exp{i(\theta_2-\phi_1)}}{\sqrt{|\mu_1|^2+|c|^2}}\Big)$$
respectively, where $\bar{c}=|c|\exp({i\phi_1}).$ Proceeding as
above, we  prove that neither $g_{(v,\,\lambda_0
v)}(\beta^{\prime})$ nor $g_{(v,\,\lambda_0 v)}(\beta^{\prime
\prime})$ is equal to $\inf_{\beta}g_{(v,\,\lambda_0 v)}(\beta)$
for any  $v$ with $|v|\in (0,\frac{1}{\|A_1^*\|}].$

%Similarly if either $A_1=\left(\begin{smallmatrix}
%d_1 & 0   \\
%0  & 1
%\end{smallmatrix}\right), \left(\begin{smallmatrix}
%1 & 0   \\
%0  & d_2
%\end{smallmatrix}\right)$ or
%$A_2=\left(\begin{smallmatrix}
 %0 & |c| \\
%c &  1
%\end{smallmatrix} \right)$ with either $|d_1|\neq 1$ or $|d_2|\neq 1$ and  $A_1=\left(\begin{smallmatrix}
%d_1 & 0   \\
%0  & 1
%\end{smallmatrix}\right),A_2=\left(\begin{smallmatrix}
% 1 & |c| \\
%c &  0
%\end{smallmatrix} \right)$ with $|d_1|\neq 1, $ then we can
%easily show that neither $g_{(v,\,\lambda_0 v)}(\beta^{\prime})$
%nor $g_{(v,\,\lambda_0 v)}(\beta^{\prime \prime})$ is equal to
%$\inf_{\beta}g_{(v,\,\lambda_0 v)}(\beta)$ for any fixed but
%arbitrary $v$ with $|v|\in (0,\frac{1}{\|A_1^*\|}].$

{\sf Case (iii):} Here we assume that $b= |c|$ and $|d_2|= 1.$ The
proof is similar to {\sf Case (i)} and we skip the details.

Let $A_1\in \{A_{11},A_{12}\}$ and $A_2\in \{A_{21},A_{22}\},$
where $A_{11}=\left (
\begin{smallmatrix}
1 & 0   \\
0  & d_{2}
\end{smallmatrix}\right ), A_{12}=\left (
\begin{smallmatrix}
d_{1} & 0   \\
0  & 1
\end{smallmatrix}\right ); A_{21}=\left ( \begin{smallmatrix}
 1 &  b \\
 c &  0
\end{smallmatrix}\right ),A_{22}=\left ( \begin{smallmatrix}
 0 &  b \\
 c &  1
\end{smallmatrix}\right ) .$ We have proved the theorem for
$A_1=A_{11}$ and $A_2=A_{21}.$ The proof in the remaining cases,
namely, $A_1=A_{11}$ and $A_2=A_{22};$ $A_1=A_{12}$ and
$A_2=A_{21}$ and $A_1=A_{12}$ and $A_2=A_{22}$ follow similarly.
\end{proof}
\begin{theorem}\label{Thm:main2}
Let  $A_1$ be of the form $\left (
\begin{smallmatrix}
1 & 0   \\
0  & d_{2}
\end{smallmatrix}\right )$ or $\left (
\begin{smallmatrix}
d_{1} & 0   \\
0  & 1
\end{smallmatrix}\right )$ and   $A_2$ be of the form $\left ( \begin{smallmatrix}
 0 &  b \\
 c &  0
\end{smallmatrix}\right ) $ and assume that they are not simultaneously diagonalizable.
Then there exists $(v_0, \lambda _0 v_0)$ in $\mathcal{E}$ such
that neither $g_{(v_0,\, \lambda_0 v_0)}(\beta^{\prime})$ nor
$g_{(v_0,\, \lambda_0 v_0)}(\beta^{\prime \prime})$ is equal to
$\inf_{\beta} g_{(v_0,\, \lambda_0 v_0)}(\beta).$
\end{theorem}
\begin{proof}
Suppose $A_1= \left (
\begin{smallmatrix}
1 & 0   \\
0  & d_{2}
\end{smallmatrix}\right )$  and   $A_2=\left ( \begin{smallmatrix}
 0 &  b \\
 c &  0
\end{smallmatrix}\right ).$ As above we have seen that the homomorphism $\rho_{V}$ is contractive if and only if
$|v|^2\leq \frac{1}{\|A_1^{*}\|^2}$ and $\inf_{\beta,
\|\beta\|_2=1}g_{(v,\, \lambda v)}(\beta)\geq 0 .$ Observe that
{\small\begin{eqnarray} \label{eqa12}\nonumber\inf_{\beta}g_{(v,\,
\lambda v)}(\beta)
%\nonumber&=&\inf_{\beta}\{1-|v|^2\|A_{1}^*\beta\|^2
%-\lambda^2|v|^2\|A_{2}^*\beta\|^2 +
%|v|^4\lambda^2(\|A_{1}^*\beta\|^2\|A_{2}^*\beta\|^2 -
%|\left\langle A_{1}A_{2}^*\beta, \beta\right\rangle
%|^2)\}\\\nonumber
&=&\inf_{\beta}\{1-|v|^2-\lambda^2|v|^2b^2\}|\beta_{1}|^2+\{1-|v|^2|d_2|^2-\lambda^2|v|^2|c|^2\}|\beta_{2}|^2
\\&&+|v|^4\lambda^2\left(b|\bar{\beta_1}|^2-|cd_2||
\bar{\beta_2}|^2\right)^2.
\end{eqnarray}}
To complete the proof, we follow exactly the same steps as in
Theorem \ref{Thm:main1} except the {\sf Case (iii)}. This is
because when $|d_2|=1$ and $|b|=|c|,$ as before without loss of
generality, we can take $A_1=I_2$ and $A_2=\left (
\begin{smallmatrix}
 0 &  b \\
 c &  0
\end{smallmatrix}\right )$ with $b,c \in \mathbb C,|b|=|c|.$ Since $|b|=|c|,$ we can see that $A_2$ is normal.
Therefore, conjugating $A_2$ by a unitary $U$ we can assume $A_2$
is a diagonal matrix. This contradicts the fact that $A_1$ and
$A_2$ are not simultaneously diagonalizable. Hence the proof of
the theorem involves  two cases.

{\sf Case (i):} Here $b\neq |c|$ and $|d_2|= 1,$ that is,
$A_1=\left (
\begin{smallmatrix}
 1 &  0 \\
 0 &  \exp(i\theta)
\end{smallmatrix}\right )$ and $A_2=\left ( \begin{smallmatrix}
 0 &  b \\
 c &  0
\end{smallmatrix}\right )$ with $b \neq |c|.$ Let  $U=\left (
\begin{smallmatrix}
 1 &  0 \\
 0 &  \exp(-i\theta)
\end{smallmatrix}\right ).$ Then $U$ is a unitary and the pair $(A_1U, A_2U)$ determines the same set
$\Omega _{\mathbf A}.$ So, we may assume without loss of
generality that $\mathbf A$ is of the form  $(I_2, \left (
\begin{smallmatrix}
 0 &  b \\
 c &  0
\end{smallmatrix}\right ) )$  with $b, c\in \mathbb C, |b| \neq |c|.$

\begin{enumerate}
\item[(a)] Suppose $c=0.$ % $ \mu_1=\mu_2=0$ are the eigen values
%of $A_2^*$ and corresponding eigen vectors are
%$\beta^{\prime}=\beta^{''}=(0, 1).$
%The contractivity condition is
%{\small\begin{align}\label{eqa16}\inf_{\beta}g_{\lambda}(\beta,
%v)=\inf_{\beta}\{1-|v|^2-\lambda^2|v|^2b^2\}|\beta_{1}|^2+\{1-|v|^2|d_2|^2\}|\beta_{2}|^2
%+|v|^4\lambda^2b^2|\bar{\beta_1}|^4 \end{align}}
% where  $\|\beta\|_2=|\beta_1|^2+|\beta_2|^2=1.$ From Equation (\ref{eqa16}) it is also easy to see that
 % $g_{\lambda}(\beta^{'}, v)=(1-|v|^2|d_2|^2).$
The roots of $\det( A_{2}^*-\mu A_{1}^*)=0$ are $\mu_1=\mu_2=0.$
The vectors $\beta^{\prime}, \beta^{\prime \prime}$ satisfying $(
A_{2}^*-\mu_1 A_{1}^*)\beta^{\prime}=0$ and $( A_{2}^*-\mu_2
A_{1}^*)\beta^{\prime \prime}=0$ are $\beta^{\prime}=
\beta^{\prime \prime}=(0, \exp({i\psi}))$.
Note that {\small\begin{eqnarray}\label{equaaa} g_{(v,\,\lambda
v)}((0,\exp({i\psi})) )-g_{v,\,\lambda
v)}(\beta)&=&\lambda^2|v|^2|b|^2|\beta_{1}|^2(1-
|v|^2|\bar{\beta_1}|^2).
\end{eqnarray}}
From Equation (\ref{equaaa}) we have for all $v,$ $|v| \leq 1,$
for all $\lambda,$ there exists a $\beta$ with  $|\beta|<1$ such
that $g_{(v,\,\lambda v)}((0,\exp({i\psi})) )>g_{v,\,\lambda
v)}(\beta).$  We therefore conclude that neither $g_{(v,\,\lambda
v)}(\beta^{\prime})$ nor $g_{(v,\,\lambda v)}(\beta^{\prime
\prime})$ is equal to $\inf_{\beta}g_{(v,\, \lambda v)}(\beta)$
for any $\lambda$ and for any $v$ with $|v|\leq 1.$

Also, we observe that $\det(\nu A_{2}^*-A_{1}^*) \neq 0.$ Thus
there is no vector $\gamma$ with $\|\gamma\|=1$ satisfying $(\nu
A_{2}^*- A_{1}^*)\gamma=0.$ We therefore conclude that there
exists no vector $\gamma$ such that  $g_{(v,\,\lambda v)}(\gamma)$
is equal to $\inf_{\beta}g_{(v,\, \lambda v)}(\beta)$ for any
$\lambda$ and for any $v$ with $|v|\leq 1.$ We arrive at the same
conclusion whenever $b=0.$

\item[(b)] Suppose $b, c$ are not simultaneously zero. If we
consider $\det( A_{2}^*-\mu A_{1}^*)=0,$  then we see that
$\mu_1=\sqrt{\bar{b}\bar{c}}, \mu_2=-\sqrt{\bar{b}\bar{c}}$ are
the roots of $\det( A_{2}^*-\mu A_{1}^*)=0.$ The vectors
$\beta^{\prime}, \beta^{\prime \prime}$ satisfying $(
A_{2}^*-\mu_1 A_{1}^*)\beta^{\prime}=0$ and  $( A_{2}^*-\mu_2
A_{1}^*)\beta^{\prime \prime}=0$ are
$$\beta^{\prime}=\big(\frac{\sqrt{|c|}\exp(i\theta_3)}{\sqrt{|c|+|b|}},
\frac{\sqrt{\bar{b}}\exp{i(\theta_3-\phi_2)}}{\sqrt{|c|+|b|}}\big)$$
and
$$\beta^{\prime \prime}=\big(\frac{-\sqrt{|c|}\exp(i\theta_3)}{\sqrt{|c|+|b|}},
\frac{\sqrt{\bar{b}}\exp{i(\theta_3-\phi_2)}}{\sqrt{|c|+|b|}}\big)$$
respectively, where $\bar{b}=|b|\exp{i(\phi_2)}.$ Substituting
$\beta=\beta^{\prime}$ and $\beta=\beta^{\prime \prime}$ in
Equation (\ref{eqa12}) we have {\small\begin{align*}
g_{(v,\,\lambda v)}(\beta^{\prime
\prime})\nonumber&=g_{(v,\,\lambda
v)}(\beta^{\prime})\\\nonumber&=\{1-|v|^2-\lambda^2|v|^2|b|^2\}|\beta^{\prime}_{1}|^2+\{1-|v|^2-\lambda^2|v|^2|c|^2\}|
\beta^{\prime}_{2}|^2,
\end{align*}}
 where $\beta^{\prime}=(\beta^{\prime}_{1}, \beta^{\prime}_{2}), \beta^{\prime \prime}=(\beta^{\prime \prime}_{1},
 \beta^{\prime \prime}_{2}).$
Now, {\small\begin{eqnarray}\label{eqa13}\nonumber&&
g_{(v,\,\lambda v)}(\beta^{\prime \prime})-g_{(v,\,\lambda
v)}(\beta)\\\nonumber&=&\{1-|v|^2-\lambda^2|v|^2|b|^2\}(|\beta^{\prime
\prime}_{1}|^2-|\beta_1|^2)
+\{1-|v|^2-\lambda^2|v|^2|c|^2\}(|\beta^{\prime
\prime}_{2}|^2-|\beta_2|^2)
\\&&-|v|^4\lambda^2\left(|b||\bar{\beta_1}|^2-|c||
\bar{\beta_2}|^2\right)^2.
\end{eqnarray}}

\begin{itemize}
\item Assume that  $|b|^2>|c|^2.$
%If $ |\beta^{''}_{1}|= 1,$ then
%$\sqrt{b} =0,$ implying
 %$b=0,$ which is a contradiction.
Since $b, c$ are not simultaneously zero, it follows that $
\beta^{\prime \prime}_{1} \neq 0,$ or equivalently $
|\beta^{\prime \prime}_{1}|\neq 1.$ Since $ |\beta^{''}_{1}|\neq
1,$ we can choose $\beta_{1} $ such that
$(|\beta^{''}_{1}|^2-|\beta_{1}|^2)=-\delta_2,$ where
$\delta_2>0.$ Then $(|\beta^{''}_{2}|^2-|\beta_{2}|^2)=\delta_2.$
Hence from Equation (\ref{eqa13}) we have
\noindent{\small\begin{eqnarray} g_{(v,\,\lambda v)}(\beta^{\prime
\prime})-g_{(v,\,\lambda
v)}(\beta)&=&\{\lambda^2|v|^2(|b|^2-|c|^2)\}\delta _2-
|v|^4\lambda^2(|b|+|c|)^2\delta_{2}^{2}.
\end{eqnarray}}
If we choose $\delta_{2}<\frac{(|b|^2-|c|^2)}{(|b|+|c|)^2},$ then
we have $g_{(v,\,\lambda v)}(\beta^{\prime
\prime})>g_{(v,\,\lambda v)}(\beta)$ for all $\lambda$ and for all
$v,$ $|v|^2\leq1.$ We therefore conclude that neither
$g_{(v,\,\lambda v)}(\beta^{\prime})$ nor $g_{(v,\,\lambda
v)}(\beta^{\prime \prime})$ is equal to
$\inf_{\beta}g_{(v,\,\lambda v)}(\beta)$ for any $\lambda$ and for
any $v$ with $|v|\leq 1.$

\item Suppose  $|b|^2<|c|^2.$ We can also choose $\beta_{1}$ such
that $(|\beta^{''}_{1}|^2-|\beta_{1}|^2)=\delta_3,\delta_3>0.$
Therefore from Equation ( \ref{eqa13}) we have
\noindent{\small\begin{eqnarray}\label{eqaaeqaa} g_{(v,\,\lambda
v)}(\beta^{\prime \prime})-g_{(v,\,\lambda
v)}(\beta)&=&\{\lambda^2|v|^2(|c|^2-|b|^2)\}\delta _3-
|v|^4\lambda^2(|b|+|c|)^2\delta_{3}^{2}.
\end{eqnarray}}
From Equation (\ref{eqaaeqaa}) we have $g_{(v,\,\lambda
v)}(\beta^{\prime \prime})>g_{(v,\,\lambda v)}(\beta)$ for all
$\lambda,$  for all $v,$ $|v|^2\leq1$ and for all $\delta_{3},$
$\delta_{3}<\frac{(|c|^2-|b|^2)}{(|b|+|c|)^2}.$ Therefore we
conclude that neither $g_{(v,\,\lambda v)}(\beta^{\prime})$ nor
$g_{(v,\,\lambda v)}(\beta^{\prime \prime})$ is equal to
$\inf_{\beta}g_{(v,\,\lambda v)}(\beta)$ for any $\lambda$ and for
any $v$ with $|v|\leq 1.$
\end{itemize}
\end{enumerate}
%If we compute $\det(\nu A_{2}^*-A_{1}^*)=0$ then
%$\nu_1=\frac{1}{\mu_1}, \nu_2=\frac{1}{\mu_2}$ be two roots and
%corresponding vectors are
%$$\beta^{\prime}=(-\frac{\sqrt{\bar{c}}}{\sqrt{|c|+|b|}},
%\frac{\sqrt{b}}{\sqrt{|c|+|b|}})$$ and
%$$\beta^{\prime \prime}=(\frac{\sqrt{\bar{c}}}{\sqrt{|c|+|b|}},
%\frac{\sqrt{b}}{\sqrt{|c|+|b|}})$$
The roots of $\det(\nu A_{2}^*- A_{1}^*)=0$ are
$\nu_1=\frac{1}{\mu_1}, \nu_2=\frac{1}{\mu_2}.$ The vectors
$\beta^{\prime}, \beta^{\prime \prime}$ satisfying $(\nu_1
A_{2}^*- A_{1}^*)\beta^{\prime}=0$ and $(\nu_2A_{2}^*-
A_{1}^*)\beta^{\prime \prime}=0$ are
$$\beta^{\prime}=\Big(\frac{\sqrt{|c|}\exp(i\theta_3)}{\sqrt{|c|+|b|}},
\frac{\sqrt{\bar{b}}\exp{i(\theta_3-\phi_2)}}{\sqrt{|c|+|b|}}\Big)$$
and
$$\beta^{\prime \prime}=\Big(\frac{-\sqrt{|c|}\exp(i\theta_3)}{\sqrt{|c|+|b|}},
\frac{\sqrt{\bar{b}}\exp{i(\theta_3-\phi_2)}}{\sqrt{|c|+|b|}}\Big)$$
respectively, where $\bar{b}=|b|\exp{i(\phi_2)}.$ Proceeding as
above, we also find that neither $g_{(v,\,\lambda_0
v)}(\beta^{\prime})$ nor $g_{(v,\,\lambda_0 v)}(\beta^{\prime
\prime})$ is equal to $\inf_{\beta}g_{(v,\,\lambda_0 v)}(\beta)$
for any $v$ with $|v|<1.$

{\sf Case $(ii)$:} In this case, $A_1=\left(\begin{smallmatrix}
1 & 0   \\
0  & d_2
\end{smallmatrix}\right),$
$A_2=\left(\begin{smallmatrix}
 0 & |c| \\
c &  0
\end{smallmatrix} \right)$ with $|d_2|\neq 1.$
%$\mu_1=\sqrt{\frac{|c|\bar{c}}{\bar{d_2}}},
%\mu_2=-\sqrt{\frac{|c|\bar{c}}{\bar{d_2}}}$ are two roots of
%$\det(A_2^*-\mu A_1^*)=0$ and corresponding vectors are
The roots of $\det( A_{2}^*-\mu A_{1}^*)=0,$ are
$\mu_1=\sqrt{\frac{|c|\bar{c}}{\bar{d_2}}},
\mu_2=\sqrt{\frac{|c|\bar{c}}{\bar{d_2}}}.$ The vectors
$\beta^{\prime}, \beta^{\prime \prime}$ satisfying $(
A_{2}^*-\mu_1 A_{1}^*)\beta^{\prime}=0$ and $( A_{2}^*-\mu_2
A_{1}^*)\beta^{\prime \prime}=0$ are
$$\beta^{\prime}=\Big(\frac{\sqrt{|c|}\exp(i\theta_4)}{\sqrt{|c|+|\frac{|c|}{d_2}|}},
\frac{\sqrt{\frac{|c|}{\bar{d_2}}}\exp{i(\theta_4-\phi_3)}}{\sqrt{|c|+|\frac{|c|}{d_2}|}}
\Big)$$ and
$$\beta^{\prime \prime}=\Big(\frac{-\sqrt{|c|}\exp(i\theta_4)}{\sqrt{|c|+|\frac{|c|}{d_2}|}},
\frac{\sqrt{\frac{|c|}{\bar{d_2}}}\exp{i(\theta_4-\phi_3)}}{\sqrt{|c|+|\frac{|c|}{d_2}|}}
\Big)
$$ respectively, where $\frac{|c|}{\bar{d_2}}=\frac{|c|}{|d_2|}\exp{i(\phi_3)}.$

\begin{itemize}
\item Suppose $1<|d_2|.$ From Equation (\ref{eqa13}) we have
{\small\begin{eqnarray}\label{eqa14} g_{(v,\,\lambda
v)}(\beta^{\prime \prime})-g_{(v,\,\lambda
v)}((0,1))\nonumber&=&(|d_2|^2-1)|\beta^{\prime \prime}_{1}|^2
-|v|^2\lambda^2|cd_2|^2\\&=&
|d_2|\{(|d_2|-1)-|v|^2\lambda^2|c|^2|d_2|\}.
\end{eqnarray}}

From Equation (\ref{eqa14}) it follows the $g_{(v,\,\lambda
v)}(\beta^{\prime \prime})>g_{(v,\,\lambda v)}((0,1))$ for all
$v,$ $|v|$ in $(0, \frac{1}{\|A_1^*\|}]$ and for all $\lambda,$
$\lambda^2<\frac{(|d_2|-1)|d_2|}{|c|^2}.$

We therefore conclude that neither $g_{(v,\,\lambda
v)}(\beta^{\prime})$ nor $g_{(v,\,\lambda v)}(\beta^{\prime
\prime})$ is equal to $\inf_{\beta}g_{(v,\,\lambda v)}(\beta)$ for
any $\lambda$ with $\lambda^2<\frac{(|d_2|-1)|d_2|}{|c|^2}$ and
for any $v$ with $|v|$ in $(0, \frac{1}{\|A_1^*\|}].$

\item Let $1>|d_2|.$ From Equation (\ref{eqa13}) we have
{\small\begin{eqnarray}\label{eqa15} g_{(v,\,\lambda
v)}(\beta^{\prime \prime})-g_{(v,\,\lambda
v)}((1,0))\nonumber&=&(1-|d_2|^2)|\beta^{\prime \prime}_{2}|^2
-|v|^2\lambda^2|c|^2\\&=& (1-|d_2|)-|v|^2\lambda^2|c|^2.
\end{eqnarray}}
From Equation (\ref{eqa15}) it is also easy to see that
$g_{(v,\,\lambda v)}(\beta^{\prime \prime})>g_{(v,\,\lambda
v)}((0,1))$ for all $\lambda,$ $\lambda^2<\frac{(|d_2|-1)}{|c|^2}$
and for all $v,$ $|v|$ in $(0, \frac{1}{\|A_1^*\|}].$

We therefore conclude that neither $g_{(v,\,\lambda
v)}(\beta^{\prime})$ nor $g_{(v,\,\lambda v)}(\beta^{\prime
\prime})$ is equal to $\inf_{\beta}g_{(v,\,\lambda v)}(\beta)$ for
any $\lambda$ with $\lambda^2<\frac{(|d_2|-1)}{|c|^2}$ and for any
$v$ with $|v|$ in $(0, \frac{1}{\|A_1^*\|}].$
\end{itemize}
The roots of $\det(\nu A_{2}^*-A_{1}^*)=0,$ are
$\nu_1=\frac{1}{\mu_1}, \nu_2= \frac{1}{\mu_2}.$ The vectors
$\beta^{\prime}, \beta^{\prime \prime}$ satisfying $(\nu_1A_{2}^*-
A_{1}^*)\beta^{\prime}=0$ and $(\nu_2A_{2}^*-
A_{1}^*)\beta^{\prime \prime}=0$ are
$$\beta^{\prime}=\Big(\frac{\sqrt{|c|}\exp(i\theta_4)}{\sqrt{|c|+|\frac{|c|}{d_2}|}},
\frac{\sqrt{\frac{|c|}{\bar{d_2}}}\exp{i(\theta_4-\phi_3)}}{\sqrt{|c|+|\frac{|c|}{d_2}|}}
\Big)$$ and
$$\beta^{\prime \prime}=\Big(\frac{-\sqrt{|c|}\exp(i\theta_4)}{\sqrt{|c|+|\frac{|c|}{d_2}|}},
\frac{\sqrt{\frac{|c|}{\bar{d_2}}}\exp{i(\theta_4-\phi_3)}}{\sqrt{|c|+|\frac{|c|}{d_2}|}}
\Big)
$$ respectively, where $\frac{|c|}{\bar{d_2}}=\frac{|c|}{|d_2|}\exp{i(\phi_3)}.$ Proceeding as
above, we also find that neither $g_{(v,\,\lambda
v)}(\beta^{\prime})$ nor $g_{(v,\,\lambda v)}(\beta^{\prime
\prime})$ is equal to $\inf_{\beta}g_{(v,\,\lambda v)}(\beta)$ for
any $v$ with $|v|$ in $(0, \frac{1}{\|A_1^*\|}].$

Let $A_1\in \{A_{11},A_{12}\}$ and $A_2=\left (
\begin{smallmatrix}
 0 &  b \\
 c &  0
\end{smallmatrix}\right ),$ where $A_{11}=\left (
\begin{smallmatrix}
1 & 0   \\
0  & d_{2}
\end{smallmatrix}\right ), A_{12}=\left (
\begin{smallmatrix}
d_{1} & 0   \\
0  & 1
\end{smallmatrix}\right ).$ We have proved the theorem for
$A_1=A_{11}$ and $A_2=\left ( \begin{smallmatrix}
 0 &  b \\
 c &  0
\end{smallmatrix}\right ).$ The proof in the remaining case, namely,
$A_1=A_{12}$ and $A_2=\left ( \begin{smallmatrix}
 0 &  b \\
 c &  0
\end{smallmatrix}\right )$ follows similarly.
\end{proof}

%By virtue of following theorem, we can conclude that there exit a
%contractive homomorphism of $\mathcal O(\Omega_{\mathbf A})$ which
%is not completely contractive.
The following theorem gives the existence of a $v$ say $v_0,$ such
that $(v_0,\lambda_0 v_0)$ is in $\mathcal E_{0}.$

\begin{theorem}If $A_1$ is either $ \left(\begin{smallmatrix}
1 & 0   \\
0  & d_{2}
\end{smallmatrix}\right)$ or $\left(\begin{smallmatrix}
d_{1} & 0   \\
0  & 1 \end{smallmatrix}\right)$ and
 $A_2 $ is one of $\left(\begin{smallmatrix}
1 &  b \\
c &  0
\end{smallmatrix}\right), \left(\begin{smallmatrix}
 0 &  b \\
 c &  1
\end{smallmatrix}\right)$ or $\left(\begin{smallmatrix}
 0 &  b \\
 c &  0
\end{smallmatrix}\right)$ \mbox{( $A_1$ and $A_2$ are not
simultaneously diagonalizable)}, then there exist a $v_0$ such
that  $L_V:(\mathbb
C^2,\|\,\cdot\,\|^{*} _{\Omega _{\mathbf A}})\rightarrow (\mathbb
C^2,\|\,\cdot\,\|_2)$ defines a contractive linear map which is
not completely contractive, where $V= \Big ( \begin{smallmatrix} v_0 & 0\\ 0 & \lambda_0 v_0 \end{smallmatrix}\Big )$.
\end{theorem}
\begin{proof}
%The homomorphism $\rho_{V}$ is contractive if and only if
%$$\inf_{\beta}\{1-|v |^2\|A_1^*\beta\|^2 -\lambda^2|v|^2\|A_2^*\beta\|^2
 %+ \lambda^2|v|^4(\|A_1^*\beta\|^2\|A_2^*\beta\|^2 -
%|\left\langle A_1A_2^*\beta, \beta\right\rangle |^2)\}\geq 0.$$
%Also, $\|P_{\mathbf A}(T_1, T_2)\|>1$ implies that
%$$\inf_{\beta}\{1-|v|^2\|A_1^*\beta\|^2
%-\lambda^2|v|^2\|A_2^*\beta\|^2\}<0,$$ where
%$\|\beta\|_2=|\beta_1|^2+|\beta_2|^2=1.$ We assume that
%$\dim\ker(A_2^*-\mu A_1^*)=1.$ Let
%$\mu_1(\beta^{'}),\mu_2(\beta^{''})$ be two roots of
%$\det(A_2^*-\mu(\beta) A_1^*)=0$ and $\beta{'}$ and $\beta^{''}$
%are the corresponding vectors of $\mu_1(\beta^{'}),
%\mu_2(\beta^{''})$ respectively.
As we have seen in Theorems \ref{Thm:main1} and Theorem \ref{Thm:main2},
for all $v,$ $|v|$ in $(0, \frac{1}{\|A_1^*\|})$ there exists a
$\lambda
> 0,$ say $\lambda_0,$ such that $(v, \lambda_0 v)$ is in
$\mathcal E$ with the property:

$g_{(v,\lambda_0v)}(\beta^{\prime\prime})>g_{(v,\lambda_0v)}(\beta^{\prime})>g_{(v,\lambda_0v)}(\beta)$
or
$g_{(v,\lambda_0v)}(\beta^{\prime})>g_{(v,\lambda_0v)}(\beta^{\prime\prime})>g_{(v,\lambda_0v)}(\beta)$
whenever $\beta^{\prime} ,\beta^{\prime\prime}\in \mathcal B.$

Let $\mathbf B$ denote the set $\{|v|^2:\inf_{\beta}
g_{(v,\,\lambda_0 v)}(\beta)\leq 0\}.$ This set is bounded below by
$\frac{1}{\|A_{1}^*\|^2+\lambda_0^2\|A_{2}^*\|^2}.$ Therefore the
infimum of $\mathbf B$ is positive. Let
$$\alpha=\inf_{|v|}\{|v|^2:\inf_{\beta}
g_{(v,\,\lambda_0 v)}(\beta)\leq 0\}.$$ Hence there exists a $v_0$
such that $|v_0|^2=\alpha.$

We claim that $g_{(v_0,\,\lambda_0 v_0)}(\beta)\geq 0$ for all
$\beta$ with $\|\beta\|_{2}=1.$

Assume there exists a $\hat{\beta}$ such that $g_{(v_0,\,\lambda_0
v_0)}(\hat{\beta})<0.$ Then there exists a neighborhood $U$ of
$v_0$ such that $g_{(v,\,\lambda_0 v)}(\hat{\beta})<0$ for all $v
\in U.$ For any $v\in U,$ $\inf_{\beta} g_{(v,\,\lambda_0
v)}(\beta)< 0,$ since the function $g_{(v,\,\lambda_0 v)}$ is
negative at $\hat{\beta}$ for all $v\in U.$ Hence $|v|^2$ is in
$\mathbf B$ for every $v\in U.$ Since $U$ is a neighborhood of
$v_0$ there exists a $v\in U$ such that $|v|^2<|v_0|^2.$ By the
previous assertion, this smaller value of $|v|^2$ also lies in
$\mathbf B,$ which is a contradiction. Also, we have $\inf_{\beta}
g_{(v_0,\,\lambda_0 v_0)}(\beta)\leq 0.$ %Therefore we conclude
%that $g_{(v_0,\,\lambda_0 v)}(\beta)= 0$ for all $\beta$ with
%$\|\beta\|_{2}=1.$
%We claim that $\inf_{\beta} g_{(v_0,\,\lambda_0 v_0)}(\beta)=0.$
%Suppose $\inf_{\beta} g_{(v_0,\,\lambda_0 v_0)}(\beta)>0.$ Then
%$g_{(v_0,\,\lambda_0 v_0)}(\beta)>0$ for all $\beta \in S^{1},$
%where $S^{1}=\{(\beta_1, \beta_2):|\beta_1|^2+|\beta_2|^2=1\}.$
%For each $\beta$ there exists a neighborhood $U_{\beta}\times
%V_{\beta}\subseteq S^{1}\times \mathbb C$ of $(\beta, v_0)$ such
%that $g_{(v_0,\,\lambda_0 v_0)}(\gamma)>0$ for all $(\gamma,
%v_0)\in U_{\beta}\times V_{\beta}.$ Let $\{U_{\beta}:\beta \in
%S^{1}\}$ covers $S^{1}.$ As $S^{1}$ is compact, then there exist $
%U_{\beta_1}, \ldots, U_{\beta_n}$ such that
%$S^{1}=\bigcup_{i=1}^{n}U_{\beta_i}.$ Let
%$\widetilde{V}=\bigcap_{i=1}^{n}V_{\beta_i}.$ Then we have
%$g_{(v,\,\lambda_0 v)}(\gamma)>0$ for all $\gamma \in S^{1}$ and
%$v\in \widetilde{V}$ which implies that
%$\inf_{\beta}g_{(v,\,\lambda_0 v)}(\beta)>0$ for all $v$ in
%$\widetilde{V}.$ This contradicts the fact that $v_0$ is the
%infimum of $\mathbf B.$
Hence $\inf_{\beta} g_{(v_0,\,\lambda_0 v_0)}(\beta)=0.$

From all possible choice for $\lambda_0$, in accordance with
Theorem \ref{Thm:main1} and Theorem \ref{Thm:main2}, we further
restrict it to satisfy $\lambda_0 \leq \frac{1}{|v_0|\|A_2^*\|}.$
This will make  $L_V$ contractive. 
%that is, $$\inf_{\beta}\{1-|\tilde{v} |^2\|A_1^*\beta\|^2
%-\lambda^2_{0}|\tilde{v}|^2\|A_2^*\beta\|^2 +
%\lambda^2_{0}|\tilde{v}|^4(\|A_1^*\beta\|^2\|A_2^*\beta\|^2 -
%|\left\langle A_1A_2^*\beta, \beta\right\rangle |^2)\}= 0$$ which
%is equivalent to $$\inf_{\beta}1-|\tilde{v} |^2\|A_1^*\beta\|^2
%-\lambda^2_{0}|\tilde{v}|^2\|A_2^*\beta\|^2 =-\inf_{\beta}
%\lambda^2_{0}|\tilde{v}|^4(\|A_1^*\beta\|^2\|A_2^*\beta\|^2 -
%|\left\langle A_1A_2^*\beta, \beta\right\rangle |^2).$$  Therefore
%we find a contractive homomorphism which is not completely
%contractive. If $\dim\ker(A_2^*-\mu(\beta) A_1^*)=0,$ then neither
%$g_{\lambda}(\beta^{'})$  nor $g_{\lambda}(\beta^{"})$ is equal to
%$\inf_{\beta}g_{\lambda}(\beta)$ for all $\lambda>0.$ Hence we
%arrive the same conclusion for $\dim\ker(A_2^*-\mu(\beta)
%A_1^*)=0.$
The choice of $\lambda_0, v_0$ ensure that the infimum of $g_{(v_0, \lambda_0 v_0)}(\beta)$ is equal to
neither $g_{(v_0, \lambda_0 v_0)}(\beta^\prime)$ nor $g_{(v_0, \lambda_0 v_0)}(\beta^{\prime\prime)}.$
Thus $L_V$ is not completely contractive.
\end{proof}
\begin{remark}
In chapter $3,$ we have shown the existence of two distinct
operator space on $\mathbf V_{\mathbf A},$ where $\mathbf A$ is of
the form $\left(\left(\begin{smallmatrix}
1 & 0   \\
0  & 0
\end{smallmatrix}\right),\left(\begin{smallmatrix}
1 &  0 \\
c &  0
\end{smallmatrix}\right)\right), \left(\left(\begin{smallmatrix}
0 & 0   \\
0  & 1
\end{smallmatrix}\right),\left(\begin{smallmatrix}
0 &  b \\
0&  1
\end{smallmatrix}\right) \right), \left(\left(\begin{smallmatrix}
1 & 0   \\
0  & 0
\end{smallmatrix}\right),\left(\begin{smallmatrix}
0 &  0 \\
c &  0
\end{smallmatrix}\right)\right)$ or $\left(\left(\begin{smallmatrix}
0 & 0   \\
0  & 1
\end{smallmatrix}\right),\left(\begin{smallmatrix}
0 &  b \\
0 &  0
\end{smallmatrix}\right)\right).$
%Alternative way, we also see
%here the existence of of two distinct operator space on $\mathbf
%V_{\mathbf A}.$ If $\mathbf A$ is of the form
%$\left(\left(\begin{smallmatrix}
%1 & 0   \\
%0  & 0
%\end{smallmatrix}\right),\left(\begin{smallmatrix}
%1 &  0 \\
%c &  0
%\end{smallmatrix}\right)\right), \left(\left(\begin{smallmatrix}
%0 & 0   \\
%0  & 1
%\end{smallmatrix}\right),\left(\begin{smallmatrix}
%0 &  b \\
%0&  1
%\end{smallmatrix}\right) \right), \left(\left(\begin{smallmatrix}
%1 & 0   \\
%0  & 0
%\end{smallmatrix}\right),\left(\begin{smallmatrix}
%0 &  0 \\
%c &  0
%\end{smallmatrix}\right)\right)$ and $\left(\left(\begin{smallmatrix}
%0 & 0   \\
%0  & 1
%\end{smallmatrix}\right),\left(\begin{smallmatrix}
%0 &  b \\
%0 &  0
%\end{smallmatrix}\right)\right),$ then we can see easily that
%$\dim\ker(A_2^*-\mu A_1^*)=2.$ Hence we have $\|\rho_{V}\|\leq 1$
%if and only if $\|\rho^{(2)}(P_{\mathbf A})\|\leq 1.$ Also, if we
%take $\mathbf A^{\rm t},$ then $\dim\ker(A_2^*-\mu A_1^*)=1.$
%Therefore we can construct a contractive homomorphism which is not
%completely contractive. Hence the existence of of two distinct
%operator space on $\mathbf V_{\mathbf A}$ follows from this.
In this case, $\det(A_2^*-\mu A_1^*)\equiv 0$ and $\det (\nu
A_2^*- A_1^*)\equiv 0,$ therefore
\mbox{$\|A_1^*\beta\|^2\|A_2^*\beta\|^2 = |\left\langle
A_1A_2^*\beta,\beta \right \rangle|^2$} for all $\beta$ in
$\mathbb C^2.$ Consequently, we have $\|\rho_{V}\|=
\|\rho_{V}^{(2)}(P_{\mathbf A})\|.$ To obtain a counter example in
this case, following methods of this chapter, one simply use $P_{\mathbf A^{\rm t}}$ instead of $P_{\mathbf A}.$
%\beta\right\rangle |^2=0,$
\end{remark}
\begin{example}
If $A_1=I_2$ and  $A_2=\left(\begin{smallmatrix}
0 & 1   \\
0  & 0 \end{smallmatrix}\right),$ then the homomorphism $\rho_{V}$
is contractive if and only if $|v|^2\leq 1$ and
$$\inf_{\beta}\{1-|v |^2 -\lambda^2|v|^2|\beta_1|^2
 + \lambda^2|v|^4|\beta_1|^4\}\geq 0.$$ Also, $\|P_{\mathbf A}(T_1, T_2)\|\leq 1$ implies that
$$\inf_{\beta}\{1-|v|^2 -\lambda^2|v|^2|\beta_1|^2\}\geq 0.$$ The roots of $\det(A_2^*-\mu A_1^*)=0$
are $\mu_1=\mu_2=0.$ The vectors $\beta^{\prime}, \beta^{\prime
\prime}$ satisfying $( A_{2}^*-\mu_1 A_{1}^*)\beta^{\prime}=0$ and
$( A_{2}^*-\mu_2 A_{1}^*)\beta^{\prime \prime}=0$ are
$\beta^{\prime}= \beta^{\prime \prime}=(0, \exp({i\psi}))$
respectively. Note that
\begin{equation}\label{glambda} g_{(v,\,\lambda v)}(\beta^{\prime \prime})
-g_{(v,\,\lambda
v)}(\beta)=\lambda^2|v|^2|\beta|^2(1-|v|^2|\beta|^2).\end{equation}
From Equation (\ref{glambda}) we have for all $v,$ $|v| \leq 1,$
for all $\lambda,$ there exists a $\beta$ with  $|\beta|<1$ such
that $g_{(v,\,\lambda v)}((0,\exp({i\psi})) )>g_{v,\,\lambda
v)}(\beta).$ Hence there exists $(v, \lambda_0 v)$ in
$\mathcal{E}$  such that neither $g_{(v,\,\lambda_0
v)}(\beta^{\prime})$  nor $g_{(v,\,\lambda_0 v)}(\beta^{\prime
\prime})$ is equal to $\inf_{\beta}g_{(v,\,\lambda_0 v)}(\beta).$

Also, we observe that $\det(\nu A_{2}^*-A_{1}^*) \neq 0.$ Thus
there is no vector $\gamma$ satisfying $(\nu A_{2}^*-
A_{1}^*)\gamma=0.$ We therefore conclude that there exists no
vector $\gamma$ such that  $g_{(v,\,\lambda v)}(\gamma)$ is equal
to $\inf_{\beta}g_{(v,\, \lambda v)}(\beta)$ for any choice of $\lambda$ and $v$ with $|v|\leq 1.$ 

Now, $\inf_{\beta}g_{(v,\,\lambda_0 v)}(\beta) \leq 0$ is
equivalent to $4|v|^2-4+\lambda^2\geq 0$ with $2|v|^2\geq 1.$ If
$\lambda^2=1, |v|^2=\frac{3}{4},$ then $\inf_{\beta}
g_{(v,\,\lambda_0 v)}(\beta)=0.$ Also, we have $\|P_{\mathbf
A}(T_1, T_2)\|>1.$ Hence this contractive homomorphism $\rho_{V}$
is not completely contractive.

%Therefore, the homomorphism $\rho_{V}$ is contractive but not completely contractive as before.
\end{example}

In chapter $1,$ we have seen that if $A_1$ and $A_2$ are
simultaneously diagonalizable, then every contractive linear map
from $(\mathbb C^2, \|\cdot\|_{\Omega_{\mathbf A}})$ to $\mathcal
M_n(\mathbb C)$ is completely contractive. We have also seen that
the particular matrix valued polynomial $P_{\mathbf A}$ plays an
important role for constructing a contractive homomorphisms which
is not complete contractive. The following theorem says that  if
$A_1$ and $A_2$ are simultaneously diagonalizable, then
$\|\rho_{V}\|\leq 1$ if and only if $\|\rho_{V}^{(2)}(P_{\mathbf
A})\|\leq 1.$
\begin{theorem}\label{dia}
 If $A_1= \left(\begin{smallmatrix}
d_{1} & 0   \\
0  & d_{2}
\end{smallmatrix}\right)$ and
 $A_2 =\left(\begin{smallmatrix}
a &  0 \\
0 &  d
\end{smallmatrix}\right) ,$ then there exists a $(v, \lambda v)$ in $\mathcal{E}$ such that the infimum is
attained at either  $\beta^{\prime} $ or $\beta^{\prime \prime}.$
\end{theorem}

\begin{proof}
As above we have seen that the homomorphism $\rho_{V}$ is
contractive if and only if $|v|^2\leq \frac{1}{\|A_1^{*}\|^2}$ and
$\inf_{\beta, \|\beta\|_2=1}g_{(v,\, \lambda v)}(\beta)\geq 0 .$
Observe that {\small
\begin{align}\label{comcontraction} \inf_{\beta}g_{(v,\,\lambda v)}(\beta)\nonumber
&=\inf_{\beta}(1-|d_{1}|^2|v|^2-|a|^2\lambda^2|v|^2)|\beta_1|^2+
(1-|d_2v|^2-|d|^2\lambda^2|v|^2)|\beta_2|^2\\&+|v|^4\lambda^2|\bar{ad_{1}}-\bar{dd_2}|^2|\beta_1\beta_2|^2
\end{align}}
 where $|\beta_1|^2+|\beta_2|^2=1.$

The roots of $\det(A_2^*-\mu A_1^*)=0$ are
$\mu_1={\frac{\bar{c}}{\bar{d}_{1}}},
\mu_2={\frac{\bar{d}}{\bar{d}_{2}}}.$ The vectors $\beta^{\prime},
\beta^{\prime \prime}$ satisfying $(\nu_1A_{2}^*-
A_{1}^*)\beta^{\prime}=0,(\nu_2A_{2}^*- A_{1}^*)\beta^{\prime
\prime}=0$ are $\beta^{\prime}=(1,0), \beta^{\prime \prime}=(0,1)$
respectively. In this case, $\beta^{\prime}=\beta^{\prime
\prime}_{\bot}, \beta^{\prime \prime}=\beta^{\prime}_\bot$, where
$\beta^{\prime \prime}_{\bot}, \beta^{\prime}_\bot$ are orthogonal
to $\beta^{\prime \prime}, \beta^{\prime}$ respectively.
Substituting $\beta^{\prime}=(1,0)$ and $ \beta^{\prime
\prime}=(0,1)$ in Equation (\ref{comcontraction}) we have
$$g_{(v,\,\lambda v)}(\beta^{\prime})=(1-|d_{1}|^2|v|^2-|a|^2\lambda^2|v|^2)$$ and
$$g_{(v,\,\lambda v)}(\beta^{\prime \prime})
 =(1-|d_2v|^2-|d|^2\lambda^2|v|^2).$$
 Without loss generality we can assume that $g_{(v,\,\lambda
v)}(\beta^{\prime})\leq g_{(v,\,\lambda v)}(\beta^{\prime
\prime}).$ Note that  {\small\begin{align}\label{diagonal}
 g_{(v,\,\lambda v)}(\beta^{\prime })- g_{(v,\,\lambda v)}(\beta)&=(g_{(v,\,\lambda v)}(\beta^{\prime })
 - g_{(v,\,\lambda v)}(\beta^{\prime
 \prime}))|\beta_2|^2-|v|^4\lambda^2|\bar{ad_{1}}
 -\bar{dd_2}|^2|\beta_1\beta_2|^2.\end{align}}

Since $g_{(v,\,\lambda v)}(\beta^{\prime})\leq g_{(v,\,\lambda
v)}(\beta^{\prime \prime}),$ from Equation \eqref{diagonal} we
observe that $g_{(v,\,\lambda v)}(\beta^{\prime })\leq
g_{(v,\,\lambda v)}(\beta)$ for all $\beta.$ Hence we conclude
that the infimum is attained at $\beta^{\prime}.$ Similarly we
also prove that the infimum is attained at $\beta^{\prime
\prime}.$

The roots of $\det(\nu A_{2}^*-A_{1}^*)=0$ are $\frac{1}{\mu_1},
\frac{1}{\mu_2}.$ The vectors $\beta^{\prime}, \beta^{\prime
\prime}$ satisfying $(\nu_1A_{2}^*- A_{1}^*)\beta^{\prime}=0$ and
$(\nu_2A_{2}^*- A_{1}^*)\beta^{\prime \prime}=0$ are
$\beta^{\prime}=(1,0)$ and $\beta^{\prime \prime}=(0,1)$
respectively. We also therefore conclude that infimum is attained
at either $\beta^{\prime} $ or $\beta^{\prime \prime}.$ This
completes the proof.
\end{proof}
The following Corollary is an immediate consequence of Theorem
\ref{dia} and Corollary \ref{corolll}.
\begin{corollary}Suppose $A_1$ and $A_2$ are simultaneously diagonalizable. Then $\|\rho_{V}\|\leq 1$ if
and only if $\|\rho_{V}^{(2)}(P_{\mathbf A})\|\leq 1.$
\end{corollary}

%\thispagestyle{empty} \cleardoublepage
%\addcontentsline{toc}{chapter}{Bibliography}
%\bibliographystyle{amsplain}
%\bibliography{ref}

\end{document}